\documentclass[reqno]{amsart}

\usepackage[mathscr]{euscript}
\usepackage{graphicx}
\usepackage[shortlabels]{enumitem}
\usepackage{mathtools}
\usepackage{booktabs}
\usepackage[dvipsnames]{xcolor}
\usepackage{tikz}
\usepackage{tikz-cd}
\usepackage{tikz-3dplot}
\usetikzlibrary{decorations.pathreplacing}
\usetikzlibrary{decorations.markings}
\usetikzlibrary{positioning,cd,arrows, patterns}
\usetikzlibrary{arrows.meta}
\usepackage{subfigure}
\usepackage{ragged2e}
\usepackage{a4wide}
\usepackage{hyperref}
\hypersetup{colorlinks}
\usepackage{amsbsy}
\usepackage{xargs}
\usepackage{multirow}
\usepackage{array}
\usepackage{float}

\hypersetup{
    unicode=false,          
    pdftoolbar=true,        
    pdfmenubar=true,        
    pdffitwindow=false,     
    pdfnewwindow=true,      
    colorlinks=true,       
    linkcolor=blue,          
    citecolor=red,        
    filecolor=violet,      
    urlcolor=violet       
}
\usepackage{caption} 
\theoremstyle{plain}
\newtheorem{theorem}{Theorem}[section]
\newtheorem{lemma}[theorem]{Lemma}
\newtheorem{proposition}[theorem]{Proposition}
\newtheorem{corollary}[theorem]{Corollary}

\theoremstyle{remark}
\newtheorem{remark}[theorem]{Remark}

\theoremstyle{definition}
\newtheorem{definition}[theorem]{Definition}
\newtheorem{assumption}[theorem]{Assumption}
\newtheorem{example}[theorem]{Example}

\DeclareMathOperator{\mmon}{\mu mon}
\DeclareMathOperator{\cl}{Cl}
\DeclareMathOperator{\Sh}{\mathscr{S}\mathsf{h}}
\DeclareMathOperator{\Loc}{\mathscr{L}\mathsf{oc}}

\newcommand\bbA{\mathbb{A}}
\newcommand\R{\mathbb{R}}
\newcommand\C{\mathbb{C}}
\newcommand\D{\mathbb{D}}
\newcommand\Z{\mathbb{Z}}
\newcommand\N{\mathbb{N}}
\newcommand\bS{\mathbb{S}}

\newcommand\bfa{\mathbf{a}}
\newcommand\bfb{\mathbf{b}}
\newcommand\bfc{\mathbf{c}}
\newcommand\bfd{\mathbf{d}}
\newcommand\bfe{\mathbf{e}}
\newcommand\bfs{\mathbf{s}}
\newcommand\bfx{\mathbf{x}}
\newcommand\bfy{\mathbf{y}}
\newcommand\bfz{\mathbf{z}}

\newcommand\cA{\mathcal{A}}
\newcommand\cC{\mathcal{C}}
\newcommand\cF{\mathcal{F}}
\newcommand\cL{\mathcal{L}}
\newcommand\cM{\mathcal{M}}
\newcommand\cX{\mathcal{X}}

\newcommand\sfI{\mathsf{I}}
\newcommand\sfT{\mathsf{T}}
\newcommand\sfY{\mathsf{Y}}

\newcommand{\clusterfont}{\mathcal}
\newcommand{\dynkinfont}{\mathsf}
\newcommand{\ngraphfont}{\mathscr}

\newcommand{\quiver}{\clusterfont{Q}}
\newcommand{\qbasis}{\clusterfont{B}}
\newcommand{\qcoxeter}{\mutation_{\quiver}}
\DeclareMathOperator{\exchange}{Ex}
\DeclareMathOperator{\pr}{pr}
\newcommand{\qbpr}{\clusterfont{B}^{\pr}}
\newcommand{\exchangesub}[2]{\exchange({#1},{#2})}

\newcommand{\ngraph}{\ngraphfont{G}}
\newcommand{\nbasis}{\ngraphfont{B}}
\newcommand{\ncoxeter}{\mutation_\ngraph}
\newcommand{\coxeterpadding}{\ngraphfont{C}}
\newcommand{\nannulus}{\ngraphfont{A}}

\newcommand{\dynA}{\dynkinfont{A}}
\newcommand{\dynB}{\dynkinfont{B}}
\newcommand{\dynC}{\dynkinfont{C}}
\newcommand{\dynD}{\dynkinfont{D}}
\newcommand{\dynE}{\dynkinfont{E}}
\newcommand{\dynF}{\dynkinfont{F}}
\newcommand{\dynG}{\dynkinfont{G}}
\newcommand{\dynX}{\dynkinfont{X}}
\newcommand{\dynY}{\dynkinfont{Y}}

\newcommand{\dynADE}{\dynkinfont{ADE}}
\newcommand{\dynBCFG}{\dynkinfont{BCFG}}

\newcommand{\exdynA}{\widetilde{\dynA}}
\newcommand{\exdynB}{\widetilde{\dynB}}
\newcommand{\exdynC}{\widetilde{\dynC}}
\newcommand{\exdynD}{\widetilde{\dynD}}
\newcommand{\exdynE}{\widetilde{\dynE}}
\newcommand{\exdynF}{\widetilde{\dynF}}
\newcommand{\exdynG}{\widetilde{\dynG}}
\newcommand{\exdynX}{\widetilde{\dynX}}

\newcommand{\facet}{\mathscr{F}}
\newcommand{\seed}{\Sigma}
\newcommand{\initialseed}{\Sigma_{t_0}}
\newcommand{\flags}{\mathcal{F}_\legendrian}
\newcommand{\mutation}{\mu}
\DeclareMathOperator{\Fan}{Fan}
\DeclareMathOperator{\re}{re}

\newcommand{\annulus}{\mathbb{A}}
\newcommand{\disk}{\mathbb{D}}
\newcommand{\sphere}{\mathbb{S}}

\newcommand{\legendrian}{\lambda}
\newcommand{\Legendrian}{\Lambda}

\newcommand{\boundary}{\partial}
\newcommand{\cycle}{\gamma}

\newcommand{\field}{\mathbb{F}}

\newcommand{\Roots}{\Phi}
\newcommand{\SRoots}{\Pi}

\newcommand{\Move}[1]{{\rm{(#1)}}}

\newcommand{\wavefront}{\Gamma}

\newcommand{\boundellipse}[3]
{(#1) ellipse (#2 and #3)}

\newcommand{\1}{\sigma_1}
\newcommand{\2}{\sigma_2}
\newcommand{\3}{\sigma_3}
\newcommand{\4}{\sigma_4}

\newcommand{\reducedto}{\succ}

\numberwithin{equation}{section}

\tikzset{ynode/.style = {circle, fill = yellow, inner sep = 2pt, opacity = 0.5}}
\tikzset{gnode/.style = {circle, fill=green, inner sep = 2pt, opacity = 0.5}}
\tikzstyle{Dnode}=[draw, circle, inner sep = 0.07cm]
\tikzstyle{double line} = [
double distance = 1.5pt, 
double=\pgfkeysvalueof{/tikz/commutative diagrams/background color}
]
\tikzstyle{triple line} = [
double distance = 2pt, 
double=\pgfkeysvalueof{/tikz/commutative diagrams/background color}
]

\newcommand*\circled[1]{\tikz[baseline=(char.base)]{
\node[shape=circle,draw,inner sep=0.8pt] (char) {#1};}}

\newcommand{\todo}[2][1]{\vspace{1 mm}\par \noindent
\marginpar{\textsc{To Do}}
\framebox{\begin{minipage}[c]{#1 \textwidth} \tt
#2 \end{minipage}}\vspace{1 mm}\par}

\newcommand{\aha}[2][0.8]{\vspace{1 mm}\par \noindent
\framebox{\begin{minipage}[c]{#1 \textwidth}  #2
\end{minipage}}\vspace{1 mm}\par}
\newcommand{\fix}[1]{\marginpar{\textsc{Fix!}}\framebox{#1}}

\makeatletter 
\tikzset{curlybrace/.style={rounded corners=2pt,line cap=round}}%
\pgfkeys{
/curlybrace/.cd,%
tip angle/.code     =  \def\cb@angle{#1},
/curlybrace/.unknown/.code ={\let\searchname=\pgfkeyscurrentname
                              \pgfkeysalso{\searchname/.try=#1,
                              /tikz/\searchname/.retry=#1}}}  
\def\curlybrace{\pgfutil@ifnextchar[{\curly@brace}{\curly@brace[]}}%

\def\curly@brace[#1]#2#3#4{%
\pgfkeys{/curlybrace/.cd,
tip angle = 0.75}%
\pgfqkeys{/curlybrace}{#1}%
\ifnum 1>#4 \def\cbrd{0.05} \else \def\cbrd{0.075} \fi
\draw[/curlybrace/.cd,curlybrace,#1]  (#2:#4-\cbrd) -- (#2:#4) arc (#2:{(#2+#3)/2-\cb@angle}:#4) --({(#2+#3)/2}:#4+\cbrd) coordinate (curlybracetipn);
\draw[/curlybrace/.cd,curlybrace,#1] ({(#2+#3)/2}:#4+\cbrd) -- ({(#2+#3)/2+\cb@angle}:#4) arc ({(#2+#3)/2+\cb@angle} :#3:#4) --(#3:#4-\cbrd)
}
\makeatother

\newlength{\myline}
\newcommandx*{\triplearrow}[4][1=0, 2=1]{
  \draw[double distance=3\myline,#3] #4;
  \draw[shorten <=#1\myline,shorten >=#2\myline,#3] #4
}

\title{Lagrangian fillings for Legendrian links of affine type}

\author{Byung Hee An}
\email{anbyhee@knu.ac.kr}
\address{Department of Mathematics Education, Kyungpook National University, Republic of Korea}

\author{Youngjin Bae}
\email{yjbae@inu.ac.kr}
\address{Department of Mathematics, Incheon National University, Republic of Korea}

\author{Eunjeong Lee}
\email{eunjeong.lee@ibs.re.kr}
\address{Center for Geometry and Physics, Institute for Basic Science (IBS), Pohang 37673, Republic of Korea}

\begin{document}

\begin{abstract}
We prove that there are at least as many exact embedded Lagrangian fillings 
as seeds for Legendrian links of affine type $\exdynD \exdynE$. We also provide 
as many Lagrangian fillings with certain symmetries as seeds of type $\exdynB_n$, $\exdynF_4$, $\exdynG_2$, and $\dynE_6^{(2)}$.
These families are the first known Legendrian links with infinitely many fillings that exhaust all seeds in the corresponding cluster structures.
Furthermore, we show that Legendrian realization of Coxeter mutation of type $\exdynD$ corresponds to the Legendrian loop considered by Casals and Ng.
\end{abstract}
\maketitle

\tableofcontents

\section{Introduction}
\subsection{Background}

Interaction between symplectic geometry and cluster algebra has become increasingly fruitful. The study of Lagrangian fillings for Legendrian links is the one of supporting areas in symplectic geometry. Many interesting connections between these two fields are revealed and strengthened as follows:

In \cite{STWZ2019}, exact Lagrangian fillings are represented by alternating diagrams of Legendrian links, and Lagrangian surgeries are realized by square moves of the diagram which correspond to quiver mutations. In addition, the boundary measurement map assigns each alternating diagram a toric chart in the moduli space of constructible sheaves adapted to Legendrian links which we think of cluster variety as well as the space of exact Lagrangian fillings.

By the series of works in \cite{SW2019, GSW2020a, GSW2020b}, they argue that the moduli space of constructible sheaves adapted to Legendrian links of positive braid closure, admits a structure of cluster algebra. Moreover, they construct infinitely many Lagrangian fillings by considering the effect of Donaldson--Thomas transformation on the cluster variety.

In the work of \cite{TZ2018, CZ2020}, the authors introduce $N$-graphs to describe exact Lagrangian fillings in a systematic and combinatorial way. They also develop Legendrian mutations which realize Lagrangian surgeries in the geometric side, and show that its induced operation in the algebraic side coincides with the cluster mutation.

Our previous work \cite{ABL2021} mainly use $N$-graphs and Legendrian mutations 
to produce distinct Lagrangian fillings. We focus on the \emph{Coxeter 
mutations} in order to see that there is no obstruction  to realize Legendrian 
mutations in $N$-graphs. As a result, we show that there are at least as 
many exact Lagrangian fillings as seeds for the Legendrian links which admits cluster 
algebra of finite type. 

On the other hand, there is a parallel strategy to study Lagrangian fillings, the Legendrian contact differential graded algebra. By the functoriality of Legendrian DGA under exact Lagrangian cobordism \cite{EHK2016}, each Lagrangian filling gives an augmentation of the DGA. Moreover, a loop of Legendrians defines an automorphism of the DGA, and it has been used to find distinct Lagrangian fillings \cite{Kal2006,CN2021}. 

\subsection{The results}
The main result is to construct as many exact embedded Lagrangian fillings as seeds for Legendrian links of affine type $\exdynD \exdynE$. We mainly use 
$N$-graphs and their Legendrian mutations to produce distinct Lagrangian 
fillings. An $N$-graph on $\D^2$ represent Legendrian surface in $J^1\D^2$ 
whose Lagrangian projection gives an exact Lagrangian surface bounding a 
Legendrian link in $J^1\partial \D^2$.
We provide the Legendrian links of type $\exdynD \exdynE$ as follows:

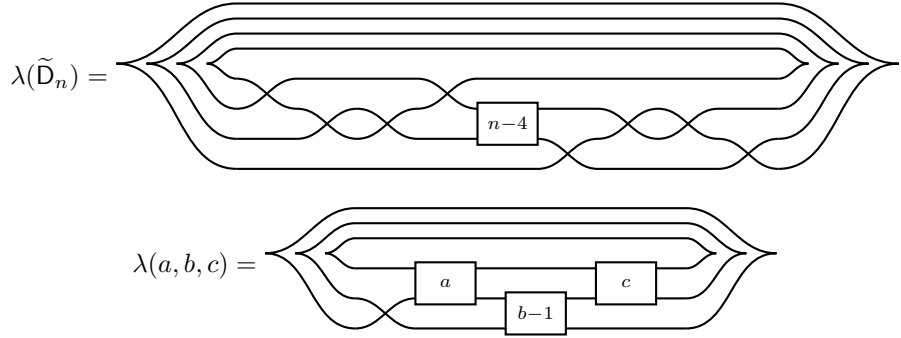
\begin{align*}
\legendrian({{\exdynD}_n})=
\begin{tikzpicture}[baseline=-0.5ex,scale=0.8]
\draw[thick] (0,0) to[out=0,in=180] (1,-0.5) to[out=0,in=180] (2,-1) to[out=0,in=180] (3,-0.5) to[out=0,in=180] (4,0) to[out=0,in=180] (9,0);
\draw[thick] (0,-0.5) to[out=0,in=180] (1,0) to[out=0,in=180] (3,0) to[out=0,in=180] (4,-0.5) (5,-0.5) to[out=0,in=180] (6,-0.5) to[out=0,in=180] (7,-1) to[out=0,in=180] (8,-0.5) to[out=0,in=180] (9,-0.5);
\draw[thick] (0,-1) to[out=0,in=180] (1,-1) to[out=0,in=180] (2,-0.5) to[out=0,in=180] (3,-1) to[out=0,in=180] (4,-1) (5,-1) to[out=0,in=180] (6,-1.5) to[out=0,in=180] (8,-1.5) to[out=0,in=180] (9,-1);
\draw[thick] (0,-1.5) to[out=0,in=180] (5,-1.5) to[out=0,in=180] (6,-1) to[out=0,in=180] (7,-0.5) to[out=0,in=180] (8,-1) to[out=0,in=180] (9,-1.5);
\draw[thick] (4,-0.4) rectangle node {$\scriptstyle{n-4}$} (5, -1.1);
\draw[thick] (0,0) to[out=180,in=0] (-0.5,0.25) to[out=0,in=180] (0,0.5) to[out=0,in=180] (9,0.5) to[out=0,in=180] (9.5,0.25) to[out=180,in=0] (9,0);
\draw[thick] (0,-0.5) to[out=180,in=0] (-1,0.25) to[out=0,in=180] (0,0.75) to[out=0,in=180] (9,0.75) to[out=0,in=180] (10,0.25) to[out=180,in=0] (9,-0.5);
\draw[thick] (0,-1) to[out=180,in=0] (-1.5,0.25) to[out=0,in=180] (0,1) to[out=0,in=180] (9,1) to[out=0,in=180] (10.5,0.25) to[out=180,in=0] (9,-1);
\draw[thick] (0,-1.5) to[out=180,in=0] (-2,0.25) to[out=0,in=180] (0,1.25) to[out=0,in=180] (9,1.25) to[out=0,in=180] (11,0.25) to[out=180,in=0] (9,-1.5);
\end{tikzpicture}
\end{align*}

\begin{align*}
\legendrian(a,b,c)=
\begin{tikzpicture}[baseline=5ex,scale=0.8]
\draw[thick] (0,0) to[out=0,in=180] (1,0.5) (2,0.5) to (2.5,0.5) (3.5,0.5) to (4,0.5) (5,0.5) to (5.5,0.5);
\draw[thick] (0,0.5) to[out=0,in=180] (1,0) to (2.5,0) (3.5,0) to (5.5,0);
\draw[thick] (0,1) to[out=0,in=180] (1,1) (2,1) to (4,1) (5,1) to (5.5,1);
\draw[thick] (1,0.4) rectangle node {$\scriptstyle a$} (2, 1.1);
\draw[thick] (2.5,-0.1) rectangle node {$\scriptstyle{b-1}$} (3.5, 0.6);
\draw[thick] (4,0.4) rectangle node {$\scriptstyle c$} (5, 1.1);
\draw[thick] (0,1) to[out=180, in=0] (-0.5,1.25) to[out=0,in=180] (0,1.5) to (5.5,1.5) to[out=0,in=180] (6,1.25) to[out=180,in=0] (5.5,1);
\draw[thick] (0,0.5) to[out=180, in=0] (-1,1.25) to[out=0,in=180] (0,1.75) to (5.5,1.75) to[out=0,in=180] (6.5,1.25) to[out=180,in=0] (5.5,0.5);
\draw[thick] (0,0) to[out=180, in=0] (-1.5,1.25) to[out=0,in=180] (0,2) to (5.5,2) to[out=0,in=180] (7,1.25) to[out=180,in=0] (5.5,0);
\end{tikzpicture}
\end{align*}
Here, $\legendrian({{\exdynE}_6})=\legendrian(3,3,3)$, $\legendrian({{\exdynE}_7})=\legendrian(2,4,4)$, and $\legendrian({{\exdynE}_8})=\legendrian(2,3,6)$, which come from the triples $(a,b,c)$ satisfying ${\frac1 a} +{ \frac1 b} + {\frac1 c} =1$.

Note that the above Legendrians are the rainbow closure of \emph{positive 
braids}. By the work of Shen--Weng \cite{SW2019}, it is direct to check that 
the corresponding cluster structure of Legendrian $\legendrian(\dynX)$ is 
indeed of type $\dynX$ for $\dynX=\exdynD$ or $\exdynE$. More precisely, the coordinate 
ring of the moduli space $\cM_1(\legendrian(\dynX))$ of microlocal rank one 
sheaves in $\Sh^\bullet_{\legendrian(\dynX)}(\R^2)$ admits the aforementioned 
cluster structure.
By the way, the (candidate) Legendrians of type $\exdynA$ are not the rainbow 
closure of positive braids, in general. 
Indeed, Casals--Ng~\cite{CN2021} considered a Legendrian link of type 
$\exdynA_{1,1}$ which is not the rainbow closure of a  positive braid. 
So we can not directly apply the subsequent argument to Legendrians of type 
$\exdynA$.

By applying a sequence of Reidemeister moves to the above Legendrian link 
$\legendrian(\dynX)$, we have the $(N-1)$-colored points in $S^1$ which 
represent a Legendrian braid in $J^1 S^1$.
Now we consider the $N$-graph $\ngraph(\dynX)$ depicted in 
Figures~\ref{ngraph_nbasis_D} and~\ref{ngraph_nbasis_E} extending the boundary 
data $\legendrian(\dynX)$ with decorated edges $\nbasis(\dynX)$ to indicate an 
exact Lagrangian filling of the starting Legendrian link together with a tuple 
of one-cycles on that fillings.

\begin{figure}[ht]
\subfigure[$(\ngraph(\exdynD_{2\ell+4}),\nbasis(\exdynD_{2\ell+4}))$\label{ngraph_nbasis_D}]{
$
\begin{tikzpicture}[baseline=-.5ex,scale=0.5]
\useasboundingbox (-7,-4.25) rectangle (7,4.25);
\draw [decorate,decoration={brace,amplitude=10pt}]
(-2.2,3.1) -- (3.2,3.1) node [black,midway,yshift=0.5cm] {$\ell$} ;
\begin{scope}[xshift=-4cm]
\draw[yellow, line width=5, opacity=0.5] (-1,1)--(0,0) (-1,-1)--(0,0) (1,0)--(0,0);
\draw[green, line width=5, opacity=0.5] (-1,1)--(-2,1) (-1,-1)--(-1,-2) (1,0)--(2,0)
(3,0)--(3.5,0) (4.5,0)--(5,0) (6,0)--(7,0) (9,1)--(9,2) (9,-1)--(10,-1);
\draw[blue,thick, rounded corners] (-1.75, 3) -- (-1.5, 2.5) -- (-1.25, 3);
\draw[green, thick] (0,3) -- (0,-3) (0,0) -- (-3,0);
\draw[red,thick, fill] (0,0) -- (-1,1) circle (2pt) -- +(0,2) (-1,1) -- ++(-1,0) circle (2pt) -- +(-1,0) (-2,1) -- +(0,2);
\draw[red,thick, fill] (0,0) -- (-1,-1) circle (2pt) -- +(-2,0) (-1,-1) -- ++(0,-1) circle (2pt) -- +(0,-1) (-1,-2) -- +(-2,0);
\draw[red,thick] (0,0) -- (1,0);
\draw[thick, fill=white] (0,0) circle (2pt);
\draw[thick](1, 3) -- (-2,3) to[out=180,in=90] (-3,2) --(-3,-2) to[out=-90,in=180] (-2, -3) -- (10,-3) to[out=0,in=-90] (11,-2) -- (11,2) to[out=90,in=0] (10,3)-- (1,3);
\end{scope}
\begin{scope}[xshift=4cm,rotate=180]
\draw[yellow, line width=5, opacity=0.5] (-1,1)--(0,0) (-1,-1)--(0,0) (1,0)--(0,0);
\draw[blue, thick] (0,3) -- (0,-3) (0,0) -- (-3,0);
\draw[red,thick, fill] (0,0) -- (-1,1) circle (2pt) -- +(0,2) (-1,1) -- ++(-1,0) circle (2pt) -- +(-1,0) (-2,1) -- +(0,2);
\draw[red,thick, fill] (0,0) -- (-1,-1) circle (2pt) -- +(-2,0) (-1,-1) -- ++(0,-1) circle (2pt) -- +(0,-1) (-1,-2) -- +(-2,0);
\draw[red,thick] (0,0) -- (1,0);
\draw[thick, fill=white] (0,0) circle (2pt);
\end{scope}
\begin{scope}
\draw[yellow, line width=5, opacity=0.5] (-2,0) -- (-1,0) (1,0) -- (2,0);
\draw[red, thick, fill] (-3,0) circle (2pt) -- (-3, -3) (-3,0) -- (-2,0) circle (2pt) -- (-2,3) (-2,0) -- (-1,0) circle (2pt) -- (-1,-3) (-1,0) -- (-0.5,0) (0.5,0) -- (1,0) circle (2pt) -- (1,3) (1,0) -- (2, 0) circle (2pt) -- (2,-3) (2,0) -- (3,0) circle (2pt) -- (3,3);
\draw[red, thick, dashed] (-0.5,0) -- (0.5,0);
\end{scope}
\end{tikzpicture}
$
}
\subfigure[$(\ngraph(a,b,c),\nbasis(a,b,c))$\label{ngraph_nbasis_E}]{
$
\begin{tikzpicture}[baseline=-.5ex,scale=0.6]
\draw[thick] (0,0) circle (3cm);
\draw[yellow, line cap=round, line width=5, opacity=0.5] (60:1) -- (50:1.5) 
(70:1.75) -- (50:2) (180:1) -- (170:1.5) (190:1.75) -- (170:2) (300:1) -- 
(290:1.5) (310:1.75) -- (290:2);
\draw[green, line cap=round, line width=5, opacity=0.5] (0,0) -- (60:1) (0,0) -- (180:1) (0,0) -- (300:1) (50:1.5) -- (70:1.75) (170:1.5) -- (190:1.75) (290:1.5) -- (310:1.75);
\draw[red, thick] (0,0) -- (0:3) (0,0) -- (120:3) (0,0) -- (240:3);
\draw[blue, thick, fill] (0,0) -- (60:1) circle (2pt) -- (100:3) (60:1) -- (50:1.5) circle (2pt) -- (20:3) (50:1.5) -- (70:1.75) circle (2pt) -- (80:3) (70:1.75) -- (50:2) circle (2pt) -- (40:3);
\draw[blue, thick, dashed] (50:2) -- (60:3);
\draw[blue, thick, fill] (0,0) -- (180:1) circle (2pt) -- (220:3) (180:1) -- (170:1.5) circle (2pt) -- (140:3) (170:1.5) -- (190:1.75) circle (2pt) -- (200:3) (190:1.75) -- (170:2) circle (2pt) -- (160:3);
\draw[blue, thick, dashed] (170:2) -- (180:3);
\draw[blue, thick, fill] (0,0) -- (300:1) circle (2pt) -- (340:3) (300:1) -- (290:1.5) circle (2pt) -- (260:3) (290:1.5) -- (310:1.75) circle (2pt) -- (320:3) (310:1.75) -- (290:2) circle (2pt) -- (280:3);
\draw[blue, thick, dashed] (290:2) -- (300:3);
\draw[thick, fill=white] (0,0) circle (2pt);
\curlybrace[]{10}{110}{3.2};
\draw (60:3.5) node[rotate=-30] {\small ${ a+1}$};
\curlybrace[]{130}{230}{3.2};
\draw (180:3.5) node[rotate=90] {\small $b+1$};
\curlybrace[]{250}{350}{3.2};
\draw (300:3.5) node[rotate=30] {\small $c+1$};
\end{tikzpicture}
$
}\\
\subfigure[$\mutation_\ngraph(\ngraph(\exdynD_{2\ell+4}),\nbasis(\exdynD_{2\ell+4}))$
\label{coxeter_mutation_D}]{
$
\begin{tikzpicture}[baseline=-.5ex,scale=0.47]
\useasboundingbox (-3,-4) rectangle (15,4);
\draw[fill=black, opacity=0.1](3,3) -- (-1,3) to[out=180, in=90] (-3,1) to (-3,-1) to[out=-90,in=180] (-1,-3) to (13,-3) to[out=0, in=-90] (15,-1) to (15,1) to[out=90,in=0] (13,3) to (3,3) ;
\draw[fill=white,dashed, rounded corners]
(3,1.75) -- (0.25,1.75) -- (0.25, -1.75) -- (11.75,-1.75) -- (11.75,1.75) -- (3,1.75)
;
\draw[blue,thick, rounded corners]
(1/3,3)-- (4/3,2) --(1,4/3)-- (2/3,2) -- (-1/3,3)
;
\begin{scope}[yscale=-1]
\draw[green, line width=5, opacity=0.5, line cap=round] 
(1,0.5) -- (1,1.5)
(0.5,-1) --(1.5,-1)
;
\draw[yellow, line width=5, opacity=0.5] 
(1,0.5) -- (2,0) --(1.5,-1)
(2,0) -- (3,0)
;
\draw[green, thick] 
(2,3) -- (2,-3)
(2,2) -- (0,2) -- (-2,0) -- (-3,0)
(2,-2) -- (0,-2) -- (-2,0)
(0,2) -- (0,-2)
(0,0) -- (2,0)
;
\draw[red, thick] 
(-3,1) -- (-2,0) to[out=15,in=-105] (0,2) (0,-2) -- (-1,-3)
(0,2) -- (1,1.5) -- (2,2)
(1,3) -- (2,2)
(1,1.5) -- (1,0.5) -- (2,0) -- (3,0)
(1,0.5) -- (0,0) -- (0.5,-1) -- (1.5,-1) -- (2,0)
(0,0) to[out=-120,in=120] (0,-2) (0,2) -- (-1,3)
(0,-2) -- (0.5,-1)
(1.5,-1) -- (2,-2)
(2,-2) -- (1,-3)
(-3,-1) -- (-2,0)
;
\draw[thick, fill=white] 
(-2,0) circle (2pt) 
(0,2) circle (2pt) 
(0,-2) circle (2pt) 
(0,0) circle (2pt) 
(2,2) circle (2pt) 
(2,0) circle (2pt) 
(2,-2) circle (2pt);
\draw[red, thick, fill]
(1,0.5) circle (2pt)
(1,1.5) circle (2pt)
(0.5,-1) circle (2pt)
(1.5,-1) circle (2pt)
;
\draw[thick]
(3,3) -- (-1,3) to[out=180, in=90] (-3,1) to (-3,-1) to[out=-90,in=180] (-1,-3) to (13,-3) to[out=0, in=-90] (15,-1) to (15,1) to[out=90,in=0] (13,3) to (3,3) 
;
\end{scope}
\begin{scope}[xshift=12cm,rotate=180,yscale=-1]
\draw[green, line width=5, opacity=0.5, line cap=round] 
(1,0.5) -- (1,1.5)
(0.5,-1) --(1.5,-1)
;
\draw[yellow, line width=5, opacity=0.5] 
(1,0.5) -- (2,0) --(1.5,-1)
(2,0) -- (3,0)
;
\draw[blue, thick] 
(2,3) -- (2,-3)
(2,2) -- (0,2) -- (-2,0) -- (-3,0)
(2,-2) -- (0,-2) -- (-2,0)
(0,2) -- (0,-2)
(0,0) -- (2,0)
;
\draw[red, thick] 
(-3,1) -- (-2,0) to[out=15,in=-105] (0,2) (0,-2) -- (-1,-3)
(0,2) -- (1,1.5) -- (2,2)
(1,3) -- (2,2)
(1,1.5) -- (1,0.5) -- (2,0) -- (3,0)
(1,0.5) -- (0,0) -- (0.5,-1) -- (1.5,-1) -- (2,0)
(0,0) to[out=-120,in=120] (0,-2) (0,2) -- (-1,3)
(0,-2) -- (0.5,-1)
(1.5,-1) -- (2,-2)
(2,-2) -- (1,-3)
(-3,-1) -- (-2,0)
;
\draw[thick, fill=white] 
(-2,0) circle (2pt) 
(0,2) circle (2pt) 
(0,-2) circle (2pt) 
(0,0) circle (2pt) 
(2,2) circle (2pt) 
(2,0) circle (2pt) 
(2,-2) circle (2pt);
\draw[red, thick, fill]
(1,0.5) circle (2pt)
(1,1.5) circle (2pt)
(0.5,-1) circle (2pt)
(1.5,-1) circle (2pt)
;
\end{scope}
\begin{scope}[xshift=6cm]
\draw[green, line width=5, opacity=0.5] (-3,0) -- (-2,0) (-1,0) -- (-0.5, 0);
\draw[yellow, line width=5, opacity=0.5] (-2,0) -- (-1,0);
\draw[red, thick,rounded corners] 
(-4,-2) -- (-3, -2) -- (-3,0)
(-4,2) -- (-2,2) -- (-2,3) 
(-2,0) -- (-3,0)
(-3,-3) -- (-2,-2) -- (-2,-1) -- (-1,-1) -- (-1,0)
(-2,0) -- (-1,0) 
(-2,0) -- (-2,1) -- (-1,1) -- (-1,2) -- (-0.5,2)
(-1,0) -- (-0.5, 0) 
(-1,-3) -- (-1,-2) -- (-0.5,-2);
\draw[red, thick, fill] (-3,0) circle (2pt) (-2,0) circle (2pt) (-1,0) circle (2pt);
\draw[red, thick, dashed] (-0.5, 0) -- (0.5, 0) (-0.5, 2) -- (0.5, 2) (-0.5, -2) -- (0.5, -2);
\end{scope}
\begin{scope}[rotate=180,xshift=-6cm]
\draw[green, line width=5, opacity=0.5] (-3,0) -- (-2,0) (-1,0) -- (-0.5, 0);
\draw[yellow, line width=5, opacity=0.5] (-2,0) -- (-1,0);
\draw[red, thick,rounded corners] 
(-4,-2) -- (-3, -2) -- (-3,0)
(-4,2) -- (-2,2) -- (-2,3) 
(-2,0) -- (-3,0)
(-3,-3) -- (-2,-2) -- (-2,-1) -- (-1,-1) -- (-1,0)
(-2,0) -- (-1,0) 
(-2,0) -- (-2,1) -- (-1,1) -- (-1,2) -- (-0.5,2)
(-1,0) -- (-0.5, 0) 
(-1,-3) -- (-1,-2) -- (-0.5,-2);
\draw[red, thick, fill] (-3,0) circle (2pt) (-2,0) circle (2pt) (-1,0) circle (2pt);
\end{scope}
\end{tikzpicture} 
$
}
\subfigure[$\mutation_\ngraph(\ngraph(a,b,c),\nbasis(a,b,c))$\label{coxeter_mutation_E}]{
$
\begin{tikzpicture}[baseline=-.5ex,scale=0.47]
\useasboundingbox (-5,-5.5) rectangle (5,5.5);
\draw[thick] (0,0) circle (5cm);
\draw[dashed]  (0,0) circle (3cm);
\fill[opacity=0.1, even odd rule] (0,0) circle (3) (0,0) circle (5);
\foreach \i in {1,2,3} {
\begin{scope}[rotate=\i*120]
\draw[yellow, line cap=round, line width=5, opacity=0.5] (60:1) -- (50:1.5) (70:1.75) -- (50:2);
\draw[green, line cap=round, line width=5, opacity=0.5] (0,0) -- (60:1) (50:1.5) -- (70:1.75);
\draw[blue, thick, rounded corners] (0,0) -- (0:3.4) to[out=-75,in=80] (-40:4);
\draw[red, thick, fill] (0,0) -- (60:1) circle (2pt) (60:1) -- (50:1.5) circle (2pt) -- (70:1.75) circle (2pt) -- (50:2) circle (2pt);
\draw[red, thick, dashed, rounded corners] (50:2) -- (60:2.8) -- (60:3.3) to[out=0,in=220] (40:4) (40:4) to[out=120,in=-20] (60:4);
\draw[red, thick, rounded corners] (50:2) -- (40:2.8) -- (40:3.3) to[out=-20,in=200] (20:4) (70:1.75) -- (80:2.8) -- (80:3.3) to[out=20,in=240] (60:4) (60:1) -- (100:2.8) -- (100:3.3) to[out=40,in=260] (80:4);
\draw[red, thick, rounded corners] (50:1.5) -- (20:3) -- (20:3.5) to[out=-70,in=50] (-40:4) (20:4) to[out=-50,in=120] (0:4.5) -- (0:5);
\draw[red, thick] (20:4) to[out=100,in=-40] (40:4) (60:4) to[out=140,in=0] (80:4);
\draw[blue, thick] (20:5) -- (20:4) to[out=140,in=-80] (40:4) (60:5) -- (60:4) to[out=180,in=-40] (80:4) -- (80:5);
\draw[blue, thick, rounded corners] (20:4) to[out=-70,in=100] (-20:4.5) -- (-20:5);
\draw[blue, thick, dashed] (40:4) to[out=160,in=-60] (60:4) (40:4) -- (40:5);
\draw[fill=white, thick] (20:4) circle (2pt) (40:4) circle (2pt) (60:4) circle (2pt) (80:4) circle (2pt) (-40:4) circle (2pt);
\end{scope}
\draw[fill=white, thick] (0,0) circle (2pt);
}
\end{tikzpicture}
$}
\caption{$N$-graphs of type $\exdynD \exdynE$ and their Coxeter 
mutations.}\label{fig_N_graphs_of_DE_and_Coxeter}
\end{figure}
Note that the pair of an $N$-graph and a tuple of cycles for $\exdynD_n$ differ depending on the parity of~$n$, see Table~\ref{table:affine D type}.

The pair $(\ngraph(\dynX),\nbasis(\dynX))$ in Figure~\ref{ngraph_nbasis_D} 
or~\ref{ngraph_nbasis_E} determines the initial seed $\initialseed= 
(\bfx_{t_0}, \qbasis_{t_0})$ in the corresponding cluster structure. The 
intersection pattern of the one-cycles defines a quiver $\quiver_{t_0}$
and the exchange matrix $\qbasis_{t_0}$, which is the adjacency matrix of $\quiver_{t_0}$, and 
the microlocal monodromy assign the tuple of cycles to a tuple of regular 
functions $\bfx_{t_0}$ in the coordinate ring of the moduli space  $\cM_1(\legendrian(\dynX))$.
In order to guarantee the existence of as many exact Lagrangian fillings as seeds, it remains to apply mutations in all possible ways.

A subtle point arises from the difference between mutation in cluster structure and the corresponding operation, \emph{Legendrian mutation}, in $N$-graph. The Legendrian mutation is well-defined when the geometric intersections numbers between cycles coincide with the algebraic intersections, while there is no obstruction to mutate in the cluster structure. 

Let $\qbpr_{t_0}$ be the principal part of the exchange matrix having 
$n=|\nbasis(\dynX)|$ columns in the initial seed $\initialseed$ determined by 
$(\ngraph(\dynX),\nbasis(\dynX))$. Then the combinatorial structure of the exchange 
graph $\exchange(\qbpr_{t_0})$ plays the crucial role to realize Legendrian 
mutation on $N$-graphs. Namely, any seed in the cluster pattern is obtained by 
iterating \emph{Coxeter mutation} followed by the mutations in a certain 
induced subgraph $\exchange(\qbpr_{t_0},x_{\ell})$ of degree $n-1$. The upshot is to 
use the induction argument on the number of cycles $|\nbasis(\dynX)|$ as long as the 
Coxeter mutation is possible in the $N$-graph setup.

Now the problem boils down to realize Coxeter mutations in $N$-graphs. Let us 
consider a partition $\nbasis_+$, $\nbasis_-$ of one cycles $\nbasis$ which 
are  {\color{green!50!black} green}, {\color{orange!80!yellow} yellow}-shaded 
cycles, 
respectively. Then the $N$-graph version of the Coxeter mutation 
$\mutation_\quiver$, called the \emph{Legendrian Coxeter mutation}, is defined by a 
sequence of 
the mutations
$\mutation_\quiver=\prod_{i\in \nbasis_-}\mutation_i \cdot \prod_{i\in \nbasis_+}\mutation_i$.

By applying Legendrian mutations together with a sequence of $N$-graph moves 
(II) and (V) in Figure~\ref{fig:move1-6}, we have the resulting pair 
$\mutation_\quiver(\ngraph(\dynX),\nbasis(\dynX))$ as in 
Figures~\ref{coxeter_mutation_D} and~\ref{coxeter_mutation_E}.
The key observation is that the Legendrian Coxeter mutations are nothing but 
attaching annulus type $N$-graphs, the gray-shaded region in 
Figures~\ref{coxeter_mutation_D} and~\ref{coxeter_mutation_E}. There are no 
obstruction to realize these attaching procedure. The similar holds for 
$\mutation_\quiver^{-1}(\ngraph(\dynX),\nbasis(\dynX))$ and other 
$\exdynD$-types.

\begin{theorem}[Theorem~\ref{thm:seed many fillings}]\label{theorem:seed many}
There are at least as many distinct exact embedded Lagrangian fillings as seeds for Legendrian links of type $\exdynD\exdynE$. 
\end{theorem}
There are many results showing the existence of infinitely many distinct Lagrangian fillings for Legendrian links, see \cite{CG2020,CZ2020,GSW2020b,CN2021}. To the best of authors' knowledge, this is the first result of infinitely many Lagrangian fillings which exhaust all seeds in the corresponding cluster structure. 

The attached $N$-graph annuli can be seen as exact Lagrangian cobordisms. Indeed, the $N$-graph annulus corresponds to the loop $\vartheta(\exdynD)$ of Legendrians $\legendrian(\exdynD)$ in Figure~\ref{fig:legendrian loop of D_intro}.
Note that this coincides with the Legendrian loop described in \cite[Figure 
2]{CN2021} up to Reidemeister moves. For the type of $\exdynE$, the twice of 
Legendrian Coxeter mutation on the pair $(\ngraph(a,b,c),\nbasis(a,b,c))$ gives
a loop $\vartheta(\exdynE)$ of $\legendrian(\exdynE)$ in Figure~\ref{fig:legendrian loop of E_intro}.
This loop of Legendrian is obtained by encoding a closed path of the half twist $\Delta$ in the three-strand braid. This path of $\Delta$ can be seen as a generalization of the path of a single crossing, the half twist of the two-strand braid, depicted in Figure~\ref{fig:legendrian loop of D_intro}.

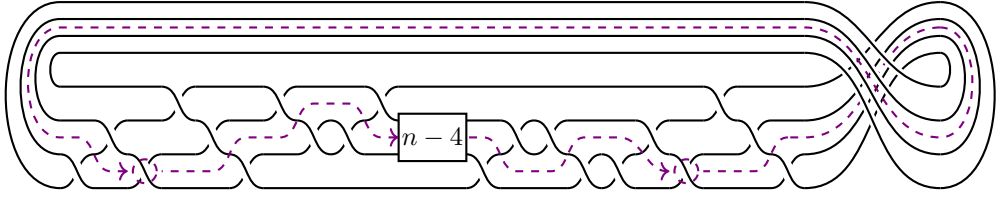
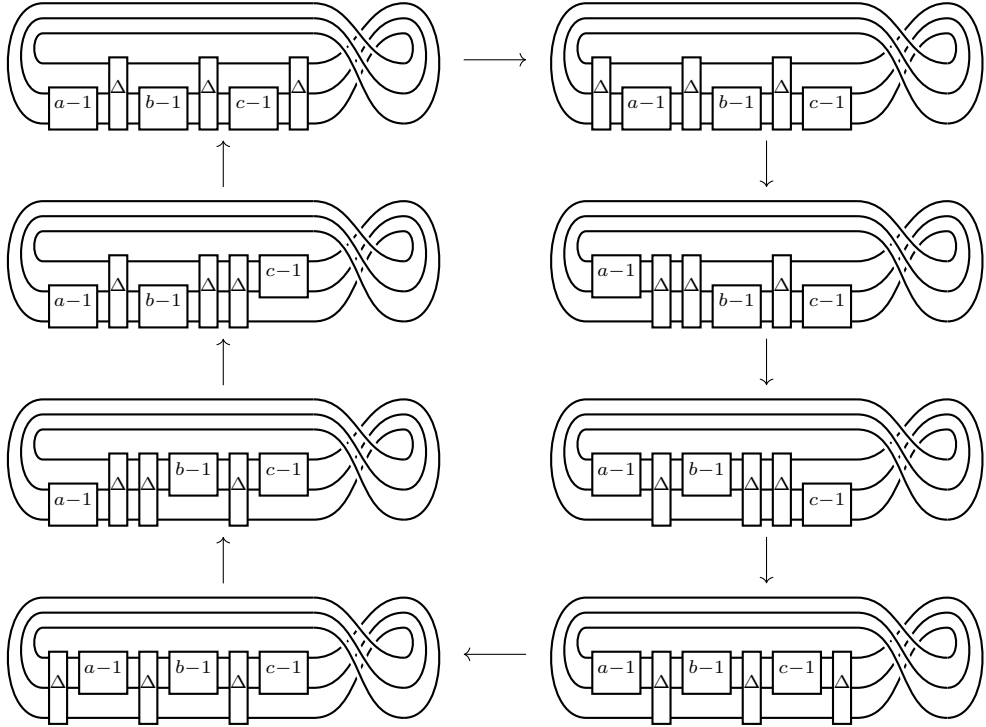
\begin{figure}[ht]
\subfigure[A Legendrian loop $\vartheta(\exdynD)$ induced from Legendrian Coxeter mutation $\mutation_\ngraph$ on $(\ngraph(\exdynD),\nbasis(\exdynD))$.
\label{fig:legendrian loop of D_intro}]{
\begin{tikzpicture}[baseline=-.5ex,scale=0.45]
\foreach \x in {0,2,5,12,15,16,18,21}
{
\begin{scope}[xshift=\x cm]
\draw[thick] (0,0) to[out=0,in=180] (1,1);
\draw[line width=5, white] (0,1) to[out=0,in=180] (1,0);
\draw[thick] (0,1) to[out=0,in=180] (1,0);
\end{scope}
}
\foreach \x in {1,4,7,8,13,14,17,20}
{
\begin{scope}[xshift=\x cm, yshift=1cm]
\draw[thick] (0,0) to[out=0,in=180] (1,1);
\draw[line width=5, white] (0,1) to[out=0,in=180] (1,0);
\draw[thick] (0,1) to[out=0,in=180] (1,0);
\end{scope}
}\foreach \x in {3,6,9,19}
{
\begin{scope}[xshift=\x cm, yshift=2cm]
\draw[thick] (0,0) to[out=0,in=180] (1,1);
\draw[line width=5, white] (0,1) to[out=0,in=180] (1,0);
\draw[thick] (0,1) to[out=0,in=180] (1,0);
\end{scope}
}
\draw[thick] 
(1,0) to (2,0) (3,0) to (5,0) (6,0) to (12,0) (13,0) to (15,0) (17,0) to (18,0) (19,0) to (21,0)
(3,1) to (4,1) (6,1) to (7,1) (9,1) to (12,1) (19,1) to (20,1)
(0,2) to (1,2) (2,2) to (3,2) (5,2) to (6,2) (10,2) to (13,2) (15,2) to (17,2) (18,2) to (19,2) (21,2) to (22,2)
(0,3) to (3,3) (4,3) to (6,3) (7,3) to (9,3) (10,3) to (19,3) (20,3) to (22,3)
;
\draw[thick,fill=white] (10,0.8) rectangle node {$n-4$} (12, 2.2);
\draw[thick,violet,dashed]
(26,4.75) to[out=180,in=0] (22,1.5) -- (21.5,1.5) 
;
\foreach \x in {0,1,2,3}
{
\draw[thick] 
(0,\x) to[out=180,in=180] (0,5.5-0.5*\x)
(0,5.5-0.5*\x) to (22,5.5-0.5*\x)
(26,\x) to[out=0,in=0] (26,5.5-0.5*\x)
(22,\x) to[out=0,in=180] (26,4+0.5*\x)
;
}
\foreach \x in {0,1,2,3}
{
\draw[line width=3, white]
(22,4+0.5*\x) to[out=0,in=180] (26,\x)
;
\draw[thick]
(22,4+0.5*\x) to[out=0,in=180] (26,\x)
;
}
\draw[thick, violet, dashed]
(2.5,0.5) circle (10pt)
(18.5,0.5) circle (10pt);
\draw[thick, violet, dashed,<-]
(2,0.5) to (1.5,0.5) to[out=180,in=0] (0.5,1.5) to (0,1.5) to[out=180,in=180] (0,4.75) to (22,4.75) to[out=0,in=180] (26,1.5) to[out=0,in=0] (26,4.75)
;
\draw[thick,violet,dashed]
(21.5,1.5) to[out=180,in=0] (20.5,0.5) -- (19,0.5);
\draw[thick,violet,dashed,<-]
(18,0.5) to[out=180,in=0] (16.5,1.5) to (15.5, 1.5)  to[out=180,in=0] (14.5, 0.5) to (13.5,0.5) to[out=180,in=0] (12.5,1.5) to (12,1.5);
\draw[thick,violet,dashed,<-]
(10,1.5) to (9.5, 1.5) to[out=180,in=0] (8.5,2.5) to (7.5,2.5) to[out=180,in=0] (6.5,1.5) to (5.5,1.5) to[out=180,in=0] (4.5,0.5) to (3,0.5)
;
\end{tikzpicture}
}

\subfigure[A Legendrian loop $\vartheta(\exdynE)$ induced from Legendrian Coxeter mutation $\mutation_\ngraph^{2}$ on $(\ngraph(\exdynE),\nbasis(\exdynE))$.
\label{fig:legendrian loop of E_intro}]{
\begin{tikzcd}[ampersand replacement=\&]
\begin{tikzpicture}[baseline=5ex,scale=0.4]
\draw[thick](0,0) -- (9,0) (0,1) -- (9,1) (0,2) -- (9,2) 
(0,3) -- (9,3) (0,3.5) -- (9,3.5) (0,4) -- (9,4)  
(0,0) to[out=180,in=180] (0,4)
(0,1) to[out=180,in=180] (0,3.5)
(0,2) to[out=180,in=180] (0,3)
(9,0) to[out=0,in=180] (12,3)
(9,1) to[out=0,in=180] (12,3.5)
(9,2) to[out=0,in=180] (12,4)
(12,4) to[out=0,in=0] (12,0)
(12,3.5) to[out=0,in=0] (12,1)
(12,3) to[out=0,in=0] (12,2);
\draw[line width=3, white]
(9,4) to[out=0,in=180] (12,2)
(9,3.5) to[out=0,in=180] (12,1)
(9,3) to[out=0,in=180] (12,0);
\draw[thick]
(9,4) to[out=0,in=180] (12,2)
(9,3.5) to[out=0,in=180] (12,1)
(9,3) to[out=0,in=180] (12,0);
\draw[thick,fill=white] (0.2,-0.2) rectangle node {$\scriptstyle{a-1}$} (1.8,1.2)
(2.2,-0.2) rectangle node {$\scriptstyle{\Delta}$} (2.8,2.2)
(3.2,-0.2) rectangle node {$\scriptstyle{b-1}$} (4.8,1.2)
(5.2,-0.2) rectangle node {$\scriptstyle{\Delta}$} (5.8,2.2)
(6.2,-0.2) rectangle node {$\scriptstyle{c-1}$} (7.8,1.2)
(8.2,-0.2) rectangle node {$\scriptstyle{\Delta}$} (8.8,2.2);
\end{tikzpicture}
\arrow[r] \&
\begin{tikzpicture}[baseline=5ex,scale=0.4]
\draw[thick](0,0) -- (9,0) (0,1) -- (9,1) (0,2) -- (9,2) 
(0,3) -- (9,3) (0,3.5) -- (9,3.5) (0,4) -- (9,4)  
(0,0) to[out=180,in=180] (0,4)
(0,1) to[out=180,in=180] (0,3.5)
(0,2) to[out=180,in=180] (0,3)
(9,0) to[out=0,in=180] (12,3)
(9,1) to[out=0,in=180] (12,3.5)
(9,2) to[out=0,in=180] (12,4)
(12,4) to[out=0,in=0] (12,0)
(12,3.5) to[out=0,in=0] (12,1)
(12,3) to[out=0,in=0] (12,2);
\draw[line width=3, white]
(9,4) to[out=0,in=180] (12,2)
(9,3.5) to[out=0,in=180] (12,1)
(9,3) to[out=0,in=180] (12,0);
\draw[thick]
(9,4) to[out=0,in=180] (12,2)
(9,3.5) to[out=0,in=180] (12,1)
(9,3) to[out=0,in=180] (12,0);
\draw[thick,fill=white] 
(0.2,-0.2) rectangle node {$\scriptstyle{\Delta}$} (0.8,2.2)
(1.2,-0.2) rectangle node {$\scriptstyle{a-1}$} (2.8,1.2)
(3.2,-0.2) rectangle node {$\scriptstyle{\Delta}$} (3.8,2.2)
(4.2,-0.2) rectangle node {$\scriptstyle{b-1}$} (5.8,1.2)
(6.2,-0.2) rectangle node {$\scriptstyle{\Delta}$} (6.8,2.2)
(7.2,-0.2) rectangle node {$\scriptstyle{c-1}$} (8.8,1.2);
\end{tikzpicture}
\arrow[d] \\
\begin{tikzpicture}[baseline=5ex,scale=0.4]
\draw[thick](0,0) -- (9,0) (0,1) -- (9,1) (0,2) -- (9,2) 
(0,3) -- (9,3) (0,3.5) -- (9,3.5) (0,4) -- (9,4)  
(0,0) to[out=180,in=180] (0,4)
(0,1) to[out=180,in=180] (0,3.5)
(0,2) to[out=180,in=180] (0,3)
(9,0) to[out=0,in=180] (12,3)
(9,1) to[out=0,in=180] (12,3.5)
(9,2) to[out=0,in=180] (12,4)
(12,4) to[out=0,in=0] (12,0)
(12,3.5) to[out=0,in=0] (12,1)
(12,3) to[out=0,in=0] (12,2);
\draw[line width=3, white]
(9,4) to[out=0,in=180] (12,2)
(9,3.5) to[out=0,in=180] (12,1)
(9,3) to[out=0,in=180] (12,0);
\draw[thick]
(9,4) to[out=0,in=180] (12,2)
(9,3.5) to[out=0,in=180] (12,1)
(9,3) to[out=0,in=180] (12,0);
\draw[thick,fill=white] (0.2,-0.2) rectangle node {$\scriptstyle{a-1}$} (1.8,1.2)
(2.2,-0.2) rectangle node {$\scriptstyle{\Delta}$} (2.8,2.2)
(3.2,-0.2) rectangle node {$\scriptstyle{b-1}$} (4.8,1.2)
(5.2,-0.2) rectangle node {$\scriptstyle{\Delta}$} (5.8,2.2)
(6.2,-0.2) rectangle node {$\scriptstyle{\Delta}$} (6.8,2.2)
(7.2,0.8) rectangle node {$\scriptstyle{c-1}$} (8.8,2.2);
\end{tikzpicture}
\arrow[u] \& 
\begin{tikzpicture}[baseline=5ex,scale=0.4]
\draw[thick](0,0) -- (9,0) (0,1) -- (9,1) (0,2) -- (9,2) 
(0,3) -- (9,3) (0,3.5) -- (9,3.5) (0,4) -- (9,4)  
(0,0) to[out=180,in=180] (0,4)
(0,1) to[out=180,in=180] (0,3.5)
(0,2) to[out=180,in=180] (0,3)
(9,0) to[out=0,in=180] (12,3)
(9,1) to[out=0,in=180] (12,3.5)
(9,2) to[out=0,in=180] (12,4)
(12,4) to[out=0,in=0] (12,0)
(12,3.5) to[out=0,in=0] (12,1)
(12,3) to[out=0,in=0] (12,2);
\draw[line width=3, white]
(9,4) to[out=0,in=180] (12,2)
(9,3.5) to[out=0,in=180] (12,1)
(9,3) to[out=0,in=180] (12,0);
\draw[thick]
(9,4) to[out=0,in=180] (12,2)
(9,3.5) to[out=0,in=180] (12,1)
(9,3) to[out=0,in=180] (12,0);
\draw[thick,fill=white] 
(0.2,0.8) rectangle node {$\scriptstyle{a-1}$} (1.8,2.2)
(2.2,-0.2) rectangle node {$\scriptstyle{\Delta}$} (2.8,2.2)
(3.2,-0.2) rectangle node {$\scriptstyle{\Delta}$} (3.8,2.2)
(4.2,-0.2) rectangle node {$\scriptstyle{b-1}$} (5.8,1.2)
(6.2,-0.2) rectangle node {$\scriptstyle{\Delta}$} (6.8,2.2)
(7.2,-0.2) rectangle node {$\scriptstyle{c-1}$} (8.8,1.2);
\end{tikzpicture}
\arrow[d] \\
\begin{tikzpicture}[baseline=5ex,scale=0.4]
\draw[thick](0,0) -- (9,0) (0,1) -- (9,1) (0,2) -- (9,2) 
(0,3) -- (9,3) (0,3.5) -- (9,3.5) (0,4) -- (9,4)  
(0,0) to[out=180,in=180] (0,4)
(0,1) to[out=180,in=180] (0,3.5)
(0,2) to[out=180,in=180] (0,3)
(9,0) to[out=0,in=180] (12,3)
(9,1) to[out=0,in=180] (12,3.5)
(9,2) to[out=0,in=180] (12,4)
(12,4) to[out=0,in=0] (12,0)
(12,3.5) to[out=0,in=0] (12,1)
(12,3) to[out=0,in=0] (12,2);
\draw[line width=3, white]
(9,4) to[out=0,in=180] (12,2)
(9,3.5) to[out=0,in=180] (12,1)
(9,3) to[out=0,in=180] (12,0);
\draw[thick]
(9,4) to[out=0,in=180] (12,2)
(9,3.5) to[out=0,in=180] (12,1)
(9,3) to[out=0,in=180] (12,0);
\draw[thick,fill=white] (0.2,-0.2) rectangle node {$\scriptstyle{a-1}$} (1.8,1.2)
(2.2,-0.2) rectangle node {$\scriptstyle{\Delta}$} (2.8,2.2)
(3.2,-0.2) rectangle node {$\scriptstyle{\Delta}$} (3.8,2.2)
(4.2,0.8) rectangle node {$\scriptstyle{b-1}$} (5.8,2.2)
(6.2,-0.2) rectangle node {$\scriptstyle{\Delta}$} (6.8,2.2)
(7.2,0.8) rectangle node {$\scriptstyle{c-1}$} (8.8,2.2);
\end{tikzpicture}
\arrow[u] \&
\begin{tikzpicture}[baseline=5ex,scale=0.4]
\draw[thick](0,0) -- (9,0) (0,1) -- (9,1) (0,2) -- (9,2) 
(0,3) -- (9,3) (0,3.5) -- (9,3.5) (0,4) -- (9,4)  
(0,0) to[out=180,in=180] (0,4)
(0,1) to[out=180,in=180] (0,3.5)
(0,2) to[out=180,in=180] (0,3)
(9,0) to[out=0,in=180] (12,3)
(9,1) to[out=0,in=180] (12,3.5)
(9,2) to[out=0,in=180] (12,4)
(12,4) to[out=0,in=0] (12,0)
(12,3.5) to[out=0,in=0] (12,1)
(12,3) to[out=0,in=0] (12,2);
\draw[line width=3, white]
(9,4) to[out=0,in=180] (12,2)
(9,3.5) to[out=0,in=180] (12,1)
(9,3) to[out=0,in=180] (12,0);
\draw[thick]
(9,4) to[out=0,in=180] (12,2)
(9,3.5) to[out=0,in=180] (12,1)
(9,3) to[out=0,in=180] (12,0);
\draw[thick,fill=white] 
(0.2,0.8) rectangle node {$\scriptstyle{a-1}$} (1.8,2.2)
(2.2,-0.2) rectangle node {$\scriptstyle{\Delta}$} (2.8,2.2)
(3.2,0.8) rectangle node {$\scriptstyle{b-1}$} (4.8,2.2)
(5.2,-0.2) rectangle node {$\scriptstyle{\Delta}$} (5.8,2.2)
(6.2,-0.2) rectangle node {$\scriptstyle{\Delta}$} (6.8,2.2)
(7.2,-0.2) rectangle node {$\scriptstyle{c-1}$} (8.8,1.2);
\end{tikzpicture}
\arrow[d] \\
\begin{tikzpicture}[baseline=5ex,scale=0.4]
\draw[thick](0,0) -- (9,0) (0,1) -- (9,1) (0,2) -- (9,2) 
(0,3) -- (9,3) (0,3.5) -- (9,3.5) (0,4) -- (9,4)  
(0,0) to[out=180,in=180] (0,4)
(0,1) to[out=180,in=180] (0,3.5)
(0,2) to[out=180,in=180] (0,3)
(9,0) to[out=0,in=180] (12,3)
(9,1) to[out=0,in=180] (12,3.5)
(9,2) to[out=0,in=180] (12,4)
(12,4) to[out=0,in=0] (12,0)
(12,3.5) to[out=0,in=0] (12,1)
(12,3) to[out=0,in=0] (12,2);
\draw[line width=3, white]
(9,4) to[out=0,in=180] (12,2)
(9,3.5) to[out=0,in=180] (12,1)
(9,3) to[out=0,in=180] (12,0);
\draw[thick]
(9,4) to[out=0,in=180] (12,2)
(9,3.5) to[out=0,in=180] (12,1)
(9,3) to[out=0,in=180] (12,0);
\draw[thick,fill=white] 
(0.2,-0.2) rectangle node {$\scriptstyle{\Delta}$} (0.8,2.2)
(1.2,0.8) rectangle node {$\scriptstyle{a-1}$} (2.8,2.2)
(3.2,-0.2) rectangle node {$\scriptstyle{\Delta}$} (3.8,2.2)
(4.2,0.8) rectangle node {$\scriptstyle{b-1}$} (5.8,2.2)
(6.2,-0.2) rectangle node {$\scriptstyle{\Delta}$} (6.8,2.2)
(7.2,0.8) rectangle node {$\scriptstyle{c-1}$} (8.8,2.2);
\end{tikzpicture}
\arrow[u] \&
\begin{tikzpicture}[baseline=5ex,scale=0.4]
\draw[thick](0,0) -- (9,0) (0,1) -- (9,1) (0,2) -- (9,2) 
(0,3) -- (9,3) (0,3.5) -- (9,3.5) (0,4) -- (9,4)  
(0,0) to[out=180,in=180] (0,4)
(0,1) to[out=180,in=180] (0,3.5)
(0,2) to[out=180,in=180] (0,3)
(9,0) to[out=0,in=180] (12,3)
(9,1) to[out=0,in=180] (12,3.5)
(9,2) to[out=0,in=180] (12,4)
(12,4) to[out=0,in=0] (12,0)
(12,3.5) to[out=0,in=0] (12,1)
(12,3) to[out=0,in=0] (12,2);
\draw[line width=3, white]
(9,4) to[out=0,in=180] (12,2)
(9,3.5) to[out=0,in=180] (12,1)
(9,3) to[out=0,in=180] (12,0);
\draw[thick]
(9,4) to[out=0,in=180] (12,2)
(9,3.5) to[out=0,in=180] (12,1)
(9,3) to[out=0,in=180] (12,0);
\draw[thick,fill=white] 
(0.2,0.8) rectangle node {$\scriptstyle{a-1}$} (1.8,2.2)
(2.2,-0.2) rectangle node {$\scriptstyle{\Delta}$} (2.8,2.2)
(3.2,0.8) rectangle node {$\scriptstyle{b-1}$} (4.8,2.2)
(5.2,-0.2) rectangle node {$\scriptstyle{\Delta}$} (5.8,2.2)
(6.2,0.8) rectangle node {$\scriptstyle{c-1}$} (7.8,2.2)
(8.2,-0.2) rectangle node {$\scriptstyle{\Delta}$} (8.8,2.2);
\end{tikzpicture}
\arrow[l]
\end{tikzcd}
}
\caption{Legendrian loops induced from Legendrian Coxeter mutation}
\label{fig:legendrian loops}
\end{figure}

\begin{theorem}[Theorem~\ref{thm:legendrian loop}]\label{theorem:legendrian loop}
The Legendrian Coxeter mutation $\mutation_\ngraph^{\pm1}$ on 
$(\ngraph(\exdynD),\nbasis(\exdynD))$ and twice of Legendrian mutation 
$\mutation_\ngraph^{\pm 2}$ on $(\ngraph(\exdynE),\nbasis(\exdynE))$ induce 
Legendrian loops $\vartheta(\exdynD)$ and $\vartheta(\exdynE)$ in Figure~\ref{fig:legendrian loops}, respectively. 
In particular, the order of the Legendrian loops are infinite as elements of the fundamental group of the space of Legendrians isotopic to $\lambda(\exdynD)$ and $\lambda(\exdynE)$, respectively.
\end{theorem}

Any cluster pattern of non-simply-laced affine type can be obtained by folding 
a cluster pattern of type $\exdynA$, $\exdynD$, or $\exdynE$. In other words, 
those cluster pattern of non-simply-laced 
affine type can be seen as sub-patterns of $\exdynA \exdynD \exdynE$-types 
consisting of seeds with certain symmetries of finite group $G$ action. We call 
such seeds or $N$-graphs \emph{$G$-admissible}, and the mutation in the folded 
cluster structure is a sequence of mutations respecting the $G$-orbits. 
We say that a seed (or an $N$-graph) is \emph{globally foldable} if it is 
$G$-admissible and its arbitrary mutations along $G$-orbits are again 
$G$-admissible. 

The followings $N$-graphs with tuples of cycles represent folding process of 
type $\exdynG_2$, $\dynE_6^{(2)}$, and $\exdynF_4$, respectively.
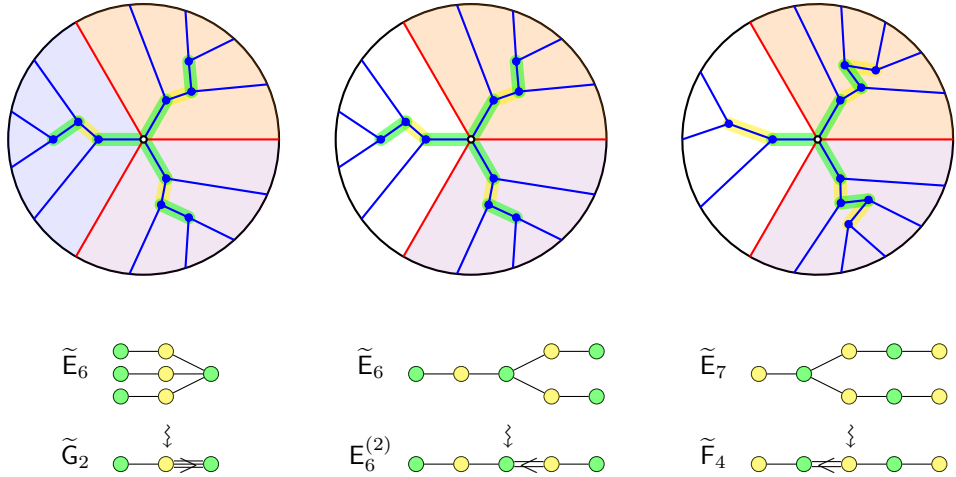
\[
\begin{tikzcd}[column sep=small, row sep=small]
\begin{tikzpicture}[baseline=-.5ex,scale=0.6]
\draw[thick] (0,0) circle (3cm);
\draw[orange, opacity=0.2, fill] (0:3) arc(0:120:3) (120:3) -- (0,0) -- (0:3);
\draw[violet, opacity=0.1, fill] (0:3) -- (0,0) -- (-120:3) arc(-120:0:3) (0:3);
\draw[blue, opacity=0.1, fill] (120:3) arc(120:240:3) (240:3) -- (0,0) -- (120:3);
\draw[yellow, line cap=round, line width=5, opacity=0.5] (60:1) -- (45:1.5) (180:1) -- (165:1.5) (300:1) -- (285:1.5);
\draw[green, line cap=round, line width=5, opacity=0.5] (0,0) -- (60:1) (0,0) -- (180:1) (0,0) -- (300:1) (45:1.5) -- (60:2) (165:1.5) -- (180:2) (285:1.5) -- (300:2);
\draw[red, thick] (0,0) -- (0:3) (0,0) -- (120:3) (0,0) -- (240:3);
\draw[blue, thick, fill] (0,0) -- (60:1) circle (2pt) -- (96:3) (60:1) -- (45:1.5) circle (2pt) -- (24:3) (45:1.5) -- (60:2) circle (2pt) -- (72:3) (60:2) -- (48:3);
\draw[blue, thick, fill] (0,0) -- (180:1) circle (2pt) -- (216:3) (180:1) -- (165:1.5) circle (2pt) -- (144:3) (165:1.5) -- (180:2) circle (2pt) -- (192:3) (180:2) -- (168:3);
\draw[blue, thick, fill] (0,0) -- (300:1) circle (2pt) -- (336:3) (300:1) -- (285:1.5) circle (2pt) -- (264:3) (285:1.5) -- (300:2) circle (2pt) -- (312:3) (300:2) -- (288:3);
\draw[thick, fill=white] (0,0) circle (2pt);
\end{tikzpicture}
&
\begin{tikzpicture}[baseline=-.5ex,scale=0.6]
\draw[thick] (0,0) circle (3cm);
\draw[orange, opacity=0.2, fill] (0:3) arc(0:120:3) (120:3) -- (0,0) -- (0:3);
\draw[violet, opacity=0.1, fill] (0:3) -- (0,0) -- (-120:3) arc(-120:0:3) (0:3);
\draw[yellow, line cap=round, line width=5, opacity=0.5] (60:1) -- (45:1.5) (180:1) -- (165:1.5) (300:1) -- (285:1.5);
\draw[green, line cap=round, line width=5, opacity=0.5] (0,0) -- (60:1) (0,0) -- (180:1) (0,0) -- (300:1) (45:1.5) -- (60:2) (165:1.5) -- (180:2) (285:1.5) -- (300:2);
\draw[red, thick] (0,0) -- (0:3) (0,0) -- (120:3) (0,0) -- (240:3);
\draw[blue, thick, fill] (0,0) -- (60:1) circle (2pt) -- (96:3) (60:1) -- (45:1.5) circle (2pt) -- (24:3) (45:1.5) -- (60:2) circle (2pt) -- (72:3) (60:2) -- (48:3);
\draw[blue, thick, fill] (0,0) -- (180:1) circle (2pt) -- (216:3) (180:1) -- (165:1.5) circle (2pt) -- (144:3) (165:1.5) -- (180:2) circle (2pt) -- (192:3) (180:2) -- (168:3);
\draw[blue, thick, fill] (0,0) -- (300:1) circle (2pt) -- (336:3) (300:1) -- (285:1.5) circle (2pt) -- (264:3) (285:1.5) -- (300:2) circle (2pt) -- (312:3) (300:2) -- (288:3);
\draw[thick, fill=white] (0,0) circle (2pt);
\end{tikzpicture}
&
\begin{tikzpicture}[baseline=-.5ex,scale=0.6,rotate=120]
\draw[thick] (0,0) circle (3cm);
\draw[orange, opacity=0.2, fill] (0:3) -- (0,0) -- (-120:3) arc(-120:0:3) (0:3);
\draw[violet, opacity=0.1, fill] (120:3) arc(120:240:3) (240:3) -- (0,0) -- 
(120:3);
\draw[yellow, line cap=round, line width=5, opacity=0.5] (60:1) -- (50:2) 
(180:1) -- (170:1.5) (190:1.75) -- (170:2) (300:1) -- (290:1.5) (310:1.75) -- 
(290:2);
\draw[green, line cap=round, line width=5, opacity=0.5] (0,0) -- (60:1) (0,0) 
-- (180:1) (0,0) -- (300:1) (170:1.5) -- (190:1.75) (290:1.5) -- (310:1.75);
\draw[red, thick] (0,0) -- (0:3) (0,0) -- (120:3) (0,0) -- (240:3);
\draw[blue, thick, fill] (0,0) -- (60:1) circle (2pt) -- (90:3) (60:1) -- 
(50:2)  circle (2pt) -- (30:3) (50:2) -- (60:3);
\draw[blue, thick, fill] (0,0) -- (180:1) circle (2pt) -- (220:3) (180:1) -- 
(170:1.5)  circle (2pt) -- (140:3) (170:1.5) -- (190:1.75) circle (2pt) -- 
(200:3) (190:1.75) -- (170:2) circle (2pt) -- (160:3) (170:2) -- (180:3);
\draw[blue, thick, dashed] ;
\draw[blue, thick, fill] (0,0) -- (300:1) circle (2pt) -- (340:3) (300:1) -- 
(290:1.5) circle (2pt) -- (260:3) (290:1.5) -- (310:1.75) circle (2pt) -- 
(320:3) (310:1.75) -- (290:2) circle (2pt) -- (280:3) (290:2) -- (300:3);
\draw[blue, thick, dashed] ;
\draw[thick, fill=white] (0,0) circle (2pt);
\end{tikzpicture}
\\
\begin{tikzpicture}[baseline=-.5ex,scale=0.6]
\tikzstyle{triple line} = [
double distance = 2pt, 
double=\pgfkeysvalueof{/tikz/commutative diagrams/background color}
]    
\node[Dnode] (a1) at (0,0) {};
\node[Dnode] (a2) at (-1,0.5) {};
\node[Dnode] (a3) at (-2,0.5) {};
\node[Dnode] (a4) at (-1,0) {};
\node[Dnode] (a5) at (-2,0) {};
\node[Dnode] (a6) at (-1,-0.5) {};
\node[Dnode] (a7) at (-2,-0.5) {};

\node[gnode] at (a1) {};
\node[ynode] at (a2) {};
\node[gnode] at (a3) {};
\node[ynode] at (a4) {};
\node[gnode] at (a5) {};
\node[ynode] at (a6) {};
\node[gnode] at (a7) {};

\draw (a1)--(a2)--(a3) (a1)--(a4)--(a5) (a1)--(a6)--(a7);

\node at (-3,0) {$\exdynE_{6}$};
\node at (-1,-1.5) {\rotatebox[origin=c]{-90}{$\rightsquigarrow$}};

\node at (-3,-2) {$\exdynG_{2}$};

\coordinate[Dnode] (3) at (-2,-2) {};
\coordinate[Dnode] (2) at (-1,-2) {};
\coordinate[Dnode] (1) at (0,-2) {};

\node[gnode] at (1) {};
\node[ynode] at (2) {};
\node[gnode] at (3) {};

\draw (2)--(3);
\draw[triple line] (1)-- ++ (2) node[midway,yshift=-0.1cm ] {$>$};
\draw (1)--(2);

\end{tikzpicture}
&
\begin{tikzpicture}[baseline=-.5ex,scale=0.6]

\coordinate[Dnode] (d1) at (-1,0) {};
\coordinate[Dnode] (d2) at (0,0) {};
\coordinate[Dnode] (d3) at (1,0) {};
\coordinate[Dnode] (d4) at (2,0.5) {};
\coordinate[Dnode] (d5) at (3,0.5) {};
\coordinate[Dnode] (d6) at (2,-0.5) {};
\coordinate[Dnode] (d7) at (3,-0.5) {};

\node[gnode] at (d1) {};
\node[ynode] at (d2) {};
\node[gnode] at (d3) {};
\node[ynode] at (d4) {};
\node[gnode] at (d5) {};
\node[ynode] at (d6) {};
\node[gnode] at (d7) {};

\draw  (d1)--(d2)--(d3)--(d4)--(d5)
(d3)--(d6)--(d7);

\node at (-2,0) {$\exdynE_{6}$};
\node at (1,-1.5) {\rotatebox[origin=c]{-90}{$\rightsquigarrow$}};

\node at (-2,-2) {$\dynE_{6}^{(2)}$};
\coordinate[Dnode] (1) at (-1,-2) {};
\coordinate[Dnode] (2) at (0,-2) {};
\coordinate[Dnode] (3) at (1,-2) {};
\coordinate[Dnode] (4) at (2,-2) {};
\coordinate[Dnode] (5) at (3,-2) {};

\node[gnode] at (1) {};
\node[ynode] at (2) {};
\node[gnode] at (3) {};
\node[ynode] at (4) {};
\node[gnode] at (5) {};

\draw (1)--(2)--(3) (4)--(5) ;
\draw[double line] (3)-- ++ (4) node[midway,yshift=-0.1cm ] {$<$};

\end{tikzpicture}
&
\begin{tikzpicture}[baseline=-.5ex,scale=0.6]

\coordinate[Dnode] (e1) at (-2,0) {};
\coordinate[Dnode] (e2) at (-1,0) {};
\coordinate[Dnode] (e3) at (0,0.5) {};
\coordinate[Dnode] (e4) at (1,0.5) {};
\coordinate[Dnode] (e5) at (2,0.5) {};
\coordinate[Dnode] (e6) at (0,-0.5) {};
\coordinate[Dnode] (e7) at (1,-0.5) {};
\coordinate[Dnode] (e8) at (2,-0.5) {};

\foreach \g in {e2,e4,e7} {
\node[gnode] at (\g) {};
}
\foreach \y in {e1,e3,e5,e6,e8}{
\node[ynode] at (\y) {};
}
\draw (e1)--(e2)--(e3)--(e4)--(e5)
(e2)--(e6)--(e7)--(e8);

\node at (-3,0) {$\exdynE_{7}$};
\node at (0,-1.5) {\rotatebox[origin=c]{-90}{$\rightsquigarrow$}};

\node at (-3,-2) {$\exdynF_{4}$};
\coordinate[Dnode] (1) at (-2,-2) {};
\coordinate[Dnode] (2) at (-1,-2) {};
\coordinate[Dnode] (3) at (0,-2) {};
\coordinate[Dnode] (4) at (1,-2) {};
\coordinate[Dnode] (5) at (2,-2) {};

\foreach \y in {1,3,5} {
\node[ynode] at (\y) {};
}
\foreach \g in {2,4}{
\node[gnode] at (\g) {};
}

\draw (1)--(2)
(3)--(4)--(5);
\draw[double line] (2)-- ++ (3) node[midway,yshift=-0.1cm ] {$<$};
\end{tikzpicture}
\end{tikzcd}
\]
The three colored regions in the first $N$-graph represent rotational 
$\Z/3\Z$-symmetry, and the two colored regions in the remaining three 
$N$-graphs indicate $\Z/2\Z$-symmetry given by partial rotation. All the above 
symmetries induce that the corresponding $N$-graphs are globally foldable, and 
hence we can realize the folded seeds via $N$-graphs with symmetries. 

\begin{theorem}[Theorem~\ref{theorem:folded Lagrangian fillings}]
The following holds:
\begin{enumerate}
\item There exists a set of $\Z/2\Z$-admissible $4$-graphs of the Legendrian link $\legendrian(\exdynD_{2n})$ admits the cluster pattern of type $\exdynB_n$.
\item There exists a set of $\Z/3\Z$-admissible $3$-graphs of the Legendrian link $\legendrian(\exdynE_6)$ admits the cluster pattern of type $\exdynG_2$.
\item There exists a set of $\Z/2\Z$-admissible $3$-graphs of the Legendrian link $\legendrian(\exdynE_6)$ admits the cluster pattern of type $\dynE_6^{(2)}$.
\item There exists a set of $\Z/2\Z$-admissible $3$-graphs of the Legendrian link $\legendrian(\exdynE_7)$ admits the cluster pattern of type $\exdynF_4$.
\end{enumerate}
\end{theorem}

\subsection{Organization of the paper}
The rest of the paper is divided into six sections including appendix. 
We review, in Section~\ref{sec:cluster algebra}, some basics on affine cluster algebra. Especially we focus on structural results about the combinatorics of exchange graphs using Coxeter mutations.
In Section~\ref{sec:N-graph}, we recall how $N$-graphs and their moves encode Legendrian surfaces and the Legendrian isotopies. After that we review the assignment of seed in the cluster structure from $N$-graphs and certain flag moduli.
In Section~\ref{sec:affine DE type}, we investigate Legendrian links and $N$-graphs of type $\exdynD \exdynE$. 
We discuss $N$-graph realization of the Coxeter mutation and prove Theorem~\ref{theorem:legendrian loop} on the relationship between Coxeter mutations and Legendrian loops.
We also construct as many Lagrangian fillings as seeds for Legendrian links of type $\exdynD \exdynE$ and prove Theorem~\ref{thm:seed many fillings}.
In Section~\ref{section:folding}, we discuss the folded cluster patterns and prove Theorems~\ref{theorem:folded seed} and \ref{theorem:folded Lagrangian fillings}.
Finally, in Appendix~\ref{sec:Coxeter padding affine D_n}, the pictorial proof of $N$-graph realization for the Coxeter mutation of type $\exdynD$ will be given.

If some readers are familiar with the notion of cluster algebra and $N$-graph, then one may skip Section~\ref{sec:cluster algebra} and Section~\ref{sec:N-graph}, respectively, and start from Section~\ref{sec:affine DE type}.

\subsection*{Acknowledgement}
We thank Roger Casals for useful conversations and 
Salvatore Stella for explaining the result on the affine almost positive roots 
model.  
B. H. An and Y. Bae were supported by the National Research Foundation of Korea (NRF) grant funded by the Korea government (MSIT) (No. 2020R1A2C1A0100320).
E. Lee was supported by the Institute for Basic Science (IBS-R003-D1).

\section{Cluster algebras}\label{sec:cluster algebra}
Cluster algebras, introduced by Fomin and Zelevinsky~\cite{FZ1_2002}, are commutative algebras with specific generators, called \emph{cluster variables}, defined recursively. In this section, we recall basic notions in the theory of cluster algebras. For more details, we refer the reader to~\cite{FZ1_2002, FZ2_2003,BFZ3_2005,FZ4_2007}.

Throughout this section, we fix $m, n \in \Z_{>0}$ such that $n \leq m$, and we let $\field$ be the rational function field with $m$ independent variables over $\C$.

\subsection{Basics on cluster algebras}

\begin{definition}[{cf. \cite{FZ1_2002, FZ2_2003}}]
A \emph{seed} $\seed = (\bfx, \qbasis)$ is a pair of 
\begin{itemize}
\item a tuple $\mathbf x = (x_1,\dots,x_m)$ of algebraically independent generators of $\field$, that is, $\field = \C(x_1,\dots,x_m)$;
\item an $m \times n$ integer matrix $\qbasis = (b_{i,j})_{i,j}$ such that the \emph{principal part} $\qbasis^{\textrm{pr}} \coloneqq
(b_{i,j})_{1\leq i,j\leq n}$ is skew-symmetrizable, that is, there exist positive integers $d_1,\dots,d_n$ such that
$$\textrm{diag}(d_1,\dots,d_n) \cdot \qbasis^{\textrm{pr}}$$ is a skew-symmetric matrix.
\end{itemize}
We call elements $x_1,\dots,x_m$ \emph{cluster variables} and call $\qbasis$ \emph{exchange matrix}. Moreover, we call $x_1,\dots,x_n$ \emph{unfrozen} (or, \emph{mutable}) variables and $x_{n+1},\dots,x_m$ \emph{frozen} variables.
\end{definition}

We say that two seeds $\seed = (\bfx, \qbasis)$ and $\seed' = (\bfx', \qbasis')$ are \textit{equivalent}, denoted by $\seed\sim\seed'$ if there exists a permutation $\sigma$ of indices $1,\dots,n$ such that
\[
x_i' = x_{\sigma(i)},\quad b_{i,j}' = b_{\sigma(i),\sigma(j)},
\]
where $\bfx = (x_{1},\dots,x_{m})$, $\bfx' = (x_1',\dots,x_m')$, $\qbasis = (b_{i,j})$, and $\qbasis' = (b_{i,j}')$.

To define cluster algebras, we introduce mutations on seeds, exchange matrices, and quivers as follows. 
\begin{enumerate}
\item (Mutation on seeds)	For a seed $\seed = (\bfx, \qbasis)$ and an integer $k \in [n] \coloneqq \{1,\dots,n\}$, the \emph{mutation} $\mutation_k(\seed) = (\bfx', \qbasis')$ is defined as follows:
\begin{equation*}
\begin{split}
x_i' &= \begin{cases}
x_i &\text{ if } i \neq k,\\
\displaystyle 
x_k^{-1}\left( \prod_{b_{j,k} > 0} x_j^{b_{j,k}} + \prod_{b_{j,k} < 0}x_j^{-b_{j,k}}
\right) & \text{ otherwise}.
\end{cases}\\[1em]
b_{i,j}' &= \begin{cases}
-b_{i,j} & \text{ if } i = k \text{ or } j = k, \\
\displaystyle b_{i,j} + \frac{|b_{i,k}| b_{k,j} + b_{i,k} | b_{k,j}|} {2} & \text{ otherwise}.
\end{cases}
\end{split}
\end{equation*}
\item (Mutation on exchange matrices)
We define $\mutation_k(\qbasis) = (b_{i,j}')$, and say that \emph{$\qbasis' =(b_{i,j}')$ is the mutation of $\qbasis$ at $k$}.
\item (Mutation on quivers) We call a finite directed multigraph $\quiver$ a \emph{quiver} if it does not have directed cycles of length at most $2$. 
The adjacency matrix $\qbasis(\quiver)$ of a quiver is always skew-symmetric. Moreover, $\mutation_k(\qbasis(\quiver))$ is again the adjacency matrix of a quiver $\quiver'$. We define $\mutation_k(\quiver)$ to be the quiver satisfying 
\[
\qbasis(\mutation_k(\quiver)) = \mutation_k(\qbasis(\quiver)),
\]
and say that \emph{$\mutation_k(\quiver)$ is the mutation of $\quiver$ at $k$}.
\end{enumerate}

\begin{example}
Let $n = m = 2$. Suppose that an initial seed is given by
\[
\initialseed = \left(
(x_1,x_2), \begin{pmatrix}
0 & 1 \\ -3 & 0
\end{pmatrix}
\right).
\]
Considering mutations $\mu_1(\initialseed)$ and $\mu_2\mu_1(\initialseed)$, we obtain the following.
\[
\mu_1(\initialseed) = \left(
\left(
\frac{1+x_2^3}{x_1},x_2
\right), \begin{pmatrix}
0 & -1 \\3 & 0
\end{pmatrix}
\right), \quad
\mu_2\mu_1(\initialseed) = \left(
\left(
\frac{1+x_2^3}{x_1}, \frac{1+x_1+x_2^3}{x_1x_2}
\right),
\begin{pmatrix}
0 & 1  \\ -3 & 0
\end{pmatrix}
\right).
\]
\end{example}

\begin{remark}
Let $k$ be a vertex in a quiver $\quiver$.
The mutation $\mutation_k(\quiver)$ can also be described via a sequence of three steps:
\begin{enumerate}
\item For each directed two-arrow path $i \to k \to j$, add a new arrow $i \to j$.
\item Reverse the direction of all arrows incident to the vertex $k$.
\item Repeatedly remove directed $2$-cycles until unable to do so.
\end{enumerate}
\end{remark}

We say a quiver $\quiver'$ is \emph{mutation equivalent} to another quiver $\quiver$ if there exists a sequence  of mutations $\mutation_{j_1},\dots,\mutation_{j_{\ell}}$ which connects $\quiver'$ and $\quiver$, that is,
\[
\quiver' = (\mutation_{j_{\ell}} \cdots \mutation_{j_1})(\quiver).
\]

An immediate check shows that $\mutation_k(\seed)$ is again a seed, and a mutation is an involution, that is, its square is the identity. Since the adjacency matrix of a quiver $\quiver$ is skew-symmetric, we sometimes denote by 
\[
\seed = (\mathbf{x},\quiver) = (\mathbf{x}, \qbasis(\quiver)).
\] 
Also, note that the mutation on seeds does not change frozen variables $x_{n+1},\dots,x_m$. Let $\mathbb{T}_n$ denote the $n$-regular tree whose edges are labeled by $1,\dots,n$. Except for $n = 1$, there are infinitely many vertices on the tree $\mathbb{T}_n$. For example, we present regular trees $\mathbb{T}_2$ and $\mathbb{T}_3$ in Figure~\ref{figure_regular_trees_2_and_3}.

\begin{figure}[ht]
\begin{tabular}{cc}
\begin{tikzpicture}
\tikzset{every node/.style={scale=0.8}}
\tikzset{cnode/.style = {circle, fill,inner sep=0pt, minimum size= 1.5mm}}
\node[cnode] (1) {};
\node[cnode, right of=1 ] (2) {};
\node[cnode, right of=2 ] (3) {};
\node[cnode, right of=3 ] (4) {};	
\node[cnode, right of=4 ] (5) {};	
\node[cnode, right of=5 ] (6) {};	

\node[left of=1] {$\cdots$};
\node[right of=6] {$\cdots$};

\draw (1)--(2) node[above, midway] {$1$};
\draw (2)--(3) node[above, midway] {$2$};
\draw (3)--(4) node[above, midway] {$1$};
\draw (4)--(5) node[above, midway] {$2$};
\draw (5)--(6) node[above, midway] {$1$};
\end{tikzpicture} &
\begin{tikzpicture}
\tikzset{every node/.style={scale=0.8}}
\tikzset{cnode/.style = {circle, fill,inner sep=0pt, minimum size= 1.5mm}}
\node[cnode] (1) {};
\node[cnode, below right of =1] (2) {};
\node[cnode, below of =2] (3) {};
\node[cnode, above right of=2] (4){};
\node[cnode, above of =4] (5) {};
\node[cnode, below right of = 4] (6) {};
\node[cnode, below of= 6] (7) {};
\node[cnode, above right of = 6] (8) {};
\node[cnode, above of = 8] (9) {};
\node[cnode, below right of = 8] (10) {};
\node[cnode, below of = 10] (11) {};
\node[cnode, above right of = 10] (12) {};

\node[left of = 1] {$\cdots$};
\node[right of = 12] {$\cdots$};

\draw (1)--(2) node[above, midway, sloped] {$1$};
\draw (4)--(6) node[above, midway, sloped] {$1$};
\draw (8)--(10) node[above, midway, sloped] {$1$};

\foreach \x [evaluate ={ \x as \y using int(\x +2)} ] in {2, 6, 10}{
\draw (\x)--(\y)  node[below, midway, sloped] {$2$}; 
}
\foreach \x [evaluate = {\x as \y using int(\x +1)}] in {2, 4, 6, 8, 10}{
\draw (\x)--(\y) node[above, midway, sloped] {$3$};
}
\end{tikzpicture}\\[2ex]
$\mathbb{T}_2$ & $\mathbb{T}_3$
\end{tabular}
\caption{The $n$-regular trees for $n=2$ and $n = 3$.}
\label{figure_regular_trees_2_and_3}	
\end{figure}
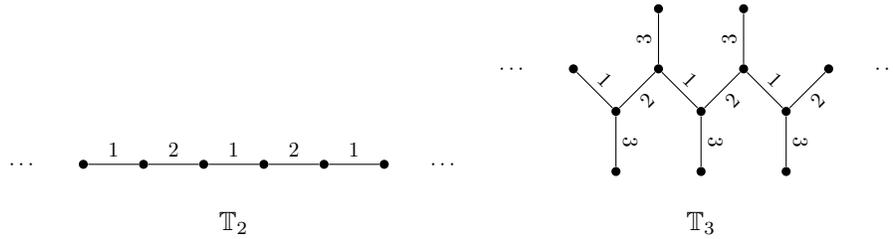
A \emph{cluster pattern} (or \emph{seed pattern}) is an assignment
\[
\mathbb{T}_n \to \{\text{seeds in } \field\}, \quad t \mapsto \seed_t = (\bfx_t, \qbasis_t)
\]
such that if $\begin{tikzcd} t \arrow[r,dash, "k"] & t' \end{tikzcd}$ in $\mathbb{T}_n$, then $\mutation_k(\seed_t) = \seed_{t'}$. Let $\{ \seed_t = (\bfx_t, \qbasis_t)\}_{t \in \mathbb{T}_n}$ be a cluster pattern with $\bfx_t = (x_{1;t},\dots,x_{m;t})$. Since the mutation does not change frozen variables, we may let $x_{n+1} = x_{n+1;t},\dots,x_m = x_{m;t}$.

\begin{definition}[{cf. \cite{FZ2_2003}}]
Let $\{ \seed_t = (\bfx_t, \qbasis_t)\}_{t \in \mathbb{T}_n}$ be a cluster pattern with $\bfx_t = (x_{1;t},\dots,x_{m;t})$. The \emph{cluster algebra \textup{(}of geometric type\textup{)}} $\cA(\{\seed_t\}_{t \in \mathbb{T}_n})$ is defined to be the $\C[x_{n+1},\dots,x_m]$-subalgebra of $\field$ generated by all the cluster variables $\bigcup_{t \in \mathbb{T}_n} \{x_{1;t},\dots,x_{n;t}\}$.
\end{definition}

If we fix a vertex $t_0 \in \mathbb{T}_n$, then a cluster pattern $\{ \seed_t \}_{t \in \mathbb{T}_n}$ is constructed from the seed $\initialseed$
and thus we simply denote by $\cA(\initialseed) = \cA(\{\seed_t\}_{t \in \mathbb{T}_n})$. 
In this case, we call $\initialseed$ an \emph{initial seed}. 
The cluster algebra does not depend on the choice of initial seed.

\begin{example}\label{example_A2_example}
Let $n = m = 2$. Suppose that an initial seed is given by 
\[
\initialseed = \left(
(x_1,x_2), \begin{pmatrix}
0 & 1  \\ -1 & 0
\end{pmatrix}
\right).
\]
We present a part of the cluster pattern obtained by the initial seed $\initialseed$.
\begin{center}
\begin{tikzcd}
\seed_{t_5} = \left( (x_2,x_1), 
\begin{pmatrix}
0 & -1 \\ 1 & 0
\end{pmatrix}
\right)
\arrow[r, color=white, "\textcolor{black}{\sim}" description]
& \initialseed
= \left(
(x_1,x_2), 
\begin{pmatrix}
0 & 1 \\ -1 & 0
\end{pmatrix}
\right) \arrow[d,<->, "\mutation_1"]
\\
\seed_{t_4} = \left(
(\frac{1+x_1}{x_2}, x_1), \begin{pmatrix}
0 & 1 \\ -1 & 0
\end{pmatrix}
\right)  \arrow[u,<->, "\mutation_1"]
& 
\left(
\left(\frac{1+x_2}{x_1}, x_2\right), \begin{pmatrix}
0 & -1 \\ 1 & 0
\end{pmatrix}
\right)  = \seed_{t_1} \arrow[d, <->,"\mutation_2"]\\
\seed_{t_3} = \left(
\left(\frac{1+x_1}{x_2}, \frac{1+x_1+x_2}{x_1x_2}\right),
\begin{pmatrix}
0 & -1 \\ 1 & 0
\end{pmatrix}
\right) \arrow[u,<->, "\mutation_2"]
&
\left(
\left(\frac{1+x_2}{x_1}, \frac{1+x_1+x_2}{x_1x_2}\right),
\begin{pmatrix}
0 & 1  \\ -1 & 0
\end{pmatrix}
\right) = \seed_{t_2} \arrow[l,<->, "\mutation_1"]
\end{tikzcd}
\end{center}
Accordingly, we have that
\[
\cA(\initialseed) = \cA(\{\seed_t\}_{t \in \mathbb{T}_n}) =  \C\left[x_1,x_2,\frac{1+x_2}{x_1}, \frac{1+x_1+x_2}{x_1x_2}, \frac{1+x_1}{x_2}\right].
\]
\end{example}

\begin{remark}\label{rmk_x_cluster_mutation}
There is another mutation operation called the \emph{cluster $\mathcal{X}$-mutation}. Let $\{ \seed_t = (\bfx_t, \qbasis_t)\}_{t \in \mathbb{T}_n}$ be a cluster pattern with $\bfx_t = (x_{1;t},\dots,x_{m;t})$. For $t \in \mathbb{T}_n$ and $j \in [n]$, we set $\mathbf{y}_t = (y_{1;t},\dots,y_{n;t})$ by
\[
y_{j;t} = \prod_{i \in [m]} x_{i;t}^{b^{(t)}_{i,j}}
\]
where $\qbasis_t = (b^{(t)}_{i,j})$. Then, the assignment $t \mapsto (\mathbf{y}_t, \qbasis_t)$ is called a cluster \emph{$Y$-pattern} and for $\begin{tikzcd} t \arrow[r,dash, "k"] & t' \end{tikzcd}$ in $\mathbb{T}_n$, we have
\[
y_{i;t'} = \begin{cases}
\displaystyle y_{i;t} y_{k;t}^{\max\{b_{k,i}^{(t)},0\}}(1+y_{k;t})^{-b_{k,i}^{(t)}} & \text{ if }i \neq k, \\
y_{k;t}^{-1} &\text{ otherwise;}
\end{cases}
\]
see~\cite[Proposition~3.9]{FZ4_2007}. For $\begin{tikzcd} t \arrow[r,dash, "k"] & t' \end{tikzcd}$ in $\mathbb{T}_n$, the operation sends $(\mathbf y_t, \qbasis_t)$ to $(\mathbf y_{t'}, \qbasis_{t'})$ is called the \emph{cluster $\mathcal{X}$-mutation} (or, \emph{$\mathcal{X}$-cluster mutation}). For exchange matrices and quivers, the cluster $\mathcal{X}$-mutation is defined the same as before. 
\end{remark}

We say that a quiver $\quiver$ is \emph{acyclic} if it does not have directed cycles.
Similarly, for a skew-symmetrizable matrix $\qbasis = (b_{i,j})$, we say that it is \emph{acyclic} if there are no sequences $j_1,j_2,\dots,j_{\ell}$ with $\ell \ge 3$ such that
\[
b_{j_1,j_2}, b_{j_2,j_3},\dots,b_{j_{\ell-1},j_{\ell}},b_{j_{\ell},j_1} > 0.
\]
We say a seed $\seed = (\mathbf x, \qbasis)$ is \emph{acyclic} if so is $\qbasis$.
The \textit{Cartan counterpart} $C(\qbpr) = (c_{i,j})$ of the principal part $\qbpr$ of an exchange matrix $\qbasis$ is defined by 
\[
c_{i,j} = \begin{cases}
2 & \text{ if } i = j,  \\
-|b_{i,j}| & \text{ if } i \neq j.
\end{cases}
\]

\begin{definition}
For a Dynkin type $\dynX$, we define a quiver $\quiver$, a matrix $\qbasis$, a cluster pattern $\{\seed_t\}_{t \in \mathbb{T}_n}$, or a cluster algebra $\cA(\initialseed)$ \emph{of type~$\dynX$} as follows.
\begin{enumerate}
\item A quiver is \textit{of type~$\dynX$} 
if it is mutation equivalent to an \emph{acyclic} quiver whose underlying graph is isomorphic to the Dynkin 
diagram of type $\dynX$.
\item A skew-symmetrizable matrix is \textit{of 
type $\dynX$} if it is mutation equivalent to an acyclic skew-symmetrizable matrix whose 
Cartan counterpart $C(\qbasis)$ is isomorphic to the Cartan matrix of type~$\dynX$. 
\item A cluster pattern $\{\seed_t\}_{t \in \mathbb{T}_n}$ is \textit{of type $\dynX$} if for some $t \in \mathbb{T}_n$, the principal part $\qbpr_t$ of the exchange matrix $\qbasis_t$ is of type $\dynX$.
\item A cluster algebra $\cA(\initialseed)$ is \textit{of type $\dynX$} if its cluster pattern is of type $\dynX$.
\end{enumerate}
\end{definition}

Here, we say that two matrices $C_1$ and $C_2$ are \emph{isomorphic} if they are conjugate to each other via a permutation matrix, that is, $C_2 = P^{-1} C_1 P$ for some permutation matrix~$P$. 
It is proved in~\cite[Corollary~4]{CalderoKeller06} that if two acyclic skew-symmetrizable matrices are mutation equivalent, then there exists a sequence of mutations from one to other such that intermediate skew-symmetrizable matrices are all acyclic. 
Indeed, if two acyclic skew-symmetrizable matrices are mutation equivalent, then their Cartan counterparts are isomorphic. 
Accordingly, a quiver or a matrix of type $\dynX$ is well-defined.

\begin{assumption}\label{assumption_affine}
Throughout this paper, we assume that for any cluster algebra, the principal part $\qbpr_{t_0}$ of the initial exchange matrix is acyclic of \textit{affine} type unless mentioned otherwise. 
\end{assumption}

\subsection{Combinatorics of exchange graphs}
\label{sec_comb_of_exchange_graphs}

The \emph{exchange graph} of a cluster pattern is the $n$-regular (finite or infinite) connected graph whose vertices are the seeds of the cluster pattern and whose edges connect the seeds related by a single mutation. 

\begin{definition} 
The \emph{exchange graph} $\exchange(\cA)$ of the cluster algebra $\cA$ is a quotient of the tree $\mathbb{T}_n$ modulo the equivalence relation on vertices defined by setting $t \sim t'$ if and only if $\seed_t \sim \seed_{t'}$. 
\end{definition} 

For example, the exchange graph in Example~\ref{example_A2_example} is a cycle graph with~$5$~vertices. We regard a seed as a vertex of the exchange graph.
For $\initialseed = (\mathbf x_{t_0}, \qbasis_{t_0})$,
the cluster algebra $\cA(\initialseed)$ is said to have \emph{principal coefficients} if the exchange matrix $\qbasis_{t_0}$ is a $(2n \times n)$-matrix of the form $\begin{pmatrix}
\qbpr_{t_0} \\ \clusterfont{I}_n
\end{pmatrix}$, and have \emph{trivial coefficients} if $\qbasis_{t_0}=\qbpr_{t_0}$.
Here $\clusterfont{I}_n$ is the identity matrix of size~$n \times n$.
We recall the following result on the combinatorics of exchange graphs.
\begin{theorem}[{\cite[Theorem~4.6]{FZ4_2007}}]\label{thm_exchange_graph_covering}
The exchange graph of an arbitrary cluster algebra $\cA$ is covered by the exchange graph of the cluster algebra $\cA(\initialseed)$ having principal coefficients and the set of principal part of exchange matrices are the same. 
\end{theorem}

One of the direct consequence is that the exchange graph of the cluster algebra $\cA(\initialseed)$ having trivial coefficients is covered by the exchange graph of the cluster algebra $\cA(\tilde\seed_{t_0})$ whose exchange matrix has the same principal part of $\initialseed$.
Therefore, for a fixed principal part of the exchange matrix, the cluster algebra having principal coefficients has the largest exchange graph while that having trivial coefficients has the smallest one (see~\cite[Section~4]{FZ4_2007}).

However, it is unknown whether the largest exchange graph is strictly larger than the smallest one or not. Indeed, it is conjectured in \cite[Conjecture~4.3]{FZ4_2007} that the exchange graph $\exchange(\cA)$ is determined by the principal part $\qbpr_{t_0}$ only.
The conjecture is confirmed for finite cases~\cite{FZ2_2003} or exchange matrices coming from quivers~\cite{IKLP13} as follows:

\begin{theorem}[{\cite[Theorem~1.13]{FZ2_2003}; \cite[Theorem~4.6]{IKLP13}}]\label{thm_exchange_graph_skew_symetric}
Let $\initialseed = (\mathbf x_{t_0}, \qbasis_{t_0})$ be an initial seed.
If the principal part~$\qbpr_{t_0}$ of $\qbasis_{t_0}$ is \emph{of finite type} or \emph{skew-symmetric}, then the exchange graph of a cluster algebra $\mathcal A(\initialseed)$ only depends on the principal part $\qbpr_{t_0}$ of the exchange matrix $\qbasis_{t_0}$.
\end{theorem}

We furthermore extend this result to cluster algebras whose initial exchange matrices are of affine type. We will prove this theorem later in Section~\ref{sec:folding}.
\begin{theorem}\label{theorem:principal coefficients and trivial coefficients}
Let $\initialseed = (\mathbf x_{t_0}, \qbasis_{t_0})$ be an initial seed.
If the principal part $\qbpr_{t_0}$ of $\qbasis_{t_0}$ is \emph{of affine type}, then the exchange graph of a cluster algebra $\mathcal A(\initialseed)$ only depends on the principal part $\qbpr_{t_0}$ of the exchange matrix $\qbasis_{t_0}$.
\end{theorem}

Because of Assumption~\ref{assumption_affine} and Theorem~\ref{theorem:principal coefficients and trivial coefficients}, we simply denote the exchange graph by
$\exchange(\qbpr) = \exchange(\cA(\mathbf x, \qbasis))$.
In Tables~\ref{Dynkin} and~\ref{table_twisted_affine}, we present lists of standard affine root 
systems and twisted affine root systems, respectively. They are the same as 
presented in Tables Aff 1, Aff 2, and Aff 3 of~\cite[Chapter~4]{Kac83}, and we 
denote by $\exdynX = \dynX^{(1)}$. 
We notice that the number of vertices of the standard affine Dynkin diagram of type $\exdynX_{n-1}$ is $n$ while we do not specify the vertex numbering. 

For a Dynkin type $\dynX$, we say that $\dynX$ is \emph{simply-laced} if its Dynkin diagram has only single edges, otherwise, $\dynX$ is \emph{non-simply-laced}.
Recall that the Cartan matrix associated to a Dynkin diagram $\dynX$ can be read directly from the diagram $\dynX$ as follows:
\begin{center}
\setlength{\tabcolsep}{20pt}
\begin{tabular}{ccccc}
\begin{tikzpicture}[scale =.5, baseline=-.5ex]
\tikzset{every node/.style={scale=0.7}}

\node[Dnode, label=below:{$i$}] (2) {};
\node[Dnode, label=below:{$j$}] (3) [right=of 2] {};

\draw (2)-- (3);

\end{tikzpicture}&
\begin{tikzpicture}[scale =.5, baseline=-.5ex]
\tikzset{every node/.style={scale=0.7}}

\node[Dnode, label=below:{$i$}] (2) {};
\node[Dnode, label=below:{$j$}] (3) [right=of 2] {};

\draw[double line] (2)-- node{\scalebox{1.3}{$>$}} (3);

\end{tikzpicture}&
\begin{tikzpicture}[scale =.5, baseline=-.5ex]
\tikzset{every node/.style={scale=0.7}}

\node[Dnode,  label=below:{$i$}] (2) {};
\node[Dnode, label=below:{$j$}] (3) [right=of 2] {};

\draw[triple line] (2)-- node{\scalebox{1.3}{$>$}} (3);
\draw (2)--(3);

\end{tikzpicture}
&
\begin{tikzpicture}[scale=.5, baseline=-.5ex]
\tikzset{every node/.style={scale=0.7}}

\node[Dnode, label=below:{$i$}] (1) {};
\node[Dnode, label=below:{$j$}] (2) [right = of 1] {};

\draw[double distance = 2.7pt] (1)--(2);
\draw[double distance = 0.9pt] (1)-- node{\scalebox{1.3}{ $>$}} (2);			
\end{tikzpicture} 
&
\begin{tikzpicture}[scale=.5, baseline=-.5ex]
\tikzset{every node/.style={scale=0.7}}
\node[Dnode, label=below:{$i$}] (1) {};
\node[Dnode, label=below:{$j$}] (2) [right = of 1] {};

\draw[double line] (1)-- node[pos=0.2]{\scalebox{1.3}{ $<$}} 
node[pos=0.7]{\scalebox{1.3}{ $>$}} (2);

\end{tikzpicture}\\
$c_{i,j} = -1$
&$c_{i,j} = -2$
& $c_{i,j} = -3$
& $c_{i,j} = -4$
& $c_{i,j} = -2$ \\
$c_{j,i} = -1$
&$c_{j,i} = -1$
& $c_{j,i} = -1$
& $c_{j,i} = -1$
& $c_{j,i} = -2$
\end{tabular}
\end{center}
For example, the Cartan matrix
of the diagram \begin{tikzpicture}[scale =.5, baseline=-.5ex]
\tikzset{every node/.style={scale=0.7}}

\node[Dnode, label=below:{$1$}] (1) {};
\node[Dnode, label=below:{$2$}] (2) [right = of 1] {};
\node[Dnode, label=below:{$3$}] (3) [right=of 2] {};

\draw[triple line] (2)-- node{\scalebox{1.3}{$>$}} (3);
\draw (1)--(2);
\draw (2)--(3);

\end{tikzpicture} of type
$\exdynG_2$ is 
\begin{equation}\label{eq_Cartan_G2}
\begin{bmatrix}
2 & -1 & 0 \\
-1 & 2 & -3 \\
0 & -1 & 2
\end{bmatrix}.
\end{equation}
Therefore, for each non-simply-laced Dynkin diagram $\dynX$, any exchange matrix $\qbasis$ of type $\dynX$ is \emph{not} skew-symmetric but skew-symmetrizable. Hence it never come from any quiver.

The Dynkin diagrams of standard affine root systems do not have cycles except 
of type $\exdynA_{n-1}$ for $n \geq 3$. We consider \textit{bipartite} coloring 
on affine Dynkin diagrams except of type $\exdynA$, that is, we color the set 
of vertices with black or white such that for any edge connecting $i$ and $j$, 
two vertices $i$ and $j$ have different colors. The coloring defines an 
orientation on the directed graph $\Gamma(\qbasis)$ such that sinks are colored 
in black. This is equivalent to saying that each nonzero entry $b_{i,j}$ of the 
matrix $\qbpr$ has positive sign if and only if $i$ is white and $j$ is black. 
Accordingly, the bipartite coloring on each affine Dynkin diagram of type 
$\dynX$ determines a quiver of type $\dynX$.

\begin{table}[ht]
\begin{center}
\begin{tabular}{l|l}
\toprule
$\Roots$ & Dynkin diagram \\
\midrule
$\exdynA_1$  &
\begin{tikzpicture}[scale=.5, baseline=-.5ex]
\tikzset{every node/.style={scale=0.7}}
\node[Dnode] (1) {};
\node[Dnode] (2) [right = of 1] {};

\draw[double line] (1)-- node[pos=0.2]{\scalebox{1.3}{ $<$}} node[pos=0.7]{\scalebox{1.3}{ $>$}} (2);

\end{tikzpicture}  \\ 
$\exdynA_{n-1}$ $(n \geq 3)$  &
\begin{tikzpicture}[scale=.5, baseline=-.5ex]
\tikzset{every node/.style={scale=0.7}}
\node[Dnode] (1) {};
\node[Dnode] (2) [right = of 1] {};
\node[Dnode] (3) [right = of 2] {};
\node[Dnode] (4) [right =of 3] {};
\node[Dnode] (5) [right =of 4] {};			
\node[Dnode] (6) [above =of 3] {};			

\draw (1)--(2)--(3)
(4)--(5)
(1)--(6)--(5);
\draw[dotted] (3)--(4);

\end{tikzpicture}  \\ 

$\exdynB_{n-1}$ $(n \geq 4)$  &
\begin{tikzpicture}[scale=.5, baseline=-.5ex]
\tikzset{every node/.style={scale=0.7}}

\node[Dnode, fill=black] (1) {};
\node[Dnode] (2) [right = of 1] {};
\node[Dnode] (3) [right = of 2] {};
\node[Dnode] (4) [right =of 3] {};
\node[Dnode] (5) [right =of 4] {};
\node[Dnode] (6) [below left = 0.6cm and 0.6cm of 1] {};
\node[Dnode] (7) [above left = 0.6cm and 0.6cm of 1] {};			

\draw (1)--(2)
(3)--(4)
(6)--(1)--(7);
\draw [dotted] (2)--(3);
\draw[double line] (4)-- node{\scalebox{1.3}{ $>$}} (5);
\end{tikzpicture}  \\ 			
$\exdynC_{n-1}$ $(n \geq 3)$ & 
\begin{tikzpicture}[scale=.5, baseline=-.5ex]
\tikzset{every node/.style={scale=0.7}}

\node[Dnode,fill=black] (1) {};
\node[Dnode] (2) [right = of 1] {};
\node[Dnode] (3) [right = of 2] {};
\node[Dnode] (4) [right =of 3] {};
\node[Dnode] (5) [right =of 4] {};			
\node[Dnode] (6) [left =of 1] {};
\draw (1)--(2)
(3)--(4);
\draw [dotted] (2)--(3);
\draw[double line] (4)-- node{\scalebox{1.3}{ $<$}} (5);
\draw[double line] (1)-- node{\scalebox{1.3}{ $>$}} (6);			
\end{tikzpicture}  \\ 	

$\exdynD_{n-1}$ $(n \geq 5)$ & 
\begin{tikzpicture}[scale=.5, baseline=-.5ex]
\tikzset{every node/.style={scale=0.7}}

\node[Dnode, fill=black,] (1) {};
\node[Dnode] (2) [right = of 1] {};
\node[Dnode] (3) [right = of 2] {};
\node[Dnode] (4) [right =of 3] {};			

\node[Dnode] (5) [ below right = 0.6cm and 0.6cm of 4] {};
\node[Dnode] (6) [above right=  0.6cm and 0.6cm of 4] {};

\node[Dnode] (7) [above left = 0.6cm and 0.6cm of 1] {};
\node[Dnode] (8) [below left=0.6cm and 0.6cm of 1] {};

\draw(1)--(2)
(3)--(4)--(5)
(4)--(6)
(7)--(1)--(8);
\draw[dotted] (2)--(3);
\end{tikzpicture}
\\ 
$\exdynE_6$ & 

\begin{tikzpicture}[scale=.5, baseline=-.5ex]
\tikzset{every node/.style={scale=0.7}}

\node[Dnode] (1) {};
\node[Dnode, fill=black] (3) [right=of 1] {};
\node[Dnode] (4) [right=of 3] {};
\node[Dnode, fill=black] (2) [above=of 4] {};
\node[Dnode,fill=black] (5) [right=of 4] {};
\node[Dnode] (6) [right=of 5]{};
\node[Dnode] (7) [above=of 2]{};

\draw(1)--(3)--(4)--(5)--(6)
(7)--(2)--(4);
\end{tikzpicture}			
\\
$\exdynE_7$ & 
\begin{tikzpicture}[scale=.5, baseline=-.5ex]
\tikzset{every node/.style={scale=0.7}}

\node[Dnode] (1) {};
\node[Dnode,fill=black] (8) [left=of 1] {};
\node[Dnode, fill=black] (3) [right=of 1] {};
\node[Dnode] (4) [right=of 3] {};
\node[Dnode, fill=black] (2) [above=of 4] {};
\node[Dnode,fill=black] (5) [right=of 4] {};
\node[Dnode] (6) [right=of 5]{};
\node[Dnode, fill=black] (7) [right=of 6]{};

\draw (8)--(1)--(3)--(4)--(5)--(6)--(7)
(2)--(4);
\end{tikzpicture}	
\\
$\exdynE_8$ & 
\begin{tikzpicture}[scale=.5, baseline=-.5ex]
\tikzset{every node/.style={scale=0.7}}

\node[Dnode] (1) {};
\node[Dnode, fill=black] (3) [right=of 1] {};
\node[Dnode] (4) [right=of 3] {};
\node[Dnode, fill=black] (2) [above=of 4] {};
\node[Dnode, fill=black] (5) [right=of 4] {};
\node[Dnode] (6) [right=of 5]{};
\node[Dnode, fill=black] (7) [right=of 6]{};
\node[Dnode] (8) [right=of 7]{};
\node[Dnode, fill=black] (9) [right=of 8]{};

\draw(1)--(3)--(4)--(5)--(6)--(7)--(8)--(9)
(2)--(4);
\end{tikzpicture}
\\
$\exdynF_4$ 
&
\begin{tikzpicture}[scale = .5, baseline=-.5ex]
\tikzset{every node/.style={scale=0.7}}

\node[Dnode,fill=black] (1) {};
\node[Dnode] (2) [right = of 1] {};
\node[Dnode, fill=black] (3) [right = of 2] {};
\node[Dnode] (4) [right =of 3] {};
\node[Dnode, fill=black] (5) [right =of 4] {};

\draw (1)--(2)
(2)--(3)
(4)--(5);
\draw[double line] (3)-- node{\scalebox{1.3}{$>$}} (4);
\end{tikzpicture}  \\
$\exdynG_2$
&
\begin{tikzpicture}[scale =.5, baseline=-.5ex]
\tikzset{every node/.style={scale=0.7}}

\node[Dnode] (1) {};
\node[Dnode, fill=black] (2) [right = of 1] {};
\node[Dnode] (3) [right=of 2] {};

\draw[triple line] (2)-- node{\scalebox{1.3}{$>$}} (3);
\draw (1)--(2);
\draw (2)--(3);

\end{tikzpicture}\\
\bottomrule
\end{tabular}
\end{center}
\caption{Dynkin diagrams of standard affine root systems}\label{Dynkin}
\end{table}

\begin{table}[ht]
\begin{tabular}{l|l}
\toprule
$\Roots$ & Dynkin diagram \\
\midrule
$\dynA_2^{(2)}$ & 
\begin{tikzpicture}[scale=.5, baseline=-.5ex]
\tikzset{every node/.style={scale=0.7}}

\node[Dnode] (1) {};
\node[Dnode] (2) [right = of 1] {};

\draw[double distance = 2.7pt] (1)--(2);
\draw[double distance = 0.9pt] (1)-- node{\scalebox{1.3}{ $<$}} (2);			
\end{tikzpicture} 
\\
$\dynA_{2(n-1)}^{(2)}$ ($n \ge 3$)
&
\begin{tikzpicture}[scale=.5, baseline=-.5ex]
\tikzset{every node/.style={scale=0.7}}

\node[Dnode, fill=black] (1) {};
\node[Dnode] (2) [right = of 1] {};
\node[Dnode] (3) [right = of 2] {};
\node[Dnode] (4) [right =of 3] {};
\node[Dnode] (5) [right =of 4] {};			
\node[Dnode] (6) [left =of 1] {};
\draw (1)--(2)
(3)--(4);
\draw [dotted] (2)--(3);
\draw[double line] (4)-- node{\scalebox{1.3}{ $>$}} (5);
\draw[double line] (1)-- node{\scalebox{1.3}{ $>$}} (6);			
\end{tikzpicture}  \\ 
$\dynA_{2(n-1)-1}^{(2)}$ ($n \ge 4$) & 
\begin{tikzpicture}[scale=.5, baseline=-.5ex]
\tikzset{every node/.style={scale=0.7}}

\node[Dnode, fill=black] (1) {};
\node[Dnode] (2) [right = of 1] {};
\node[Dnode] (3) [right = of 2] {};
\node[Dnode] (4) [right =of 3] {};
\node[Dnode] (5) [right =of 4] {};
\node[Dnode] (6) [below left = 0.6cm and 0.6cm of 1] {};
\node[Dnode] (7) [above left = 0.6cm and 0.6cm of 1] {};			

\draw (1)--(2)
(3)--(4)
(6)--(1)--(7);
\draw [dotted] (2)--(3);
\draw[double line] (4)-- node{\scalebox{1.3}{ $<$}} (5);
\end{tikzpicture}  \\ 		
$\dynD_{n}^{(2)}$ ($n \ge 3$)  & 
\begin{tikzpicture}[scale=.5, baseline=-.5ex]
\tikzset{every node/.style={scale=0.7}}

\node[Dnode] (1) {};
\node[Dnode] (2) [right = of 1] {};
\node[Dnode] (3) [right = of 2] {};
\node[Dnode] (4) [right =of 3] {};
\node[Dnode] (5) [right =of 4] {};			
\node[Dnode] (6) [left =of 1] {};
\draw (1)--(2)
(3)--(4);
\draw [dotted] (2)--(3);
\draw[double line] (4)-- node{\scalebox{1.3}{ $>$}} (5);
\draw[double line] (1)-- node{\scalebox{1.3}{ $<$}} (6);			
\end{tikzpicture}  \\ 
$\dynE_6^{(2)}$  &
\begin{tikzpicture}[scale = .5, baseline=-.5ex]
\tikzset{every node/.style={scale=0.7}}

\node[Dnode, fill=black] (1) {};
\node[Dnode] (2) [right = of 1] {};
\node[Dnode, fill=black] (3) [right = of 2] {};
\node[Dnode] (4) [right =of 3] {};
\node[Dnode, fill=black] (5) [right =of 4] {};

\draw (1)--(2)
(2)--(3)
(4)--(5);
\draw[double line] (3)-- node{\scalebox{1.3}{$<$}} (4);
\end{tikzpicture}  \\
$\dynD_4^{(3)}$  &
\begin{tikzpicture}[scale =.5, baseline=-.5ex]
\tikzset{every node/.style={scale=0.7}}

\node[Dnode] (1) {};
\node[Dnode, fill=black] (2) [right = of 1] {};
\node[Dnode] (3) [right=of 2] {};

\draw[triple line] (2)-- node{\scalebox{1.3}{$<$}} (3);
\draw (1)--(2);
\draw (2)--(3);

\end{tikzpicture}\\
\bottomrule
\end{tabular}
\caption{Dynkin diagrams of twisted affine root systems}\label{table_twisted_affine}
\end{table}

\begin{remark}
Any cycle with $n$ vertices defines a quiver of type $\exdynA_{n-1}$. If a 
quiver is a directed $n$-cycle, then it is mutation equivalent to a quiver of 
type $\dynD_n$ (see~Type IV in \cite{Vatne10}). Recall 
from~\cite[Lemma~6.8]{FST08} the mutation equivalence class in this case. Let 
$\quiver$ and $\quiver'$ are two $n$-cycles for $n \geq 3$. Suppose that in 
$\quiver$, there are $p$ edges of one direction and $q = n - p$ edges of the 
opposite direction. Also, in $\quiver'$, there are $p'$ edges of one direction 
and $q' = n - p'$ edges of the opposite direction. Then two quivers $\quiver$ 
and $\quiver'$ are mutation equivalent if and only if the unordered pairs 
$\{p,q\}$ and $\{p',q'\}$ coincide. As we already mentioned, if $p = 0$ or $p = 
n$, then the quiver $\quiver'$ is of type $\dynD_n$. We say that a quiver 
$\quiver$ is of type $\exdynA_{p,q}$ if it has $p$ edges of one direction and 
$q$ edges of the opposite direction. We depict some examples for quivers of 
type $\exdynA_{p,q}$ in Figure~\ref{fig_example_Apq}.
\end{remark}

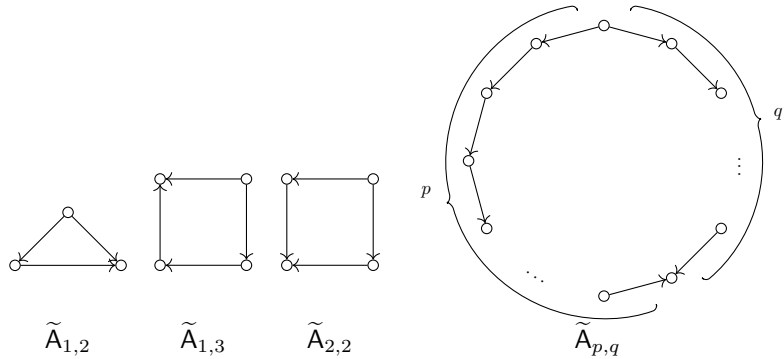
\begin{figure}[ht]
\begin{tabular}{cccc}
\begin{tikzpicture}[scale = 0.5]			
\tikzset{every node/.style={scale=0.7}}
\node[Dnode] (3) {};
\node[Dnode] (1) [below left = 0.6cm and 0.6cm of 3]{};
\node[Dnode] (2) [below right = 0.6cm and 0.6cm  of 3] {};

\draw[->] (3)--(1);
\draw[->] (3)--(2);
\draw[->] (1)--(2);
\end{tikzpicture} 
&	\begin{tikzpicture}[scale = 0.5]			
\tikzset{every node/.style={scale=0.7}}
\node[Dnode] (1) {};
\node[Dnode] (2) [below = of 1] {};
\node[Dnode] (3) [right = of 2] {};
\node[Dnode] (4) [above = of 3] {};

\draw[->] (4) -- (1);
\draw[->]	(4) -- (3);
\draw[->]	(3) -- (2);
\draw[->]	(2) --(1);
\end{tikzpicture} 
&	\begin{tikzpicture}[scale = 0.5]			
\tikzset{every node/.style={scale=0.7}}
\node[Dnode] (1) {};
\node[Dnode] (2) [below = of 1] {};
\node[Dnode] (3) [right = of 2] {};
\node[Dnode] (4) [above = of 3] {};

\draw[->] (4) -- (1);
\draw[->]	(4) -- (3);
\draw[->]	(3) -- (2);
\draw[->]	(1)--(2);
\end{tikzpicture} 
&
\raisebox{4em}{	\begin{tikzpicture}[baseline=-.5ex,scale=0.6]
\tikzstyle{state}=[draw, circle, inner sep = 0.07cm]
\tikzset{every node/.style={scale=0.7}}
\foreach \x in {1,...,7, 9,10,11}{
\node[Dnode] (\x) at (\x*30:3) {};
}
\node(12) at (12*30:3) {$\vdots$};
\node[rotate=-30] (8) at (8*30:3) {$\cdots$};

\foreach \x [evaluate={\y=int(\x+1);}] in {3,...,6,9}{
\draw[->] (\x)--(\y);
}
\foreach \x [evaluate={\y=int(\x-1);}] in {3,2,11}{
\draw[->] (\x)--(\y);
} 
\curlybrace[]{100}{290}{3.5};
\draw (190:4) node[rotate=0] {$p$};
\curlybrace[]{-50}{80}{3.5};
\draw (15:4) node[rotate=0] {$q$};
\end{tikzpicture}}
\\	
$\exdynA_{1,2}$ & $\exdynA_{1,3}$ & $\exdynA_{2,2}$ & $\exdynA_{p,q}$
\end{tabular}
\caption{Quivers of type $\exdynA_{p,q}$.}\label{fig_example_Apq}
\end{figure}

Let $\quiver$ be a quiver having bipartite coloring, that is, each vertex is either source or sink. Let $I_+ \subset [n]$ be the set of sources (that is, white vertices); and let $I_- \subset [n]$ be the set of sinks (that is, black vertices). Then we have $[n] = I_+ \sqcup I_-$. We consider the composition
$\qcoxeter = \mutation_- \mutation_+$ of a sequence of mutations where
\[
\mutation_{\varepsilon} = \prod_{i \in I_{\varepsilon}} \mutation_i \qquad \text{ for } \varepsilon \in \{ +, -\}.
\]
We call $\qcoxeter$ the \textit{Coxeter mutation}. Because of the definition, we have 
\[
\qcoxeter(\qbpr) = (\mutation_+ \mutation_-)(\qbpr) = \qbpr \quad \text{ and } 
\quad \qcoxeter^{-1}(\qbpr)= (\mutation_- \mutation_+)(\qbpr) = \qbpr.
\] 
The initial seed $\seed_{t_0} = \seed_0 = (\bfx_0, \qbasis_0)$ is included in a \textit{bipartite belt} consisting of the seeds $\seed_r = (\bfx_r, \qbasis_0)$ for $r \in \Z$ defined by 
\[
\seed_r = (\bfx_r, \qbasis_0) = \begin{cases}
\qcoxeter^r(\seed_0) & \text{ if } r > 0, \\
(\mutation_+ \mutation_-)^{-r}(\seed_0) & \text{ if } r < 0.
\end{cases}
\]
We write 
\[
\bfx_r = (x_{1;r},\dots,x_{n;r}) \quad \text{ for }r \in \Z.
\]

Let $\Phi$ be the root system defined by the Cartan counterpart of $\qbpr$. Let $\Pi$ be the set of simple roots $\alpha_1,\dots,\alpha_n$. We denote by $\Phi^+$ the set of positive roots.  
The positivity of Laurent phenomenon, which was conjectured by Fomin and Zelevinsky in~\cite{FZ1_2002}, and proved by Gross, Hacking, Keel, and Kontsevich in~\cite[Corollary~0.4]{GHKK18}, states that every non-zero cluster variable $z$ can be uniquely written as 
\[
z = \frac{f(\bfx_{t_0})}{x_{1;0}^{d_1} \cdots x_{n;0}^{d_n}}
\]
where $f$ is a polynomial with nonnegative integer coefficients in the cluster variables $x_{1;0},\dots,x_{n;0}$ and it is not divisible by any cluster variables $x_{1;0},\dots,x_{n;0}$. 
The \textit{denominator vector} $\mathbf d(z) =  \mathbf d_{\bfx_{t_0}}(z)$ of $z$ with respect to the cluster $\bfx_{t_0}$ is defined by
\[
\mathbf d(z) = \mathbf d_{\bfx_{t_0}}(z) = \sum_{i=1}^n d_i \alpha_i. 
\]
For example, for the initial seed $\initialseed = (\bfx_{t_0}, \qbasis_{t_0})$, we have $\mathbf d(x_{i;0}) = -\alpha_i$ for all $i \in [n]$. Using these terminologies, we recall the following:
\begin{theorem}[{\cite[Theorems~1.1 and 1.2]{ReadingStella20}}]\label{thm_RS_theorems} 
Suppose that the principal part $\qbpr_{t_0}$ of the exchange matrix in the initial seed is acyclic and its Cartan counterpart is of affine type. Let $\Phi$ be the associated root system with simple roots $\alpha_1,\dots,\alpha_n$. Then, the map from the cluster variables in $\cA(\qbpr_{t_0})$ defined by
\(
z \mapsto \mathbf d(z) 
\)
is injective and the image lies in $\Phi$. Moreover, collecting the nonnegative linear span of $\mathbf{d}$-vectors of cluster variables in each seed, we get a simplicial fan such that the dual graph of its underlying simplicial complex is isomorphic to the exchange graph~$\exchange(\qbpr)$.
\end{theorem}

The above theorem provides so-called \textit{affine almost positive roots model} for an affine root system. By analyzing the affine almost positive roots, they also provide the following results.
\begin{theorem}[{\cite[Propositions 5.4 and 5.14]{ReadingStella20}}]\label{thm_ReadingStella}
Suppose that the principal part $\qbpr_{t_0}$ of the exchange matrix in the initial seed $\initialseed= (\bfx_{t_0}, \qbasis_{t_0})$ is acyclic and its Cartan counterpart is of affine type. 
\begin{enumerate}
\item The Coxeter mutation $\qcoxeter$ acts on 
the exchange graph $\exchange(\qbpr_{t_0})$.
\item For $\ell \in [n]$ and $r \in \Z$, we denote by $\exchangesub{\qbpr_{t_0}}{x_{\ell;r}}$ the induced subgraph  of $\exchange(\qbpr_{t_0})$ consisting of seeds having the cluster variable $x_{\ell;r}$. Then, we have
\[
\exchangesub{\qbpr_{t_0}}{x_{\ell;r}} \cong \exchange(\qbpr_{t_0}|_{[n] \setminus \{\ell\}}).
\]
\item For a seed $\seed = (\bfx, \qbasis)$, there exists $r\in \Z$ such that 
\[
|\{ x_{1;r},\dots,x_{n;r}\} \cap \{ x_{1},\dots,x_{n}  \}| \geq 2.
\]
\end{enumerate}
\end{theorem}

Because we rephrase statements in the paper~\cite{ReadingStella20} in terms of exchange graphs, we briefly explain how we convert their theorem in this form. Reading and Stella denoted by $\Fan^{\re}_c(\Phi)$ the fan in Theorem~\ref{thm_RS_theorems}, that is, each maximal cone of $\Fan^{\re}_c(\Phi)$ is nonnegative linear span of $\mathbf{d}$-vectors of
cluster variables in a seed. Moreover, it is also proved in~\cite{ReadingStella20} that  $\Fan^{\re}_c(\Phi)$ is isomorphic to \textit{the fan of $\mathbf{g}$-vector cones}. Since the exchange graph $\exchange(\qbpr)$ is isomorphic the dual graph of the fan of $\mathbf{g}$-vector cones for the cluster algebra   by Reading and Speyer~\cite[Corollaries~1.2 and~1.3]{ReadingSpeyer18},  the combinatorics of exchange graph can be obtained by considering the fan $\Fan^{\re}_c(\Phi)$. The paper~\cite{ReadingStella20} provides several properties of $\Fan^{\re}_c(\Phi)$, and we rephrase them in terms of exchange graphs.

As a direct consequence of Theorem~\ref{thm_ReadingStella}, we have the following lemma which will be used later. 
\begin{lemma}\label{lemma:normal form}
Suppose that the principal part $\qbpr_{t_0}$ of the exchange matrix in the initial seed is acyclic and its Cartan counterpart is of affine type. For any seed $\seed_t = (\bfx_t, \qbasis_t)$, there exist $\ell \in [n]$ and $r \in \Z$ such that two seeds $\seed_t$ and $\seed_r = \qcoxeter^r(\initialseed)$ are in the induced subgraph $\exchangesub{\qbpr_{t_0}}{x_{\ell;r}}$. Indeed, there is a sequence $j_1,\dots,j_L \in [n] \setminus \{\ell\}$ of indices such that
\begin{align*}
\qcoxeter^r(\initialseed), \mutation_{j_1}(\qcoxeter^r(\initialseed)),
(\mutation_{j_2} \mutation_{j_1})(\qcoxeter^r(\initialseed)), \dots,
(\mutation_{j_L} \cdots \mutation_{j_1})(\qcoxeter^r(\initialseed)) \in  \exchangesub{\qbpr_{t_0}}{x_{\ell;r}}
\end{align*}
and 
\[
\seed_t = (\mutation_{j_L} \cdots \mutation_{j_1})(\qcoxeter^r(\initialseed)).
\]
\end{lemma}
\begin{proof}
By Theorem~\ref{thm_ReadingStella}(3), there exists $r \in \Z$ and $\ell \in [n]$ such that  both seeds $\seed_t$ and $\qcoxeter^r(\initialseed)$ have the same cluster variable $x_{\ell;r}$. By Theorem~\ref{thm_ReadingStella}(2), the induced subgraph $\exchangesub{\qbpr_{t_0}}{x_{\ell;r}}$ consisting of seeds having the cluster variable $x_{\ell;r}$ is isomorphic to the exchange graph $\exchange(\qbpr_{t_0}|_{[n] \setminus \{\ell\}})$. Therefore, there exists a sequence of indices $j_1,\dots,j_L \in [n] \setminus \{\ell\}$ such that the sequence $\mutation_{j_1},\dots,\mutation_{j_L}$ of mutations connects the seed $\qcoxeter^r(\initialseed)$ and $\seed_t$ inside the graph $\exchangesub{\qbpr_{t_0}}{x_{\ell;r}}$ as desired.
\end{proof}

\begin{remark}\label{rmk_inifinitely_many_seeds_for_affine}
In general, there are infinitely many seeds in the bipartite belt $\{ \seed_r \mid r \in \Z\}$. It is proved in~\cite[Theorem~8.8]{FZ4_2007} that there are finitely many seeds in the bipartite belt if and only if the Cartan counterpart $C(\qbpr_{t_0})$ is a Cartan matrix of finite type. Indeed, there are finitely many seeds in the cluster pattern if and only if the Cartan counter part is a Cartan matrix of finite type.
\end{remark}

\subsection{Folding}\label{sec:folding}
Under certain conditions, one can \textit{fold} cluster patterns to produce new ones. This procedure is used to study cluster algebras of non-simply-laced affine type from those of simply-laced affine type (see Table~\ref{figure:all possible foldings}). As before, we fix $m, n \in \Z_{>0}$ such that $n \leq m$. In this section, we recall \textit{folding} of cluster algebras from~\cite{FWZ_chapter45}. We refer the reader to~\cite{Dupont08}.

Let $\quiver$ be a quiver on $[m]$. Let $G$ be a finite group acting on the set $[m]$. For $i, i' \in [m]$, the notation $i \sim i'$ will mean that $i$ and $i'$ lie in the same
$G$-orbit. To study folding of cluster algebras, we prepare some terminologies.

For each $g \in G$, let $\quiver' = g \cdot \quiver$ be the quiver whose adjacency matrix $\qbasis(\quiver') = (b_{i,j}')$ is given by 
\[
b_{i,j}' = b_{g(i),g(j)}.
\]
\begin{definition}[{cf.~\cite[\S4.4]{FWZ_chapter45} and~\cite[\S 3]{Dupont08}}]\label{definition:admissible quiver}
Let $\quiver$ be a quiver on $[m]$ and $G$ a finite
group acting on the set $[m]$.
\begin{enumerate} 
\item A quiver $\quiver$ is \emph{$G$-invariant} if $g \cdot \quiver = \quiver$ for any $g \in G$.
\item A $G$-invariant quiver $\quiver$ is \emph{$G$-admissible} if
\begin{enumerate}
\item for any $i \sim i'$, index $i$ is mutable if and only if so is $i'$;
\item for mutable indices $i \sim i'$, we have $b_{i,i'} = 0$;
\item for any $i \sim i'$, and any mutable $j$, we have $b_{i,j} b_{i',j} \geq 0$.
\end{enumerate}
\item For a $G$-admissible quiver $\quiver$, we call a $G$-orbit \emph{mutable} (respectively, \emph{frozen}) if it consists of mutable (respectively, frozen) vertices. 
\end{enumerate}        
\end{definition}
For a $G$-admissible quiver $\quiver$, we define the matrix $\qbasis^G = \qbasis(\quiver)^G = (b_{I,J}^G)$ whose rows (respectively, columns) are labeled by the $G$-orbits (respectively, mutable $G$-orbits) by
\[
b_{I,J}^G = \sum_{i \in I} b_{i,j}
\]
where $j$ is an arbitrary index in $J$. We then say $\qbasis^G$ is obtained from $\qbasis$ (or from the quiver $\quiver$) by \textit{folding} with respect to the given $G$-action.

\begin{remark}
We note that the $G$-admissibility and the folding can also be defined for exchange matrices. 
\end{remark}

\begin{example}\label{example_D4_to_G2}
Let $\quiver$ be a quiver of type $\exdynE_6$ whose adjacency matrix 
$\qbasis(\quiver)$ is 
\[
\qbasis(\quiver) = \begin{pmatrix}
0 & 1 & 0 & 1 & 0 & 1 & 0\\
-1 & 0 & -1 & 0 & 0 & 0 & 0 \\
0 & 1 & 0 & 0 & 0 & 0 & 0 \\
-1 & 0 & 0 & 0 & -1 & 0 & 0 \\
0 & 0 & 0  & 1 & 0 & 0 &  0\\
-1 & 0 & 0 & 0 & 0 & 0 & -1 \\
0 & 0 & 0 & 0 & 0 & 1 & 0 
\end{pmatrix}.
\]
Suppose that the finite group $G = \Z / 3 \Z$ acts on $[7]$ as depicted in Figure~\ref{fig_E6_Z3}. Here, we denote the generator of $G$ by $\tau$. We decorate vertices of the quiver $\quiver$ with white and black for presenting sources and sinks, respectively. One may check that the quiver $\quiver$ is $G$-admissible. By setting $I_1 = \{1\}$, $I_2 = \{2, 4, 6\}$, and $I_3 = \{3, 5, 7\}$, we obtain
\[
\begin{split}
b_{I_1,I_2}^G &= \sum_{i \in I_1} b_{i,2} = b_{1,2} = 1, \\
b_{I_1, I_3}^G &= \sum_{i \in I_1} b_{i,3} = b_{1,3} = 0, \\
b_{I_2, I_3}^G &= \sum_{i \in I_2} b_{i,3} = b_{2,3} + b_{4,3} + b_{6,3} = -1, \\
b_{I_2, I_1}^G &= \sum_{i \in I_2} b_{i,1} = b_{2,1} + b_{4,1} + b_{6,1} = -3, \\
b_{I_3, I_1}^G &= \sum_{i \in I_3} b_{i,1} = b_{3,1} + b_{5,1} + b_{7,1} = 0, \\
b_{I_3,I_2}^G &= \sum_{i \in I_3} b_{i, 2} = b_{3,2} + b_{5,2} + b_{7,2} = 1.
\end{split}
\]
Accordingly, we obtain the matrix 
\[
\qbasis^G = \begin{pmatrix} 
0 & 1 & 0 \\ -3 & 0 & -1 \\ 0 & 1 & 0 
\end{pmatrix}
\] 
whose Cartan counterpart is the Cartan matrix of type
$\exdynG_2$ (cf.~\eqref{eq_Cartan_G2}).
\end{example}

For a $G$-admissible quiver $\quiver$ and a mutable $G$-orbit $I$, we consider a composition of mutations given by
\[
\mutation_I = \prod_{i \in I} \mutation_i
\]
which is well-defined because of the definition of admissible quivers. We call $\mutation_I$ an \textit{orbit mutation}. If $\mutation_I(\quiver)$ is again $G$-admissible, then we have that
\begin{equation*}
(\mutation_I(\qbasis))^G = \mutation_I(\qbasis^G).
\end{equation*}
We notice that the quiver $\mutation_I(\quiver)$ may \textit{not} be $G$-admissible in general. Therefore, we present the following definition.
\begin{definition}
Let $G$ be a group acting on the vertex set of a quiver $\quiver$. We say that $\quiver$ is \emph{globally foldable} with respect to $G$ if $\quiver$ is $G$-admissible, and moreover, for any sequence of mutable $G$-orbits $I_1,\dots,I_\ell$, the quiver $(\mutation_{I_\ell} \dots \mutation_{I_1})(\quiver)$ is $G$-admissible.
\end{definition}

For a globally foldable quiver, we can fold all the seeds in the corresponding cluster pattern. 
Let $\field^G$ be the field of rational functions in $\#[m]/G$ independent variables. Let $\psi \colon \field \to \field^G$ be a surjective homomorphism. A seed $\seed = (\mathbf{x}, \qbasis(\quiver))$ is called \emph{$(G, \psi)$-invariant} (respectively, \emph{$(G,\psi)$-admissible}) if 
\begin{itemize}
\item for any $i \sim i'$, we have $\psi(x_i) = \psi(x_{i'})$;
\item $\quiver$ is $G$-invariant (respectively, $G$-admissible).
\end{itemize}
In this situation, we define a new ``folded'' seed $\seed^G = (\bfx^G, \qbasis^G)$ in $\field^G$ whose exchange matrix is given as before and cluster variables $\bfx^G = (x_I)$ are indexed by the $G$-orbits and given by $x_I = \psi(x_i)$.

\begin{proposition}[{\cite[Corollary~4.4.11]{FWZ_chapter45}}]\label{proposition:folded cluster pattern}
Let $\quiver$ be a quiver which is globally foldable with respect to a group $G$ acting on the set of its vertices. Let $\initialseed = (\mathbf{x}, \qbasis(\quiver))$ be a seed in the field $\field$ of rational functions freely generated by a cluster $\mathbf{x} = (x_1,\dots,x_m)$. Define $\psi \colon \field  \to \field^G$ so that $\initialseed$ is a $(G, \psi)$-admissible seed. Then, for any mutable $G$-orbits $I_1,\dots,I_\ell$, the seed $(\mutation_{I_\ell} \dots \mutation_{I_1})(\initialseed)$ is $(G, \psi)$-admissible, and moreover, the folded seeds $((\mutation_{I_\ell} \dots \mutation_{I_1})(\initialseed))^G$ form a cluster pattern in $\field^G$ with the initial seed $\initialseed^G=(\bfx^G, (\qbasis(\quiver))^G)$.
\end{proposition}

\begin{example}\label{example_folding_E}
The quiver in Example~\ref{example_D4_to_G2} is globally foldable, and 
moreover, the corresponding cluster pattern is of type $\exdynG_2$. In fact, seed 
patterns of type~$\exdynB\exdynC\exdynF\exdynG$ are obtained by folding quivers 
of type~$\exdynD\exdynE$ in general 
(cf.~\cite{FeliksonShapiroTumarkin12_unfoldings}). In 
Figure~\ref{figure:G-actions}, we present some examples of foldings. 
We decorate vertices of quivers with white and black colors for presenting 
source and sink, respectively. We denote the generator of $G$ by $\tau$. For 
each case, the finite group action that makes each quiver globally foldable is 
depicted in Figure~\ref{figure:G-actions}. Note that the alternating coloring 
on quivers of type $\exdynE_6$ or $\exdynE_7$ provide 
that on quivers of type  $\exdynG_2$, $\dynE_6^{(2)}$, or $\exdynF_4$. 
Here, we decorate the vertices of folded quivers with orbits $I_i \coloneqq G 
\cdot i \subset [n] $. 
All possible foldings between simply-laced affine Dynkin diagrams and 
non-simply-laced affine Dynkin diagrams are given in 
Table~\ref{figure:all possible foldings}.
\end{example}

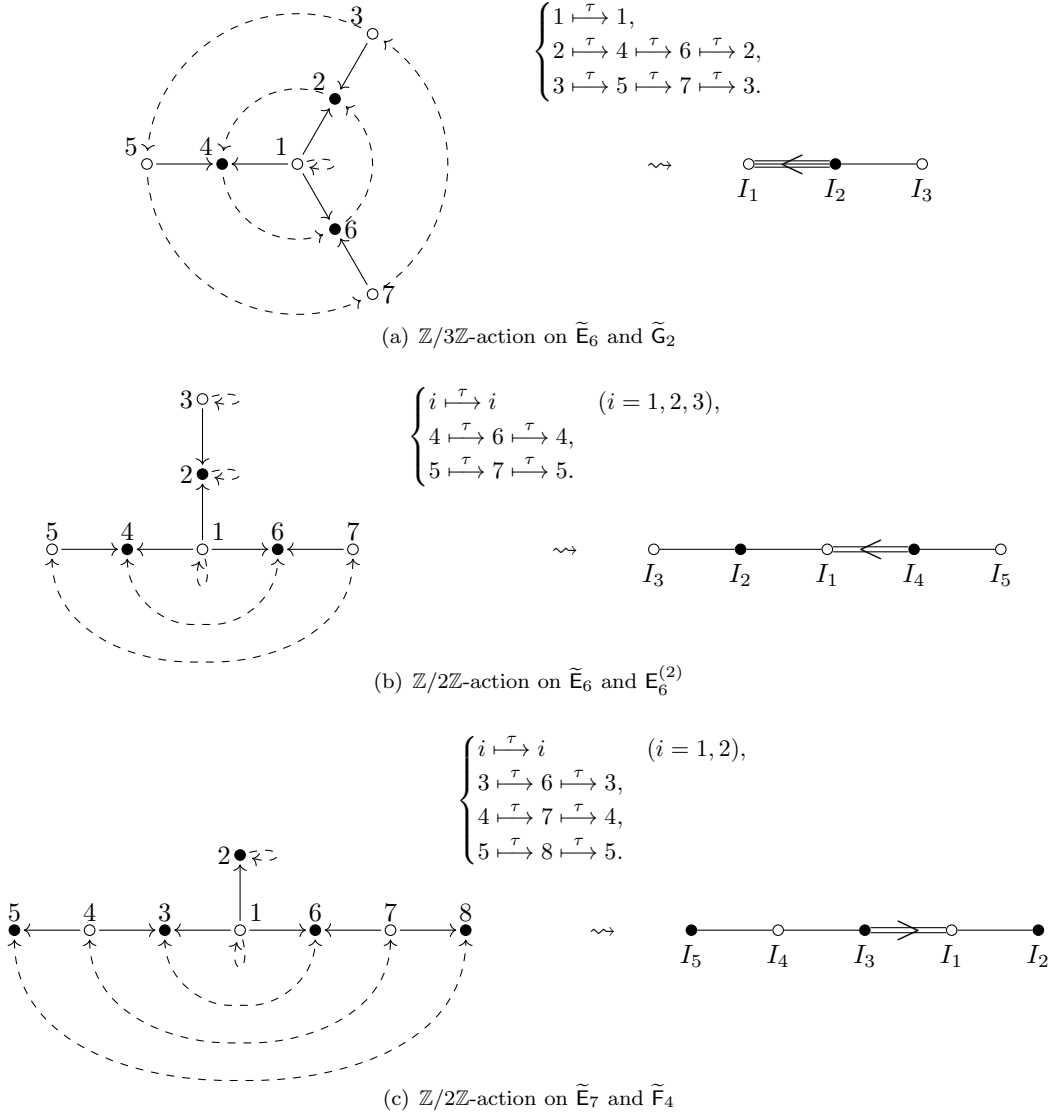
\begin{figure}[ht]
\subfigure[$\Z/3\Z$-action on $\exdynE_6$ and $\exdynG_2$\label{fig_E6_Z3}]{
$
\begin{tikzpicture}[baseline=-.5ex]	
\draw[fill=white] (0,0) circle (2pt) node (A1) {} (60:2) circle (2pt) node (A3) {} (180:2) circle (2pt) node (A5) {} (300:2) circle (2pt) node (A7) {};
\draw[fill] (60:1) circle (2pt) node (A2) {} (180:1) circle (2pt) node (A4) {} (300:1) circle (2pt) node (A6) {};
\draw[->] (A1) node[above left] {$1$} -- (A2) node[above left] {$2$};
\draw[->] (A1) -- (A4) node[above left] {$4$};
\draw[->] (A1) -- (A6) node[right] {$6$};
\draw[->] (A3) node[above left] {$3$} -- (A2);
\draw[->] (A5) node[above left] {$5$} -- (A4);
\draw[->] (A7) node[right] {$7$} -- (A6);
\draw[->, dashed, thin] (A1) edge[loop right] (A1);
\draw[->, dashed, thin] (70:1) arc (70:170:1) node[midway, above left] {};
\draw[->, dashed, thin] (65:2) arc (65:175:2) node[midway, above left] {};
\draw[->, dashed, thin] (190:1) arc (190:290:1) node[midway, below left] {};
\draw[->, dashed, thin] (185:2) arc (185:295:2) node[midway, below left] {};
\draw[->, dashed, thin] (310:1) arc (310:410:1) node[midway, right] {};
\draw[->, dashed, thin] (305:2) arc (305:415:2) node[midway, right] {};

\begin{scope}[xshift = 6cm]
\tikzstyle{state}=[draw, circle, inner sep = 1.4pt]
\tikzstyle{double line} = [
double distance = 1.5pt, 
double=\pgfkeysvalueof{/tikz/commutative diagrams/background color}
]
\tikzstyle{triple line} = [
double distance = 2pt, 
double=\pgfkeysvalueof{/tikz/commutative diagrams/background color}
]

\node[state, label=below:{$I_1$}] (1) {};
\node[state, fill=black, label=below:{$I_2$}] (2) [right = of 1] {};
\node[state, label=below:{$I_3$}] (3) [right=of 2] {};

\draw[triple line] (1)-- node{\scalebox{1.3}{$<$}} (2);
\draw (1)--(2);
\draw (2)--(3);

\node[label={below:\normalsize{$\rightsquigarrow$}}] [above left = 0.1cm and 1cm of 1] (arrow) {};	

\node[above = 0.3cm of arrow, scale=0.9] {$\begin{cases}
1\stackrel{\tau}{\longmapsto}1,\\
2\stackrel{\tau}{\longmapsto}4\stackrel{\tau}{\longmapsto}6\stackrel{\tau}{\longmapsto}2,\\
3\stackrel{\tau}{\longmapsto}5\stackrel{\tau}{\longmapsto}7\stackrel{\tau}{\longmapsto}3.
\end{cases}$};

\end{scope}

\end{tikzpicture}
$
}

\subfigure[$\Z/2\Z$-action on $\exdynE_6$ and $\dynE_6^{(2)}$]{
$
\begin{tikzpicture}[baseline=-.5ex]
\draw[fill=white] (0,0) circle (2pt) node (A1) {} (90:2) circle (2pt) node (A3) {} (180:2) circle (2pt) node (A5) {} (0:2) circle (2pt) node (A7) {};
\draw[fill] (90:1) circle (2pt) node (A2) {} (180:1) circle (2pt) node (A4) {} (0:1) circle (2pt) node (A6) {};
\draw[->] (A1) node[above right] {$1$} -- (A2) node[left] {$2$};
\draw[->] (A1) -- (A4) node[above] {$4$};
\draw[->] (A1) -- (A6) node[above] {$6$};
\draw[->] (A3) node[left] {$3$} -- (A2);
\draw[->] (A5) node[above] {$5$} -- (A4);
\draw[->] (A7) node[above] {$7$} -- (A6);
\draw[->, dashed, thin] (A1) edge[loop below] (A1);
\draw[->, dashed, thin] (A2) edge[loop right] (A2);
\draw[->, dashed, thin] (A3) edge[loop right] (A3);
\draw[<->, dashed, thin, rounded corners] (A4) to[out=-90,in=180] (0,-1) to[out=0,in=-90] (A6);
\draw[<->, dashed, thin, rounded corners] (A5) to[out=-90,in=180] (0,-1.5) to[out=0,in=-90] (A7);

\begin{scope}[xshift = 6cm]
\tikzstyle{state}=[draw, circle, inner sep = 1.4pt]
\tikzstyle{double line} = [
double distance = 1.5pt, 
double=\pgfkeysvalueof{/tikz/commutative diagrams/background color}
]

\node[state,label=below:{$I_3$}] (1) {};
\node[state,label=below:{$I_2$},fill=black] (2) [right = of 1] {};
\node[state,label = below:{$I_1$}] (3) [right = of 2] {};
\node[state,label = below:{$I_4$},fill=black] (4) [right=of 3] {};
\node[state,label = below:{$I_5$}] (5) [right=of 4] {};

\draw (1)--(2)
(2)--(3)
(4)--(5);
\draw[double line] (3)-- node{\scalebox{1.3}{$<$}} (4);

\node[label={below:\normalsize{$\rightsquigarrow$}}] [above left = 0.1cm and 1cm of 1] (arrow) {};	
\end{scope}

\node[above = 0.3cm of arrow, scale=0.9] {$		\begin{cases}
i\stackrel{\tau}{\longmapsto}i & (i=1,2,3),\\
4\stackrel{\tau}{\longmapsto}6\stackrel{\tau}{\longmapsto}4,\\
5\stackrel{\tau}{\longmapsto}7\stackrel{\tau}{\longmapsto}5.
\end{cases}$};

\end{tikzpicture}
$
}

\subfigure[$\Z/2\Z$-action on $\exdynE_7$ and $\exdynF_4$]{
$
\begin{tikzpicture}[baseline=-.5ex]
\draw[fill=white] (0,0) circle (2pt) node (A1) {} (180:2) circle (2pt) node (A4) {} (0:2) circle (2pt) node (A7) {};
\draw[fill] (90:1) circle (2pt) node (A2) {} (180:1) circle (2pt) node (A3) {} (180:3) circle (2pt) node (A5) {} (0:1) circle (2pt) node (A6) {} (0:3) circle (2pt) node (A8) {};
\draw[->] (A1) node[above right] {$1$} -- (A2) node[left] {$2$};
\draw[->] (A1) -- (A3) node[above] {$3$};
\draw[->] (A1) -- (A6) node[above] {$6$};
\draw[->] (A4) node[above] {$4$} -- (A3);
\draw[->] (A4) -- (A5) node[above] {$5$};
\draw[->] (A7) node[above] {$7$} -- (A6);
\draw[->] (A7) -- (A8) node[above] {$8$};
\draw[->, dashed, thin] (A1) edge[loop below] (A1);
\draw[->, dashed, thin] (A2) edge[loop right] (A2);
\draw[<->, dashed, thin, rounded corners] (A3) to[out=-90,in=180] (0,-1) to[out=0,in=-90] (A6);
\draw[<->, dashed, thin, rounded corners] (A4) to[out=-90,in=180] (0,-1.5) to[out=0,in=-90] (A7);
\draw[<->, dashed, thin, rounded corners] (A5) to[out=-90,in=180] (0,-2) to[out=0,in=-90] (A8);

\begin{scope}[xshift = 6cm]
\tikzstyle{state}=[draw, circle, inner sep = 1.4pt]
\tikzstyle{double line} = [
double distance = 1.5pt, 
double=\pgfkeysvalueof{/tikz/commutative diagrams/background color}
]

\node[state,fill=black,  label=below:{$I_5$}] (1) {};
\node[state, label=below:{$I_4$}] (2) [right = of 1] {};
\node[state, fill=black, label=below:{$I_3$}] (3) [right = of 2] {};
\node[state, label=below:{$I_1$}] (4) [right =of 3] {};
\node[state, fill=black, label=below:{$I_2$}] (5) [right =of 4] {};

\draw (1)--(2)
(2)--(3)
(4)--(5);
\draw[double line] (3)-- node{\scalebox{1.3}{$>$}} (4);
\node[label={below:\normalsize{$\rightsquigarrow$}}] [above left = 0.1cm and 1cm of 1] (arrow) {};	

\node[above = 0.3cm of arrow, scale=0.9] {$		\begin{cases}
i\stackrel{\tau}{\longmapsto}i & (i=1,2),\\
3\stackrel{\tau}{\longmapsto}6\stackrel{\tau}{\longmapsto}3,\\
4\stackrel{\tau}{\longmapsto}7\stackrel{\tau}{\longmapsto}4,\\
5\stackrel{\tau}{\longmapsto}8\stackrel{\tau}{\longmapsto}5.
\end{cases}$};

\end{scope}

\end{tikzpicture}
$
}
\caption{$G$-actions on Dynkin diagrams of affine type}
\label{figure:G-actions}
\end{figure}

\newcolumntype{?}{!{\vrule width 0.6pt}}
\begin{table}[ht]
{
\setlength{\tabcolsep}{4pt}
\renewcommand{\arraystretch}{1.5}		
\begin{tabular}{c||c?c?c?c?c?c?c?c?c?c?c}
\toprule
$\dynX$ & $\exdynA_{2,2}$ & $\exdynA_{n,n}$ 
& \multicolumn{2}{c?}{$\exdynD_4$ }
&  \multicolumn{2}{c?}{$\exdynD_n$ } 
&  \multicolumn{2}{c?}{$\exdynD_{2n}$ } 
&  \multicolumn{2}{c?}{$\exdynE_6$ }
& $\exdynE_7$ \\
\hline 
$G$ & $\Z/2\Z$& $\Z/2\Z$  
& $(\Z/2\Z)^2$ & $\Z/3\Z$ 
& $\Z/ 2\Z$ & $\Z/2\Z$
& $\Z/2\Z$  & $(\Z/2\Z)^2$
& $\Z/3\Z$ & $\Z/2\Z$ 
& $\Z/2\Z$ \\
\hline 
$\dynY$ & $\exdynA_1$ & $\dynD_{n+1}^{(2)}$ 
& $\dynA_2^{(2)}$ & $\dynD_4^{(3)}$ 
& $\exdynC_{n-2}$ & $\dynA_{2(n-1)-1}^{(2)}$
& $\exdynB_n$ & $\dynA_{2n-2}^{(2)}$
& $\exdynG_2$ & $\dynE_6^{(2)}$
& $\exdynF_4$ \\
\bottomrule
\end{tabular}}
\caption{Foldings appearing in affine Dynkin diagrams. For 
$(\dynX, G, \dynY)$ in each column, the quiver of type~$\dynX$ is globally 
foldable with respect to $G$, and the corresponding folded cluster pattern is of 
type~$\dynY$.}
\label{figure:all possible foldings}
\end{table}

\begin{remark}\label{remark:folding and Coxeter mutation}
Suppose that the alternating coloring on quivers of type $\dynX$ provide that on quivers of type $\dynY$. If a cluster pattern of simply-laced type $\dynX$ gives a cluster pattern of type $\dynY$ via the folding procedure, then the Coxeter mutation of type~$\dynY$ is the same as that of type~$\dynX$. More precisely, for a globally foldable seed $\seed$ with respect to $G$ defining a cluster algebra of type $\dynX$ and its Coxeter mutation $\qcoxeter^{\dynX}$, we have
\[
\qcoxeter^{\dynY}(\seed^G) = (\qcoxeter^{\dynX}(\seed))^G.
\]
Here,  $\qcoxeter^{\dynY}$ is the Coxeter mutation on the cluster pattern determined by $\seed^G$.	This observation implies that the bipartite belt of the cluster pattern of type $\dynY$ can be identified with that of type $\dynX$.
\end{remark}

As we saw in Definition~\ref{definition:admissible quiver}, if a seed $\seed = 
(\mathbf{x}, \quiver)$ is $(G,\psi)$-admissible, then $\seed$ is 
$(G,\psi)$-invariant. 
The converse holds when we consider the foldings presented in Table~\ref{figure:all possible foldings}, and moreover they form the folded cluster pattern.

\begin{theorem}[{\cite{AL2021}}]\label{thm_invariant_seeds_form_folded_pattern}
Let $(\dynX, G, \dynY)$ be a triple given by a column of Table~\ref{figure:all 
possible foldings}.
Let $\initialseed = (\mathbf{x}_{t_0},\quiver_{t_0})$ be a seed in the field $\field$. Suppose that $\quiver_{t_0}$ is of type $\dynX$. Define $\psi 
\colon \field  \to \field^G$ so that $\initialseed$ is a $(G, \psi)$-admissible 
seed. Then, for any seed $\seed = (\mathbf{x}, \quiver)$ in the cluster pattern, 
if the quiver $\quiver$ is $G$-invariant, then it is $G$-admissible.
Moreover, any $(G,\psi)$-invariant seed $\seed = (\mathbf{x}, \quiver)$ can be 
reached with a sequence of orbit mutations from the 
initial seed. Indeed,  the set of such seeds forms the cluster 
pattern of the `folded' cluster algebra $\cA(\initialseed^G)$ of type $\dynY$.
\end{theorem}

Under the aid of Proposition~\ref{proposition:folded cluster pattern}, we will prove Theorem~\ref{theorem:principal coefficients and trivial coefficients}.

\begin{proof}[Proof of Theorem~\ref{theorem:principal coefficients and trivial coefficients}]
By Theorem~\ref{thm_exchange_graph_skew_symetric}, it is enough to consider the case where the principal part is of \emph{non-simply-laced affine type}. 
Let $(\dynX, G, \dynY)$ be a column in Table~\ref{figure:all 
possible foldings}.
Let $\quiver(\dynX)$ be the quiver of type $\dynX$ and $\qbasis(\dynX)=\qbasis(\quiver(\dynX))$ be the adjacency matrix of $\quiver(\dynX)$, which is a square matrix of size $n$.
Let $\tilde\qbasis(\dynX) = \begin{pmatrix}
\qbasis(\dynX)\\ \clusterfont{I}_n
\end{pmatrix}$ be the $(2n\times n)$ matrix having principal coefficients whose principal part is given by~$\qbasis(\dynX)$. 
On the other hand, we consider a quiver $\overline{\quiver}(X)$ by adding $n' \coloneqq \#([n]/G)$ frozen vertices and arrows. Here, each frozen vertex is indexed by a $G$-orbit and we draw an arrow from the frozen vertex to each mutable vertex in the corresponding $G$-orbit. 
For some algebraic independent elements $\bfx=(x_1,\dots, x_n)$, $\overline{\bfx} = (x_1,\dots,x_n,x_{n+1},\dots,x_{n+n'})$, and $\tilde\bfx=(x_1,\dots, x_n, x_{n+1},\dots, x_{2n})$ in $\field$, we denote cluster algebras by
\[
\tilde\cA(\dynX)=\cA(\tilde\bfx, \tilde\qbasis(\dynX)),
\quad 
\overline\cA(\dynX)=\cA(\overline\bfx, \qbasis(\overline\quiver(\dynX))), 
\quad\text{ and }\quad
\cA(\dynX)=\cA(\bfx,\qbasis(\dynX)).
\]
Then, by Theorem~\ref{thm_exchange_graph_skew_symetric}, their exchange graphs are isomorphic.

Similarly, for an exchange matrix $\qbasis(\dynY)$ of type $\dynY$ of size $n'$, let $\tilde\qbasis(\dynY)=\begin{pmatrix}
\qbasis(\dynY)\\ \clusterfont{I}_{n'}
\end{pmatrix}$ and we denote cluster algebras by
\[
\tilde\cA(\dynY)=\cA(\tilde\bfx', \tilde\qbasis(\dynY))\quad\text{ and }\quad
\cA(\dynY)=\cA(\bfx',\qbasis(\dynY)).
\]
Here, $\tilde\bfx' = (x_1',\dots,x_{n'}',x_{n'+1}',\dots,x_{2n'}')$ and $\bfx' = (x_1',\dots,x_{n'}')$.

Extending the action of $G$ on $\quiver$ of type $\dynX$ to  $\overline\quiver(\dynX)$ such that $G$ acts trivially on frozen vertices, the quiver $\overline\quiver(\dynX)$ becomes a globally foldable quiver with respect to $G$ (see~\cite[Lemma~5.5.3]{FWZ_chapter45}).
Moreover, via $\psi \colon \field \to \field^G$, the folded seed $(\overline\bfx, \overline\quiver(\dynX))^G$ produces the principal coefficient cluster algebra $\tilde\cA(\dynY)$ of type $\dynY$.
This produces the following diagram.
\[
\begin{tikzcd}
\exchange(\tilde\cA(\dynX)) 
	\arrow[r,twoheadrightarrow, "\cong"]
& \exchange(\overline\cA(\dynX)) 
	\arrow[r, twoheadrightarrow,"\cong"]
& \exchange(\cA(\dynX)) \\
& \exchange(\overline\cA(\dynX))|_{\text{$(G,\psi)$-admissible}}  
	\arrow[u, hookrightarrow]  
	\arrow[r,rightarrowtail]
	\arrow[d,equal]
& \exchange(\cA(\dynX))|_{\text{$(G,\psi)$-admissible}}  
	\arrow[u, hookrightarrow]
	\arrow[d,equal] \\
& \exchange(\tilde\cA(\dynY)) \arrow[r,twoheadrightarrow]
& \exchange(\cA(\dynY))
\end{tikzcd}
\] 
Here, the graphs in the second row are the graphs whose vertices are the $(G,\psi)$-admissible seeds in the graphs $ \exchange(\overline\cA(\dynX))$ and  $\exchange(\cA(\dynX))$, respectively; each pair of vertices is connected if and only if they are related via an \emph{orbit mutation}. 
The inclusion from the second row to the first row means that there is an inclusion between the set of vertices.
The surjectivity in the top and bottom row is induced by the maximality of the exchange graph of a cluster algebra having principal coefficients in Theorem~\ref{thm_exchange_graph_covering}. 
Moreover, the equalities connecting the second and third rows are given by Theorem~\ref{thm_invariant_seeds_form_folded_pattern}. 
This proves the theorem.
\end{proof}

\section{\texorpdfstring{$N$-graphs}{N-graphs} and seeds}\label{sec:N-graph}

\subsection{\texorpdfstring{$N$}{N}-graphs}\label{subsec:N-graph}
Let us recall the notion of $N$-graphs and its moves which present Legendrian surfaces and Legendrian isotopies in $\R^5$.

\begin{definition}\cite[Definition~2.2]{CZ2020}
An  $N$-graph $\ngraph$ on a smooth surface $S$ is an $(N-1)$-tuple of graphs $(\ngraph_1,\dots, \ngraph_{N-1})$ satisfying the following conditions:
\begin{enumerate}
\item Each graph $\ngraph_i$ is embedded, trivalent, possibly empty and non necessarily connected.
\item Any consecutive pair of graphs $(\ngraph_i,\ngraph_{i+1})$, $1\leq i \leq N-2$, intersects only at hexagonal points depicted as in Figure~\ref{fig:hexagonal_point}.
\item Any pair of graphs $(\ngraph_i, \ngraph_j)$ with $1\leq i,j\leq N-1$ and $|i-j|>1$  intersects transversely at edges.
\end{enumerate}
\end{definition}

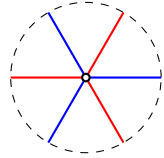
\begin{figure}[ht]
\begin{tikzpicture}
\begin{scope}
\draw[dashed] (0,0) circle (1cm);
\draw[red, thick] (60:1)--(0,0) (180:1)--(0,0) (-60:1)--(0,0);
\draw[blue, thick] (0:1)--(0,0) (120:1)--(0,0) (240:1)--(0,0);
\draw[thick,black,fill=white] (0,0) circle (0.05);
\end{scope}
\end{tikzpicture}
\caption{A hexagonal point}
\label{fig:hexagonal_point}
\end{figure}

Let $\ngraph\subset S$ be an $N$-graph.
A finite cover $\{U_i\}_{i\in I}$ is called {\em $\ngraph$-compatible} if
\begin{enumerate}
\item each $U_i$ is diffeomorphic to the open disk $\mathring{\disk}^2$,
\item $U_i \cap \ngraph$ is connected, and
\item $U_i \cap \ngraph$ contains at most one vertex or a hexagonal point.
\end{enumerate}

\begin{definition}\cite[Definition~2.7]{CZ2020}\label{def:free}
Let $\ngraph$ be an $N$-graph on a surface $S$. The {\em Legendrian weave} $\Legendrian(\ngraph)\subset J^1 S$ is an embedded Legendrian surface whose wavefront $\wavefront(\ngraph)\subset S\times \R$ is constructed by weaving the wavefronts $\{\wavefront(U_i)\}_{i\in I}$ as depicted in Figure~\ref{fig:local_chart_3-graphs}  from a $\ngraph$-compatible cover $\{U_i\}_{i\in I}$ with respect to the gluing data given by $\ngraph$.
\end{definition}

\begin{figure}[ht]
\begin{tikzpicture}

\begin{scope}[xshift=-3.5cm]
\draw[thick] \boundellipse{0,1}{1.25}{0.25};

\draw[thick] \boundellipse{0,-1}{1.25}{0.25};
\draw[thick] \boundellipse{0,-2}{1.25}{0.25};
\draw[thick, ->] (0,-1.35)--(0,-1.9);

\node at (1.5,1){\tiny $N$};
\node at (0,0){\tiny $\vdots$};
\node at (1.5,-1){\tiny $1$};

\end{scope}

\begin{scope}
\draw[thick] \boundellipse{0,1}{1.25}{0.25};

\draw[blue, thick] (-1,0)--(1,0);
\draw[thick] (-3/4,1/3)--(5/4,1/3);
\draw[thick] (5/4,1/3)--(3/4,-1/3);
\draw[thick] (3/4,-1/3)--(-5/4,-1/3);
\draw[thick] (-5/4,-1/3)--(-3/4,1/3);
\draw[thick] (-5/4,3/12)--(3/4,3/12);
\draw[thick] (3/4,3/12)--(5/4,-3/12);
\draw[thick] (5/4,-3/12)--(-3/4,-3/12);
\draw[thick] (-3/4,-3/12)--(-5/4,3/12);

\draw[thick] \boundellipse{0,-1}{1.25}{0.25};
\draw[thick] \boundellipse{0,-2}{1.25}{0.25};
\draw[blue, thick] (-5/4,-2)--(5/4,-2);
\draw[thick, ->] (0,-1.35)--(0,-1.9);

\node at (1.5,1){\tiny $N$};
\node at (0,2/3){\tiny $\vdots$};
\node at (1.5,1/3){\tiny $i+1$};
\node at (1.5,-1/3){\tiny $i$};
\node at (0,-0.45){\tiny $\vdots$};
\node at (1.5,-1){\tiny $1$};

\end{scope}

\begin{scope}[xshift=3.5cm]
\draw[thick] \boundellipse{0,1}{1.25}{0.25};

\draw[blue, thick] (-1.25,0)--(0,0);
\draw[blue, thick] (0,0)--(1.25,1/3);
\draw[blue, thick] (0,0)--(1/2,-1/3);
\draw[thick,blue,fill=blue] (0,0) circle (0.05);

\draw[thick] (-1.25,0) to[out=25,in=175] (1.25,1/3);
\draw[dotted, thick] (-1.25,0) to[out=10,in=190] (1.25,1/3);

\draw[thick] (-1.25,0) to[out=-25,in=180] (1/2,-1/3);
\draw[thick] (-1.25,0) to[out=-10,in=160] (1/2,-1/3);

\draw[dotted,thick] (1/2,-1/3) to[out=80,in=220] (1.25,1/3);
\draw[thick] (1/2,-1/3) to[out=30,in=250] (1.25,1/3);

\draw[thick] \boundellipse{0,-1}{1.25}{0.25};
\draw[thick] \boundellipse{0,-2}{1.25}{0.25};
\draw[blue, thick] (-5/4,-2)--(0,-2);
\draw[blue, thick] (0,-2)--(0.90,-1.825);
\draw[blue, thick] (0,-2)--(1/2,-2.225);

\draw[thick,blue,fill=blue] (0,-2) circle (0.05);
\draw[thick, ->] (0,-1.35)--(0,-1.9);

\node at (1.5,1){\tiny $N$};
\node at (0,2/3){\tiny $\vdots$};
\node at (1.5,1/3){\tiny $i+1$};
\node at (1.5,-1/3){\tiny $i$};
\node at (0,-0.45){\tiny $\vdots$};
\node at (1.5,-1){\tiny $1$};

\end{scope}

\begin{scope}[xshift=7cm]
\draw[thick] \boundellipse{0,1}{1.25}{0.25};

\draw[blue, thick] (0,0)--(1,0);
\draw[blue, thick] (-5/4,3/12)--(0,0);
\draw[blue, thick] (-3/4,1/3)--(0,0);

\draw[red, thick] (-1,0)--(0,0);
\draw[red, thick] (0,0)--(5/4,-3/12);
\draw[red, thick] (0,0)--(3/4,-1/3);

\draw[thick] (-3/4,1/3)--(5/4,1/3);
\draw[thick] (5/4,1/3)--(3/4,-1/3);
\draw[thick] (3/4,-1/3)--(-5/4,-1/3);
\draw[thick] (-5/4,-1/3)--(-3/4,1/3);
\draw[thick,black,fill=white] (0,0) circle (0.05);
\draw[thick] (-5/4,3/12)--(3/4,3/12);
\draw[thick] (3/4,3/12)--(5/4,-3/12);
\draw[thick] (5/4,-3/12)--(-3/4,-3/12);
\draw[thick] (-3/4,-3/12)--(-5/4,3/12);

\draw[thick] (-5/4,3/12)--(-3/4,1/3);
\draw[thick] (-3/4,1/3)--(5/4,-3/12);
\draw[thick] (5/4,-3/12)--(3/4,-1/3);
\draw[thick] (3/4,-1/3)--(-5/4,3/12);

\draw[thick] \boundellipse{0,-1}{1.25}{0.25};
\draw[thick] \boundellipse{0,-2}{1.25}{0.25};
\draw[red, thick] (-5/4,-2)--(0,-2);
\draw[red, thick] (0,-2)--(0.90,-1.825);
\draw[red, thick] (0,-2)--(1/2,-2.225);

\draw[blue, thick] (5/4,-2)--(0,-2);
\draw[blue, thick] (0,-2)--(-0.90,-2.175);
\draw[blue, thick] (0,-2)--(-1/2,-1.775);
\draw[thick,black,fill=white] (0,-2) circle (0.05);

\draw[thick, ->] (0,-1.35)--(0,-1.9);

\node at (1.5,1){\tiny $N$};
\node at (0,2/3){\tiny $\vdots$};
\node at (1.5,1/3){\tiny $i+2$};
\node at (1.5,0){\tiny $i+1$};
\node at (1.5,-1/3){\tiny $i$};
\node at (0,-0.45){\tiny $\vdots$};
\node at (1.5,-1){\tiny $1$};

\end{scope}

\end{tikzpicture}
\caption{Four-types of local charts for $N$-graphs.}
\label{fig:local_chart_3-graphs}
\end{figure}
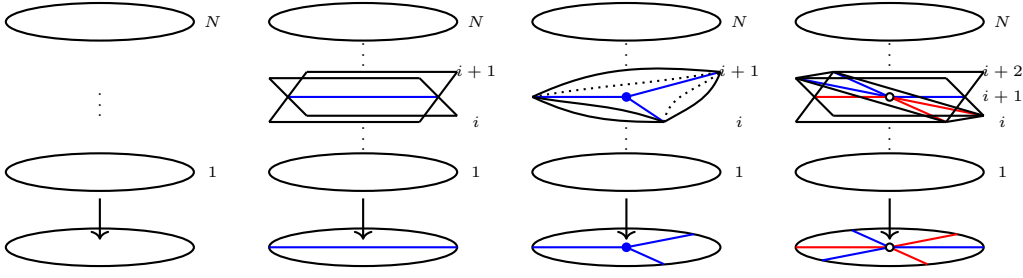

\begin{definition}
An $N$-graph $\ngraph\subset \disk^2$ is called {\em free} if the induced Legendrian weave $\Legendrian(\ngraph)\subset J^1\disk^2$ can be woven without interior Reeb chord. 
\end{definition}

Let $\annulus$ be the oriented annulus with two boundaries $\partial_+\annulus$ and $\partial_-\annulus$ homeomorphic to $\sphere^1$. Consider an $N$-graph $\ngraph$ on $\annulus$, then its restriction on the boundaries give Legendrian links $\legendrian(\partial_+\ngraph)$ and $\legendrian(\partial_-\ngraph)$ in $J^1(\partial_+\annulus)$ and $J^1(\partial_-\annulus)$, respectively.

Note that a Legendrian link in $J^1\sphere^1$ can be considered as a Legendrian link in $\R^3$ under an embedding $\iota:J^1\sphere^1 \to \R^3$. A free $N$-graph $\ngraph$ on $\annulus$ induces an embedded exact Lagrangian cobordism in $\R^3\times [0,1]$ from $\iota(\legendrian(\partial_+\ngraph))\subset \R^3\times \{1\}$ to $\iota(\legendrian(\partial_-\ngraph))\subset \R^3\times \{0\}$. Moreover, a free $N$-graph $\ngraph$ on $\disk^2$ gives a Legendrian weave $\Legendrian(\ngraph)$ in $J^1\disk^2$ which can be regarded as an {\em embedded} Lagrangian filling in $\R^4$ of a Legendrian link $\iota(\legendrian(\partial_+\ngraph))$.

On the other hand, Legendrian isotopies in $J^1\sphere^1$ produce \emph{elementary annulus $N$-graph}s. The following two Legendrian Reidemeister moves \Move{RIII} and \Move{R0} can be interpreted as $N$-graphs $\ngraph_{\Move{RIII}}$ and $\ngraph_{\Move{R0}}$ on the annulus $\annulus$, respectively, as depicted in Figure~\ref{fig:elementary annulus N-graph}. The Move \Move{I} and \Move{V} of $N$-graphs in Figure~\ref{fig:move1-6} imply that the inverses $\ngraph_{\Move{RIII}}^{-1}$ and $\ngraph_{\Move{R0}}^{-1}$ can be obtained by reversing the role of the inner- and outer boundaries.

Let $\ngraph_1$, $\ngraph_2$ be two $N$-graphs on $\annulus$ with $\partial_-\ngraph_1=\partial_+\ngraph_2$. Then we can glue $\ngraph_1$, $\ngraph_2$ along $\partial_-\ngraph_1=\partial_+\ngraph_2$ to obtain a new $N$-graph $\ngraph_1\cdot \ngraph_2$ with two boundaries $\partial_+\ngraph_1$ and $\partial_-\ngraph_2$. If $\partial_-\ngraph_1$ is rotationally symmetric, then the gluing $\ngraph_1\cdot \ngraph_2$ is only well-defined up to that symmetry.

\begin{figure}[ht]
\begin{tikzcd}[row sep=0.25cm]
\begin{tikzpicture}[baseline=-.5ex,scale=1]
\draw [thick] (0,0) circle [radius=0.7];
\draw [thick] (0,0) circle [radius=1.5];
\draw [thick, dotted] (260:1) arc (260:280:1);
\draw [thick, blue] (60:0.7) to[out=60, in=-30] (90:1.1) (120:0.7) to[out=120, in=-150] (90:1.1) -- (90:1.5);
\draw [thick, red] (90:0.7) -- (90:1.1) to[out=150,in=-60] (120:1.5) (90:1.1) to[out=30,in=-120] (60:1.5);
\draw [thick, black]
(30:0.7) -- (30:1.5) 
(150:0.7) -- (150:1.5)
(0:0.7) -- (0:1.5) 
(180:0.7) -- (180:1.5);
\node at (0,0) {$\ngraph_{\Move{RIII}}$};
\end{tikzpicture}
\cdot
\begin{tikzpicture}[baseline=-.5ex,scale=1]
\draw [thick] (0,0) circle [radius=0.7];
\draw [thick, black]
(30:0.5) -- (30:0.7) 
(150:0.5) -- (150:0.7)
(0:0.5) -- (0:0.7) 
(180:0.5) -- (180:0.7);
\draw [thick, red] (90:0.5) -- (90:0.7);
\draw [thick, blue] (60:0.5) -- (60:0.7) (120:0.5) -- (120:0.7) ;
\draw [thick, dotted] (250:0.6) arc (250:290:0.6);
\draw [double] (0,0) circle [radius=0.5];
\node at (0,0) {$\ngraph$};
\end{tikzpicture}
=
\begin{tikzpicture}[baseline=-.5ex,scale=1]
\draw [thick] (0,0) circle [radius=1.5];
\draw [thick, dotted] (260:1) arc (260:280:1);
\draw [thick, blue] (60:0.7) to[out=60, in=-30] (90:1.1) (120:0.7) to[out=120, in=-150] (90:1.1) -- (90:1.5);
\draw [thick, red] (90:0.7) -- (90:1.1) to[out=150,in=-60] (120:1.5) (90:1.1) to[out=30,in=-120] (60:1.5);
\draw [thick, black]
(30:0.7) -- (30:1.5) 
(150:0.7) -- (150:1.5)
(0:0.7) -- (0:1.5) 
(180:0.7) -- (180:1.5);
\draw [thick, black]
(30:0.5) -- (30:0.7) 
(150:0.5) -- (150:0.7)
(0:0.5) -- (0:0.7) 
(180:0.5) -- (180:0.7);
\draw [thick, red] (90:0.5) -- (90:0.7);
\draw [thick, blue] (60:0.5) -- (60:0.7) (120:0.5) -- (120:0.7) ;
\draw [double] (0,0) circle [radius=0.5];
\node at (0,0) {$\ngraph$};
\end{tikzpicture}
\\
\begin{tikzpicture}[baseline=-.5ex,scale=1]
\draw [thick] (0,0) circle [radius=0.7];
\draw [thick] (0,0) circle [radius=1.5];
\draw [thick, dotted] (260:1) arc (260:280:1);
\draw [thick, blue] (120:0.7) to[out=120, in=-120] (60:1.5);
\draw [thick, yellow] (60:0.7) to[out=60,in=-60] (120:1.5);
\draw [thick, black]
(30:0.7) -- (30:1.5) 
(150:0.7) -- (150:1.5)
(0:0.7) -- (0:1.5) 
(180:0.7) -- (180:1.5);
\node at (0,0) {$\ngraph_{\Move{R0}}$};
\end{tikzpicture}
\cdot
\begin{tikzpicture}[baseline=-.5ex,scale=1]
\draw [thick] (0,0) circle [radius=0.7];
\draw [thick, black]
(30:0.5) -- (30:0.7) 
(150:0.5) -- (150:0.7)
(0:0.5) -- (0:0.7) 
(180:0.5) -- (180:0.7);
\draw [thick, yellow] (60:0.5) -- (60:0.7);
\draw [thick, blue] (120:0.5) -- (120:0.7) ;
\draw [thick, dotted] (250:0.6) arc (250:290:0.6);
\draw [double] (0,0) circle [radius=0.5];
\node at (0,0) {$\ngraph$};
\end{tikzpicture}
=
\begin{tikzpicture}[baseline=-.5ex,scale=1]
\draw [thick] (0,0) circle [radius=1.5];
\draw [thick, dotted] (260:1) arc (260:280:1);
\draw [thick, blue] (120:0.7) to[out=120, in=-120] (60:1.5);
\draw [thick, yellow] (60:0.7) to[out=60,in=-60] (120:1.5);
\draw [thick, black]
(30:0.7) -- (30:1.5) 
(150:0.7) -- (150:1.5)
(0:0.7) -- (0:1.5) 
(180:0.7) -- (180:1.5);
\draw [thick, black]
(30:0.5) -- (30:0.7) 
(150:0.5) -- (150:0.7)
(0:0.5) -- (0:0.7) 
(180:0.5) -- (180:0.7);
\draw [thick, yellow] (60:0.5) -- (60:0.7);
\draw [thick, blue] (120:0.5) -- (120:0.7) ;
\draw [double] (0,0) circle [radius=0.5];
\node at (0,0) {$\ngraph$};
\end{tikzpicture}
\end{tikzcd}
\caption{Elementary annulus operations on $N$-graphs on $\disk^2$.}
\label{fig:elementary annulus N-graph}
\end{figure}
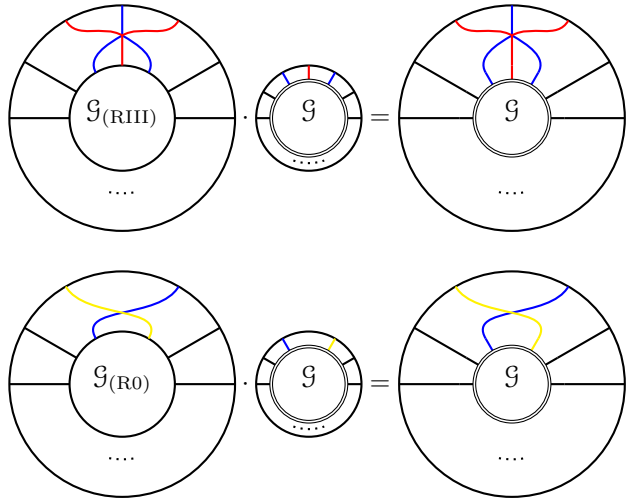

\begin{theorem}\cite[Theorem~1.1]{CZ2020}\label{thm:N-graph moves and legendrian isotopy}
Let $\ngraph$ be a local $N$-graph. The combinatorial moves in Figure~\ref{fig:move1-6} are Legendrian isotopies for $\Legendrian(\ngraph)$.
\end{theorem}

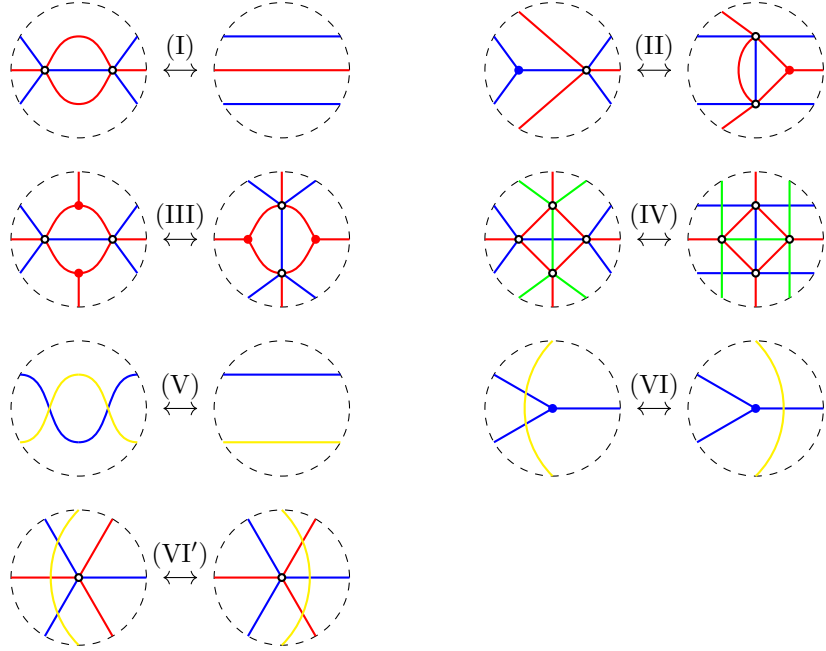
\begin{figure}[ht]
\begin{tikzpicture}[scale=0.9]
\begin{scope}
\draw [dashed] (0,0) circle [radius=1];
\draw [dashed] (3,0) circle [radius=1];
\draw [<->] (1.25,0) -- (1.75,0) node[midway, above] {\Move{I}};

\draw [blue, thick] ({-sqrt(3)/2},1/2)--(-1/2,0);
\draw [blue, thick] ({-sqrt(3)/2},-1/2)--(-1/2,0);
\draw [blue, thick] ({sqrt(3)/2},1/2)--(1/2,0);
\draw [blue, thick] ({sqrt(3)/2},-1/2)--(1/2,0);
\draw [blue, thick] (-1/2,0)--(1/2,0);

\draw [red, thick] (-1,0)--(-1/2,0) to[out=60,in=180] (0,1/2) to[out=0,in=120] (1/2,0)--(1,0);
\draw [red, thick] (-1/2,0) to[out=-60,in=180] (0, -1/2) to[out=0, in=-120] (1/2,0); 

\draw[thick,black,fill=white] (-1/2,0) circle (0.05);
\draw[thick,black,fill=white] (1/2,0) circle (0.05);

\draw [blue, thick] ({3-sqrt(3)/2},1/2)--({3+sqrt(3)/2},1/2);
\draw [blue, thick] ({3-sqrt(3)/2},-1/2)--({3+sqrt(3)/2},-1/2);
\draw [red, thick] (2,0)--(4,0);

\end{scope}

\begin{scope}[xshift=7cm]
\draw [dashed] (0,0) circle [radius=1];
\draw [dashed] (3,0) circle [radius=1];
\draw [<->] (1.25,0) -- (1.75,0) node[midway, above] {\Move{II}};

\draw [blue, thick] ({-sqrt(3)/2},1/2)--(-1/2,0);
\draw [blue, thick] ({-sqrt(3)/2},-1/2)--(-1/2,0);
\draw [blue, thick] ({sqrt(3)/2},1/2)--(1/2,0);
\draw [blue, thick] ({sqrt(3)/2},-1/2)--(1/2,0);
\draw [blue, thick] (-1/2,0)--(1/2,0);

\draw [red, thick] (-1/2,{sqrt(3)/2}) -- (1/2,0)--(1,0);
\draw [red, thick] (-1/2,{-sqrt(3)/2}) -- (1/2,0);

\draw[thick,blue,fill=blue] (-1/2,0) circle (0.05);
\draw[thick,black,fill=white] (1/2,0) circle (0.05);

\draw [blue, thick] ({3-sqrt(3)/2},1/2)--({3+sqrt(3)/2},1/2);
\draw [blue, thick] ({3-sqrt(3)/2},-1/2)--({3+sqrt(3)/2},-1/2);
\draw [blue, thick] (3,1/2)--(3,-1/2);

\draw [red, thick] (5/2,{sqrt(3)/2})--(3,1/2) to[out=-150,in=150] (3,-1/2)--(5/2,{-sqrt(3)/2});
\draw [red, thick] (3,1/2)--(7/2,0) -- (4,0);
\draw [red, thick] (3,-1/2)--(7/2,0);

\draw[thick,black,fill=white] (3,1/2) circle (0.05);
\draw[thick,black,fill=white] (3,-1/2) circle (0.05);
\draw[thick,red,fill=red] (7/2,0) circle (0.05);

\end{scope}

\begin{scope}[yshift=-2.5cm]
\draw [dashed] (0,0) circle [radius=1];
\draw [<->] (1.25,0) -- (1.75,0) node[midway, above] {\Move{III}};

\draw [blue, thick] ({-sqrt(3)/2},1/2)--(-1/2,0);
\draw [blue, thick] ({-sqrt(3)/2},-1/2)--(-1/2,0);
\draw [blue, thick] ({sqrt(3)/2},1/2)--(1/2,0);
\draw [blue, thick] ({sqrt(3)/2},-1/2)--(1/2,0);
\draw [blue, thick] (-1/2,0)--(1/2,0);

\draw [red, thick] (-1,0)--(-1/2,0) to[out=60,in=180] (0,1/2) to[out=0,in=120] (1/2,0)--(1,0);
\draw [red, thick] (-1/2,0) to[out=-60,in=180] (0, -1/2) to[out=0, in=-120] (1/2,0); 
\draw [red, thick] (0,1) to (0,1/2);
\draw [red, thick] (0,-1) to (0,-1/2);

\draw[thick,black,fill=white] (-1/2,0) circle (0.05);
\draw[thick,black,fill=white] (1/2,0) circle (0.05);

\draw[thick,red,fill=red] (0,1/2) circle (0.05);
\draw[thick,red,fill=red] (0,-1/2) circle (0.05);

\end{scope}

\begin{scope}[yshift=-2.5cm, xshift=3cm]

\draw [dashed] (0,0) circle [radius=1];

\draw [blue, thick] (-1/2,{sqrt(3)/2}) to (0,1/2) to (0,-1/2) to (-1/2,-{sqrt(3)/2});
\draw [blue, thick] (1/2,{sqrt(3)/2})--(0,1/2);
\draw [blue, thick] (1/2,-{sqrt(3)/2})--(0,-1/2);

\draw [red, thick] (-1,0)--(-1/2,0) to[out=60,in=180] (0,1/2) to[out=0,in=120] (1/2,0)--(1,0);
\draw [red, thick] (-1/2,0) to[out=-60,in=180] (0, -1/2) to[out=0, in=-120] (1/2,0); 
\draw [red, thick] (0,1) to (0,1/2);
\draw [red, thick] (0,-1) to (0,-1/2);

\draw[thick,red,fill=red] (-1/2,0) circle (0.05);
\draw[thick,red,fill=red] (1/2,0) circle (0.05);

\draw[thick,black,fill=white] (0,1/2) circle (0.05);
\draw[thick,black,fill=white] (0,-1/2) circle (0.05);
\end{scope}

\begin{scope}[xshift=7cm, yshift=-2.5cm]
\draw [dashed] (0,0) circle [radius=1];
\draw [<->] (1.25,0) -- (1.75,0) node[midway, above] {\Move{IV}};

\draw [blue, thick] ({-sqrt(3)/2},1/2)--(-1/2,0);
\draw [blue, thick] ({-sqrt(3)/2},-1/2)--(-1/2,0);
\draw [blue, thick] ({sqrt(3)/2},1/2)--(1/2,0);
\draw [blue, thick] ({sqrt(3)/2},-1/2)--(1/2,0);
\draw [blue, thick] (-1/2,0)--(1/2,0);

\draw [red, thick] (-1,0)--(-1/2,0) to (0,1/2) to (1/2,0)--(1,0);
\draw [red, thick] (-1/2,0) to (0, -1/2) to (1/2,0); 
\draw [red, thick] (0,1) to (0,1/2);
\draw [red, thick] (0,-1) to (0,-1/2);

\draw [green, thick] (-1/2,{sqrt(3)/2}) to (0,1/2) to (0,-1/2) to (-1/2,-{sqrt(3)/2});
\draw [green, thick] (1/2,{sqrt(3)/2})--(0,1/2);
\draw [green, thick] (1/2,-{sqrt(3)/2})--(0,-1/2);

\draw[thick,black,fill=white] (-1/2,0) circle (0.05);
\draw[thick,black,fill=white] (1/2,0) circle (0.05);

\draw[thick,black,fill=white] (0,1/2) circle (0.05);
\draw[thick,black,fill=white] (0,-1/2) circle (0.05);

\end{scope}

\begin{scope}[xshift=10cm, yshift=-2.5cm]
\draw [dashed] (0,0) circle [radius=1];
\draw [blue, thick] ({-sqrt(3)/2},1/2)--({+sqrt(3)/2},1/2);
\draw [blue, thick] ({-sqrt(3)/2},-1/2)--({+sqrt(3)/2},-1/2);
\draw [blue, thick] (0,1/2)--(0,-1/2);

\draw [red, thick] (-1,0)--(-1/2,0) to (0,1/2) to (1/2,0)--(1,0);
\draw [red, thick] (-1/2,0) to (0, -1/2) to (1/2,0); 
\draw [red, thick] (0,1) to (0,1/2);
\draw [red, thick] (0,-1) to (0,-1/2);

\draw [green, thick] (-1/2,{sqrt(3)/2}) to (-1/2,{-sqrt(3)/2});
\draw [green, thick] (1/2,{sqrt(3)/2}) to (1/2,{-sqrt(3)/2});
\draw [green, thick] (-1/2,0) to (1/2,0);

\draw[thick,black,fill=white] (-1/2,0) circle (0.05);
\draw[thick,black,fill=white] (1/2,0) circle (0.05);

\draw[thick,black,fill=white] (0,1/2) circle (0.05);
\draw[thick,black,fill=white] (0,-1/2) circle (0.05);

\end{scope}

\begin{scope}[xshift=0cm, yshift=-5cm]
\draw [dashed] (0,0) circle [radius=1];
\draw [<->] (1.25,0) -- (1.75,0) node[midway, above] {\Move{V}};

\draw [blue, thick] ({-sqrt(3)/2},1/2)to[out=0,in=180](0,-1/2);
\draw [blue, thick] ({sqrt(3)/2},1/2)to[out=180,in=0](0,-1/2);

\draw [yellow, thick] ({-sqrt(3)/2},-1/2)to[out=0,in=180](0,1/2);
\draw [yellow, thick] ({sqrt(3)/2},-1/2)to[out=180,in=0](0,1/2);

\end{scope}

\begin{scope}[xshift=3cm, yshift=-5cm]
\draw [dashed] (0,0) circle [radius=1];

\draw [blue, thick] ({-sqrt(3)/2},1/2) to ({sqrt(3)/2},1/2);

\draw [yellow, thick] ({-sqrt(3)/2},-1/2) to ({sqrt(3)/2},-1/2);
\end{scope}

\begin{scope}[xshift=7cm, yshift=-5cm]
\draw [dashed] (0,0) circle [radius=1];
\draw [<->] (1.25,0) -- (1.75,0) node[midway, above] {\Move{VI}};
\draw [blue, thick] ({-sqrt(3)/2},1/2) to (0,0) to(1,0);
\draw [blue, thick] ({-sqrt(3)/2},-1/2) to (0,0);

\draw[thick,blue,fill=blue] (0,0) circle (0.05);

\draw[yellow, thick] (0,1) to[out=-135,in=135] (0,-1);
\end{scope}

\begin{scope}[xshift=10cm, yshift=-5cm]
\draw [dashed] (0,0) circle [radius=1];
\draw [blue, thick] ({-sqrt(3)/2},1/2) to (0,0) to(1,0);
\draw [blue, thick] ({-sqrt(3)/2},-1/2) to (0,0);

\draw[thick,blue,fill=blue] (0,0) circle (0.05);

\draw[yellow, thick] (0,1) to[out=-45,in=45] (0,-1);
\end{scope}

\begin{scope}[xshift=0cm, yshift=-7.5cm]
\draw [dashed] (0,0) circle [radius=1];
\draw [<->] (1.25,0) -- (1.75,0) node[midway, above] {$\Move{VI'}$};
\draw [blue, thick] ({-1/2},{sqrt(3)/2}) to (0,0) to(1,0);
\draw [blue, thick] ({-1/2},{-sqrt(3)/2}) to (0,0);

\draw [red, thick] (-1,0) to (0,0) to(1/2,{sqrt(3)/2});
\draw [red, thick] (0,0) to(1/2,{-sqrt(3)/2});

\draw[thick,black,fill=white] (0,0) circle (0.05);

\draw[yellow, thick] (0,1) to[out=-135,in=135] (0,-1);
\end{scope}

\begin{scope}[xshift=3cm, yshift=-7.5cm]
\draw [dashed] (0,0) circle [radius=1];
\draw [blue, thick] ({-1/2},{sqrt(3)/2}) to (0,0) to(1,0);
\draw [blue, thick] ({-1/2},{-sqrt(3)/2}) to (0,0);

\draw [red, thick] (-1,0) to (0,0) to(1/2,{sqrt(3)/2});
\draw [red, thick] (0,0) to(1/2,{-sqrt(3)/2});

\draw[thick,black,fill=white] (0,0) circle (0.05);

\draw[yellow, thick] (0,1) to[out=-45,in=45] (0,-1);
\end{scope}
\end{tikzpicture}
\caption{Combinatorial moves for Legendrian isotopies of surface $\Legendrian(\ngraph)$. Here the pairs ({\color{blue} blue}, {\color{red} red}) and ({\color{red} red}, {\color{green} green}) are consecutive. Other pairs are not.}
\label{fig:move1-6}
\end{figure}

Let us denote the equivalence class of an $N$-graph $\ngraph$ up to the moves $\Move{I},\dots,\Move{VI'}$ by $[\ngraph]$.

Note that the moves in Figure~\ref{fig:move1-6}  preserve the freeness of $N$-graphs. There are other combinatorial moves in $N$-graphs involving cusps which induces Legendrian isotopies of Legendrian weaves, see \cite[Figure~3]{CZ2020}.

\subsection{One-cycles and flag moduli of \texorpdfstring{$N$}{N}-graphs}
Let us recall from \cite{CZ2020,ABL2021} the construction of a seed, a quiver together with cluster variables, from a free $N$-graphs $\ngraph\subset \disk^2$. Let $\Legendrian(\ngraph)$ be the corresponding Legendrian surface, then the set of one-cycles in $\Legendrian(\ngraph)$ and their intersection data define a quiver, and a monodromy along each cycle assigns a coordinate function to each vertex which plays a role of cluster variable.

There is an operation in $N$-graph, so-called \emph{Legendrian mutation}, which is analogous to the mutation in the cluster structure. This Legendrian mutation is important in producing as many distinct $N$-graphs as seeds which can be interpreted as Lagrangian fillings of the Legendrian link $\legendrian(\boundary\ngraph)$.

We present one-cycles of the Legendrian surface $\Legendrian(\ngraph)$ in terms of subgraphs of $\ngraph$. Instead of giving general definition of subgraphs which gives one-cycles of $\Legendrian(\ngraph)$, let us focus on certain type of cycles which are of main interest in the current article. See \cite{CZ2020,ABL2021} for the general construction of one-cycles.

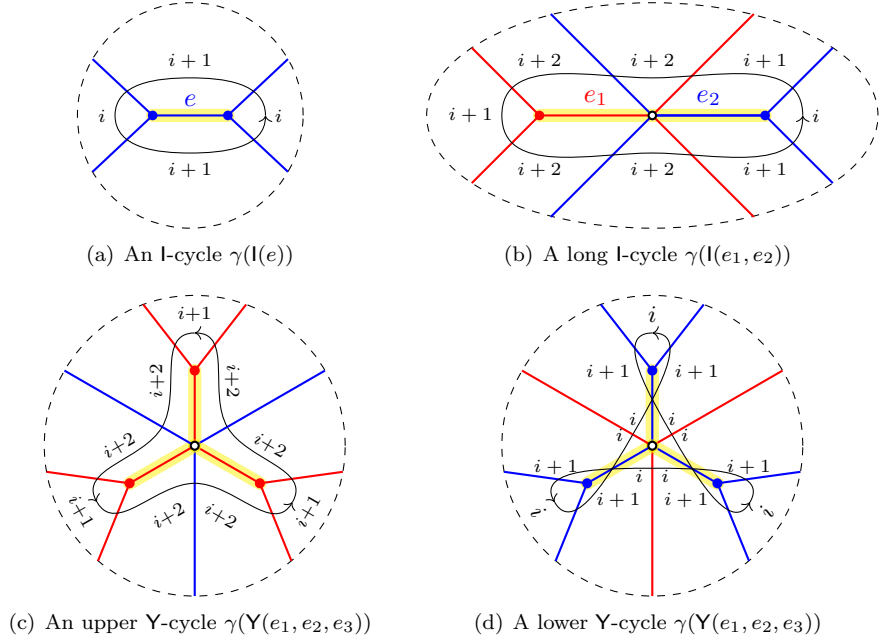
\begin{figure}[ht]
\subfigure[An $\sfI$-cycle $\cycle(\sfI(e))$\label{figure:I-cycle}]{\makebox[.4\textwidth]{
\begin{tikzpicture}
\draw [dashed] (0,0) circle [radius=1.5];

\draw [yellow, line cap=round, line width=5, opacity=0.5] (-1/2,0) to (1/2,0);
\draw [blue, thick] ({-3*sqrt(3)/4},3/4)--(-1/2,0);
\draw [blue, thick] ({-3*sqrt(3)/4},-3/4)--(-1/2,0);
\draw [blue, thick] ({3*sqrt(3)/4},3/4)--(1/2,0);
\draw [blue, thick] ({3*sqrt(3)/4},-3/4)--(1/2,0);
\draw [blue, thick] (-1/2,0)--(1/2,0) node[above, midway] {$e$};

\draw[thick,blue,fill=blue] (-1/2,0) circle (0.05);
\draw[thick,blue,fill=blue] (1/2,0) circle (0.05);

\draw[->] (1,0) node[right] {\scriptsize $i$} to[out=90,in=0] (0,0.5) node[above] {\scriptsize $i+1$} to[out=180,in=90] (-1,0) node[left] {\scriptsize $i$} to[out=-90,in=180] (0,-0.5)node[below] {\scriptsize $i+1$} to[out=0,in=-90] (1,0);

\end{tikzpicture}
}}
\subfigure[A long $\sfI$-cycle $\cycle(\sfI(e_1,e_2))$\label{figure:long I-cycle}]{\makebox[.4\textwidth]{
\begin{tikzpicture}

\draw[dashed] \boundellipse{0,0}{3}{1.5};

\draw [yellow, line cap=round, line width=5, opacity=0.5] (-1.5,0) to (1.5,0);

\draw[red, thick] (0,0)--(1.35,1.35);
\draw[red, thick] (0,0)--(1.35,-1.35);
\draw[red, thick] (0,0)--(-1.5,0) node[above, midway] {$e_1$};
\draw[red, thick] (-1.5,0)--(-1.5-0.9,0.9);
\draw[red, thick] (-1.5,0)--(-1.5-0.9,-0.9);

\draw[blue, thick] (0,0)--(-1.35,1.35);
\draw[blue, thick] (0,0)--(-1.35,-1.35);
\draw[blue, thick] (0,0)--(1.5,0) node[above, midway] {$e_2$};
\draw[blue, thick] (1.5,0)--(1.5+0.9,0.9);
\draw[blue, thick] (0,0)--(1.5,0)--(1.5+0.9,-0.9);

\draw[thick,red,fill=red] (-1.5,0) circle (0.05);
\draw[thick,blue,fill=blue] (1.5,0) circle (0.05);
\draw[thick,black,fill=white] (0,0) circle (0.05);

\draw[->] (2,0) node[right]{\scriptsize $i$} to[out=90,in=0] (0,0.5) node[above]{\scriptsize $i+2$} to[out=180,in=90] (-2,0) node[left]{\scriptsize $i+1$} to[out=-90,in=180] (0,-0.5) node[below]{\scriptsize $i+2$} to[out=0,in=-90] (2,0);
\node[above] at (1.5,0.5) {\scriptsize $i+1$};
\node[below] at (1.5,-0.5) {\scriptsize $i+1$};
\node[above] at (-1.5,0.5) {\scriptsize $i+2$};
\node[below] at (-1.5,-0.5) {\scriptsize $i+2$};

\end{tikzpicture}}}

\subfigure[An upper $\sfY$-cycle $\cycle(\sfY(e_1,e_2,e_3))$\label{figure:Y-cycle_1}]{\makebox[.4\textwidth]{
\begin{tikzpicture}
\draw [dashed] (0,0) circle [radius=2];

\begin{scope}
\draw [yellow, line cap=round, line width=5, opacity=0.5] 
(0:0) -- (-30:1);
\draw [blue, thick](0:0)--(270:2);
\draw [red, thick](0,0)--(330:1) (310:2)--(330:1)--(350:2);
\draw[->] (270:0.5) to[out=0,in=-120] (330:1.5);
\draw(330:1.5) to[out=60,in=-60] (30:0.5);
\node[rotate=60] at (330:1.75) {$\scriptstyle i+1$};
\node[rotate=-30] at (290:1) {$\scriptstyle i+2$};
\node[rotate=-30] at (0:1) {$\scriptstyle i+2$};
\end{scope}
\begin{scope}[rotate=120]
\draw [yellow, line cap=round, line width=5, opacity=0.5] 
(0:0) -- (-30:1);
\draw [blue, thick](0:0)--(270:2);
\draw [red, thick](0,0)--(330:1) (310:2)--(330:1)--(350:2);
\draw[->] (270:0.5) to[out=0,in=-120] (330:1.5);
\draw(330:1.5) to[out=60,in=-60] (30:0.5);
\end{scope}
\begin{scope}[rotate=240]
\draw [yellow, line cap=round, line width=5, opacity=0.5] 
(0:0) -- (-30:1);
\draw [blue, thick](0:0)--(270:2);
\draw [red, thick](0,0)--(330:1) (310:2)--(330:1)--(350:2);
\draw[->] (270:0.5) to[out=0,in=-120] (330:1.5);
\draw(330:1.5) to[out=60,in=-60] (30:0.5);
\end{scope}
\draw[thick, red, fill] 
(90:1) circle (1.5pt)
(210:1) circle (1.5pt)
(330:1) circle (1.5pt);

\node at (90:1.75) {$\scriptstyle i+1$};
\node[rotate=-90] at (60:1) {$\scriptstyle i+2$};
\node[rotate=90] at (120:1) {$\scriptstyle i+2$};
\node[rotate=-60] at (210:1.75) {$\scriptstyle i+1$};
\node[rotate=30] at (180:1) {$\scriptstyle i+2$};
\node[rotate=30] at (250:1) {$\scriptstyle i+2$};

\draw[thick,black, fill=white] (0:0) circle (1.5pt);
\end{tikzpicture}}}
\subfigure[A lower $\sfY$-cycle $\cycle(\sfY(e_1,e_2,e_3))$\label{figure:Y-cycle_2}]{\makebox[.4\textwidth]{
\begin{tikzpicture}
\draw [dashed] (0,0) circle [radius=2];

\begin{scope}
\draw [yellow, line cap=round, line width=5, opacity=0.5] 
(0:0) -- (-30:1);
\draw [red, thick](0:0)--(270:2);
\draw [blue, thick](0,0)--(330:1) (310:2)--(330:1)--(350:2);
\draw (270:0.3) to[out=0,in=60] (330:1.5);
\draw[->] (30:0.3) to[out=-60,in=-120] (330:1.5)  ;
\end{scope}
\begin{scope}[rotate=120]
\draw [yellow, line cap=round, line width=5, opacity=0.5] 
(0:0) -- (-30:1);
\draw [red, thick](0:0)--(270:2);
\draw [blue, thick](0,0)--(330:1) (310:2)--(330:1)--(350:2);
\draw (270:0.3) to[out=0,in=60] (330:1.5);
\draw[->] (30:0.3) to[out=-60,in=-120] (330:1.5)  ;
\end{scope}
\begin{scope}[rotate=240]
\draw [yellow, line cap=round, line width=5, opacity=0.5] 
(0:0) -- (-30:1);
\draw [red, thick](0:0)--(270:2);
\draw [blue, thick](0,0)--(330:1) (310:2)--(330:1)--(350:2);
\draw (270:0.3) to[out=0,in=60] (330:1.5);
\draw[->] (30:0.3) to[out=-60,in=-120] (330:1.5)  ;
\end{scope}
\draw[thick, blue, fill] 
(90:1) circle (1.5pt)
(210:1) circle (1.5pt)
(330:1) circle (1.5pt);

\node[rotate=60] at (330:1.75) {\small$i$};
\node at (90:1.75) {\small$i$};
\node[rotate=-60] at (210:1.75) {\small$i$};

\node[right] at (80:1) {\scriptsize$i+1$};
\node[left] at (100:1) {\scriptsize$i+1$};
\node[above right]  at (330:1) {\scriptsize$i+1$};
\node[below left]  at (330:1) {\scriptsize$i+1$};
\node[above left]  at (210:1) {\scriptsize$i+1$};
\node[below right]  at (210:1) {\scriptsize$i+1$};
\node[above]  at (30:0.3) {\scriptsize$i$};
\node[right]  at (30:0.3) {\scriptsize$i$};
\node[above]  at (150:0.3) {\scriptsize$i$};
\node[left]  at (150:0.3) {\scriptsize$i$};
\node[below left]  at (270:0.2) {\scriptsize$i$};
\node[below right]  at (270:0.2) {\scriptsize$i$};

\draw[thick,black, fill=white] (0:0) circle (1.5pt);
\end{tikzpicture}}}
\caption{(Long) $\sfI$- and $\sfY$-cycles}
\label{fig:I and Y cycle}
\end{figure}

\begin{definition}[(Long) $\sfI$-cycles]
For an edge $e$ of $\ngraph$ connecting two trivalent vertices, let $\sfI(e)$ be the subgraph of $\ngraph$ consisting of a single edge $e$. Then the cycle $[\cycle(\sfI(e))]$ depicted in  Figure~\ref{figure:I-cycle} is called an \emph{$\sfI$-cycle}.

Consider a linear chain of edges $(e_1,e_2,\dots, e_n)$ satisfying
\begin{itemize}
\item $e_i$ connects a trivalent vertex and a hexagonal point for $i=1,n$;
\item $e_i$ and $e_{i+1}$ meet at a hexagonal point in the opposite way, see Figure~\ref{figure:long I-cycle}, for $i=2,\dots, n-1$.
\end{itemize}
Then the cycle $[\cycle(\sfI(e_1,\dots, e_n))]$ is called a \emph{long $\sfI$-cycle}.
\end{definition}

\begin{definition}[$\sfY$-cycles]
Let $e_1,e_2,e_3$ be monochromatic edges joining a hexagonal point $h$ and trivalent vertices $v_i$ for $i=1,2,3$. Then the subgraph $\sfY(e_1,e_2,e_3)$ consisting of three edges $e_1, e_2$ and $e_3$ defines a cycle $[\cycle(\sfY(e_1,e_2,e_3))]$ called an \emph{upper} or \emph{lower} \emph{$\sfY$-cycle} according to the relative position of sheets that edges represent. See Figures~\ref{figure:Y-cycle_1} and \ref{figure:Y-cycle_2}.
\end{definition}

\begin{remark}
Black thin lines in Figure~\ref{fig:I and Y cycle} represent the lift of a circle in $\disk^2$ to a circle in $\disk^2\times \R$ and labels on each region of the black thin line indicate the index of the lift on that region, see Figure~\ref{fig:local_chart_3-graphs}.
\end{remark}

\begin{definition}\label{def:good cycle}
Let $\ngraph\subset \disk^2$ be an $N$-graph, and $\Legendrian(\ngraph)$ be an induced Legendrian surface in $J^1\disk^2$. A cycle $[\cycle]\in H_1(\Legendrian(\ngraph))$ is \emph{good} if $[\cycle]$ can be transformed to an $\sfI$-cycle in $H_1(\Legendrian(\ngraph'))$ for some $[\ngraph']=[\ngraph]$.

A tuple of linearly independent good cycles $\nbasis=\{[\gamma_i]\}_{i\in I}$ in $H_1(\Legendrian(\ngraph))$ is \emph{good} if for any pair of dictinct cycles $[\gamma_i]$ and $[\gamma_j]$, two cycles $[\cycle_i]$ and $[\cycle_j]$ can be simultaneously transformed to $\sfI$-cycles in $H_1(\Legendrian(\ngraph'))$ for some $[\ngraph']=[\ngraph]$.
\end{definition}

\begin{definition}\label{def:equiv on N-graph and N-basis}
Let $(\ngraph, \nbasis)$ and $(\ngraph', \nbasis')$ be pairs of an $N$-graph and good tuples of one-cycles. We say that $(\ngraph, \nbasis)$ and $(\ngraph', \nbasis')$ are \emph{equivalent} if $[\ngraph]=[\ngraph']$ and the induced isomorphism $H_1(\Legendrian(\ngraph))\cong H_1(\Legendrian(\ngraph'))$ identifies $\nbasis$ with $\nbasis'$. We denote the equivalent class of $(\ngraph, \nbasis)$ by $[\ngraph, \nbasis]$.
\end{definition}

Let us recall from \cite{CZ2020} the construction of the algebraic invariant $\cM(\ngraph)$ of the Legendrian weave $\Legendrian(\ngraph)$ by considering legible model of the moduli spaces of constructible sheaves associated to $\Legendrian(\ngraph)$ as follows:

\begin{definition}[\cite{CZ2020}]\label{def:flag moduli space}
Let $\ngraph\subset \disk^2$ be an $N$-graph. Let $\{F_i\}_{i\in I}$ be a set of closures of connected components of $\disk^2\setminus \ngraph$, call each closure a \emph{face}. The \emph{framed flag moduli space} $\widetilde \cM(\ngraph)$ is a collection of \emph{flags} $\cF_{\Legendrian(\ngraph)}=\{\cF^\bullet(F_i)\}_{i\in I}$ in $\C^N$ such that for any pair of faces $F_1$ and $F_2$  sharing an edge in $\ngraph_i$, the corresponding flags $\cF^\bullet(F_1)$ and $\cF^\bullet(F_2)$ satisfy
\begin{align}\label{equation:flag conditions}
\begin{cases}
\cF^j(F_1)=\cF^j(F_2), \qquad 0\leq j \leq N, \quad j\neq i;\\
\cF^i(F_1)\neq \cF^i(F_2).
\end{cases}
\end{align}

Let us consider the general linear group $\operatorname{GL}_N$ action on $\cM(\ngraph)$ by acting on all flags at once. The \emph{flag moduli space} of the $N$-graph $\ngraph$ is defined by the quotient space (a stack, in general)
\[
\cM(\ngraph):=\widetilde{\cM}(\ngraph)/\operatorname{GL}_N.
\] 
\end{definition}

From now on, we will regard flags $\cF_{\Legendrian(\ngraph)}$ as a formal parameter for the flag moduli space $\cM(\ngraph)$.

\begin{theorem}[{\cite[Theorem~5.3]{CZ2020}}]
The flag moduli space $\cM(\ngraph)$ is
a Legendrian isotopy invariant of $\Legendrian(\ngraph)$.
\end{theorem}

Let $\legendrian=\legendrian_\beta$ be a Legendrian in $J^1\sphere^1$, which gives us an $(N-1)$-tuple $X=(X_1,\dots, X_{N-1})$ of points  in $\sphere^1$ which given by the alphabet $\sigma_1,\dots,\sigma_{N-1}$ of the braid word $\beta$. Let $\{f_j\}_{j\in J}$ be the set of closures of connected components of $\sphere^1\setminus X$. The flags $\flags=\{\cF^\bullet(f_j)\}_{j\in J}$ in $\C^N$ satisfying exactly the same conditions in \eqref{equation:flag conditions} will be called simply by \emph{flags on $\legendrian$}.
As before, we will regard $\flags$ as a formal parameter for the flag moduli space $\cM(\boundary \ngraph)$ of $\legendrian(\boundary\ngraph)$.

\begin{definition}\label{def:good N-graph}
Let $\ngraph\subset \disk^2$ be an $N$-graph, and let $\flags$ be flags adapted to $\legendrian\subset J^1\boundary\disk^2$ given by $\boundary\ngraph$. An $N$-graph $\ngraph$ is \emph{good}, if the flags $\flags$ uniquely determine flags $\cF_{\Legendrian(\ngraph)}$ in Definition~\ref{def:flag moduli space}.
\end{definition}

Note that $\ngraph(a,b,c)$ in the introduction is good in an obvious way. If an $N$-graph $\ngraph\subset \disk^2$ is good and $[\ngraph]=[\ngraph']$, then $\ngraph'$ is also good.

\subsection{Seeds from \texorpdfstring{$N$}{N}-graphs and their mutations}

\begin{definition}
For each a pair $(\ngraph, \nbasis)$ of an $N$-graph and a good tuple of cycles, we define a quiver $\quiver=\quiver(\ngraph,\nbasis)$ as follows:
\begin{enumerate}
\item the set of vertices is $[n]$ where $\nbasis=\{[\cycle_i]\mid i\in[n]\}\subset H_1(\Legendrian(\ngraph))$, and
\item the $(i,j)$-entry $b_{i,j}$ for $\qbasis(\quiver)=(b_{i,j})$ is the algebraic intersection number between $[\cycle_i]$ and $[\cycle_j]$, see Figure~\ref{fig:I-cycle with orientation and intersections}.
\end{enumerate}
\end{definition}

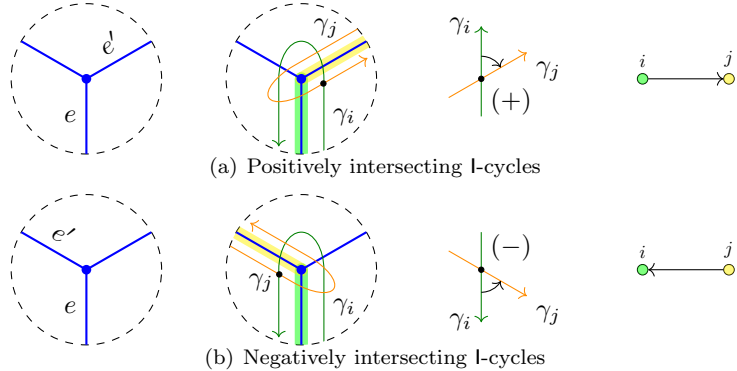
\begin{figure}[ht]
\subfigure[Positively intersecting $\sfI$-cycles]{
$
\def\arraycolsep{1pc}
\begin{array}{cccc}
\begin{tikzpicture}[baseline=-.5ex]
\draw [dashed] (0,0) circle [radius=1];
\clip (0,0) circle (1);
\draw [blue, thick] (0,0)--({cos(30)},{sin(30)}) node[above, midway, rotate=30] {\color{black}$e'$};
\draw [blue, thick] (0,0)--({cos(-90)},{sin(-90)}) node[left, midway] {\color{black}$e$};
\draw [blue, thick] (0,0)--({cos(150)},{sin(150)});
\draw[thick,blue,fill=blue] (0,0) circle (0.05);
\end{tikzpicture}
&
\begin{tikzpicture}[baseline=-.5ex]
\draw [dashed] (0,0) circle [radius=1];
\clip (0,0) circle (1);
\draw [green, line cap=round, line width=5, opacity=0.5] (0,0) to (0,-1);
\draw [yellow, line cap=round, line width=5, opacity=0.5] (0,0) to ({cos(30)},{sin(30)});
\draw [blue, thick] (0,0)--({cos(30)},{sin(30)});
\draw [blue, thick] (0,0)--({cos(-90)},{sin(-90)});
\draw [blue, thick] (0,0)--({cos(150)},{sin(150)});
\draw[thick,blue,fill=blue] (0,0) circle (0.05);
\draw [green!50!black,->] (0.3,-1) -- node[midway,right] {\color{black}$\cycle_i$} (0.3, 0) arc (0:180:0.3 and 0.5) -- (-0.3,-0.9);
\begin{scope}[rotate=120]
\draw [orange!90!yellow,->] (0.2,-1) -- node[midway,above] {\color{black}$\cycle_j$} (0.2, 0) arc (0:180:0.2 and 0.5) -- (-0.2,-0.9);
\end{scope}
\draw [fill] (-12:0.3) circle (1pt);
\end{tikzpicture}
&
\begin{tikzpicture}[baseline=-.5ex]
\begin{scope}
\draw [green!50!black,->] (90:-0.5) -- (90:0.7) node[left] {\color{black}$\cycle_i$};
\draw [orange!90!yellow,->] (30:-0.5) -- (30:0.7) node[below right] {\color{black}$\cycle_j$};
\draw [fill] (0,0) circle (1pt);
\draw [->] (90:0.3) arc (90:30:0.3);
\draw (0,0) node[below right] {$(+)$};
\end{scope}
\end{tikzpicture}
&
\begin{tikzpicture}[baseline=-.5ex]
\tikzstyle{state}=[draw, circle, inner sep = 0.07cm]
\tikzset{every node/.style={scale=0.7}}    
\node[state, label=above:{$i$}] (1) at (3,0) {};
\node[state, label=above:{$j$}] (2) [right = of 1] {};
\node[ynode] at (2) {};
\node[gnode] at (1) {};
\draw[->] (1)--(2);
\end{tikzpicture}
\end{array}
$
}
\subfigure[Negatively intersecting $\sfI$-cycles]{
$
\def\arraycolsep{1pc}
\begin{array}{cccc}
\begin{tikzpicture}[baseline=-.5ex]
\draw [dashed] (0,0) circle [radius=1];
\clip (0,0) circle (1);
\draw [blue, thick] (0,0)--({cos(30)},{sin(30)});
\draw [blue, thick] (0,0)--({cos(-90)},{sin(-90)}) node[left, midway] {\color{black}$e$};
\draw [blue, thick] (0,0)--({cos(150)},{sin(150)}) node[above, midway, rotate=-30] {\color{black}$e'$};
\draw[thick,blue,fill=blue] (0,0) circle (0.05);
\end{tikzpicture}
&
\begin{tikzpicture}[baseline=-.5ex]
\draw [dashed] (0,0) circle [radius=1];
\clip (0,0) circle (1);
\draw [green, line cap=round, line width=5, opacity=0.5] (0,0) to (0,-1);
\draw [yellow, line cap=round, line width=5, opacity=0.5] (0,0) to ({cos(150)},{sin(150)});
\draw [blue, thick] (0,0)--({cos(30)},{sin(30)});
\draw [blue, thick] (0,0)--({cos(-90)},{sin(-90)});
\draw [blue, thick] (0,0)--({cos(150)},{sin(150)});
\draw[thick,blue,fill=blue] (0,0) circle (0.05);
\draw [green!50!black,->] (0.3,-1) -- node[midway,right] {\color{black}$\cycle_i$} (0.3, 0) arc (0:180:0.3 and 0.5) -- (-0.3,-0.9);
\begin{scope}[rotate=-120]
\draw [orange!90!yellow,->] (0.2,-1) -- node[midway,below] {\color{black}$\cycle_j$} (0.2, 0) arc (0:180:0.2 and 0.5) -- (-0.2,-0.9);
\end{scope}
\draw [fill] (-168:0.3) circle (1pt);
\end{tikzpicture}
&
\begin{tikzpicture}[baseline=-.5ex]
\begin{scope}
\draw [green!50!black,->] (-90:-0.5) -- (-90:0.7) node[left] {\color{black}$\cycle_i$};
\draw [orange!90!yellow,->] (-30:-0.5) -- (-30:0.7) node[below right] {\color{black}$\cycle_j$};
\draw [fill] (0,0) circle (1pt);
\draw [->] (-90:0.3) arc (-90:-30:0.3);
\draw (0,0) node[above right] {$(-)$};
\end{scope}
\end{tikzpicture}
&
\begin{tikzpicture}[baseline=-.5ex]
\tikzstyle{state}=[draw, circle, inner sep = 0.07cm]
\tikzset{every node/.style={scale=0.7}}    
\node[state, label=above:{$i$}] (1) at (3,0) {};
\node[state, label=above:{$j$}] (2) [right = of 1] {};
\node[ynode] at (2) {};
\node[gnode] at (1) {};
\draw[->] (2)--(1);
\end{tikzpicture}
\end{array}
$
}
\caption{$\sfI$-cycles with intersections.}
\label{fig:I-cycle with orientation and intersections}
\end{figure}

In order to assign a cluster variable to each one-cycle, let us consider the microlocal monodromy functor
\[
\mmon:\cM(\ngraph) \to\Loc(\Legendrian(\ngraph)),
\]
which sends flags $\{\cF^\bullet(F_i)\}_{i \in I}\in\cM(\ngraph)$ to rank-one local systems $\mmon(\cF_{\Legendrian(\ngraph)})$ on the Legendrian surface $\Legendrian(\ngraph)$. Then (cluster) variables $\bfx$ for the triple $(\ngraph, \nbasis, \cF_{\Legendrian(\ngraph)})$ are defined by
\[
\bfx=\left(
\mmon(\cF_{\Legendrian(\ngraph)})([\cycle_1]),
\dots,
\mmon(\cF_{\Legendrian(\ngraph)})([\cycle_n])\right).
\]
Let us denote the above assignment by 
\[
\Psi(\ngraph, \nbasis, \cF_{\Legendrian(\ngraph)})=(\bfx(\Legendrian(\ngraph), \nbasis, \cF_{\Legendrian(\ngraph)}),\quiver(\Legendrian(\ngraph),\nbasis)).
\]

Especially when an $N$-graph $\ngraph$ is good, see Definition~\ref{def:good N-graph}, $\cF_{\Legendrian(\ngraph)}$ is determined by the flags $\flags\in \Sh_{\legendrian}^1(\boundary\disk^2\times\R)$ at the boundary, where the Legendrian link $\legendrian$ is given by $\boundary\ngraph$. Then, by the functorial property of the microlocal monodromy functor $\mmon$, we have
\begin{theorem}[{\cite[\S7.2.1]{CZ2020}}]\label{thm:N-graph to seed}
Let $\ngraph\subset \disk^2$ be a good $N$-graph with a good tuple $\nbasis$ of cycles in $H_1(\Legendrian(\ngraph))$, and with flags $\flags$ on $\legendrian\subset J^1\sphere^1$ at the boundary. Then the assignment $\Psi$ to a seed in a cluster structure
\[
\Psi(\ngraph,\nbasis,\flags)= (\bfx(\Legendrian(\ngraph),\nbasis,\flags),\quiver(\Legendrian(\ngraph),\nbasis))
\]
is well-defined up to Legendrian isotopy.
\end{theorem}

In turn, this gives a tool to distinguish exact Lagrangian fillings as follows:

\begin{corollary}\label{corollary:distinct seeds imples distinct fillings}
As in the above setup, if two triples $(\ngraph, \nbasis, \flags)$, $(\ngraph', \nbasis', \flags)$ with the same boundary condition define different seeds, then two induced Lagrangian fillings $\pi\circ\iota(\Legendrian(\ngraph))$, $\pi\circ\iota(\Legendrian(\ngraph'))$ bounding $\iota(\legendrian)$ are not exact Lagrangian isotopic to each other.
\end{corollary}

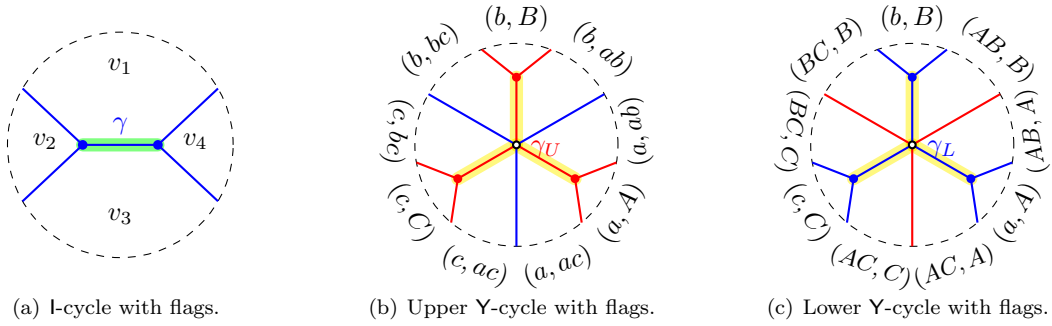
\begin{figure}[ht]
\begin{tikzcd}
\subfigure[$\sfI$-cycle with flags.\label{figure:I-cycle with flags}]{
\begin{tikzpicture}[baseline=.5ex]
\useasboundingbox (-2,-2) rectangle (2,2);
\begin{scope}
\draw [dashed] (0,0) circle [radius=1.5];

\draw [green, line cap=round, line width=5, opacity=0.5] (-1/2,0) to (1/2,0);

\draw [blue, thick] ({-3*sqrt(3)/4},3/4)--(-1/2,0);
\draw [blue, thick] ({-3*sqrt(3)/4},-3/4)--(-1/2,0);
\draw [blue, thick] ({3*sqrt(3)/4},3/4)--(1/2,0);
\draw [blue, thick] ({3*sqrt(3)/4},-3/4)--(1/2,0);
\draw [blue, thick] (-1/2,0)--(1/2,0) node[above, midway] {$\cycle$};

\draw[thick,blue,fill=blue] (-1/2,0) circle (0.05);
\draw[thick,blue,fill=blue] (1/2,0) circle (0.05);

\node at (0,1) {$v_1$};
\node at (-1,0) {$v_2$};
\node at (0,-1) {$v_3$};
\node at (1,0) {$v_4$};

\end{scope}
\end{tikzpicture}}
&
\subfigure[Upper $\sfY$-cycle with flags.\label{figure:Y-cycle with flags I}]{
\begin{tikzpicture}[baseline=.5ex]
\useasboundingbox (-2,-2) rectangle (2,2);
\begin{scope}[scale=0.9]
\draw [dashed] (0,0) circle [radius=1.5];

\draw [yellow, line cap=round, line width=5, opacity=0.5] (0,0) to ({cos(0-30)},{sin(0-30)});
\draw [yellow, line cap=round, line width=5, opacity=0.5] (0,0) to ({cos(120-30)},{sin(120-30)});
\draw [yellow, line cap=round, line width=5, opacity=0.5] (0,0) to ({cos(240-30)},{sin(240-30)});

\draw [blue, thick] ({1.5*cos(180-30)},{1.5*sin(180-30)})--(0,0)--({1.5*cos(60-30)},{1.5*sin(60-30)});
\draw [blue, thick] (0,0)--({1.5*cos(60+30)},{-1.5*sin(60+30)});

\draw [red, thick] ({1.5*cos(20-30)},{1.5*sin(20-30)})--({cos(0-30)},{sin(0-30)})--(0,0);
\draw [red, thick] ({1.5*cos(20+30)},{-1.5*sin(20+30)})--({cos(0-30)},{sin(0-30)});

\draw [red, thick] ({1.5*cos(100-30)},{1.5*sin(100-30)})--({cos(120-30)},{sin(120-30)})--(0,0);
\draw [red, thick] ({1.5*cos(140-30)},{1.5*sin(140-30)})--({cos(120-30)},{sin(120-30)});

\draw [red, thick] ({1.5*cos(220-30)},{1.5*sin(220-30)})--({cos(240-30)},{sin(240-30)})--(0,0);
\draw [red, thick] ({1.5*cos(260-30)},{1.5*sin(260-30)})--({cos(240-30)},{sin(240-30)});

\draw[thick,red,fill=red] ({cos(0-30)},{sin(0-30)}) circle (0.05);
\draw[thick,red,fill=red] ({cos(120-30)},{sin(120-30)}) circle (0.05);
\draw[thick,red,fill=red] ({cos(240-30)},{sin(240-30)}) circle (0.05);
\draw[thick,black,fill=white] (0,0) circle (0.05);
\end{scope}
\begin{scope}[yshift=-0.1cm]
\node at (0,1.7) {$(b,B)$};
\node[rotate=60] at ({1.7*cos(-30)},{1.7*sin(-30)}) {$(a,A)$};
\node[rotate=-60] at ({1.7*cos(-150)},{1.7*sin(-150)}) {$(c,C)$};

\node[rotate=100] at ({1.7*cos(10)},{1.7*sin(10)}) {$(a,ab)$};
\node[rotate=-100] at ({1.7*cos(170)},{1.7*sin(170)}) {$(c,bc)$};

\node[rotate=20] at ({1.7*cos(-70)},{1.7*sin(-70)}) {$(a,ac)$};
\node[rotate=-20] at ({1.7*cos(-110)},{1.7*sin(-110)}) {$(c,ac)$};

\node[rotate=-40] at ({1.7*cos(50)},{1.7*sin(50)}) {$(b,ab)$};
\node[rotate=40] at ({1.7*cos(130)},{1.7*sin(130)}) {$(b,bc)$};

\draw[red] (0,0) node[xshift=0.4cm] {$\cycle_U$};
\end{scope}
\end{tikzpicture}}
&
\subfigure[Lower $\sfY$-cycle with flags.\label{figure:Y-cycle with flags II}]{
\begin{tikzpicture}[baseline=.5ex]
\useasboundingbox (-2,-2) rectangle (2,2);
\begin{scope}[scale=0.9]
\draw [dashed] (0,0) circle [radius=1.5];

\draw [yellow, line cap=round, line width=5, opacity=0.5] (0,0) to ({cos(0-30)},{sin(0-30)});
\draw [yellow, line cap=round, line width=5, opacity=0.5] (0,0) to ({cos(120-30)},{sin(120-30)});
\draw [yellow, line cap=round, line width=5, opacity=0.5] (0,0) to ({cos(240-30)},{sin(240-30)});

\draw [red, thick] ({1.5*cos(180-30)},{1.5*sin(180-30)})--(0,0)--({1.5*cos(60-30)},{1.5*sin(60-30)});
\draw [red, thick] (0,0)--({1.5*cos(60+30)},{-1.5*sin(60+30)});

\draw [blue, thick] ({1.5*cos(20-30)},{1.5*sin(20-30)})--({cos(0-30)},{sin(0-30)})--(0,0);
\draw [blue, thick] ({1.5*cos(20+30)},{-1.5*sin(20+30)})--({cos(0-30)},{sin(0-30)});

\draw [blue, thick] ({1.5*cos(100-30)},{1.5*sin(100-30)})--({cos(120-30)},{sin(120-30)})--(0,0);
\draw [blue, thick] ({1.5*cos(140-30)},{1.5*sin(140-30)})--({cos(120-30)},{sin(120-30)});

\draw [blue, thick] ({1.5*cos(220-30)},{1.5*sin(220-30)})--({cos(240-30)},{sin(240-30)})--(0,0);
\draw [blue, thick] ({1.5*cos(260-30)},{1.5*sin(260-30)})--({cos(240-30)},{sin(240-30)});

\draw[thick,blue,fill=blue] ({cos(0-30)},{sin(0-30)}) circle (0.05);
\draw[thick,blue,fill=blue] ({cos(120-30)},{sin(120-30)}) circle (0.05);
\draw[thick,blue,fill=blue] ({cos(240-30)},{sin(240-30)}) circle (0.05);
\draw[thick,black,fill=white] (0,0) circle (0.05);
\end{scope}
\begin{scope}[yshift=-0.1cm]
\node at (0,1.7) {$(b,B)$};
\node[rotate=60] at ({1.7*cos(-30)},{1.7*sin(-30)}) {$(a,A)$};
\node[rotate=-60] at ({1.7*cos(-150)},{1.7*sin(-150)}) {$(c,C)$};

\node[rotate=100] at ({1.7*cos(10)},{1.7*sin(10)}) {\small$(AB,A)$};
\node[rotate=-100] at ({1.7*cos(170)},{1.7*sin(170)}) {\small$(BC,C)$};

\node[rotate=20] at ({1.7*cos(-70)},{1.7*sin(-70)}) {\small$(AC,A)$};
\node[rotate=-20] at ({1.7*cos(-110)},{1.7*sin(-110)}) {\small$(AC,C)$};

\node[rotate=-40] at ({1.7*cos(50)},{1.7*sin(50)}) {\small$(AB,B)$};
\node[rotate=40] at ({1.7*cos(130)},{1.7*sin(130)}) {\small$(BC,B)$};

\draw[blue] (0,0) node[xshift=0.4cm] {$\cycle_L$};
\end{scope}
\end{tikzpicture}}
\end{tikzcd}
\caption{$\sfI$- and $\sfY$-cycles with flags.}
\label{fig:I and Y cycle with flags}
\end{figure}

Let us consider an $\sfI$-cycle $[\cycle]$ represented by a loop $\cycle(e)$ for some monochromatic edge $e$ as in Figure~\ref{figure:I-cycle with flags}. Let us denote four flags corresponding to each region by $F_1,F_2,F_3,F_4$, respectively. Suppose that $e \subset \ngraph_i$, then by the construction of flag moduli space $\cM(\ngraph)$, a two-dimensional vector space $V:=\cF^{i+1}(F_*)/\cF^{i-1}(F_*)$ is independent of $*=1,2,3,4$. Moreover, $\cF^{i}(F_*)/\cF^{i-1}(F_*)$ defines a one-dimensional subspace $v_*\subset V$ for $*=1,2,3,4$, satisfying
\[
v_1\neq v_2 \neq v_3 \neq v_4 \neq v_1.
\]
Then $\mmon(\cF_{\Legendrian(\ngraph)})$ along the one-cycle $[\gamma(e)]$ is defined by the cross ratio 
\[
\mmon(\cF_{\Legendrian(\ngraph)})([\cycle])\coloneqq\langle v_1,v_2,v_3,v_4 \rangle=\frac{v_1 \wedge v_2}{v_2 \wedge v_3}\cdot\frac{v_3 \wedge v_4}{v_4\wedge v_1}.
\]

Suppose that local flags $\{F_j\}_{j\in J}$ near the upper $\sfY$-cycle $[\cycle_U]$ look like in Figure~\ref{figure:Y-cycle with flags I}. Let $\ngraph_i$ and $\ngraph_{i+1}$ be the $N$-subgraphs in red and blue, respectively. Then the $3$-dimensional vector space $V=\cF^{i+2}(F_*)/\cF^{i-1}(F_*)$ is independent of $*\in J$. Now regard $a,b,c$ and $A,B,C$ are subspaces of $V$ of dimension one and two, respectively. Then the microlocal monodromy along the $\sfY$-cycle $[\cycle_U]$  becomes
\[
\mmon(\cF_{\Legendrian(\ngraph)})([\cycle_U])\coloneqq\frac{B(a)C(b)A(c)}{B(c)C(a)A(b)}.
\]
Here $B(a)$ can be seen as a paring between the vector $a$ and the covector $B$.

Now consider the lower $\sfY$-cycle $[\cycle_L]$ whose local flags given as in Figure~\ref{figure:Y-cycle with flags II}. We already have seen that the orientation convention of the loop in Figure~\ref{fig:I and Y cycle} for the upper and lower $\sfY$-cycle is different. Then microlocal monodromy along $[\cycle_L]$ follows the opposite orientation and becomes 
\[
\mmon(\cF_{\Legendrian(\ngraph)})([\cycle_L])\coloneqq\frac{C(a)B(c)A(b)}{C(b)B(a)A(c)}.
\]
Here, $B(a)$ is a pairing between the vector $B$ and covector $a$ which is the same as the above.

Let us define an operation called (\emph{Legendrian}) \emph{mutation} on $N$-graphs $\ngraph$ which corresponds to a geometric operation on the induced Legendrian surface $\Legendrian(\ngraph)$ that producing a smoothly isotopic but not necessarily Legendrian isotopic to $\Legendrian(\ngraph)$, see \cite[Definition~4.19]{CZ2020}.

\begin{definition}[{\cite[Definition~4.19]{CZ2020}}]\label{def:legendrian mutation}
Let $\ngraph$ be a (local) $N$-graph and $e\in \ngraph_i\subset \ngraph$ be an edge between two trivalent vertices corresponding to an $\sfI$-cycle $[\cycle]=[\cycle(e)]$. The mutation $\mutation_\cycle(\ngraph)$ of $\ngraph$ along $\cycle$ is obtained by applying the local change depicted in the left of Figure~\ref{fig:Legendrian mutation on N-graphs}.
\end{definition}

\begin{figure}[ht]
\begin{tikzcd}
\subfigure[A mutation along $\sfI$-cycle.\label{figure:I-mutation}]{
\begin{tikzpicture}[baseline=-.5ex,scale=1.2]
\begin{scope}
\draw [dashed] (0,0) circle [radius=1];
\draw [->,yshift=.5ex] (1.25,0) -- (1.75,0) node[midway, above] {$\mutation_\cycle$};
\draw [<-,yshift=-.5ex] (1.25,0) -- (1.75,0) node[midway, below] {$\mutation_{\cycle'}$};

\draw [green, line cap=round, line width=5, opacity=0.5] (-1/2,0) to (1/2,0);

\draw [blue, thick] ({-1*sqrt(3)/2},1*1/2)--(-1/2,0);
\draw [blue, thick] ({-1*sqrt(3)/2},-1*1/2)--(-1/2,0);
\draw [blue, thick] ({1*sqrt(3)/2},1*1/2)--(1/2,0);
\draw [blue, thick] ({1*sqrt(3)/2},-1*1/2)--(1/2,0);
\draw [blue, thick] (-1/2,0)-- node[midway,above] {$\cycle$}(1/2,0);

\draw[thick,blue,fill=blue] (-1/2,0) circle (0.05);
\draw[thick,blue,fill=blue] (1/2,0) circle (0.05);

\end{scope}

\begin{scope}[xshift=3cm]

\draw [dashed] (0,0) circle [radius=1];

\draw [green, line cap=round, line width=5, opacity=0.5] (0,-1/2) to (0,1/2);

\draw [blue, thick] (-1*1/2,{1*sqrt(3)/2}) to (0,1/2) to (0,-1/2)  node[midway,right] {$\cycle'$} to (-1*1/2,-{1*sqrt(3)/2});
\draw [blue, thick] (1*1/2,{1*sqrt(3)/2})--(0,1/2);
\draw [blue, thick] (1*1/2,-{1*sqrt(3)/2})--(0,-1/2);

\draw[thick,blue,fill=blue] (0,1/2) circle (0.05);
\draw[thick,blue,fill=blue] (0,-1/2) circle (0.05);
\end{scope}
\end{tikzpicture}
}
&
\subfigure[A mutations along $\sfY$-cycle.\label{figure:Y-mutation}]{
\begin{tikzpicture}[baseline=-.5ex,scale=1.2]
\begin{scope}

\draw [->,yshift=.5ex] (1.25,0) -- (1.75,0) node[midway, above] {$\mutation_\cycle$};
\draw [<-,yshift=-.5ex] (1.25,0) -- (1.75,0) node[midway, below] {$\mutation_{\cycle'}$};

\draw [dashed] (0,0) circle [radius=1];

\draw [yellow, line cap=round, line width=5, opacity=0.5] (0,1/2) to (0,0);
\draw [yellow, line cap=round, line width=5, opacity=0.5] ({sqrt(3)/4},-1/4) to (0,0);
\draw [yellow, line cap=round, line width=5, opacity=0.5] (-{sqrt(3)/4},-1/4) to (0,0);

\draw [blue, thick] (-1*1/2,{1*sqrt(3)/2}) to (0,1/2) to (0,0) node[right] {$\cycle$} to ({-sqrt(3)/4},-1/4) to (-1,0);
\draw [blue, thick] (1/2,{1*sqrt(3)/2}) to (0,1/2) to (0,0) to ({sqrt(3)/4},-1/4) to (1,0);
\draw [blue, thick] ({-sqrt(3)/4},-1/4) to ({-1/2},{-sqrt(3)/2});
\draw [blue, thick] ({sqrt(3)/4},-1/4) to ({1/2},{-sqrt(3)/2});

\draw [red, thick] (0,0) to ({sqrt(3)/2},1/2);
\draw [red, thick] (0,0) to ({-sqrt(3)/2},1/2);
\draw [red, thick] (0,0) to (0,-1);

\draw[thick,blue,fill=blue] (0,1/2) circle (0.05);
\draw[thick,blue,fill=blue] ({-sqrt(3)/4},-1/4) circle (0.05);
\draw[thick,blue,fill=blue] ({sqrt(3)/4},-1/4) circle (0.05);
\draw[thick,black,fill=white] (0,0) circle (0.05);
\end{scope}

\begin{scope}[xshift=3cm]

\draw [dashed] (0,0) circle [radius=1];

\draw [yellow, line cap=round, line width=5, opacity=0.5] (0,1/2) to (0,0);
\draw [yellow, line cap=round, line width=5, opacity=0.5] ({sqrt(3)/4},-1/4) to (0,0);
\draw [yellow, line cap=round, line width=5, opacity=0.5] (-{sqrt(3)/4},-1/4) to (0,0);

\draw [blue, thick] (-1/2,{sqrt(3)/2}) to ({-sqrt(3)/4},1/4) to (-1,0);
\draw [blue, thick] (1/2,{sqrt(3)/2}) to ({sqrt(3)/4},1/4) to (1,0);
\draw [blue, thick] ({-1/2},{-sqrt(3)/2}) to (0,-1/2) to ({1/2},{-sqrt(3)/2});
\draw [blue, thick] ({sqrt(3)/4},1/4) to (0,0);
\draw [blue, thick] ({-sqrt(3)/4},1/4) to (0,0);
\draw [blue, thick] (0,-1/2) to (0,0);

\draw [red, thick] ({-sqrt(3)/4},1/4) to ({-sqrt(3)/4},-1/4) to (0,-1/2) to ({sqrt(3)/4},-1/4) to ({sqrt(3)/4},1/4) to (0,1/2) to ({-sqrt(3)/4},1/4);
\draw [red, thick] ({sqrt(3)/4},1/4) to ({sqrt(3)/2},1/2);
\draw [red, thick] ({-sqrt(3)/4},1/4) to ({-sqrt(3)/2},1/2);
\draw [red, thick] (0,-1/2) to (0,-1);
\draw [red, thick] ({-sqrt(3)/4},-1/4) to (0,0);
\draw [red, thick] ({sqrt(3)/4},-1/4) to (0,0);
\draw [red, thick] (0,1/2) to (0,0) node[right] {$\cycle'$};

\draw[thick,red,fill=red] (0,1/2) circle (0.05);
\draw[thick,red,fill=red] ({-sqrt(3)/4},-1/4) circle (0.05);
\draw[thick,red,fill=red] ({sqrt(3)/4},-1/4) circle (0.05);
\draw[thick,black,fill=white] (0,0) circle (0.05);
\draw[thick,black,fill=white] ({-sqrt(3)/4},1/4) circle (0.05);
\draw[thick,black,fill=white] ({sqrt(3)/4},1/4) circle (0.05);
\draw[thick,black,fill=white] (0,-1/2) circle (0.05);
\end{scope}

\end{tikzpicture}
}
\end{tikzcd}
\caption{Legendrian mutations at $\sfI$- and $\sfY$-cycles.}
\label{fig:Legendrian mutation on N-graphs}
\end{figure}
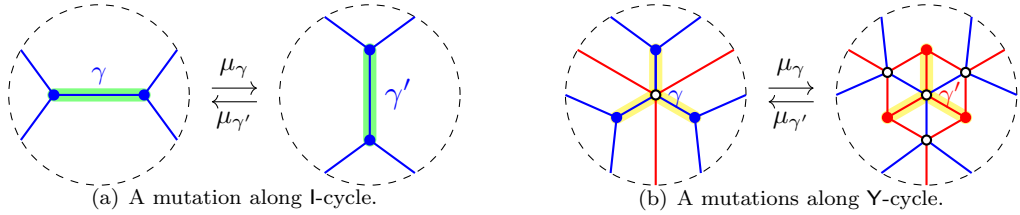

For the $\sfY$-cycle, the Legendrian mutation becomes as in the right of Figure~\ref{fig:Legendrian mutation on N-graphs}. Note that the mutation at $\sfY$-cycle can be decomposed into a sequence of Move~\Move{I} and Move~\Move{II} together with a mutation at $\sfI$-cycle.

Let us remind our main purpose of finding exact embedded Lagrangian fillings for a Legendrian links. The following lemma guarantees that Legendrian mutation preserves the embedding property of Lagrangian fillings.
\begin{proposition}[{\cite[Lemma~7.4]{CZ2020}}]
Let $\ngraph\subset \disk^2$ be a free $N$-graph. Then mutation $\mutation(\ngraph)$ at any $\sfI$- or $\sfY$-cycle is again free $N$-graph. 
\end{proposition}

\begin{proposition}
Let $\ngraph\subset \disk^2$ be a good $N$-graph. Then mutation $\mutation_\gamma(\ngraph)$ at an $\sfI$-cycle $\gamma$ is again a good $N$-graph.
\end{proposition}

\begin{proof}
The proof is straightforward from the notion of the good $N$-graph in Definition~\ref{def:good N-graph} and of the Legendrian mutation depicted in Figure~\ref{figure:I-mutation}. Note that the Legendrian mutation $\mutation_\gamma(\ngraph)$ at a $\sfY$-cycle $\gamma$ is also good, since $\mutation_\gamma(\ngraph)$ is a composition of Moves \Move{I} and \Move{II}, and a mutation at an $\sfI$-cycle.
\end{proof}

An important observation is the Legendrian mutation on $(\ngraph,\nbasis)$ induces a cluster mutation on the induced seed $(\bfx(\Legendrian(\ngraph),\nbasis,\flags),\quiver(\Legendrian(\ngraph),\nbasis))$.

\begin{proposition}[{\cite[\S7.2]{CZ2020}}]\label{proposition:equivariance of mutations}
Let $\ngraph\subset \disk^2$ be a good $N$-graph and $\nbasis$ a good tuple of cycles in~$H_1(\Legendrian(\ngraph))$. Let $\mutation_{\cycle_i}(\ngraph,\nbasis)$ be a Legendrian mutation  of $(\ngraph,\nbasis)$ along a one-cycle $\cycle_i$. Then, for flags $\flags$ on $\legendrian$, we have
\[
\Psi(\mutation_{\cycle_i}(\ngraph,\nbasis),\flags)=\mutation_{i}(\Psi(\ngraph,\nbasis,\flags)).
\]
Here, $\mutation_{i}$ is the cluster $\mathcal{X}$-mutation at the vertex $i$ \textup{(}cf. Remark~\ref{rmk_x_cluster_mutation}\textup{)}.
\end{proposition}

\section{Legendrian links and \texorpdfstring{$N$}{N}-graphs of \texorpdfstring{type ${\exdynD\exdynE}$}{Affine DE type}}
\label{sec:affine DE type}

Throughout this section, we denote $\dynX$ by Dynkin type of $\exdynD\exdynE$.
We investigate Legendrian links and $N$-graphs of type $\dynX$. 
We realize Coxeter mutations via $N$-graphs, and interpret them as Legendrian loops.
With these terminologies, we construct as many Lagrangian fillings as seeds for Legendrian links of type $\dynX$.

\subsection{Legendrian links of \texorpdfstring{type ${\exdynD\exdynE}$}{Affine DE type}}

Let us start by presenting Legendrian links of type $\dynX$.
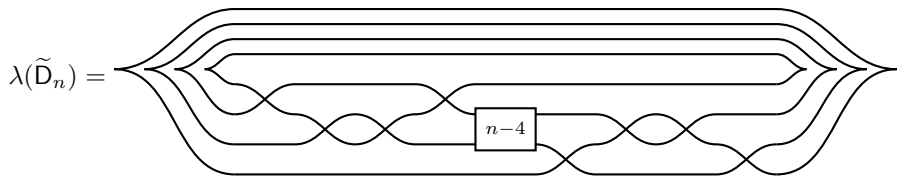
\begin{align*}
\legendrian({\exdynD}_n)=
\begin{tikzpicture}[baseline=0ex,scale=0.8]
\draw[thick] (0,0) to[out=0,in=180] (1,-0.5) to[out=0,in=180] (2,-1) to[out=0,in=180] (3,-0.5) to[out=0,in=180] (4,0) to[out=0,in=180] (9,0);
\draw[thick] (0,-0.5) to[out=0,in=180] (1,0) to[out=0,in=180] (3,0) to[out=0,in=180] (4,-0.5) (5,-0.5) to[out=0,in=180] (6,-0.5) to[out=0,in=180] (7,-1) to[out=0,in=180] (8,-0.5) to[out=0,in=180] (9,-0.5);
\draw[thick] (0,-1) to[out=0,in=180] (1,-1) to[out=0,in=180] (2,-0.5) to[out=0,in=180] (3,-1) to[out=0,in=180] (4,-1) (5,-1) to[out=0,in=180] (6,-1.5) to[out=0,in=180] (8,-1.5) to[out=0,in=180] (9,-1);
\draw[thick] (0,-1.5) to[out=0,in=180] (5,-1.5) to[out=0,in=180] (6,-1) to[out=0,in=180] (7,-0.5) to[out=0,in=180] (8,-1) to[out=0,in=180] (9,-1.5);
\draw[thick] (4,-0.4) rectangle node {$\scriptstyle{n-4}$} (5, -1.1);
\draw[thick] (0,0) to[out=180,in=0] (-0.5,0.25) to[out=0,in=180] (0,0.5) to[out=0,in=180] (9,0.5) to[out=0,in=180] (9.5,0.25) to[out=180,in=0] (9,0);
\draw[thick] (0,-0.5) to[out=180,in=0] (-1,0.25) to[out=0,in=180] (0,0.75) to[out=0,in=180] (9,0.75) to[out=0,in=180] (10,0.25) to[out=180,in=0] (9,-0.5);
\draw[thick] (0,-1) to[out=180,in=0] (-1.5,0.25) to[out=0,in=180] (0,1) to[out=0,in=180] (9,1) to[out=0,in=180] (10.5,0.25) to[out=180,in=0] (9,-1);
\draw[thick] (0,-1.5) to[out=180,in=0] (-2,0.25) to[out=0,in=180] (0,1.25) to[out=0,in=180] (9,1.25) to[out=0,in=180] (11,0.25) to[out=180,in=0] (9,-1.5);
\end{tikzpicture}
\end{align*}
\begin{align*}
\legendrian(a,b,c)=
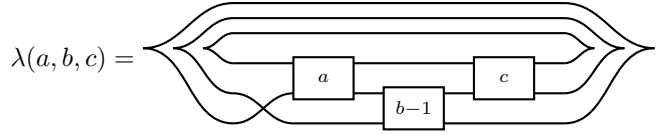
\begin{tikzpicture}[baseline=5ex,scale=0.8]
\draw[thick] (0,0) to[out=0,in=180] (1,0.5) (2,0.5) to (2.5,0.5) (3.5,0.5) to (4,0.5) (5,0.5) to (5.5,0.5);
\draw[thick] (0,0.5) to[out=0,in=180] (1,0) to (2.5,0) (3.5,0) to (5.5,0);
\draw[thick] (0,1) to[out=0,in=180] (1,1) (2,1) to (4,1) (5,1) to (5.5,1);
\draw[thick] (1,0.4) rectangle node {$\scriptstyle a$} (2, 1.1);
\draw[thick] (2.5,-0.1) rectangle node {$\scriptstyle{b-1}$} (3.5, 0.6);
\draw[thick] (4,0.4) rectangle node {$\scriptstyle c$} (5, 1.1);
\draw[thick] (0,1) to[out=180, in=0] (-0.5,1.25) to[out=0,in=180] (0,1.5) to (5.5,1.5) to[out=0,in=180] (6,1.25) to[out=180,in=0] (5.5,1);
\draw[thick] (0,0.5) to[out=180, in=0] (-1,1.25) to[out=0,in=180] (0,1.75) to (5.5,1.75) to[out=0,in=180] (6.5,1.25) to[out=180,in=0] (5.5,0.5);
\draw[thick] (0,0) to[out=180, in=0] (-1.5,1.25) to[out=0,in=180] (0,2) to (5.5,2) to[out=0,in=180] (7,1.25) to[out=180,in=0] (5.5,0);
\end{tikzpicture}
\end{align*}

Note that $\legendrian({\exdynE}_6)=\legendrian(3,3,3)$, $\legendrian({\exdynE}_7)=\legendrian(2,4,4)$, and $\legendrian({\exdynE}_8)=\legendrian(2,3,6)$ each of which comes from the triples $(a,b,c)$ satisfying ${\frac1 a} +{ \frac1 b} + {\frac1 c} =1$. By the work of \cite{SW2019,GSW2020b} the Legendrian link $\legendrian({\exdynD}_n)$ in $\R^3$ admits the brick quiver diagram $\quiver^{\mathsf{brick}}({\exdynD}_n).$

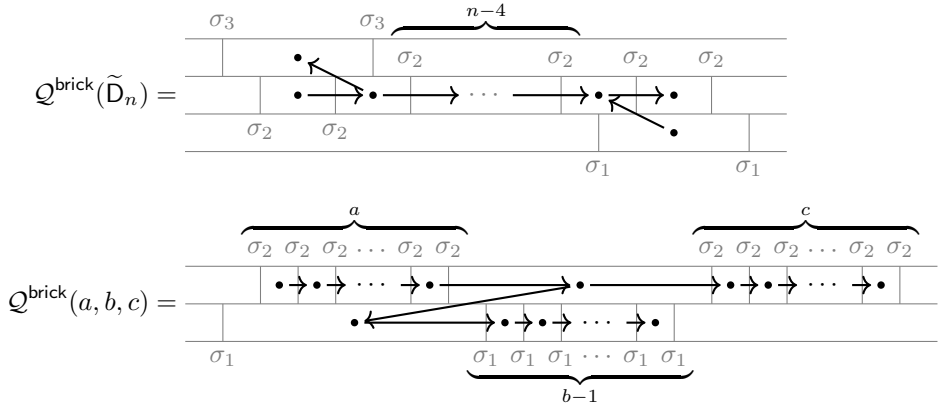
\begin{align*}
\quiver^{\mathsf{brick}}({\exdynD}_n)&=
\begin{tikzpicture}[baseline=-5.5ex]
\draw[gray] (0,0) to (8,0) (0,-0.5) to (8,-0.5) (0,-1) to (8,-1) (0,-1.5) to (8,-1.5);
\draw[gray] (0.5,0) node[above] {$\sigma_3$} to (0.5,-0.5) (2.5,0) node[above] {$\sigma_3$} to (2.5,-0.5) 
(1,-0.5) to (1,-1) node[below] {$\sigma_2$} (2,-0.5) to (2,-1)node[below] {$\sigma_2$} (3,-0.5) node[above] {$\sigma_2$} to (3,-1) (5,-0.5) node[above] {$\sigma_2$} to (5,-1) (6,-0.5)  node[above] {$\sigma_2$} to (6,-1) (7,-0.5)  node[above] {$\sigma_2$} to (7,-1)
(5.5,-1) to (5.5,-1.5)node[below] {$\sigma_1$} (7.5,-1) to (7.5,-1.5) node[below] {$\sigma_1$} (4,-0.75) node (dots) {$\cdots$};
\draw[thick,fill] 
(1.5,-0.25) circle (1pt) node (D1) {}
(1.5,-0.75) circle (1pt) node (D2) {}
(2.5,-0.75) circle (1pt) node (D3) {}
(5.5,-0.75) circle (1pt) node (D4) {}
(6.5,-0.75) circle (1pt) node (D5) {}
(6.5,-1.25) circle (1pt) node (D6) {}
;
\draw (4,-0.5) node[yshift=5ex] {$\overbrace{\hphantom{\hspace{2.5cm}}}^{n-4}$};
\draw[thick,->] (D2) to (D3);
\draw[thick,->] (D3) to (dots);
\draw[thick,->] (dots) to (D4);
\draw[thick,->] (D4) to (D5);
\draw[thick,->] (D3) to (D1);
\draw[thick,->] (D6) to (D4);
\end{tikzpicture}\\
\quiver^{\mathsf{brick}}(a,b,c)&=\begin{tikzpicture}[baseline=-.5ex,scale=1]
\draw[gray] (-5,0.5) -- (5,0.5) (-5,0) -- (5,0) (-5,-0.5) -- (5,-0.5);
\draw[gray] (-4.5,0) -- (-4.5,-0.5) node[below] {$\sigma_1$} (-4,0) -- (-4, 0.5) node[above] {$\sigma_2$} (-3.5, 0) -- (-3.5, 0.5) node[above] {$\sigma_2$} (-3, 0) -- (-3, 0.5) node[above] {$\sigma_2$} (-2.5,0.5) node[above] {$\cdots$} (-2,0) -- (-2,0.5) node[above] {$\sigma_2$} (-1.5,0) -- (-1.5,0.5) node[above] {$\sigma_2$};
\draw[gray] (-1,0) -- (-1, -0.5) node[below] {$\sigma_1$} (-0.5,0) -- (-0.5, -0.5) node[below] {$\sigma_1$} (0, 0) -- (0, -0.5) node[below] {$\sigma_1$} (0.5,-0.5) node[below] {$\cdots$} (1,0) -- (1,-0.5) node[below] {$\sigma_1$} (1.5,0) -- (1.5, -0.5) node[below] {$\sigma_1$};
\draw[gray] (2,0) -- (2,0.5) node[above] {$\sigma_2$} (2.5,0) -- (2.5,0.5) node[above] {$\sigma_2$} (3,0) -- (3,0.5) node[above] {$\sigma_2$} (3.5,0.5) node[above] {$\cdots$} (4,0) -- (4,0.5) node[above] {$\sigma_2$} (4.5,0) -- (4.5,0.5) node[above] {$\sigma_2$};
\draw (-2.75,0.5) node[yshift=4ex] {$\overbrace{\hphantom{\hspace{3cm}}}^a$};
\draw (0.25,-0.5) node[yshift=-4ex] {$\underbrace{\hphantom{\hspace{3cm}}}_{b-1}$};
\draw (3.25,0.5) node[yshift=4ex] {$\overbrace{\hphantom{\hspace{3cm}}}^c$};
\draw[thick, fill] (-3.75, 0.25) circle (1pt) node (A1) {} (-3.25, 0.25) circle (1pt) node (A2) {} (-1.75, 0.25) circle (1pt) node (A3) {} (0.25, 0.25) circle (1pt) node (Aa) {} (-2.5,0.25) node (Adots) {$\cdots$} (2.25, 0.25) circle (1pt) node (B1) {} (2.75,0.25) circle (1pt) node (B2) {} (4.25, 0.25) circle (1pt) node (Bb) {} (3.5,0.25) node (Bdots) {$\cdots$};
\draw[thick, fill] (-2.75, -0.25) circle (1pt) node (C1) {} (-0.75, -0.25) circle (1pt) node (C2) {} (-0.25, -0.25) circle (1pt) node (C3) {} (1.25, -0.25) circle (1pt) node (Cc) {} (0.5, -0.25) node (Cdots) {$\cdots$};
\draw[thick,->] (A1) -- (A2);
\draw[thick,->] (A2) -- (Adots);
\draw[thick,->] (Adots) -- (A3);
\draw[thick,->] (A3) -- (Aa);
\draw[thick,->] (Aa) -- (B1);
\draw[thick,->] (B1) -- (B2);
\draw[thick,->] (B2) -- (Bdots);
\draw[thick,->] (Bdots) -- (Bb);
\draw[thick,->] (C1) -- (C2);
\draw[thick,->] (C2) -- (C3);
\draw[thick,->] (C3) -- (Cdots);
\draw[thick,->] (Cdots) -- (Cc);
\draw[thick,->] (Aa) -- (C1);
\end{tikzpicture}
\end{align*}

Note that the moduli $\cM_1(\legendrian(\dynX))$ of microlocal rank one sheaves in $\Sh^\bullet_{\legendrian(\dynX)}(\R^2)$ is a Legendrian invariant and its coordinate ring is isomorphic to cluster algebra of type $\dynX$.

The Legendrians $\legendrian({\exdynD}_n)$ and $\legendrian(a,b,c)$ are rainbow closures of the following positive braids respectively:
\begin{align*}
\beta({\exdynD}_n)&=\sigma_3\sigma_2\sigma_2\sigma_3 \sigma_2^{n-4} \sigma_1 \sigma_2 \sigma_2 \sigma_1,&
\beta(a,b,c)&=\sigma_1 \sigma_2^a \sigma_1^{b-1} \sigma_2^c.
\end{align*}
So induced links in $J^1\sphere^1$ have the following braid presentation.
\begin{align*}
\widehat\beta({\exdynD}_n)
&=\Delta_4(\sigma_3\sigma_2\sigma_2\sigma_3 \sigma_2^{n-4} \sigma_1 \sigma_2 \sigma_2\sigma_1)\Delta_4,&
\widehat\beta(a,b,c)&=\Delta_3 \sigma_1 \sigma_2^a \sigma_1^{b-1} \sigma_2^c \Delta_3,
\end{align*}
where $\Delta_3=\sigma_2 \sigma_1 \sigma_2$ and $\Delta_4=\sigma_1 \sigma_2 \sigma_1 \sigma_3 \sigma_2 \sigma_1$ are half twists of 3- and 4-strand braid, respectively.

Before considering $N$-graphs bounding $\legendrian(\dynX)$, let us manipulate its corresponding braid presentation to obtain simpler $N$-graphs.
For $\dynX = {\exdynD_n}$, let $k=\lfloor \frac{n-3}2\rfloor$ and $\ell=\lfloor \frac{n-4}2\rfloor$. Then we have the following computation.

\begin{align}\label{eqn:front braid to N-graph braid}\allowdisplaybreaks
\widehat\beta({\exdynD}_n)
&=\Delta_4\sigma_3\sigma_2\sigma_2\sigma_3 \sigma_2^{n-4} \sigma_1 \sigma_2 \sigma_2\sigma_1 \Delta_4 \nonumber\\
&\doteq \sigma_2^k \sigma_1 \sigma_2 \sigma_2\sigma_1 \Delta_4 \Delta_4 \sigma_3\sigma_2\sigma_2\sigma_3 \sigma_2^\ell \nonumber\\
&= \sigma_2^k \sigma_1 \sigma_2 \sigma_2 \sigma_1 {\color{red} \sigma_1 \sigma_2 \sigma_1} \sigma_3 \sigma_2 \sigma_1 \sigma_1 \sigma_2 {\color{blue}\sigma_1 \sigma_3} \sigma_2 \sigma_1\sigma_3\sigma_2\sigma_2\sigma_3 \sigma_2^\ell \nonumber\\
&= \sigma_2^k \sigma_1 \sigma_2 \sigma_2 {\color{red} \sigma_1 \sigma_2 \sigma_1} \sigma_2 \sigma_3 \sigma_2 \sigma_1 \sigma_1 \sigma_2 \sigma_3 {\color{red}\sigma_1 \sigma_2 \sigma_1}\sigma_3\sigma_2\sigma_2\sigma_3 \sigma_2^\ell \nonumber\\
&= \sigma_2^k \sigma_1 \sigma_2 \sigma_2 \sigma_2 \sigma_1 \sigma_2 \sigma_2 \sigma_3 \sigma_2 \sigma_1 \sigma_1 {\color{red}\sigma_2 \sigma_3 \sigma_2} \sigma_1 \sigma_2 \sigma_3\sigma_2\sigma_2\sigma_3 \sigma_2^\ell \nonumber\\
&= \sigma_2^k \sigma_1 \sigma_2 \sigma_2 \sigma_2 \sigma_1 \sigma_2 \sigma_2 \sigma_3 \sigma_2 {\color{blue}\sigma_1 \sigma_1 \sigma_3} \sigma_2 \sigma_3 \sigma_1 \sigma_2 \sigma_3\sigma_2\sigma_2\sigma_3 \sigma_2^\ell \nonumber\\
&= \sigma_2^k \sigma_1 \sigma_2 \sigma_2 \sigma_2 \sigma_1 \sigma_2 \sigma_2 {\color{red}\sigma_3 \sigma_2  \sigma_3} \sigma_1 \sigma_1 \sigma_2 {\color{blue} \sigma_3 \sigma_1} \sigma_2 \sigma_3\sigma_2\sigma_2\sigma_3 \sigma_2^\ell \nonumber\\
&= \sigma_2^k \sigma_1 \sigma_2 \sigma_2 \sigma_2 \sigma_1 \sigma_2 \sigma_2 \sigma_2 \sigma_3  \sigma_2 \sigma_1 {\color{red}\sigma_1 \sigma_2 \sigma_1} \sigma_3 \sigma_2 \sigma_3\sigma_2\sigma_2\sigma_3 \sigma_2^\ell \nonumber\\\displaybreak
&= \sigma_2^k \sigma_1 \sigma_2 \sigma_2 \sigma_2 \sigma_1 \sigma_2 \sigma_2 \sigma_2 \sigma_3  {\color{red}\sigma_2 \sigma_1 \sigma_2} \sigma_1 \sigma_2 \sigma_3 \sigma_2 \sigma_3\sigma_2\sigma_2\sigma_3 \sigma_2^\ell \nonumber\\
&= \sigma_2^k \sigma_1 \sigma_2 \sigma_2 \sigma_2 \sigma_1 \sigma_2 \sigma_2 \sigma_2 {\color{blue}\sigma_3  \sigma_1} \sigma_2 \sigma_1 \sigma_1 \sigma_2 \sigma_3 \sigma_2 \sigma_3\sigma_2\sigma_2\sigma_3 \sigma_2^\ell \nonumber\\
&= \sigma_2^k \sigma_1 \sigma_2 \sigma_2 \sigma_2 \sigma_1 \sigma_2 \sigma_2 \sigma_2 \sigma_1  \sigma_3 \sigma_2 \sigma_1 \sigma_1 \sigma_2 \sigma_3 \sigma_2 \sigma_3\sigma_2\sigma_2\sigma_3 {\color{violet}\sigma_2^\ell} \nonumber\\
&\stackrel{\star}{=} \sigma_2^k \sigma_1 \sigma_2 \sigma_2 \sigma_2 \sigma_1 \sigma_2 \sigma_2 \sigma_2 \sigma_1 \sigma_2^\ell \sigma_3 \sigma_2 \sigma_1 \sigma_1 \sigma_2 {\color{red} \sigma_3 \sigma_2 \sigma_3} \sigma_2\sigma_2\sigma_3  \nonumber\\
&= \sigma_2^k \sigma_1 \sigma_2 \sigma_2 \sigma_2 \sigma_1 \sigma_2 \sigma_2 \sigma_2 \sigma_1 \sigma_2^\ell \sigma_3 \sigma_2 \sigma_1 \sigma_1 \sigma_2 \sigma_2 \sigma_3 \sigma_2 \sigma_2\sigma_2\sigma_3 
\end{align}

Here $\doteq$ is the braid equivalence in $J^1\sphere^1$ up to cyclic rotation. The equivalence $\stackrel{\star}{=}$ also can be checked directly. Indeed, the relation 
\[
(\sigma_3 \sigma_2 \sigma_1 \sigma_1 \sigma_2 \sigma_3 \sigma_2 \sigma_3\sigma_2\sigma_2\sigma_3 )\sigma_2^\ell
\stackrel{\star}{=}
\sigma_2^\ell(\sigma_3 \sigma_2 \sigma_1 \sigma_1 \sigma_2 \sigma_3 \sigma_2 \sigma_3\sigma_2\sigma_2\sigma_3 )
\]
is justfied by the following moves in braids:
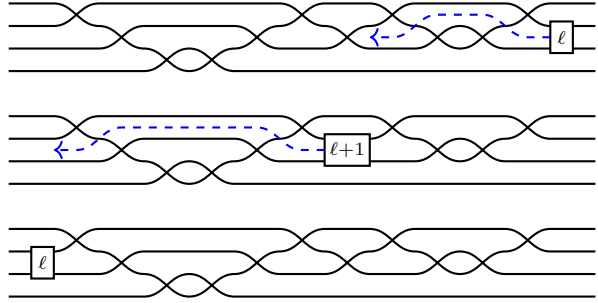
\begin{align*}
\begin{tikzpicture}[scale=0.6]
\begin{scope}
\draw[thick] 
(-1,0) to (2,0) to[out=0,in=180] (3,0.5) to[out=0,in=180] (4,0) to (12,0)
(-1,0.5) to (1,0.5) to[out=0,in=180] (2,1) to (4,1) to[out=0,in=180] (5,0.5) to (6,0.5) to[out=0,in=180] (7,1) to[out=0,in=180] (8,1.5) to (10,1.5) to[out=0,in=180] (11,1) to (12,1)
(-1,1) to (0,1) to[out=0,in=180] (1,1.5) to (5,1.5) to[out=0,in=180] (6,1) to[out=0,in=180] (7,0.5) to (8,0.5) to[out=0,in=180] (9,1) to[out=0,in=180] (10,0.5) to (12,0.5)
(-1,1.5) to (0,1.5) to[out=0,in=180] (1,1) to[out=0,in=180] (2,0.5) to[out=0,in=180] (3,0) to[out=0,in=180] (4,0.5) to[out=0,in=180] (5,1) to[out=0,in=180] (6,1.5) to (7,1.5) to[out=0,in=180] (8,1) to[out=0,in=180] (9,0.5) to[out=0,in=180] (10,1) to[out=0,in=180] (11,1.5) to (12,1.5)
;
\draw[thick,fill=white] (11,1.1) rectangle node {$\scriptstyle \ell$} (11.5, 0.4);
\draw[thick,blue,dashed,->] (11,0.75) to[out=180, in=0] (10.5,0.75) to[out=180, in=0] (9.5,1.25) to[out=180, in=0] (8.5,1.25) to[out=180, in=0] (7.5,0.75) to[out=180, in=0] (7,0.75);
\end{scope}
\begin{scope}[yshift=-2.5cm]
\draw[thick] 
(-1,0) to (2,0) to[out=0,in=180] (3,0.5) to[out=0,in=180] (4,0) to (12,0)
(-1,0.5) to (1,0.5) to[out=0,in=180] (2,1) to (4,1) to[out=0,in=180] (5,0.5) to (6,0.5) to[out=0,in=180] (7,1) to[out=0,in=180] (8,1.5) to (10,1.5) to[out=0,in=180] (11,1) to (12,1)
(-1,1) to (0,1) to[out=0,in=180] (1,1.5) to (5,1.5) to[out=0,in=180] (6,1) to[out=0,in=180] (7,0.5) to (8,0.5) to[out=0,in=180] (9,1) to[out=0,in=180] (10,0.5) to (12,0.5)
(-1,1.5) to (0,1.5) to[out=0,in=180] (1,1) to[out=0,in=180] (2,0.5) to[out=0,in=180] (3,0) to[out=0,in=180] (4,0.5) to[out=0,in=180] (5,1) to[out=0,in=180] (6,1.5) to (7,1.5) to[out=0,in=180] (8,1) to[out=0,in=180] (9,0.5) to[out=0,in=180] (10,1) to[out=0,in=180] (11,1.5) to (12,1.5)
;
\draw[thick,fill=white] (6,1.1) rectangle node {$\scriptstyle{\ell+1}$} (7, 0.4);
\draw[thick,blue,dashed,->] (6,0.75) to[out=180, in=0] (5.5,0.75) to[out=180, in=0] (4.5,1.25) to[out=180, in=0] (1.5,1.25) to[out=180, in=0] (0.5,0.75) to (0,0.75);
\end{scope}
\begin{scope}[yshift=-5cm]
\draw[thick] 
(-1,0) to (2,0) to[out=0,in=180] (3,0.5) to[out=0,in=180] (4,0) to (12,0)
(-1,0.5) to (1,0.5) to[out=0,in=180] (2,1) to (4,1) to[out=0,in=180] (5,0.5) to (6,0.5) to[out=0,in=180] (7,1) to[out=0,in=180] (8,1.5) to (10,1.5) to[out=0,in=180] (11,1) to (12,1)
(-1,1) to (0,1) to[out=0,in=180] (1,1.5) to (5,1.5) to[out=0,in=180] (6,1) to[out=0,in=180] (7,0.5) to (8,0.5) to[out=0,in=180] (9,1) to[out=0,in=180] (10,0.5) to (12,0.5)
(-1,1.5) to (0,1.5) to[out=0,in=180] (1,1) to[out=0,in=180] (2,0.5) to[out=0,in=180] (3,0) to[out=0,in=180] (4,0.5) to[out=0,in=180] (5,1) to[out=0,in=180] (6,1.5) to (7,1.5) to[out=0,in=180] (8,1) to[out=0,in=180] (9,0.5) to[out=0,in=180] (10,1) to[out=0,in=180] (11,1.5) to (12,1.5)
;
\draw[thick,fill=white] (-0.5,1.1) rectangle node {$\scriptstyle{\ell}$} (0, 0.4);
\end{scope}
\end{tikzpicture}
\end{align*}
For $\dynX =\exdynE_n= \quiver(a,b,c)$ with $n=a+b+c-3$, we have 
\begin{align*}
\widehat\beta(a,b,c)
&=\Delta_3 \sigma_1 \sigma_2^a \sigma_1^{b-1} \sigma_2^c \Delta_3\\
&=\sigma_2 \sigma_1 {\color{red} \sigma_2 \sigma_1 \sigma_2^a} \sigma_1^{b-1} {\color{red} \sigma_2^c \sigma_2 \sigma_1 \sigma_2}\\
&=\sigma_2 \sigma_1  \sigma_1^{a} \sigma_2 \sigma_1 \sigma_1^{b-1} \sigma_1 \sigma_2 \sigma_1^{c+1}\\
&=\sigma_2 \sigma_1^{a+1} \sigma_2 \sigma_1^{b+1} \sigma_2 \sigma_1^{c+1}.
\end{align*}

\subsection{\texorpdfstring{$N$}{N}-graphs of \texorpdfstring{type ${\exdynD}{\exdynE}$}{Affine Dn and En type}}

Now we consider $N$-graphs on $\disk^2$ whose boundary data come from the 
Legendrian of type ${\exdynD}_n$ or $\exdynE_n=\quiver(a,b,c)$ with $n=a+b+c-3$. More concretely, the braids
\begin{align*}
\widehat \beta({\exdynD}_n) &=\sigma_2^k \sigma_1 \sigma_2^3 \sigma_1 
\sigma_2^3 \sigma_1 \sigma_2^\ell \sigma_3 \sigma_2 \sigma_1^2 \sigma_2^2 
\sigma_3  \sigma_2^3\sigma_3,\\
\widehat \beta(a,b,c) &= \sigma_2 \sigma_1^{a+1} \sigma_2 \sigma_1^{b+1} 
\sigma_2 \sigma_1^{c+1}
\end{align*}
in $J^1\sphere^1$ give the boundary data on $\partial\disk^2$ as in 
Figure~\ref{fig:legendrian link of affine DE}. 

We define $N$-graphs on~$\disk^2$ as depicted in Tables~\ref{table:affine D type} and~\ref{table:affine E type} and denote pairs of the $N$-graphs and the set of one cycles by $(\ngraph(\dynX), \nbasis(\dynX))$. 
For $\dynX=\exdynE_n=\quiver(a,b,c)$, we also use the notation $(\ngraph(a,b,c),\nbasis(a,b,c))$ instead.

\begin{figure}[ht]
\subfigure[$\widehat \beta({\exdynD}_n)$]{
\begin{tikzpicture}[baseline=-.5ex,scale=0.6]
\draw[thick] (0,3) arc (90:270:3) (6,-3) arc (-90:90:3);
\draw[thick] (0,-3) to (6,-3) (6,3) to (0,3);
\draw[blue, thick, fill] 
(6,0)+(-90:3) circle (2pt) 
(6,0)+(0:3) circle (2pt)
(6,0)+(90:3) circle (2pt)
(90+90/4+90/12:3) circle (2pt) 
(90+90/4+90/6:3) circle (2pt) ;
\draw[green, thick, fill] 
(90:3) circle (2pt) 
(180:3) circle (2pt)
(270:3) circle (2pt);
\draw[red, thick, fill] 
(6,0)+(-90+90/4:3) circle (2pt) 
(6,0)+(-90+2*90/4:3) circle (2pt) 
(6,0)+(-90+3*90/4:3) circle (2pt) 
(6,0)+(90/4:3) circle (2pt) 
(6,0)+(2*90/4:3) circle (2pt) 
(6,0)+(3*90/4:3) circle (2pt) 
(90+90/4:3) circle (2pt) 
(90+2*90/4:3) circle (2pt) 
(90+3*90/4:3) circle (2pt)
(180+90/4:3) circle (2pt) 
(180+2*90/4:3) circle (2pt) 
(180+3*90/4:3) circle (2pt)
(1,-3) circle (2pt)
(2,-3) circle (2pt)
(4,-3) circle (2pt)
(5,-3) circle (2pt)
(1,3) circle (2pt)
(2,3) circle (2pt)
(4,3) circle (2pt)
(5,3) circle (2pt)
;
\draw [decorate,decoration={brace,amplitude=5pt,mirror,raise=1ex}]
(0.9,-3) -- (5.1,-3) node[midway,yshift=-1.5em]{$k=\lfloor \frac{n-3}2\rfloor$};
\draw [decorate,decoration={brace,amplitude=5pt,raise=1ex}]
(0.9,3) -- (5.1,3) node[midway,yshift=1.5em]{$\ell=\lfloor \frac{n-4}2\rfloor$};
\end{tikzpicture}
}
\subfigure[$\widehat \beta(a,b,c)$]{
\begin{tikzpicture}[baseline=-.5ex,scale=0.6]
\draw[thick] (0,0) circle (3cm);
\draw[red, thick, fill] (0:3) circle (2pt) (120:3) circle (2pt) (240:3) circle (2pt);
\draw[blue, thick, fill] (20:3) circle (2pt) (40:3) circle (2pt)  (80:3) circle (2pt) (100:3) circle (2pt) (140:3) circle (2pt) (160:3) circle (2pt)  (200:3) circle (2pt) (220:3) circle (2pt) (260:3) circle (2pt) (280:3) circle (2pt)  (320:3) circle (2pt) (340:3) circle (2pt);
\curlybrace[]{10}{110}{3.2};
\draw (60:3.5) node[rotate=-30] {$a+1$};
\curlybrace[]{130}{230}{3.2};
\draw (180:3.5) node[rotate=90] {$b+1$};
\curlybrace[]{250}{350}{3.2};
\draw (300:3.5) node[rotate=30] {$c+1$};
\end{tikzpicture}
}

\caption{Legendrian links in $J^1\sphere^1$ of type $\exdynD\exdynE$.}
\label{fig:legendrian link of affine DE}
\end{figure}
\begin{table}[ht]
\begin{tabular}{c|c|c}
\toprule
$\widehat \beta$ & $(\ngraph,\nbasis)$ & $\quiver(\ngraph,\nbasis)$\\
\midrule
$\widehat \beta({\exdynD}_4)$ &
\begin{tikzpicture}[baseline=-.5ex,scale=0.5]
\useasboundingbox (-2.5,-3.5) rectangle (2.5,3.5);
\begin{scope}[xshift=-1cm]
\draw[yellow,   line width=5, opacity=0.5] (-1,1)--(0,0) (-1,-1)--(0,0) 
(1,0)--(0,0);
\draw[green, line cap=round, line width=5, opacity=0.5] (-2,1)-- (-1,1) 
(-1,-1)--(-1,-2)
(3,1)--(3,2) (3,-1)--(4,-1);
\draw[blue,thick, rounded corners] (-1.75, 3) -- (-1.5, 2.5) -- (-1.25, 3);
\draw[thick] (1, 3) -- (-2,3) to[out=180,in=90] (-3,2) --(-3,-2) 
to[out=-90,in=180] (-2, -3) -- (4,-3) to[out=0,in=-90] (5,-2) -- (5,2) 
to[out=90,in=0] (4,3)-- (1,3);
\draw[green, thick] (0,3) -- (0,-3) (0,0) -- (-3,0);
\draw[red, thick, fill] (0,0) -- (-1,1) circle (2pt) -- +(0,2) (-1,1) -- 
++(-1,0) circle (2pt) -- +(-1,0) (-2,1) -- +(0,2);
\draw[red,thick, fill] (0,0) -- (-1,-1) circle (2pt) -- +(-2,0) (-1,-1) -- 
++(0,-1) circle (2pt) -- +(0,-1) (-1,-2) -- +(-2,0);
\draw[red,thick] (0,0) -- (1,0);
\draw[thick, fill=white] (0,0) circle (2pt);
\end{scope}
\begin{scope}[xshift=1cm,rotate=180]
\draw[yellow,  line width=5, opacity=0.5] (-1,1)--(0,0) (-1,-1)--(0,0) 
(1,0)--(0,0);
\draw[blue, thick] (0,3) -- (0,-3) (0,0) -- (-3,0);
\draw[red ,thick, fill] (0,0) -- (-1,1) circle (2pt) -- +(0,2) (-1,1) -- 
++(-1,0) circle (2pt) -- +(-1,0) (-2,1) -- +(0,2);
\draw[red,thick, fill] (0,0) -- (-1,-1) circle (2pt) -- +(-2,0) (-1,-1) -- 
++(0,-1) circle (2pt) -- +(0,-1) (-1,-2) -- +(-2,0);
\draw[red,thick] (0,0) -- (1,0);
\draw[thick, fill=white] (0,0) circle (2pt);
\end{scope}
\end{tikzpicture}
&
\begin{tikzpicture}[baseline=-.5ex]
\draw[fill] (0,0) circle(2pt) node (A1) {} (135:1) circle(2pt) node (A2) {} 
(225:1) circle(2pt) node (A3) {};
\draw[fill] (-45:1) circle(2pt) node (A4) {} (45:1) circle(2pt) node (A5) {};
\draw[->] (A2) node[above] {$2$} -- (A1) node[above] {$1$};
\draw[->] (A3) node[below] {$3$} -- (A1);
\draw[->] (A4) node[below] {$5$} -- (A1);
\draw[->] (A5) node[above] {$4$} -- (A1);
\end{tikzpicture}\\
\hline
$\widehat \beta({\exdynD}_5)$ &
\begin{tikzpicture}[baseline=-.5ex,scale=0.5]
\useasboundingbox (-2.5,-3.5) rectangle (2.5,3.5);
\begin{scope}[xshift=-1cm]
\draw[yellow, line width=5,opacity=0.5] (-1,1)--(0,0) (-1,-1)--(0,0) 
(1,0)--(0,0);
\draw[green, line width=5,opacity=0.5] (-2,1)--(-1,1) (-1,-1)--(-1,-2) 
(1,0)--(2,0)--(3,1) (2,0)--(3,-1);
\draw[blue,thick, rounded corners] (-1.75, 3) -- (-1.5, 2.5) -- (-1.25, 3);
\draw[thick] (1, 3) -- (-2,3) to[out=180,in=90] (-3,2) --(-3,-2) 
to[out=-90,in=180] (-2, -3) -- (4,-3) to[out=0,in=-90] (5,-2) -- (5,2) 
to[out=90,in=0] (4,3)-- (1,3);
\draw[green, thick] (0,3) -- (0,-3) (0,0) -- (-3,0);
\draw[red,thick, fill] (0,0) -- (-1,1) circle (2pt) -- +(0,2) (-1,1) -- 
++(-1,0) circle (2pt) -- +(-1,0) (-2,1) -- +(0,2);
\draw[red,thick, fill] (0,0) -- (-1,-1) circle (2pt) -- +(-2,0) (-1,-1) -- 
++(0,-1) circle (2pt) -- +(0,-1) (-1,-2) -- +(-2,0);
\draw[red,thick] (0,0) -- (1,0);
\draw[thick, fill=white] (0,0) circle (2pt);
\end{scope}
\draw[red, thick, fill] (0,0) circle (2pt) -- (0,-3);
\begin{scope}[xshift=1cm, yscale=-1,rotate=180]
\draw[yellow, line width=5, opacity=0.5] (-1,-1) -- (-1,-2) (-1,1) -- (-2,1);
\draw[blue, thick] (0,3) -- (0,-3) (0,0) -- (-3,0);
\draw[red,thick, fill] (0,0) -- (-1,1) circle (2pt) -- +(0,2) (-1,1) -- 
++(-1,0) circle (2pt) -- +(-1,0) (-2,1) -- +(0,2);
\draw[red,thick, fill] (0,0) -- (-1,-1) circle (2pt) -- +(-2,0) (-1,-1) -- 
++(0,-1) circle (2pt) -- +(0,-1) (-1,-2) -- +(-2,0);
\draw[red,thick] (0,0) -- (1,0);
\draw[thick, fill=white] (0,0) circle (2pt);
\end{scope}
\end{tikzpicture}
&
\begin{tikzpicture}[baseline=-.5ex]
\draw[fill] (0,0) circle(2pt) node (A1) {} (120:1) circle(2pt) node (A2) {} 
(240:1) circle(2pt) node (A3) {} (1,0) circle(2pt) node (A4) {};
\draw[fill] (1,0) ++(-60:1) circle(2pt) node (A5) {} (1,0) ++(60:1) circle(2pt) 
node (A6) {};
\draw[->] (A2) node[above] {$2$} -- (A1) node[left] {$1$};
\draw[->] (A3) node[below] {$3$} -- (A1) ;
\draw[->] (A4) node[right] {$4$} -- (A1) ;
\draw[->] (A4) -- (A5) node[below] {$6$};
\draw[->] (A4) -- (A6) node[above] {$5$};
\end{tikzpicture}\\
\hline
$\widehat \beta({\exdynD}_{2\ell+4})$ &
\begin{tikzpicture}[baseline=-.5ex,scale=0.5]
\useasboundingbox (-7.5,-3.5) rectangle (7.5,3.5);
\begin{scope}[xshift=-4cm]
\draw[yellow, line width=5, opacity=0.5] (-1,1)--(0,0) (-1,-1)--(0,0) 
(1,0)--(0,0);
\draw[green, line width=5, opacity=0.5] (-1,1)--(-2,1) (-1,-1)--(-1,-2) 
(1,0)--(2,0)
(3,0)--(3.5,0) (4.5,0)--(5,0) (6,0)--(7,0) (9,1)--(9,2) (9,-1)--(10,-1);
\draw[blue,thick, rounded corners] (-1.75, 3) -- (-1.5, 2.5) -- (-1.25, 3);
\draw[green, thick] (0,3) -- (0,-3) (0,0) -- (-3,0);
\draw[red,thick, fill] (0,0) -- (-1,1) circle (2pt) -- +(0,2) (-1,1) -- 
++(-1,0) circle (2pt) -- +(-1,0) (-2,1) -- +(0,2);
\draw[red,thick, fill] (0,0) -- (-1,-1) circle (2pt) -- +(-2,0) (-1,-1) -- 
++(0,-1) circle (2pt) -- +(0,-1) (-1,-2) -- +(-2,0);
\draw[red,thick] (0,0) -- (1,0);
\draw[thick, fill=white] (0,0) circle (2pt);
\draw[thick](1, 3) -- (-2,3) to[out=180,in=90] (-3,2) --(-3,-2) 
to[out=-90,in=180] (-2, -3) -- (10,-3) to[out=0,in=-90] (11,-2) -- (11,2) 
to[out=90,in=0] (10,3)-- (1,3);
\end{scope}
\begin{scope}[xshift=4cm,rotate=180]
\draw[yellow, line width=5, opacity=0.5] (-1,1)--(0,0) (-1,-1)--(0,0) 
(1,0)--(0,0);
\draw[blue, thick] (0,3) -- (0,-3) (0,0) -- (-3,0);
\draw[red,thick, fill] (0,0) -- (-1,1) circle (2pt) -- +(0,2) (-1,1) -- 
++(-1,0) circle (2pt) -- +(-1,0) (-2,1) -- +(0,2);
\draw[red,thick, fill] (0,0) -- (-1,-1) circle (2pt) -- +(-2,0) (-1,-1) -- 
++(0,-1) circle (2pt) -- +(0,-1) (-1,-2) -- +(-2,0);
\draw[red,thick] (0,0) -- (1,0);
\draw[thick, fill=white] (0,0) circle (2pt);
\end{scope}
\begin{scope}
\draw[yellow, line width=5, opacity=0.5] (-2,0) -- (-1,0) (1,0) -- (2,0);
\draw[red, thick, fill] (-3,0) circle (2pt) -- (-3, -3) (-3,0) -- (-2,0) circle 
(2pt) -- (-2,3) (-2,0) -- (-1,0) circle (2pt) -- (-1,-3) (-1,0) -- (-0.5,0) 
(0.5,0) -- (1,0) circle (2pt) -- (1,3) (1,0) -- (2, 0) circle (2pt) -- (2,-3) 
(2,0) -- (3,0) circle (2pt) -- (3,3);
\draw[red, thick, dashed] (-0.5,0) -- (0.5,0);
\end{scope}
\end{tikzpicture}&
\begin{tikzpicture}[baseline=-.5ex]
\draw[fill] (0,0) circle(2pt) node (A1) {} (120:1) circle(2pt) node (A2) {} 
(240:1) circle(2pt) node (A3) {} (1,0) circle(2pt) node (A4) {} (2,0) circle 
(2pt) node (A5) {} (3,0) circle (2pt) node (A6) {};
\draw[fill] (3,0) ++(60:1) circle(2pt) node (A7) {} (3,0) ++(-60:1) circle(2pt) 
node (A8) {};
\draw[->] (A2) node[above] {$2$} --(A1) node[left] {$1$} ;
\draw[->]  (A3) node[below] {$3$} -- (A1) ;
\draw[->]  (A4) node[above] {$4$} -- (A1);
\draw[dashed] (A4) -- (A5) node[above] {$n-2$};
\draw[->] (A5) -- (A6) node[right] {$n-1$};
\draw[->] (A7) node[above] {$n$} -- (A6);
\draw[->] (A8) node[below] {$n+1$} -- (A6);
\end{tikzpicture}\\
\hline
$\widehat \beta({\exdynD}_{2\ell+5})$ &
\begin{tikzpicture}[baseline=-.5ex,scale=0.5]
\useasboundingbox (-7.5,-3.5) rectangle (7.5,3.5);
\begin{scope}[xshift=-3.5cm]
\draw[yellow, line width=5, opacity=0.5] (-1,1)--(0,0) (-1,-1)--(0,0) 
(1,0)--(0,0);
\draw[green, line width=5, opacity=0.5] (-1,1)--(-2,1) (-1,-1)--(-1,-2) 
(1,0)--(2,0)
(3,0)--(3.5,0) (4.5,0)--(5,0) (6,0)--(7,0) (7,0)--(8,1) (7,0)--(8,-1);
\draw[blue,thick, rounded corners] (-1.75, 3) -- (-1.5, 2.5) -- (-1.25, 3);
\draw[green, thick] (0,3) -- (0,-3) (0,0) -- (-3,0);
\draw[red,thick, fill] (0,0) -- (-1,1) circle (2pt) -- +(0,2) (-1,1) -- 
++(-1,0) circle (2pt) -- +(-1,0) (-2,1) -- +(0,2);
\draw[red,thick, fill] (0,0) -- (-1,-1) circle (2pt) -- +(-2,0) (-1,-1) -- 
++(0,-1) circle (2pt) -- +(0,-1) (-1,-2) -- +(-2,0);
\draw[red,thick] (0,0) -- (1,0);
\draw[thick, fill=white] (0,0) circle (2pt);
\draw[thick](1, 3) -- (-2,3) to[out=180,in=90] (-3,2) --(-3,-2) 
to[out=-90,in=180] (-2, -3) -- (9,-3) to[out=0,in=-90] (10,-2) -- (10,2) 
to[out=90,in=0] (9,3)-- (1,3);
\end{scope}
\begin{scope}[xshift=3.5cm,rotate=180]
\draw[yellow, line width=5, opacity=0.5] (-1,-1) -- (-1,-2) (-1,1)--(-2,1);
\draw[blue, thick] (0,3) -- (0,-3) (0,0) -- (-3,0);
\draw[red,thick, fill] (0,0) -- (-1,1) circle (2pt) -- +(0,2) (-1,1) -- 
++(-1,0) circle (2pt) -- +(-1,0) (-2,1) -- +(0,2);
\draw[red,thick, fill] (0,0) -- (-1,-1) circle (2pt) -- +(-2,0) (-1,-1) -- 
++(0,-1) circle (2pt) -- +(0,-1) (-1,-2) -- +(-2,0);
\draw[red,thick] (0,0) -- (1,0);
\draw[thick, fill=white] (0,0) circle (2pt);
\end{scope}
\begin{scope}[xshift=0.5cm]
\draw[yellow, line width=5, opacity=0.5] (-2,0) -- (-1,0) (1,0) -- (2,0);
\draw[red, thick, fill] (-3,0) circle (2pt) -- (-3, -3) (-3,0) -- (-2,0) circle 
(2pt) -- (-2,3) (-2,0) -- (-1,0) circle (2pt) -- (-1,-3) (-1,0) -- (-0.5,0) 
(0.5,0) -- (1,0) circle (2pt) -- (1,3) (1,0) -- (2, 0) circle (2pt) -- (2,-3);
\draw[red, thick, dashed] (-0.5,0) -- (0.5,0);
\end{scope}
\end{tikzpicture}&
\begin{tikzpicture}[baseline=-.5ex]
\draw[fill] (0,0) circle(2pt) node (A1) {} (120:1) circle(2pt) node (A2) {} 
(240:1) circle(2pt) node (A3) {} (1,0) circle(2pt) node (A4) {} (2,0) circle 
(2pt) node (A5) {} (3,0) circle (2pt) node (A6) {};
\draw[fill] (3,0) ++(60:1) circle(2pt) node (A7) {} (3,0) ++(-60:1) circle(2pt) 
node (A8) {};
\draw[->] (A2) node[above] {$2$} -- (A1) node[left] {$1$};
\draw[->]  (A3) node[below] {$3$} -- (A1);
\draw[->] (A4) node[above] {$4$} -- (A1);
\draw[dashed] (A5) node[above] {$n-2$} -- (A4);
\draw[->] (A6) node[right] {$n-1$} -- (A5);
\draw[->] (A6) -- (A7) node[above] {$n$};
\draw[->] (A6) -- (A8) node[below] {$n+1$};
\end{tikzpicture}\\
\bottomrule
\end{tabular}
\caption{$N$-graphs and their quivers of type~${\exdynD_n}$}
\label{table:affine D type}
\end{table}
\begin{table}[ht]
\begin{tabular}{c|c|c}
\toprule
$\widehat \beta$ & $(\ngraph,\nbasis)$ & $\quiver(\ngraph,\nbasis)$\\
\midrule
$\widehat\beta(3,3,3)$ &
\begin{tikzpicture}[baseline=-.5ex,scale=0.8]
\draw[thick] (0,0) circle (3cm);
\draw[yellow, line cap=round, line width=5, opacity=0.5] (60:1) -- (45:1.5) (180:1) -- (165:1.5) (300:1) -- (285:1.5);
\draw[green, line cap=round, line width=5, opacity=0.5] (0,0) -- (60:1) (0,0) -- (180:1) (0,0) -- (300:1) (45:1.5) -- (60:2) (165:1.5) -- (180:2) (285:1.5) -- (300:2);
\draw[red, thick] (0,0) -- (0:3) (0,0) -- (120:3) (0,0) -- (240:3);
\draw[blue, thick, fill] (0,0) -- (60:1) circle (2pt) -- (96:3) (60:1) -- (45:1.5) circle (2pt) -- (24:3) (45:1.5) -- (60:2) circle (2pt) -- (72:3) (60:2) -- (48:3);
\draw[blue, thick, fill] (0,0) -- (180:1) circle (2pt) -- (216:3) (180:1) -- (165:1.5) circle (2pt) -- (144:3) (165:1.5) -- (180:2) circle (2pt) -- (192:3) (180:2) -- (168:3);
\draw[blue, thick, fill] (0,0) -- (300:1) circle (2pt) -- (336:3) (300:1) -- (285:1.5) circle (2pt) -- (264:3) (285:1.5) -- (300:2) circle (2pt) -- (312:3) (300:2) -- (288:3);
\draw[thick, fill=white] (0,0) circle (2pt);
\end{tikzpicture}
&
\begin{tikzpicture}[baseline=.5ex]
\useasboundingbox (-2.5,-3) rectangle (2.5,3);
\draw[fill] (0,0) circle (2pt) node (O) {};
\draw[fill] (60:1) circle (2pt) node (A1) {};
\draw[fill] (60:2) circle (2pt) node (A2) {};
\draw[fill] (180:1) circle (2pt) node (B1) {};
\draw[fill] (180:2) circle (2pt) node (B2) {};
\draw[fill] (-60:1) circle (2pt) node (C1) {};
\draw[fill] (-60:2) circle (2pt) node (C2) {};
\draw[->] (O) node[above left] {$ 1$}--(A1) node[above left] {$ 2$};
\draw[->] (A2) node[above left] {$ 3$}--(A1);
\draw[->] (O)--(B1) node[above] {$ 4$};
\draw[->] (B2) node[above] {$ 5$}--(B1);
\draw[->] (O)--(C1) node[above right] {$ 6$};
\draw[->] (C2) node[above right] {$ 7$}--(C1);
\end{tikzpicture}\\
\hline
$\widehat \beta(2,4,4)$ &
\begin{tikzpicture}[baseline=-.5ex,scale=0.8]
\useasboundingbox(-3.5,-3.5)rectangle(3.5,3.5);
\draw[thick] (0,0) circle (3cm);
\draw[yellow, line cap=round, line width=5, opacity=0.5] (60:1) -- (50:2) (180:1) -- (170:1.5) (190:1.75) -- (170:2) (300:1) -- (290:1.5) (310:1.75) -- (290:2);
\draw[green, line cap=round, line width=5, opacity=0.5] (0,0) -- (60:1) (0,0) -- (180:1) (0,0) -- (300:1) (170:1.5) -- (190:1.75) (290:1.5) -- (310:1.75);
\draw[red, thick] (0,0) -- (0:3) (0,0) -- (120:3) (0,0) -- (240:3);
\draw[blue, thick, fill] (0,0) -- (60:1) circle (2pt) -- (90:3) (60:1) -- (50:2)  circle (2pt) -- (30:3) (50:2) -- (60:3);
\draw[blue, thick, fill] (0,0) -- (180:1) circle (2pt) -- (220:3) (180:1) -- (170:1.5)  circle (2pt) -- (140:3) (170:1.5) -- (190:1.75) circle (2pt) -- (200:3) (190:1.75) -- (170:2) circle (2pt) -- (160:3) (170:2) -- (180:3);
\draw[blue, thick, dashed] ;
\draw[blue, thick, fill] (0,0) -- (300:1) circle (2pt) -- (340:3) (300:1) -- (290:1.5) circle (2pt) -- (260:3) (290:1.5) -- (310:1.75) circle (2pt) -- (320:3) (310:1.75) -- (290:2) circle (2pt) -- (280:3) (290:2) -- (300:3);
\draw[blue, thick, dashed] ;
\draw[thick, fill=white] (0,0) circle (2pt);
\end{tikzpicture}
&
\begin{tikzpicture}[baseline=.5ex]
\begin{scope}[yshift=1cm]
\draw[fill] (0,0) circle (2pt) node (O) {};
\draw[fill] (60:1) circle (2pt) node (A1) {};
\draw[fill] (180:1) circle (2pt) node (B1) {};
\draw[fill] (180:2) circle (2pt) node (B2) {};
\draw[fill] (180:3) circle (2pt) node (B3) {};
\draw[fill] (-60:1) circle (2pt) node (C1) {};
\draw[fill] (-60:2) circle (2pt) node (C2) {};
\draw[fill] (-60:3) circle (2pt) node (C3) {};
\draw[->] (O) node[above left] {$ 1$}--(A1) node[above left] {$ 2$};
\draw[->] (O)--(B1) node[above] {$ 3$};
\draw[->] (B2) node[above] {$ 4$}--(B1);
\draw[->] (B2)--(B3) node[above] {$ 5$};
\draw[->] (O)--(C1) node[above right] {$ 6$};
\draw[->] (C2) node[above right] {$ 7$}--(C1);
\draw[->] (C2)--(C3) node[above right] {$ 8$};
\end{scope}
\end{tikzpicture}\\
\hline
$\widehat \beta(2,3,6)$ &
\begin{tikzpicture}[baseline=-.5ex,scale=0.8]
\useasboundingbox(-3.5,-3.5)rectangle(3.5,3.5);
\draw[thick] (0,0) circle (3cm);
\draw[yellow, line cap=round, line width=5, opacity=0.5] (60:1) -- (50:2) (180:1) -- (170:1.5) (-60:1) -- (-90:1.25) (-45:1.5)--(-75:1.75) (-60:2)--(-67.5:2.5) ;
\draw[green, line cap=round, line width=5, opacity=0.5] (0,0) -- (60:1) (0,0) -- (180:1) (0,0) -- (300:1) (170:1.5) -- (190:2) (-90:1.25)--(-45:1.5) (-75:1.75)--(-60:2);
\draw[red, thick] (0,0) -- (0:3) (0,0) -- (120:3) (0,0) -- (240:3);
\draw[blue, thick, fill] (0,0) -- (60:1) circle (2pt) -- (90:3) (60:1) -- (50:2)  circle (2pt) -- (30:3) (50:2) -- (60:3);
\draw[blue, thick, fill] (0,0) -- (180:1) circle (2pt) -- (216:3) (180:1) -- (170:1.5)  circle (2pt) -- (144:3) (170:1.5) -- (190:2) circle (2pt) -- (192:3) (190:2) -- (168:3);
\draw[blue, thick, fill] (0,0) -- (-60:1) circle (2pt) -- (-15:3) (-60:1) -- (-90:1.25)  circle (2pt) -- (-105:3) (-90:1.25) -- (-45:1.5) circle (2pt) -- (-30:3) (-45:1.5) -- (-75:1.75) circle (2pt) -- (-90:3) (-75:1.75) -- (-60:2) circle (2pt) -- (-45:3) (-60:2) -- (-67.5:2.5) circle (2pt) -- (-75:3) (-67.5:2.5) -- (-60:3);
\draw[thick, fill=white] (0,0) circle (2pt);
\end{tikzpicture}
&
\begin{tikzpicture}[baseline=.5ex]
\begin{scope}[yshift=2cm]
\draw[fill] (0,0) circle (2pt) node (O) {};
\draw[fill] (60:1) circle (2pt) node (A1) {};
\draw[fill] (180:1) circle (2pt) node (B1) {};
\draw[fill] (180:2) circle (2pt) node (B2) {};
\draw[fill] (-60:1) circle (2pt) node (C1) {};
\draw[fill] (-60:2) circle (2pt) node (C2) {};
\draw[fill] (-60:3) circle (2pt) node (C3) {};
\draw[fill] (-60:4) circle (2pt) node (C4) {};
\draw[fill] (-60:5) circle (2pt) node (C5) {};
\draw[->] (O) node[above left] {$ 1$}--(A1) node[above left] {$ 2$};
\draw[->] (O)--(B1) node[above] {$ 3$};
\draw[->] (B2) node[above] {$ 4$}--(B1);
\draw[->] (O)--(C1) node[above right] {$ 5$};
\draw[->] (C2) node[above right] {$ 6$}--(C1);
\draw[->] (C2)--(C3) node[above right] {$ 7$};
\draw[->] (C4) node[above right] {$ 8$}--(C3);
\draw[->] (C4)--(C5) node[above right] {$ 9$};
\end{scope}
\end{tikzpicture}\\
\bottomrule
\end{tabular}
\caption{$N$-graphs and their quivers of type~${\exdynE_n}=\quiver(a,b,c)$ with $n=a+b+c-3$}
\label{table:affine E type}
\end{table}

In the remaining part of this subsection, we argue the following to construct a starting exact embedded Lagrangian filling for the Legendrian of type $\exdynD$.

\begin{lemma}
The $N$-graphs in Tables~\ref{table:affine D type} and~\ref{table:affine 
E type} are free.
\end{lemma}

\begin{proof}
Recall from Definition~\ref{def:free} that an $N$-graph $\ngraph$ is free if the Legendrian weave $\Legendrian(\ngraph)$ can be woven without Reeb chords. Since we have already shown in \cite[Lemma 2.11]{ABL2021} that the $3$-graphs $\ngraph(a,b,c)$ in Table~\ref{table:affine E type} are free,  we focus on the $4$-graphs of type $\exdynD$ in Table~\ref{table:affine D type}, especially of type $\exdynD_7$, as follows. In other cases, similar arguments hold.

Since $\ngraph=\ngraph(\exdynD_7)$ is a $4$-graph, for each $x \in \D^2\setminus \ngraph$, $\pi^{-1}_{\D^2}(x)\in \Legendrian(\ngraph)$ consists of four points, and their hight with respect to $\pi_{\R}:\D^2\times \R \to \R$ induce four functions $h_i:\D^2\to \R$, $i=1,2,3,4$.
 Let us consider nonnegative functions $h_{ij}:=h_j - h_i$ for $1\leq i < j\leq 4$, then we have 
\begin{align*}
h_{12}^{-1}(0)&=\ngraph_1,& h_{23}^{-1}(0)&=\ngraph_2,& h_{34}^{-1}(0)&=\ngraph_3,
\end{align*}
where $\ngraph=(\ngraph_1, \ngraph_2,\ngraph_3)$.
We may assume that the height functions $h_1$, $h_2$, $h_3$, and $h_4$ are smooth except on $\ngraph_1$, $\ngraph_1\cup\ngraph_2$, $\ngraph_2\cup\ngraph_3$, and $\ngraph_3$, respectively.
So the gradient vector fields $\nabla h_{12}$, $\nabla h_{23}$, and $\nabla h_{34}$ are defined except on $\ngraph_1\cup\ngraph_2$, $\ngraph$, and $\ngraph_2\cup\ngraph_3$, respectively.
 Note that Reeb chords on $\D^2\setminus \ngraph$ corresponds to singular points of the gradient vector fields $\nabla h_{ij}$ for $1\leq i\leq j\leq 4$. So we need to construct hight functions $h_i$, $i=1,2,3,4$ satisfying the non-vanishing conditions.
 
We construct such gradient vector fields by weaving local charts of gradient vector fields. 
\begin{enumerate}
\item Near an edge of an $N$-graph we consider the following gradient vector fields $\nabla h_{i\; i+1}$ by tilting Legendrian sheets to avoid Reeb chords.
Note here that the direction of $\nabla h_{12}$ and of $\nabla h_{34}$ may be opposite.
Even though the following figures depict the local model for $\ngraph_2$, similar local gradient configurations valid for edges of $\ngraph_1$ and $\ngraph_3$.
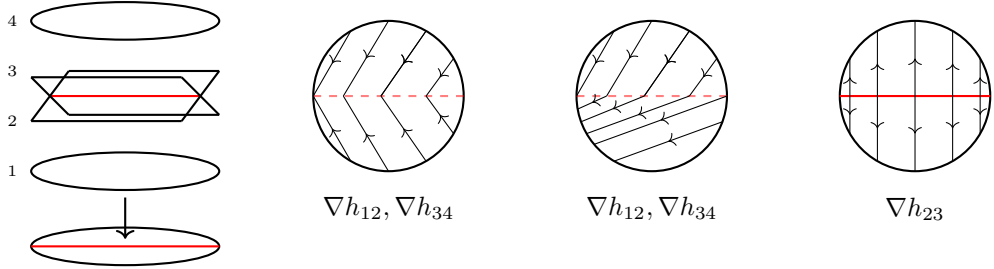
\[
\begin{tikzpicture}
\begin{scope}
\draw[thick] \boundellipse{0,1}{1.25}{0.25};

\draw[red, thick] (-1,0)--(1,0);
\draw[thick] (-3/4,1/3)--(5/4,1/3);
\draw[thick] (5/4,1/3)--(3/4,-1/3);
\draw[thick] (3/4,-1/3)--(-5/4,-1/3);
\draw[thick] (-5/4,-1/3)--(-3/4,1/3);
\draw[thick] (-5/4,3/12)--(3/4,3/12);
\draw[thick] (3/4,3/12)--(5/4,-3/12);
\draw[thick] (5/4,-3/12)--(-3/4,-3/12);
\draw[thick] (-3/4,-3/12)--(-5/4,3/12);

\draw[thick] \boundellipse{0,-1}{1.25}{0.25};
\draw[thick] \boundellipse{0,-2}{1.25}{0.25};
\draw[red, thick] (-5/4,-2)--(5/4,-2);
\draw[thick, ->] (0,-1.35)--(0,-1.9);

\node at (-1.5,1){\tiny $4$};
\node at (-1.5,1/3){\tiny $3$};
\node at (-1.5,-1/3){\tiny $2$};
\node at (-1.5,-1){\tiny $1$};

\end{scope}
\begin{scope}[xshift=3.5cm,decoration={markings, mark=at position 0.5 with {\arrow{>}}}]
\draw[thick] (0,0) circle (1cm);
\draw[red, thick, dashed, opacity=0.5] (-1,0)--(1,0);
\node at (0,-1.5) {$\nabla h_{12}, \nabla h_{34}$};
\draw[postaction={decorate}] (30:1) -- (0.5,0);
\draw[postaction={decorate}] (60:1) -- (-0.1,0);
\draw[postaction={decorate}] (90:1) -- (-0.6,0);
\draw[postaction={decorate}] (120:1) -- (-1,0);
\draw[postaction={decorate}] (60:1) -- (-0.1,0);
\draw[postaction={decorate}] (-30:1) -- (0.5,0);
\draw[postaction={decorate}] (-60:1) -- (-0.1,0);
\draw[postaction={decorate}] (-90:1) -- (-0.6,0);
\draw[postaction={decorate}] (-120:1) -- (-1,0);
\end{scope}

\begin{scope}[xshift=7cm,decoration={markings, mark=at position 0.5 with {\arrow{>}}}]
\draw[thick] (0,0) circle (1cm);
\draw[red, thick, dashed, opacity=0.5] (-1,0)--(1,0);
\node at (0,-1.5) {$\nabla h_{12}, \nabla h_{34}$};
\draw[postaction={decorate}] (30:1) -- (0.5,0);
\draw[postaction={decorate}] (60:1) -- (-0.1,0);
\draw[postaction={decorate}] (90:1) -- (-0.6,0);
\draw[postaction={decorate}] (120:1) -- (-1,0);
\draw[postaction={decorate}] (60:1) -- (-0.1,0);
\draw[postaction={decorate}]  (0.5,0) -- (-150:1);
\draw[postaction={decorate}] (-0.1,0) -- (-160:1);
\draw[postaction={decorate}] (-0.6,0) -- (-170:1);
\draw[postaction={decorate}] (1,0) -- (-140:1);
\draw[postaction={decorate}] (-20:1) -- (-120:1);
\end{scope}

\begin{scope}[xshift=10.5cm,decoration={markings, mark=at position 0.5 with {\arrow{>}}}]
\draw[thick] (0,0) circle (1cm);
\draw[red, thick] (-1,0)--(1,0);
\node at (0,-1.5) {$\nabla h_{23}$};
\draw[postaction={decorate}] (1.732*0.5,0) -- (30:1);
\draw[postaction={decorate}] (0.5,0) -- (60:1);
\draw[postaction={decorate}] (0,0) -- (90:1);
\draw[postaction={decorate}] (-0.5,0) -- (120:1);
\draw[postaction={decorate}] (-1.732*0.5,0) -- (150:1);
\draw[postaction={decorate}] (1.732*0.5,0) -- (-30:1);
\draw[postaction={decorate}] (0.5,0) -- (-60:1);
\draw[postaction={decorate}] (0,0) -- (-90:1);
\draw[postaction={decorate}] (-0.5,0) -- (-120:1);
\draw[postaction={decorate}] (-1.732*0.5,0) -- (-150:1);
\end{scope}

\end{tikzpicture}
\]
Even though the vector fields $\nabla h_{12}, \nabla h_{23}$, and $\nabla h_{34}$ in the above are not defined on $\ngraph_2$, the upper part of $\nabla h_{34}$ and the lower part of $\nabla h_{34}+\nabla h_{23}$, for example, can be smoothly extended to $\ngraph_2$.
This is because the four Legendrian sheets near the edge of the $4$-graph are smooth with distinct slope, and hence the (signed) height difference between any two sheets are well-defined even on $\ngraph_2$.
\item For the trivalent vertices in $\ngraph_2$, we consider the following gradient vector field configurations:
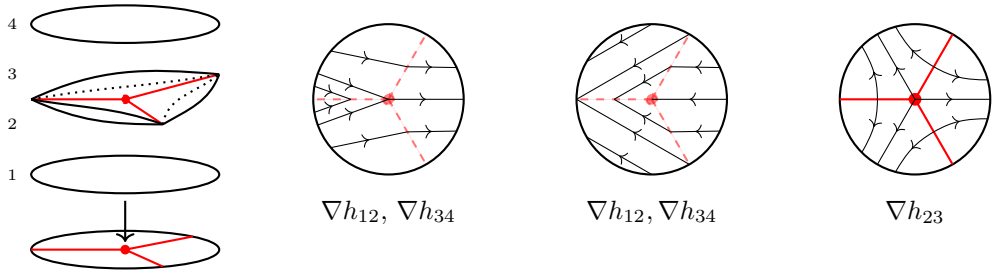
\[

\]
\item Near the hexagonal points in $\ngraph_1\cup\ngraph_2$, we consider the following model of gradient vector fields. The similar construction also works for hexagonal points contained in $\ngraph_2\cup\ngraph_3$.
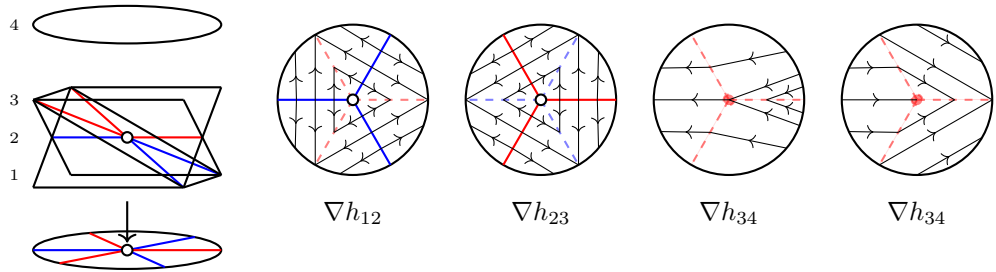
\[
%
\]
Similar as in (1), the four Legendrian sheets near the hexagonal point are smooth with distinct slope, so certain combination of $\nabla h_{12}, \nabla h_{23}$, and $\nabla h_{34}$ depending on the region can be smoothly extended to $\ngraph_1\cup \ngraph_2$.
\end{enumerate}
The upshot of the listed local model for the gradient is to avoid Reeb chords near the edges, the vertices, and the hexagonal points.
For graphical convenience, let us use the figures in the second row instead of the ones in the first row correspondingly:
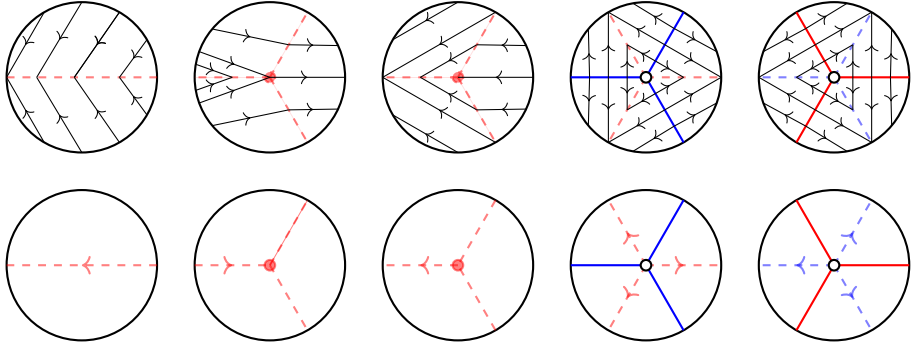
\[
%
\]
We omit arrows for other cases.
For the vector field on the boundary $\partial \D^2$, we use dotted line when the vector field inward, and use double line when it points outward.

Now we weave the above model of gradient vector fields to obtain the global gradient vector fields for $\ngraph(\exdynD_7)$:
\begin{enumerate}
\item The gradient $\nabla h_{12}$ is defined on $\D^2\setminus(\ngraph_1\cup \ngraph_2)$ and has the following configuration:
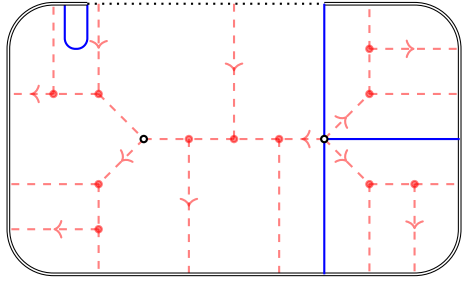
\[
%
\]
\item Let us consider the following vector field for the gradient $\nabla h_{23}$ on $\D^2\setminus\ngraph$:
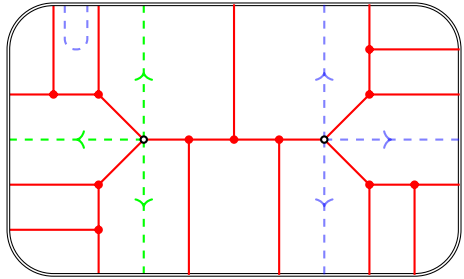
\[
%
\]
\item For the gradient $\nabla h_{34}$ on $\D^2\setminus(\ngraph_2\cup \ngraph_3)$, we consider the following vector field.
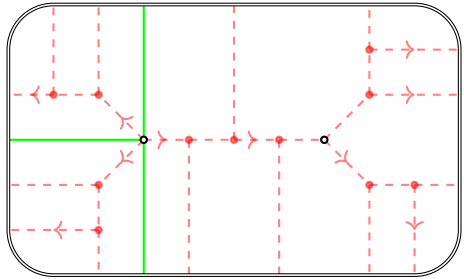
\[
%
\]
\end{enumerate}

It is direct to check that the gradient $\nabla h_{12}$, $\nabla h_{23}$, and $\nabla h_{34}$ admits nonvanishing gradient vector field on each connected component of the domain.
So it suffices to check that the same holds for the following three vector fields:
\begin{align*}
\nabla h_{13}&=\nabla h_{12}+\nabla h_{23};\\
\nabla h_{24}&=\nabla h_{23}+\nabla h_{34};\\
\nabla h_{14}&=\nabla h_{12}+\nabla h_{23}+\nabla h_{34}.
\end{align*}
By tilting the Legendrians, i.e., by adjusting the slope of each sheets of Legendrians in $\D^2\times \R$, we may assume that
\begin{align*}
\| \nabla h_{12} \| &> \| \nabla h_{23} \|,&
\| \nabla h_{23} \| &> \| \nabla h_{34} \|,&
\| \nabla h_{12} \| &> \| \nabla h_{23} \|+\| \nabla h_{34} \| 
\end{align*}
except the neighborhood of the hexagonal points.
The assumption $\| \nabla h_{12} \| > \| \nabla h_{23} \|$, for example, guarantees that there are no vanishing points of $\nabla h_{13}$, even though there exist some points on $\D^2\setminus (\ngraph_1\cup \ngraph_2)$ where the direction of $\nabla h_{12}$ and the one of $\nabla h_{23}$ opposite. The same argument holds for $\nabla h_{24}$ and $\nabla h_{14}$.
\end{proof}

\subsection{Legendrian Coxeter mutation on \texorpdfstring{$N$}{N}-graphs}\label{sec:LCM on N-graphs}
Note that the induced quivers of the $N$-graphs in Tables~\ref{table:affine D 
type} and~\ref{table:affine E type} are all bipartite. In other words, there 
are two sets of vertices $I_+$ and $I_-$ of the quiver $\quiver$ such that all 
arrows are oriented from $I_+$ to $I_-$. A Coxeter mutation 
$\mutation_{\quiver}$ is defined by the composition of the mutations
\[
\mutation_\quiver=\prod_{i\in I_-}\mutation_i \cdot \prod_{i\in I_+} \mutation_i.
\]
Note that $\prod_{i\in I_+}\mutation_i$ does not depend on the order of composition of mutations $\mutation_i$ among $i\in I_+$, and the same holds for $I_-$. It is easy to check that $\prod_{i\in I_+}\mutation_i \cdot \prod_{i\in I_-} \mutation_i$ becomes the inverse of $\mutation_\quiver$, and defines another Coxeter mutation. Let us denote it by  $\mutation_\quiver^{-1}$.

Let us consider the action of Coxeter mutation $\mutation_\quiver$ on the exchange graph of type $\exdynD \exdynE$.
Recall from Remark~\ref{rmk_inifinitely_many_seeds_for_affine} that the order of $\mutation_\quiver$ is infinite.

Now we apply the Coxeter mutation in the $N$-graph setup. We call a pair $(\ngraph, \nbasis)$ of an $N$-graph together with a set of cycles $\nbasis$ is \emph{bipartite} if the induced quiver $\quiver(\ngraph, \nbasis)$ is bipartite. Then the set of one cycles $\nbasis$ is decomposed into $\nbasis_+$ and $\nbasis_-$ regarding $I_+$ and $I_-$, respectively.
Let us define a \emph{Legendrian Coxeter mutation} $\mutation_\ngraph$ on $\ngraph$ by
\begin{align*}
\mu_\ngraph&=\prod_{\gamma\in \nbasis_-}\mutation_\gamma \cdot \prod_{\gamma\in \nbasis_+}\mutation_\gamma,&
\mu_\ngraph^{-1}&=\prod_{\gamma\in \nbasis_+}\mutation_\gamma \cdot \prod_{\gamma\in \nbasis_-}\mutation_\gamma.
\end{align*}
It is worth mentioning that $\mutation_\ngraph^{\pm1}$ is well defined if the 
set of one cycles $\nbasis_\pm$ is disjoint. That is to say that 
$\prod_{\gamma\in \nbasis_\pm}\mutation_\gamma$ is independent of the order of 
mutations among $\gamma\in \nbasis_\pm$ when it is disjoint. This directly 
implies that $\mutation_\ngraph^{-1}$ is indeed the inverse of 
$\mutation_\ngraph$. Note that all sets of one cycles $\nbasis_\pm$ of the 
pairs $(\ngraph,\nbasis)$ in Tables~\ref{table:affine D type} 
and~\ref{table:affine E type} satisfy the disjoint condition.

In order to realize the Coxeter mutation in $N$-graphs setup, we need to argue 
that there is no obstruction to apply $\mutation_\ngraph^r$  to the pairs 
$(\ngraph,\nbasis)$ listed in Tables~\ref{table:affine D type} 
and~\ref{table:affine E type} for any $r \in \Z$.

The Legendrian Coxeter mutations $\mutation_\ngraph^{\pm1}$ on $(\ngraph({\exdynE}_n), \nbasis({\exdynE}_n))$ with $n=6,7,8$, so called tripod $N$-graphs, are already discussed in \cite{ABL2021}. Let us recall some terminologies. For any pair $(\ngraph,\nbasis)$ of a $3$-graph, i.e. bicolored graph, with an ordered set of one-cycles, $(\bar\ngraph,\bar\nbasis)$ denotes the pair obtained by switching two colors.

\begin{definition}[Coxeter padding for tripod $N$-graphs]
For each triple $(a,b,c)$, the annular $N$-graph depicted in Figure~\ref{figure:coxeter padding} is denoted by $\coxeterpadding(a,b,c)$ and called the \emph{Coxeter padding} of type $(a,b,c)$. We also denote the Coxeter padding with color switched by $\bar \coxeterpadding(a,b,c)$.
\end{definition}

\begin{figure}[ht]
\subfigure[$\coxeterpadding(a,b,c)$]{\makebox[0.48\textwidth]{
$
\begin{tikzpicture}[baseline=-.5ex,scale=0.5]
\draw[thick] (0,0) circle (5) (0,0) circle (3);
\foreach \i in {1,2,3} {
\begin{scope}[rotate=\i*120]
\draw[blue, thick, rounded corners] (0:3) -- (0:3.4) to[out=-75,in=80] (-40:4);
\draw[red, thick, dashed, rounded corners] (60:3) -- (60:3.3) to[out=0,in=220] (40:4) (40:4) to[out=120,in=-20] (60:4);
\draw[red, thick, rounded corners] (40:3) -- (40:3.3) to[out=-20,in=200] (20:4) (80:3) -- (80:3.3) to[out=20,in=240] (60:4) (100:3) -- (100:3.3) to[out=40,in=260] (80:4);
\draw[red, thick, rounded corners] (20:3) -- (20:3.5) to[out=-70,in=50] (-40:4) (20:4) to[out=-50,in=120] (0:4.5) -- (0:5);
\draw[red, thick] (20:4) to[out=100,in=-40] (40:4) (60:4) to[out=140,in=0] (80:4);
\draw[blue, thick] (20:5) -- (20:4) to[out=140,in=-80] (40:4) (60:5) -- (60:4) to[out=180,in=-40] (80:4) -- (80:5);
\draw[blue, thick, rounded corners] (20:4) to[out=-70,in=100] (-20:4.5) -- (-20:5);
\draw[blue, thick, dashed] (40:4) -- (40:5) (40:4) to[out=160,in=-60] (60:4);
\draw[fill=white, thick] (20:4) circle (2pt) (40:4) circle (2pt) (60:4) circle (2pt) (80:4) circle (2pt) (-40:4) circle (2pt);
\end{scope}
\curlybrace[]{10}{110}{5.2};
\draw (60:5.5) node[rotate=-30] {$a+1$};
\curlybrace[]{130}{230}{5.2};
\draw (180:5.5) node[rotate=90] {$b+1$};
\curlybrace[]{250}{350}{5.2};
\draw (300:5.5) node[rotate=30] {$c+1$};
}
\end{tikzpicture}$
}}
\subfigure[$\bar\coxeterpadding(a,b,c)$]{\makebox[0.48\textwidth]{
$
\begin{tikzpicture}[baseline=-.5ex,scale=0.5]
\draw[thick] (0,0) circle (5) (0,0) circle (3);
\foreach \i in {1,2,3} {
\begin{scope}[rotate=\i*120]
\draw[red, thick, rounded corners] (0:3) -- (0:3.4) to[out=-75,in=80] (-40:4);
\draw[blue, thick, dashed, rounded corners] (60:3) -- (60:3.3) to[out=0,in=220] (40:4) (40:4) to[out=120,in=-20] (60:4);
\draw[blue, thick, rounded corners] (40:3) -- (40:3.3) to[out=-20,in=200] (20:4) (80:3) -- (80:3.3) to[out=20,in=240] (60:4) (100:3) -- (100:3.3) to[out=40,in=260] (80:4);
\draw[blue, thick, rounded corners] (20:3) -- (20:3.5) to[out=-70,in=50] (-40:4) (20:4) to[out=-50,in=120] (0:4.5) -- (0:5);
\draw[blue, thick] (20:4) to[out=100,in=-40] (40:4) (60:4) to[out=140,in=0] (80:4);
\draw[red, thick] (20:5) -- (20:4) to[out=140,in=-80] (40:4) (60:5) -- (60:4) to[out=180,in=-40] (80:4) -- (80:5);
\draw[red, thick, rounded corners] (20:4) to[out=-70,in=100] (-20:4.5) -- (-20:5);
\draw[red, thick, dashed] (40:4) -- (40:5) (40:4) to[out=160,in=-60] (60:4);
\draw[fill=white, thick] (20:4) circle (2pt) (40:4) circle (2pt) (60:4) circle (2pt) (80:4) circle (2pt) (-40:4) circle (2pt);
\end{scope}

\curlybrace[]{10}{110}{5.2};
\draw (60:5.5) node[rotate=-30] {$a+1$};
\curlybrace[]{130}{230}{5.2};
\draw (180:5.5) node[rotate=90] {$b+1$};
\curlybrace[]{250}{350}{5.2};
\draw (300:5.5) node[rotate=30] {$c+1$};
}
\end{tikzpicture}$
}}
\\
\subfigure[$\coxeterpadding^{-1}(a,b,c)$]{\makebox[0.48\textwidth]{
$
\begin{tikzpicture}[baseline=-.5ex,xscale=-0.5,yscale=0.5,rotate=60]
\draw[thick] (0,0) circle (5) (0,0) circle (3);
\foreach \i in {1,2,3} {
\begin{scope}[rotate=\i*120]
\draw[red, thick, rounded corners] (0:3) -- (0:3.4) to[out=-75,in=80] (-40:4);
\draw[blue, thick, dashed, rounded corners] (60:3) -- (60:3.3) to[out=0,in=220] (40:4) (40:4) to[out=120,in=-20] (60:4);
\draw[blue, thick, rounded corners] (40:3) -- (40:3.3) to[out=-20,in=200] (20:4) (80:3) -- (80:3.3) to[out=20,in=240] (60:4) (100:3) -- (100:3.3) to[out=40,in=260] (80:4);
\draw[blue, thick, rounded corners] (20:3) -- (20:3.5) to[out=-70,in=50] (-40:4) (20:4) to[out=-50,in=120] (0:4.5) -- (0:5);
\draw[blue, thick] (20:4) to[out=100,in=-40] (40:4) (60:4) to[out=140,in=0] (80:4);
\draw[red, thick] (20:5) -- (20:4) to[out=140,in=-80] (40:4) (60:5) -- (60:4) to[out=180,in=-40] (80:4) -- (80:5);
\draw[red, thick, rounded corners] (20:4) to[out=-70,in=100] (-20:4.5) -- (-20:5);
\draw[red, thick, dashed] (40:4) -- (40:5) (40:4) to[out=160,in=-60] (60:4);
\draw[fill=white, thick] (20:4) circle (2pt) (40:4) circle (2pt) (60:4) circle (2pt) (80:4) circle (2pt) (-40:4) circle (2pt);
\end{scope}

\curlybrace[]{10}{110}{5.2};
\draw (60:5.5) node[rotate=-30] {$a+1$};
\curlybrace[]{130}{230}{5.2};
\draw (180:5.5) node[rotate=30] {$c+1$};
\curlybrace[]{250}{350}{5.2};
\draw (300:5.5) node[rotate=90] {$b+1$};
}
\end{tikzpicture}$
}}
\subfigure[$\bar\coxeterpadding^{-1}(a,b,c)$]{\makebox[0.48\textwidth]{
$
\begin{tikzpicture}[baseline=-.5ex,xscale=-0.5, yscale=0.5,rotate=60]
\draw[thick] (0,0) circle (5) (0,0) circle (3);
\foreach \i in {1,2,3} {
\begin{scope}[rotate=\i*120]
\draw[blue, thick, rounded corners] (0:3) -- (0:3.4) to[out=-75,in=80] (-40:4);
\draw[red, thick, dashed, rounded corners] (60:3) -- (60:3.3) to[out=0,in=220] (40:4) (40:4) to[out=120,in=-20] (60:4);
\draw[red, thick, rounded corners] (40:3) -- (40:3.3) to[out=-20,in=200] (20:4) (80:3) -- (80:3.3) to[out=20,in=240] (60:4) (100:3) -- (100:3.3) to[out=40,in=260] (80:4);
\draw[red, thick, rounded corners] (20:3) -- (20:3.5) to[out=-70,in=50] (-40:4) (20:4) to[out=-50,in=120] (0:4.5) -- (0:5);
\draw[red, thick] (20:4) to[out=100,in=-40] (40:4) (60:4) to[out=140,in=0] (80:4);
\draw[blue, thick] (20:5) -- (20:4) to[out=140,in=-80] (40:4) (60:5) -- (60:4) to[out=180,in=-40] (80:4) -- (80:5);
\draw[blue, thick, rounded corners] (20:4) to[out=-70,in=100] (-20:4.5) -- (-20:5);
\draw[blue, thick, dashed] (40:4) -- (40:5) (40:4) to[out=160,in=-60] (60:4);
\draw[fill=white, thick] (20:4) circle (2pt) (40:4) circle (2pt) (60:4) circle (2pt) (80:4) circle (2pt) (-40:4) circle (2pt);
\end{scope}
\curlybrace[]{10}{110}{5.2};
\draw (60:5.5) node[rotate=-30] {$a+1$};
\curlybrace[]{130}{230}{5.2};
\draw (180:5.5) node[rotate=30] {$c+1$};
\curlybrace[]{250}{350}{5.2};
\draw (300:5.5) node[rotate=90] {$b+1$};
}
\end{tikzpicture}$
}}
\caption{Coxeter paddings $\coxeterpadding(a,b,c)$, $\bar\coxeterpadding(a,b,c)$ and their inverses.}
\label{figure:coxeter padding}
\end{figure}

\begin{proposition}[{\cite[Proposition 5.11]{ABL2021}}]\label{prop:coxeter mutation on tripod}
For each triple $(a,b,c)$, the Legendrian Coxeter mutation on $(\ngraph(a,b,c),\nbasis(a,b,c))$ or $(\bar\ngraph(a,b,c),\bar\nbasis(a,b,c))$ is given by concatenating the Coxeter padding~$\coxeterpadding(a,b,c)$ followed by switching two colors:
\begin{align*}
\ncoxeter(\ngraph(a,b,c), \nbasis(a,b,c)) &= \coxeterpadding(a,b,c)(\bar \ngraph(a,b,c), \bar\nbasis(a,b,c));\\
\ncoxeter(\bar \ngraph(a,b,c), \nbasis(a,b,c)) &= \bar \coxeterpadding(a,b,c) (\ngraph(a,b,c), \nbasis(a,b,c)).
\end{align*}
Similarly,
\begin{align*}
\ncoxeter^{-1}(\ngraph(a,b,c), \nbasis(a,b,c)) &= \bar \coxeterpadding^{-1}(a,b,c)(\bar \ngraph(a,b,c), \bar\nbasis(a,b,c));\\
\ncoxeter^{-1}(\bar \ngraph(a,b,c), \nbasis(a,b,c)) &= \coxeterpadding^{-1}(a,b,c) (\ngraph(a,b,c), \nbasis(a,b,c)).
\end{align*}

\end{proposition}

\begin{figure}[ht]
\begin{tikzpicture}[baseline=-.5ex,scale=0.5]
\draw[thick] (0,0) circle (5cm);
\draw[dashed]  (0,0) circle (3cm);
\fill[opacity=0.1, even odd rule] (0,0) circle (3) (0,0) circle (5);
\foreach \i in {1,2,3} {
\begin{scope}[rotate=\i*120]
\draw[yellow, line cap=round, line width=5, opacity=0.5] (60:1) -- (50:1.5) (70:1.75) -- (50:2);
\draw[green, line cap=round, line width=5, opacity=0.5] (0,0) -- (60:1) (50:1.5) -- (70:1.75);
\draw[blue, thick, rounded corners] (0,0) -- (0:3.4) to[out=-75,in=80] (-40:4);
\draw[red, thick, fill] (0,0) -- (60:1) circle (2pt) (60:1) -- (50:1.5) circle (2pt) -- (70:1.75) circle (2pt) -- (50:2) circle (2pt);
\draw[red, thick, dashed, rounded corners] (50:2) -- (60:2.8) -- (60:3.3) to[out=0,in=220] (40:4) (40:4) to[out=120,in=-20] (60:4);
\draw[red, thick, rounded corners] (50:2) -- (40:2.8) -- (40:3.3) to[out=-20,in=200] (20:4) (70:1.75) -- (80:2.8) -- (80:3.3) to[out=20,in=240] (60:4) (60:1) -- (100:2.8) -- (100:3.3) to[out=40,in=260] (80:4);
\draw[red, thick, rounded corners] (50:1.5) -- (20:3) -- (20:3.5) to[out=-70,in=50] (-40:4) (20:4) to[out=-50,in=120] (0:4.5) -- (0:5);
\draw[red, thick] (20:4) to[out=100,in=-40] (40:4) (60:4) to[out=140,in=0] (80:4);
\draw[blue, thick] (20:5) -- (20:4) to[out=140,in=-80] (40:4) (60:5) -- (60:4) to[out=180,in=-40] (80:4) -- (80:5);
\draw[blue, thick, rounded corners] (20:4) to[out=-70,in=100] (-20:4.5) -- (-20:5);
\draw[blue, thick, dashed] (40:4) to[out=160,in=-60] (60:4) (40:4) -- (40:5);
\draw[fill=white, thick] (20:4) circle (2pt) (40:4) circle (2pt) (60:4) circle (2pt) (80:4) circle (2pt) (-40:4) circle (2pt);
\end{scope}
\draw[fill=white, thick] (0,0) circle (2pt);
}
\end{tikzpicture}
\caption{The Legendrian Coxeter mutation $\mutation_\ngraph$ on $(\ngraph(a,b,c),\nbasis(a,b,c))$}
\label{figure:after Legendrian coxeter mutation}
\end{figure}

The $N$-graph after applying Legendrian Coxeter mutation $\mutation_\ngraph$ on the pair $(\ngraph(a,b,c),\nbasis(a,b,c))$ is depicted in Figure~\ref{figure:after Legendrian coxeter mutation}. We have the following corollary immediately.

\begin{corollary}\label{cor:coxeter realization E-type}
For $n=6,7,8$ and any $r\in\Z$, the Legendrian Coxeter mutation $\mutation_\ngraph^r(\ngraph(\exdynE_n),\nbasis(\exdynE_n))$ is realizable by an $N$-graph and a good tuple of cycles.
\end{corollary}

On the other hand, for the $N$-graph $(\ngraph(\exdynD_n), \nbasis(\exdynD_n))$ with cycles, the Legendrian Coxeter mutation becomes an attachment of the Coxeter padding of type $\coxeterpadding^{\pm1}(\exdynD_n)$ depicted in Table~\ref{table:coxeter paddings for affine D}.

\begin{table}[ht]
\begin{tabular}{c|c}
\toprule
$\coxeterpadding({\exdynD}_{n})$ &
\begin{tikzpicture}[baseline=-.5ex,scale=0.6]
\useasboundingbox (-3.5,-5) rectangle (14.5,5);
\draw (5.5,3.75) node {$\overbrace{\hphantom{\hspace{4cm}}}^{k=\lfloor \frac{n-3}2\rfloor}$};
\draw (5.5,-3.75) node {$\underbrace{\hphantom{\hspace{4cm}}}_{\ell=\lfloor \frac{n-4}2\rfloor}$};
\draw[blue,thick, rounded corners]
(1/3,3)-- (4/3,2) --(1,4/3)-- (2/3,2) -- (-1/3,3)
;
\begin{scope}[yscale=-1]
\draw[green, line width=5, opacity=0.5, line cap=round] 
(1,0.5) -- (1,1.5)
(0.5,-1) --(1.5,-1)
;
\draw[yellow, line width=5, opacity=0.5] 
(1,0.5) -- (2,0) --(1.5,-1)
(2,0) -- (3,0)
;
\draw[green, thick] 
(2,3) -- (2,-3)
(2,2) -- (0,2) -- (-2,0) -- (-3,0)
(2,-2) -- (0,-2) -- (-2,0)
(0,2) -- (0,-2)
(0,0) -- (2,0)
;
\draw[red, thick] 
(-3,1) -- (-2,0) to[out=15,in=-105] (0,2) (0,-2) -- (-1,-3)
(0,2) -- (1,1.5) -- (2,2)
(1,3) -- (2,2)
(1,1.5) -- (1,0.5) -- (2,0) --(3,0)
(1,0.5) -- (0,0) -- (0.5,-1) -- (1.5,-1) -- (2,0)
(0,0) to[out=-120,in=120] (0,-2) (0,2) -- (-1,3)
(0,-2) -- (0.5,-1)
(1.5,-1) -- (2,-2)
(2,-2) -- (1,-3)
(-3,-1) -- (-2,0)
;
\draw[thick, fill=white] 
(-2,0) circle (2pt) 
(0,2) circle (2pt) 
(0,-2) circle (2pt) 
(0,0) circle (2pt) 
(2,2) circle (2pt) 
(2,0) circle (2pt) 
(2,-2) circle (2pt);
\draw[red, thick, fill]
(1,0.5) circle (2pt)
(1,1.5) circle (2pt)
(0.5,-1) circle (2pt)
(1.5,-1) circle (2pt)
;
\draw[thick,]
(3,3) -- (-1,3) to[out=180, in=90] (-3,1) to (-3,-1) to[out=-90,in=180] (-1,-3) to (12,-3) to[out=0, in=-90] (14,-1) to (14,1) to[out=90,in=0] (12,3) to (3,3) 
;
\end{scope}

\begin{scope}[xshift=11cm,rotate=180]
\draw[yellow, line width=5, opacity=0.5, line cap=round] 
(1,0.5) -- (1,1.5)
(0.5,-1) --(1.5,-1)
;
\draw[green, line width=5, opacity=0.5] 
(1,0.5) -- (2,0) --(1.5,-1)
(2,0) -- (3,0)
;
\draw[blue, thick] 
(2,3) -- (2,-3)
(2,2) -- (0,2) -- (-2,0) -- (-3,0)
(2,-2) -- (0,-2) -- (-2,0)
(0,2) -- (0,-2)
(0,0) -- (2,0)
;
\draw[red, thick] 
(-3,1) -- (-2,0) to[out=-15,in=105] (0,-2) (0,2) -- (-1,3)
(0,2) -- (1,1.5) -- (2,2)
(1,3) -- (2,2)
(1,1.5) -- (1,0.5) -- (2,0) -- (3,0)
(1,0.5) -- (0,0) -- (0.5,-1) -- (1.5,-1) -- (2,0)
(0,0) to[out=120,in=-120] (0,2) (0,-2) -- (-1,-3)
(0,-2) -- (0.5,-1)
(1.5,-1) -- (2,-2)
(2,-2) -- (1,-3)
(-3,-1) -- (-2,0);
\draw[thick, fill=white] 
(-2,0) circle (2pt) 
(0,2) circle (2pt) 
(0,-2) circle (2pt) 
(0,0) circle (2pt) 
(2,2) circle (2pt) 
(2,0) circle (2pt) 
(2,-2) circle (2pt);
\draw[red, thick, fill]
(1,0.5) circle (2pt)
(1,1.5) circle (2pt)
(0.5,-1) circle (2pt)
(1.5,-1) circle (2pt)
(3,0) circle (2pt)
;
\end{scope}

\begin{scope}[xshift=6cm]
\draw[green, line width=5, opacity=0.5](-3,0)--(-2,0) (-1,0)--(-0.5,0); 
\draw[yellow, line width=5, opacity=0.5](-2,0) --(-1,0);
\draw[red, thick, rounded corners] 
(-3,0) -- (-0.5,0) (0.5,0)-- (2,0)
(-4,-2) -- (-3,-2) -- (-3,0)
(-4,2) -- (-2,2) -- (-2,3)
(-2,0) -- (-2,1) -- (-1,1) -- (-1,2) -- (-0.5,2) (0.5,2) -- (1,2) -- (1,3)
(-3,-3) -- (-3,-2) -- (-2,-2) -- (-2,-1) -- (-1,-1) -- (-1,0)
(-1,-3) -- (-1,-2) -- (-0.5,-2) (0.5,-2) -- (1,-2) -- (1,-1) -- (2,-1) -- (2,0)
(1,0) -- (1,1) -- (2,1) -- (2,2) -- (3,2)
(2,-3) -- (2,-2) -- (3,-2)
;
\draw[red, thick, fill] (-3,0) circle (2pt) (-2,0) circle (2pt) (-1,0) circle (2pt);
\draw[red, thick, dashed] (-0.5, 0) -- (0.5, 0) (-0.5, 2) -- (0.5, 2) (-0.5, -2) -- (0.5, -2);
\end{scope}
\begin{scope}[xshift=6cm,rotate=180]
\draw[green, line width=5, opacity=0.5](-1,0)--(-0.5,0); 
\draw[yellow, line width=5, opacity=0.5](-2,0) --(-1,0);
\draw[red, thick, rounded corners] 
;
\draw[red, thick, fill] (-2,0) circle (2pt) (-1,0) circle (2pt);
\end{scope}
\draw[thick, rounded corners, fill=white]
(3,1.75) -- (0.25,1.75) -- (0.25, -1.75) -- (10.75,-1.75) -- (10.75,1.75) -- (3,1.75);
\end{tikzpicture}
\\
\hline
$\coxeterpadding^{-1}({\exdynD}_{n})$ &
\begin{tikzpicture}[baseline=-.5ex,scale=0.6]
\useasboundingbox (-3.5,-5) rectangle (14.5,5);
\draw (5.5,3.75) node {$\overbrace{\hphantom{\hspace{4cm}}}^{k=\lfloor \frac{n-3}2\rfloor}$};
\draw (5.5,-3.75) node {$\underbrace{\hphantom{\hspace{4cm}}}_{\ell=\lfloor \frac{n-4}2\rfloor}$};
\draw[blue,thick, rounded corners]
(1/3,3)-- (4/3,2) --(1,4/3)-- (2/3,2) -- (-1/3,3)
;
\begin{scope}[yscale=-1]
\draw[green, line width=5, opacity=0.5, line cap=round] 
(1,0.5) -- (1,1.5)
(0.5,-1) --(1.5,-1)
;
\draw[yellow, line width=5, opacity=0.5] 
(1,0.5) -- (2,0) --(1.5,-1)
(2,0) -- (3,0)
;
\draw[green, thick] 
(2,3) -- (2,-3)
(2,2) -- (0,2) -- (-2,0) -- (-3,0)
(2,-2) -- (0,-2) -- (-2,0)
(0,2) -- (0,-2)
(0,0) -- (2,0)
;
\draw[red, thick] 
(-3,1) -- (-2,0) to[out=-15,in=105] (0,-2) -- (-1,-3)
(0,2) -- (1,1.5) -- (2,2)
(1,3) -- (2,2)
(1,1.5) -- (1,0.5) -- (2,0) --(3,0)
(1,0.5) -- (0,0) -- (0.5,-1) -- (1.5,-1) -- (2,0)
(0,0) to[out=120,in=-120] (0,2) -- (-1,3)
(0,-2) -- (0.5,-1)
(1.5,-1) -- (2,-2)
(2,-2) -- (1,-3)
(-3,-1) -- (-2,0)
;
\draw[red, thick, rounded corners]
(2,2) -- (3,2) -- (3,3)
;
\draw[thick, fill=white] 
(-2,0) circle (2pt) 
(0,2) circle (2pt) 
(0,-2) circle (2pt) 
(0,0) circle (2pt) 
(2,2) circle (2pt) 
(2,0) circle (2pt) 
(2,-2) circle (2pt);
\draw[red, thick, fill]
(1,0.5) circle (2pt)
(1,1.5) circle (2pt)
(0.5,-1) circle (2pt)
(1.5,-1) circle (2pt)
;
\draw[thick,]
(3,3) -- (-1,3) to[out=180, in=90] (-3,1) to (-3,-1) to[out=-90,in=180] (-1,-3) to (12,-3) to[out=0, in=-90] (14,-1) to (14,1) to[out=90,in=0] (12,3) to (3,3) 
;
\end{scope}

\begin{scope}[xshift=11cm,rotate=180]
\draw[yellow, line width=5, opacity=0.5, line cap=round] 
(1,0.5) -- (1,1.5)
(0.5,-1) --(1.5,-1)
;
\draw[green, line width=5, opacity=0.5] 
(1,0.5) -- (2,0) --(1.5,-1)
(2,0) -- (3,0)
;
\draw[blue, thick] 
(2,3) -- (2,-3)
(2,2) -- (0,2) -- (-2,0) -- (-3,0)
(2,-2) -- (0,-2) -- (-2,0)
(0,2) -- (0,-2)
(0,0) -- (2,0)
;
\draw[red, thick] 
(-3,1) -- (-2,0) to[out=15,in=-105] (0,2) -- (-1,3)
(0,2) -- (1,1.5) -- (2,2)
(1,3) -- (2,2)
(1,1.5) -- (1,0.5) -- (2,0) -- (3,0)
(1,0.5) -- (0,0) -- (0.5,-1) -- (1.5,-1) -- (2,0)
(0,0) to[out=-120,in=120] (0,-2) -- (-1,-3)
(0,-2) -- (0.5,-1)
(1.5,-1) -- (2,-2) -- (3,-2)
(2,-2) -- (1,-3)
(-3,-1) -- (-2,0);
\draw[red, thick, rounded corners]
(2,2) -- (3,2) -- (3,0)
;
\draw[thick, fill=white] 
(-2,0) circle (2pt) 
(0,2) circle (2pt) 
(0,-2) circle (2pt) 
(0,0) circle (2pt) 
(2,2) circle (2pt) 
(2,0) circle (2pt) 
(2,-2) circle (2pt);
\draw[red, thick, fill]
(1,0.5) circle (2pt)
(1,1.5) circle (2pt)
(0.5,-1) circle (2pt)
(1.5,-1) circle (2pt)
(3,0) circle (2pt)
;
\end{scope}

\begin{scope}[xshift=6cm]
\draw[green, line width=5, opacity=0.5](-3,0)--(-2,0) (-1,0)--(-0.5,0); 
\draw[yellow, line width=5, opacity=0.5](-2,0) --(-1,0);
\draw[red, thick, rounded corners] 
(-4,2)--(-3, 2) -- (-3,1) -- (-2,1) -- (-2,0) 
(-2,0) -- (-3,0) 
(-3,0) -- (-3,-1) -- (-2,-1) -- (-2,-2) -- (-1,-2) -- (-1,-3) 
(-2,0) -- (-1,0) 
(-2,3) -- (-2,2) -- (-1,2) -- (-1, 1) -- (-0.5, 1) 
(-1,0) -- (-0.5, 0) 
(-1,0) -- (-1,-1) -- (-0.5, -1);
\draw[red, thick, fill] (-3,0) circle (2pt) (-2,0) circle (2pt) (-1,0) circle (2pt);
\draw[red, thick, dashed] (-0.5, 0) -- (0.5, 0) (-0.5, 1) -- (0.5, 1) (-0.5, -1) -- (0.5, -1);
\end{scope}
\begin{scope}[xshift=6cm,rotate=180]
\draw[green, line width=5, opacity=0.5](-1,0)--(-0.5,0); 
\draw[yellow, line width=5, opacity=0.5](-2,0) --(-1,0);
\draw[red, thick, rounded corners] 
(-2,-2) -- (-1,-2) -- (-1,-3) 
(-2,0) -- (-1,0) 
(-2,3) -- (-1,2) -- (-1, 1) -- (-0.5, 1) 
(-1,0) -- (-0.5, 0) 
(-1,0) -- (-1, -1) -- (-0.5, -1);
\draw[red, thick, fill] (-2,0) circle (2pt) (-1,0) circle (2pt);
\end{scope}
\draw[thick, rounded corners, fill=white]
(3,1.75) -- (0.25,1.75) -- (0.25, -1.75) -- (10.75,-1.75) -- (10.75,1.75) -- (3,1.75)
;
\end{tikzpicture}
\\
\bottomrule
\end{tabular}
\caption{Coxeter paddings $\coxeterpadding^{\pm1}(\exdynD_n)$}
\label{table:coxeter paddings for affine D}
\end{table}

\begin{proposition}\label{proposition:coxeter realization D-type}
For any $r\in\Z$, the Legendrian Coxeter mutations $\mutation_\ngraph^r$ on the pairs are given by piling the Coxeter paddings $\coxeterpadding^{\pm1}(\exdynD_n)$.
\begin{align*}
\mutation_\ngraph^{\pm1}(\ngraph(\exdynD_n),\nbasis(\exdynD_n))
=\coxeterpadding^{\pm1}(\exdynD_n)(\ngraph(\exdynD_n),\nbasis(\exdynD_n)).
\end{align*}
\end{proposition}

The pictorial proof of this proposition will be given in Appendix~\ref{sec:Coxeter padding affine D_n}.
Consequently, we have the following corollary.
\begin{corollary}\label{cor:coxeter realization D-type}
For any $r\in \Z$, 
the Legendrian Coxeter mutation $\mutation_\ngraph^r(\ngraph(\exdynD_n),\nbasis(\exdynD_n))$ is realizable by $N$-graphs and good tuple of cycles.
\end{corollary}

Note that the Coxeter paddings are obtained from the Coxeter mutations $\mutation_\ngraph^{\pm1}$ conjugated by a sequence of Move (II). For the notational clarity, it is worth mentioning that $\coxeterpadding(\exdynD_n)$ and $\coxeterpadding^{-1}(\exdynD_n)$ are the inverse to each other with respect to the piling up operation introduced in Section~\ref{subsec:N-graph}.

For example, let us present the Coxeter paddings $\coxeterpadding^{\pm1}(\exdynD_4)$ as follows:
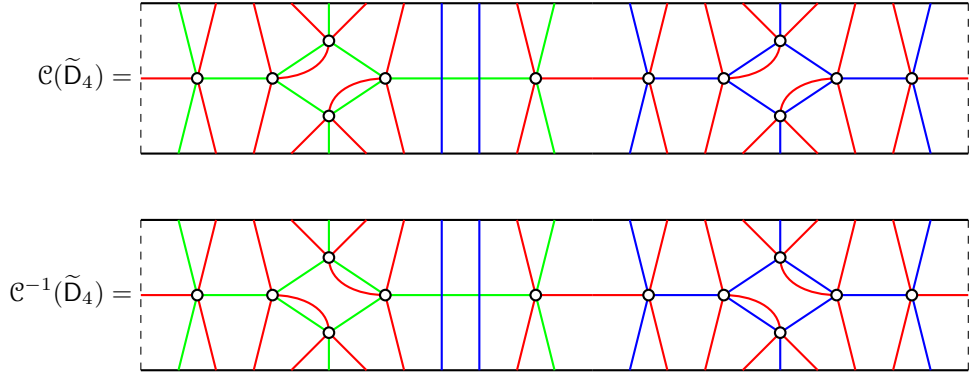
\[
\phantom{^{-1}}\coxeterpadding(\exdynD_4)=
\begin{tikzpicture}[baseline=-.5ex,yscale=-1]
\begin{scope}
\draw[thick] (0,1) -- (6,1) (0,-1) -- (6,-1);
\draw[dashed] (0,1) -- (0,-1);
\draw[thick, red] 
(0,0) -- (0.75,0) -- (1,1) (0.75,0) -- (1,-1)
(1.5,1) --(1.75,0) -- (1.5,-1)
(1.75,0) to[out=0, in=90] (2.5,-0.5) -- (2,-1)
(2.5,-0.5) --(3,-1)
(2,1) -- (2.5,0.5) -- (3,1)
(2.5,0.5) to[out=-90,in=180] (3.25,0) -- (3.5,1)
(3.25,0) -- (3.5,-1)
(5,1) -- (5.25,0) -- (5,-1)
(5.25,0) -- (6,0)
;
\draw[thick, green]
(0.5,1) -- (0.75,0) -- (0.5,-1)
(0.75,0) -- (1.75,0) -- (2.5,0.5) -- (3.25,0) -- (2.5,-0.5) -- (1.75,0)
(2.5,1) -- (2.5,0.5) (2.5,-1) -- (2.5,-0.5)
(3.25,0) -- (5.25,0) -- (5.5,1)
(5.25,0) -- (5.5,-1)
;
\draw[thick, blue]
(4,1) -- (4,-1)
(4.5,1) -- (4.5,-1)
;
\draw[thick, fill=white] 
(0.75,0) circle (2pt) 
(1.75,0) circle (2pt) 
(2.5,0.5) circle (2pt) 
(2.5,-0.5) circle (2pt) 
(3.25,0) circle (2pt) 
(5.25,0) circle (2pt) 
;
\end{scope}

\begin{scope}[xshift=6cm]
\draw[thick] (0,1) -- (5,1) (0,-1) -- (5,-1);
\draw[dashed] (5,1) -- (5,-1);
\draw[thick, red] 
(0,0) -- (0.75,0) -- (1,1) (0.75,0) -- (1,-1)
(1.5,1) --(1.75,0) -- (1.5,-1)
(1.75,0) to[out=0, in=90] (2.5,-0.5) -- (2,-1)
(2.5,-0.5) --(3,-1)
(2,1) -- (2.5,0.5) -- (3,1)
(2.5,0.5) to[out=-90,in=180] (3.25,0) -- (3.5,1)
(3.25,0) -- (3.5,-1)
(4,1) -- (4.25,0) -- (4,-1)
(4.25,0) -- (5,0)
;
\draw[thick, blue]
(0.5,1) -- (0.75,0) -- (0.5,-1)
(0.75,0) -- (1.75,0) -- (2.5,0.5) -- (3.25,0) -- (2.5,-0.5) -- (1.75,0)
(2.5,1) -- (2.5,0.5) (2.5,-1) -- (2.5,-0.5)
(3.25,0) -- (4.25,0) -- (4.5,1)
(4.25,0) -- (4.5,-1)
;
\draw[thick, fill=white] 
(0.75,0) circle (2pt) 
(1.75,0) circle (2pt) 
(2.5,0.5) circle (2pt) 
(2.5,-0.5) circle (2pt) 
(3.25,0) circle (2pt) 
(4.25,0) circle (2pt) 
;
\end{scope}
\end{tikzpicture}
\]

\[
\coxeterpadding^{-1}(\exdynD_4)=
\begin{tikzpicture}[baseline=-.5ex]
\begin{scope}
\draw[thick] (0,1) -- (6,1) (0,-1) -- (6,-1);
\draw[dashed] (0,1) -- (0,-1);
\draw[thick, red] 
(0,0) -- (0.75,0) -- (1,1) (0.75,0) -- (1,-1)
(1.5,1) --(1.75,0) -- (1.5,-1)
(1.75,0) to[out=0, in=90] (2.5,-0.5) -- (2,-1)
(2.5,-0.5) --(3,-1)
(2,1) -- (2.5,0.5) -- (3,1)
(2.5,0.5) to[out=-90,in=180] (3.25,0) -- (3.5,1)
(3.25,0) -- (3.5,-1)
(5,1) -- (5.25,0) -- (5,-1)
(5.25,0) -- (6,0)
;
\draw[thick, green]
(0.5,1) -- (0.75,0) -- (0.5,-1)
(0.75,0) -- (1.75,0) -- (2.5,0.5) -- (3.25,0) -- (2.5,-0.5) -- (1.75,0)
(2.5,1) -- (2.5,0.5) (2.5,-1) -- (2.5,-0.5)
(3.25,0) -- (5.25,0) -- (5.5,1)
(5.25,0) -- (5.5,-1)
;
\draw[thick, blue]
(4,1) -- (4,-1)
(4.5,1) -- (4.5,-1)
;
\draw[thick, fill=white] 
(0.75,0) circle (2pt) 
(1.75,0) circle (2pt) 
(2.5,0.5) circle (2pt) 
(2.5,-0.5) circle (2pt) 
(3.25,0) circle (2pt) 
(5.25,0) circle (2pt) 
;
\end{scope}

\begin{scope}[xshift=6cm]
\draw[thick] (0,1) -- (5,1) (0,-1) -- (5,-1);
\draw[dashed] (5,1) -- (5,-1);
\draw[thick, red] 
(0,0) -- (0.75,0) -- (1,1) (0.75,0) -- (1,-1)
(1.5,1) --(1.75,0) -- (1.5,-1)
(1.75,0) to[out=0, in=90] (2.5,-0.5) -- (2,-1)
(2.5,-0.5) --(3,-1)
(2,1) -- (2.5,0.5) -- (3,1)
(2.5,0.5) to[out=-90,in=180] (3.25,0) -- (3.5,1)
(3.25,0) -- (3.5,-1)
(4,1) -- (4.25,0) -- (4,-1)
(4.25,0) -- (5,0)
;
\draw[thick, blue]
(0.5,1) -- (0.75,0) -- (0.5,-1)
(0.75,0) -- (1.75,0) -- (2.5,0.5) -- (3.25,0) -- (2.5,-0.5) -- (1.75,0)
(2.5,1) -- (2.5,0.5) (2.5,-1) -- (2.5,-0.5)
(3.25,0) -- (4.25,0) -- (4.5,1)
(4.25,0) -- (4.5,-1)
;
\draw[thick, fill=white] 
(0.75,0) circle (2pt) 
(1.75,0) circle (2pt) 
(2.5,0.5) circle (2pt) 
(2.5,-0.5) circle (2pt) 
(3.25,0) circle (2pt) 
(4.25,0) circle (2pt) 
;
\end{scope}
\end{tikzpicture}
\]
Then it is direct to check that the concatenations $\coxeterpadding(\exdynD_4) \coxeterpadding^{-1}(\exdynD_4)$ and $\coxeterpadding^{-1}(\exdynD_4) \coxeterpadding(\exdynD_4)$ become trivial annulus $N$-graphs after a sequence of Move (I). The same holds for all $n\geq 4$.

\subsection{Legendrian Coxeter mutations and Legendrian loops}\label{sec:legendrian loop}

Let us start by introducing the concept of Legendrian loops. Let $\legendrian \subset (\R ^3, \xi_{\rm st})$ be a Legendrian link and $\cL(\legendrian)$ be the space of Legendrian links isotopic to that Legendrian $\legendrian$. Then a {\em Legendrian loop} $\vartheta$ is a continuous map $\vartheta\colon(\sphere^1,{\rm pt})\to (\cL(\legendrian), \legendrian)$. Note that the graph of the Legendrian loop $\vartheta$ induces a Lagrangian self-concordance of $\legendrian$ inside the symplectization $(\R\times \R^3,d(e^t \alpha_{st}))$, where $\xi_{\rm st}=\ker \alpha_{\rm st}$.

On the other hand, by the observation in Section \ref{subsec:N-graph}, $N$-graphs on annulus can be interpreted as a Lagrangian cobordism from the Legendrian of outer boundary to the one of inner boundary.

The goal of this section is to find Legendrian loops corresponding to $N$-graphs annuli coming from the Legendrian Coxeter mutations in Section \ref{sec:LCM on N-graphs}.

Let us call an $N$-graph on an annulus \emph{tame} if it is obtained by stacking elementary annulus $N$-graphs introduced in Section~\ref{subsec:N-graph}. 

\begin{lemma}
Legendrian Coxeter paddings of type ${\exdynD\exdynE}$ are tame.
\end{lemma}

\begin{proof}
We provide decompositions of the Coxeter paddings $\coxeterpadding^{-1}(\exdynD_4)$ and $\coxeterpadding^{-1}(a,b,c)$ into sequences of elementary annulus $N$-graphs in Figures~\ref{fig:coxeter padding affine D4 is tame} and~\ref{fig:coxeter padding affine E is tame}, respectively. The other cases are similar and we omit it.
\end{proof}

\begin{figure}[ht]
\begin{tikzpicture}
\begin{scope}
\draw[violet, line width=5, opacity=0.5] 
(2.5,8) -- (2.5,7) to[out=-90,in=90] (3.5,6) to[out=-90,in=90] (4,5) to[out=-90,in=90] (4.5,4) to[out=-90,in=90] (5.5,3) to[out=-90,in=150] (6.5,2.5)
(0,2.6) to[out=0,in=150] (0.5,2.5)
(0.5,2.5) to[out=-30,in=90] (1,2) to[out=-90,in=90] (2,1) to (2,0)
;
\draw[dashed] (0,8) to (0,0) ;
\draw[thick, green]
(0.5,8) -- (0.5,2.5) to[out=-150, in=0] (0,2.4)
(0.5,2.5) to[out=-30, in=90] (1,2) to[out=-90, in=150] (1.5,1.5) -- (1.5,1) to[out=-90, in=150] (2,0.5) -- (2,0)
(2.5,8) -- (2.5,7.5) to[out=-150,in=90] (2,7) -- (2,2) to[out=-90,in=30] (1.5,1.5)
(2.5,7.5) to[out=-30,in=90] (3,7) to (3,6.5) to[out=-150,in=90] (2.5,6) to (2.5,1) to[out=-90,in=30] (2,0.5)
(3,6.5) to[out=-30,in=90] (3.5,6) to[out=-90,in=90] (4,5) to[out=-90,in=90] (4.5,4)
to[out=-90,in=150] (5,3.5) to (5,0)
(5.5,8) to (5.5,4) to[out=-90,in=30] (5,3.5)
;
\draw[thick,red]
(0,2.6) to[out=0,in=150] (0.5,2.5)
(1,8) to (1,3) to[out=-90,in=30] (0.5,2.5) to (0.5,0)
(1.5,8) to (1.5,1.5) to[out=-150,in=90] (1,1) to (1,0)
(1.5,1.5) to[out=-30,in=90] (2,1) to (2,0.5) to[out=-150,in=90] (1.5,0)
(2,0.5) to[out=-30,in=90] (2.5,0)
(2,8) to[out=-90,in=150] (2.5,7.5) to (2.5,7) to[out=-90,in=150] (3,6.5) to (3,0)
(3,8) to[out=-90,in=30] (2.5,7.5)
(3.5,8) to (3.5,7) to[out=-90,in=30] (3, 6.5)
(5,8) to (5,3.5) to[out=-150,in=90] (4.5,3) to (4.5,0)
(5,3.5) to[out=-30,in=90] (5.5,3) to[out=-90,in=150] (6.5,2.5)
;
\draw[thick, blue]
(4,8) to (4,6) to[out=-90,in=90] (3.5,5) to (3.5,0)
(4.5,8) to (4.5,5) to[out=-90,in=90] (4,4) to (4,0)
(6.5,2.5) to[out=-150,in=90] (5.5,2) to (5.5,0)
;
\foreach \x in {0,1,...,8}
{
\draw[thick]
(0,\x) -- (6,\x);
}
\draw[thick, fill=white] 
(0.5,2.5) circle (2pt) 
(1.5,1.5) circle (2pt) 
(2,0.5) circle (2pt) 
(2.5,7.5) circle (2pt) 
(3,6.5) circle (2pt) 
(5,3.5) circle (2pt) 
;
\end{scope}

\begin{scope}[xshift=6cm]

\draw[violet, line width=5, opacity=0.5] 
(2.5,8) -- (2.5,7) to[out=-90,in=90] (3.5,6) to (3.5,4) to[out=-90,in=90] (4.5,3) to[out=-90,in=180] (5,2.6)
(0.5,2.5) to[out=-30,in=90] (1,2) to[out=-90,in=90] (2,1) to (2,0)
;
\draw[dashed] (5,8) to (5,0) ;
\draw[thick, blue]
(0.5,8) -- (0.5,2.5)
(0.5,2.5) to[out=-30, in=90] (1,2) to[out=-90, in=150] (1.5,1.5) -- (1.5,1) to[out=-90, in=150] (2,0.5) -- (2,0)
(2.5,8) -- (2.5,7.5) to[out=-150,in=90] (2,7) -- (2,2) to[out=-90,in=30] (1.5,1.5)
(2.5,7.5) to[out=-30,in=90] (3,7) to (3,6.5) to[out=-150,in=90] (2.5,6) to (2.5,1) to[out=-90,in=30] (2,0.5)
(3,6.5) to[out=-30,in=90] (3.5,6) to (3.5,4) to[out=-90,in=150] (4,3.5) to (4,0)
(4.5,8) to (4.5,4) to[out=-90,in=30] (4,3.5)
;
\draw[thick,red]
(1,8) to (1,3) to[out=-90,in=30] (0.5,2.5) to (0.5,0)
(1.5,8) to (1.5,1.5) to[out=-150,in=90] (1,1) to (1,0)
(1.5,1.5) to[out=-30,in=90] (2,1) to (2,0.5) to[out=-150,in=90] (1.5,0)
(2,0.5) to[out=-30,in=90] (2.5,0)
(2,8) to[out=-90,in=150] (2.5,7.5) to (2.5,7) to[out=-90,in=150] (3,6.5) to (3,0)
(3,8) to[out=-90,in=30] (2.5,7.5)
(3.5,8) to (3.5,7) to[out=-90,in=30] (3, 6.5)
(4,8) to (4,3.5) to[out=-150,in=90] (3.5,3) to (3.5,0)
(4,3.5) to[out=-30,in=90] (4.5,3) to[out=-90,in=180] (5,2.6)
;
\draw[thick,green]
(5,2.4) to[out=180,in=90] (4.5,2) to (4.5,0)
;
\foreach \x in {0,1,...,8}
{
\draw[thick]
(0,\x) -- (5,\x);
}
\draw[thick, fill=white] 
(0.5,2.5) circle (2pt) 
(1.5,1.5) circle (2pt) 
(2,0.5) circle (2pt) 
(2.5,7.5) circle (2pt) 
(3,6.5) circle (2pt) 
(4,3.5) circle (2pt) 
;
\end{scope}
\end{tikzpicture}
\caption{A sequence of elementary annulus $N$-graphs which gives $\coxeterpadding^{-1}({\exdynD}_4)$.}
\label{fig:coxeter padding affine D4 is tame} 
\end{figure}

\begin{figure}[ht]
\begin{tikzpicture}
\begin{scope}
\draw[thick] 
(0,0)--(3.5,0)
(0,1)--(3.5,1)
(0,2)--(3.5,2)
(0,3)--(3.5,3)
(0,4)--(3.5,4)
;
\draw[thick,red]
(0.5,0) -- (0.5,3) to[out=90,in=-150] (1,3.5) --(1,4)
(1,0) -- (1,2) to[out=90,in=-150] (1.5,2.5) to (1.5,3) to[out=90,in=-30] (1,3.5)
(2,0) to[out=90,in=-150] (2.5,0.5) to (2.5,1) to[out=90,in=-30] (2,1.5)
(3,0) to[out=90,in=-30] (2.5,0.5)
;
\draw[thick,red,dashed]
(1.5,0) to (1.5,1) to[out=90,in=-150] (2,1.5) to (2,2) to[out=90,in=-30] (1.5,2.5)
;
\begin{scope}[rotate around={180:(1.75,2)}]
\draw[thick,blue]
(0.5,0) -- (0.5,3) to[out=90,in=-150] (1,3.5) --(1,4)
(1,0) -- (1,2) to[out=90,in=-150] (1.5,2.5) to (1.5,3) to[out=90,in=-30] (1,3.5)
(2,0) to[out=90,in=-150] (2.5,0.5) to (2.5,1) to[out=90,in=-30] (2,1.5)
(3,0) to[out=90,in=-30] (2.5,0.5)
;
\draw[thick,blue,dashed]
(1.5,0) to (1.5,1) to[out=90,in=-150] (2,1.5) to (2,2) to[out=90,in=-30] (1.5,2.5)
;
\end{scope}
\draw[thick, fill=white]
(1,3.5) circle (2pt)
(1.5,2.5) circle (2pt)
(2,1.5) circle (2pt)
(2.5,0.5) circle (2pt)
;
\end{scope}

\begin{scope}[xshift=3cm]
\draw[thick] 
(0,0)--(3.5,0)
(0,1)--(3.5,1)
(0,2)--(3.5,2)
(0,3)--(3.5,3)
(0,4)--(3.5,4)
;
\draw[thick,red]
(0.5,0) -- (0.5,3) to[out=90,in=-150] (1,3.5) --(1,4)
(1,0) -- (1,2) to[out=90,in=-150] (1.5,2.5) to (1.5,3) to[out=90,in=-30] (1,3.5)
(2,0) to[out=90,in=-150] (2.5,0.5) to (2.5,1) to[out=90,in=-30] (2,1.5)
(3,0) to[out=90,in=-30] (2.5,0.5)
;
\draw[thick,red,dashed]
(1.5,0) to (1.5,1) to[out=90,in=-150] (2,1.5) to (2,2) to[out=90,in=-30] (1.5,2.5)
;
\begin{scope}[rotate around={180:(1.75,2)}]
\draw[thick,blue]
(0.5,0) -- (0.5,3) to[out=90,in=-150] (1,3.5) --(1,4)
(1,0) -- (1,2) to[out=90,in=-150] (1.5,2.5) to (1.5,3) to[out=90,in=-30] (1,3.5)
(2,0) to[out=90,in=-150] (2.5,0.5) to (2.5,1) to[out=90,in=-30] (2,1.5)
(3,0) to[out=90,in=-30] (2.5,0.5)
;
\draw[thick,blue,dashed]
(1.5,0) to (1.5,1) to[out=90,in=-150] (2,1.5) to (2,2) to[out=90,in=-30] (1.5,2.5)
;
\end{scope}
\draw[thick, fill=white]
(1,3.5) circle (2pt)
(1.5,2.5) circle (2pt)
(2,1.5) circle (2pt)
(2.5,0.5) circle (2pt)
;
\end{scope}

\begin{scope}[xshift=6cm]
\draw[thick] 
(0,0)--(3.5,0)
(0,1)--(3.5,1)
(0,2)--(3.5,2)
(0,3)--(3.5,3)
(0,4)--(3.5,4)
;
\draw[thick,red]
(0.5,0) -- (0.5,3) to[out=90,in=-150] (1,3.5) --(1,4)
(1,0) -- (1,2) to[out=90,in=-150] (1.5,2.5) to (1.5,3) to[out=90,in=-30] (1,3.5)
(2,0) to[out=90,in=-150] (2.5,0.5) to (2.5,1) to[out=90,in=-30] (2,1.5)
(3,0) to[out=90,in=-30] (2.5,0.5)
;
\draw[thick,red,dashed]
(1.5,0) to (1.5,1) to[out=90,in=-150] (2,1.5) to (2,2) to[out=90,in=-30] (1.5,2.5)
;
\begin{scope}[rotate around={180:(1.75,2)}]
\draw[thick,blue]
(0.5,0) -- (0.5,3) to[out=90,in=-150] (1,3.5) --(1,4)
(1,0) -- (1,2) to[out=90,in=-150] (1.5,2.5) to (1.5,3) to[out=90,in=-30] (1,3.5)
(2,0) to[out=90,in=-150] (2.5,0.5) to (2.5,1) to[out=90,in=-30] (2,1.5)
(3,0) to[out=90,in=-30] (2.5,0.5)
;
\draw[thick,blue,dashed]
(1.5,0) to (1.5,1) to[out=90,in=-150] (2,1.5) to (2,2) to[out=90,in=-30] (1.5,2.5)
;
\end{scope}
\draw[thick, fill=white]
(1,3.5) circle (2pt)
(1.5,2.5) circle (2pt)
(2,1.5) circle (2pt)
(2.5,0.5) circle (2pt)
;
\end{scope}
\draw[dashed] (0,0)--(0,4) (9.5,0)--(9.5,4);
\draw [decorate,decoration={brace,amplitude=10pt}]
(0.3,4.1) -- (3.2,4.1) node [black,midway,yshift=0.5cm] {$a+2$} ;
\draw [decorate,decoration={brace,amplitude=10pt}]
(3.3,4.1) -- (6.2,4.1) node [black,midway,yshift=0.5cm] {$b+2$} ;
\draw [decorate,decoration={brace,amplitude=10pt}]
(6.3,4.1) -- (9.2,4.1) node [black,midway,yshift=0.5cm] {$c+2$} ;
\end{tikzpicture}
\caption{A sequence of elementary annulus $N$-graphs which gives $\coxeterpadding^{-1}(a,b,c)$.}\label{fig:coxeter padding affine E is tame}
\end{figure}

In order to see the effect of Legendrian Coxeter mutation efficiently, let us present it by a sequence of braid moves together with keep tracking braid words shaded by violet color in Figure~\ref{fig:coxeter padding affine D4 is tame}.

\[
\begin{tikzcd}[column sep = -7pt, row sep = -5pt]
\widehat \beta({\exdynD}_4) 
&= &\1 & \2 & \2 &\color{red}{\2} 
&{\color{violet}\circled{\text{$\1$}}} &\color{red}{\2} 
&\2 &\2 &\1 & \3& \2& \1& \1& \2 &\color{red}{\2}
&{\color{violet}\circled{\text{$\3$}}}  &\color{red}{\2} 
&\2&\2&\3\\
&= &\1 &\2 &\color{red}{\2} & \color{red}{\1}& {\color{violet}\circled{\text{$\2$}}} 
&\1 &\2 &\2 &\1 &\3 &\2 &\1 &\1 &\color{red}{\2} &\color{red}{\3}  &{\color{violet}\circled{\text{$\2$}}} 
& \3 &\2&\2&\3\\
&= &\1 &\2 &{\color{violet}\circled{\text{$\1$}}} &\2 &\1 &\1 &\2 &\2 &\1 & \3 &\2
& {\color{blue}\1}&  \color{blue}{\1} &{\color{violet}\circled{\text{$\3$}}} &\2 
&\3 & \3 &\2&\2&\3\\
&= &{\color{red}\1} &\color{red}{\2} &{\color{violet}\circled{\text{$\1$}}}&\2 &\1 &\1
&\2& \2 &\1 &{\color{red} \3} &\color{red}{\2} &{\color{violet} \circled{\text{$\3$}}} &\1 &\1 & \2 
&\3 & \3 &\2&\2&\3\\
&= &{\color{violet} \circled{\text{$\2$}}} &\1 &\2 &\2 &\1 &\1 &\2 &\2 &\1
&  {\color{violet}\circled{\text{$\2$}}} &\3& \2 &\1 &\1 & \2& \3 & \3 &\2 
&{\color{red}\2} &\color{red}{\3}\\
&\doteq  &\1 &\2 &\2 &\1 &{\color{red}\1} &\color{red}{\2} 
&{\color{violet} \circled{\text{$\1$}}} &\2 & \1 &\3 &\2& \1& \1 & \2 &\3 
& {\color{red}\3} &\color{red}{ \2} & {\color{violet}\circled{\text{$\3$}}} & \2 &\3\\
&= &\1 &\2 &\2 &{\color{red}\1} &{\color{violet}\circled{\text{$\2$}}} &\color{red}{\1} &\2 
&\2 & \1 &\3 &\2 &\1& \1&  \2 &{\color{red}\3} 
& {\color{violet} \circled{\text{$\2$}}} &\color{red}{\3} &\2 &\2 & \3\\
&= &\1 &\2 &\2 &\2& {\color{violet}\circled{\text{$\1$}}}& \2 &\2 &\2 & \1& \3 &\2 &\1 &\1 
& \2 &\2 & {\color{violet} \circled{\text{$\3$}}} &\2 &\2 &\2& \3 \\
\end{tikzcd}
\]

In general for the Legendrian Coxeter padding $\coxeterpadding^{-1}({\exdynD}_n)$, we have the following sequence of Reidemeister moves:

\begin{align*}
\widehat \beta({\exdynD}_n)
&= \2^k \1 \2 \2 {\color{red}\2 \1 \2} \2 \2 \1 \2^\ell \3 \2 \1 \1 \2 {\color{red}\2 \3 \2} \2 \2 \3\\
&= \2^k \1 \2 {\color{red}\2 \1 \2} \1 \2 \2 \1 \2^\ell \3 \2 \1 \1 {\color{red}\2 \3 \2} \3 \2 \2 \3\\
&= \2^k {\color{red}\1 \2 \1} \2 \1 \1 \2 \2 \1 \2^\ell \3 \2 {\color{blue}\1 \1 \3} \2 \3 \3 \2 \2 \3\\
&= \2^k \2 \1 \2 \2 \1 \1 \2 \2 \1 \2^\ell {\color{red}\3 \2 \3} \1 \1 \2 \3 \3 \2 \2 \3\\
&= \2^{k+1} \1 \2 \2 \1 \1 \2 \2 \1 \2^{\ell+1} \3 \2 \1 \1 \2 \3 \3 \2 \2 \3\\
&\doteq \2^{k} \1 \2 \2 \1 \1 \2 {\color{red}\2 \1 \2} \2^{\ell} \3 \2 \1 \1 \2 \3 \3 \2 {\color{red}\2 \3 \2}\\
&= \2^{k} \1 \2 \2 \1 {\color{red}\1 \2 \1} \2 \1 \2^{\ell} \3 \2 \1 \1 \2 \3 {\color{red}\3 \2 \3} \2 \3\\
&= \2^{k} \1 \2 \2 {\color{red}\1 \2 \1} \2 \2 \1 \2^{\ell} \3 \2 \1 \1 \2 {\color{red}\3 \2 \3} \2 \2 \3\\
&= \2^{k} \1 \2 \2 \2 \1 \2 \2 \2 \1 \2^{\ell} \3 \2 \1 \1 \2 \2 \3 \2 \2 \2 \3,
\end{align*}
where $k=\lfloor \frac{n-3}2\rfloor$ and $\ell=\lfloor \frac{n-4}2\rfloor$.

Now we can translate the above sequence of moves onto
\begin{align*}
\widehat \beta({{\exdynD}_n})=\Delta_4 \3 \2 \2 \3 \2^{n-4} \1 \2 \2 \1 \Delta_4
\end{align*}
by conjugating the (cyclic) braid equivalence in~\eqref{eqn:front braid to N-graph braid}.
Then the effect of Coxeter padding $\coxeterpadding^{-1}({\exdynD}_n)$ onto $\widehat \beta({{\exdynD}_n})$ can be presented as a Legendrian loop $\vartheta(\exdynD)$ in Figure~\ref{fig:legendrian loop of D_intro} in the introduction.

Now move onto the case of the Coxeter padding of type $\exdynE$. Note that 
\begin{align*}
\widehat \beta(a,b,c) &= \2 \1^{a+1} \2 \1^{b+1} \2 \1^{c+1}\\
&\doteq \1^{a-1} \Delta \1^{b-1} \Delta \1^{c-1} \Delta
\end{align*}
and we translate the sequence of Reidemeister moves induced by $\bar\coxeterpadding^{-1}(a,b,c) \coxeterpadding^{-1}(a,b,c)$ into the Legendrian loop $\vartheta(\exdynE)$ depicted as in Figure~\ref{fig:legendrian loop of E_intro} in the introduction.
Note that the left column of the loop diagram corresponds to $\coxeterpadding^{-1}(a,b,c)$ while the right column corresponds to $\bar\coxeterpadding^{-1}(a,b,c)$. 

Note that the Legendrian loops induce annulus $N$-graphs, and their action on the space of $N$-graph by piling up the annulus has infinite order, see Remark~\ref{rmk_inifinitely_many_seeds_for_affine}.
In conclusion, we have

\begin{theorem}\label{thm:legendrian loop}
The Legendrian Coxeter mutation $\mutation_\ngraph^{\pm1}$ on 
$(\ngraph(\exdynD),\nbasis(\exdynD))$ and twice of Legendrian mutation 
$\mutation_\ngraph^{\pm 2}$ on $(\ngraph(\exdynE),\nbasis(\exdynE))$ induce 
Legendrian loops $\vartheta(\exdynD)$ and $\vartheta(\exdynE)$ in Figure~\ref{fig:legendrian loops}, respectively. In particular, the 
order of the Legendrian loops as elements in $\pi_1(\cL(\legendrian(\exdynX)),\legendrian(\exdynX))$ are infinite.
\end{theorem}

\subsection{Lagrangian fillings for Legendrian links of \texorpdfstring{type $\exdynD \exdynE$}{affine Dn and En type}}

We will prove one of our main theorem on `as many exact embedded Lagrangian fillings as seeds' (Theorem~\ref{theorem:seed many}) as follows:
\begin{theorem}[As many exact embedded Lagrangian fillings as seeds]\label{thm:seed many fillings}
There are at least as many distinct exact embedded Lagrangian fillings as seeds for Legendrian links of type $\exdynD\exdynE$.
\end{theorem}

The key ingredient of the above theorem is the following proposition.
\begin{proposition}\label{prop:every seeds come from N-graphs}
Let $\legendrian$ be a Legendrian knot or link of type $\exdynD\exdynE$. Let 
$\flags$ be flags on $\legendrian$ as a formal parameter for the moduli space 
$\cM_1(\legendrian)$. Suppose that $\seed$ is a seed in the corresponding seed 
pattern with the initial seed from Tables~\ref{table:affine D type} and 
\ref{table:affine E type}: for $\dynX=\exdynD_n, \exdynE_6, \exdynE_7$ or $\exdynE_8$,
\[
\seed_{t_0}=\Psi(\ngraph(\dynX),\nbasis(\dynX),\cF_{\legendrian(\dynX)}).
\]
Then there exists a pair $(\ngraph,\nbasis)$ such that $\seed=\Psi(\ngraph, \nbasis,\flags)$.
\end{proposition}
\begin{proof}
Note that our cases are of acyclic affine type, so we can apply Lemma~\ref{lemma:normal form} which says the following: For any seed $\seed_t$ in the cluster pattern, there exist $r\in\Z$ and $\ell\in [n]$ such that
\[
\seed_t = (\mutation_{j_L} \cdots \mutation_{j_1})(\qcoxeter^r(\initialseed)).
\]
Recall that all the mutation sequences at $j_1,\dots, j_L\in [n]\setminus\{\ell\}$ followed by the Coxeter mutations $\qcoxeter^r$ lie in the induced subgraph $\exchangesub{\qbpr_{t_0}}{x_{\ell;r}} \cong \exchange(\qbasis^{(\ell)})$.

Since we are interested in the realization of the seeds as pairs of $N$-graphs and good tuples of cycles, it suffices to check that there is no obstruction to realize each mutation. We already have shown in Corollary~\ref{cor:coxeter realization E-type} and Corollary \ref{cor:coxeter realization D-type} that the Coxeter mutations $\mutation_\quiver^r(\initialseed)$ for any $r\in\Z$ are realizable by pairs of $N$-graphs and good tuples of cycles. It remains to argue that the remaining sequence of mutations at $j_1,\dots, j_L\in [n]\setminus\{\ell\}$ can be realized by $N$-graphs and good tuples of cycles.

Now focus on the root system $\Phi([n]\setminus \{\ell\})$, and the corresponding pair $(\ngraph_{t_0},\nbasis_{t_0}\setminus\{\gamma_\ell\})$ of an $N$-graph and a (proper) sub-tuple of cycles. Simple but important observation is that the exchange graph of cluster pattern with respect to the pair $(\ngraph_{t_0},\nbasis_{t_0}\setminus\{\gamma_\ell\} )$ is isomorphic to $\exchange(\qbasis|_{[n]\setminus\{\ell\}})$.

Note that $\Phi([n]\setminus \{\ell\})$ may not be irreducible and is of the form $\Phi^{(1)}\times \cdots \times \Phi^{(k)}$, for some $k\in\N$. Let $\quiver_{t_0}$ be a quiver of rank $n$ corresponding to $\Phi([n])$, then an induced subquiver $\quiver_{t_0} \setminus \{\ell\}$ has $k$-connected components $\quiver^{(1)},\dots, \quiver^{(k)}$ which correspond to $\Phi^{(1)},\dots, \Phi^{(k)}$, respectively. Now consider the pair $(\ngraph_{t_0},\nbasis_{t_0})$ of an $N$-graph and an $n$-tuple of cycles realizing $\quiver_{t_0}$. Then ignoring the $\ell$-th cycle $\gamma_\ell$ in the pair $(\ngraph_{t_0},\nbasis_{t_0})$ produces $k$ pairs 
\[
(\ngraph^{(1)},\nbasis^{(1)}), \dots, (\ngraph^{(k)},\nbasis^{(k)})
\] 
which realize $\quiver^{(1)},\dots, \quiver^{(k)}$, respectively. In order to show the $N$-graph realizability of each seed, it suffices to check that there is no obstruction to perform mutations at each induced pair $(\ngraph^{(i)},\nbasis^{(i)})$, $i=1,2,\dots,k$.

Let us analyze possible induced pairs from $(\ngraph(\exdynE_n),\nbasis(\exdynE_n))$, $n=6,7,8$ as follows. Note that all are of type $(\ngraph(a,b,c),\nbasis(a,b,c))$ with $\frac{1}{a}+\frac{1}{b}+\frac{1}{c}=1$, see \cite[Figure 31]{ABL2021}:
\begin{enumerate}
\item If $\ell=1$, i.e. $\cycle_\ell$ corresponds to the central vertex, then we have the following three pairs:
\[
\{(\ngraph_{(3)}(\dynA_{a-1}), \nbasis_{(3)}(\dynA_{a-1})), (\ngraph_{(3)}(\dynA_{b-1}), \nbasis_{(3)}(\dynA_{b-1})),(\ngraph_{(3)}(\dynA_{c-1}), \nbasis_{(3)}(\dynA_{c-1}))\}.
\]
Here $\dynA_\bullet$ denotes the Dynkin diagram of type $\dynA$, and the subindex $(3)$ indicates that the induced pairs are $3$-graphs together with cycles even though they are monochromatic.
\item If $\cycle_\ell$ corresponds to a bivalent vertex, then for some $1\le r,s$ with $r+s+1=a$, up to permuting indices $a,b,c$, we have two following pairs:
\[
\{(\ngraph_{(3)}(\dynA_s),\nbasis_{(3)}(\dynA_s)),(\ngraph_{(3)}(r,b,c),\nbasis_{(3)}(r,b,c))\},
\]
\item If $\cycle_\ell$ corresponds to a leaf, then up to permuting indices $a,b,c$, we have the following pair:
\[
\{(\ngraph(a-1,b,c), \nbasis(a-1,b,c))\}.
\]
\end{enumerate}

Now list the possible induced pairs from $(\ngraph(\exdynD_{n}),\nbasis(\exdynD_{n}))$.
Note that the blue arc in the upper left side of $\ngraph(\exdynD_{n})$ does nothing to do with the Legendrian mutation at $\nbasis(\exdynD_{n})$. For simplicity, we ignore that blue arc when we consider the induced pairs.
\begin{enumerate}
\setcounter{enumi}{3}
\item If $n=4$ and $\ell=1$, i.e. corresponds to the central vertex, then we have four pairs: All the pairs are $(\ngraph_{(4)}(\dynA_1),\nbasis_{(4)}(\dynA_1))$, see Figure~\ref{fig:case(4)}.

\item If $n\ge 5$ and either $\ell=1$ or $n-1$, then we have the following three pairs, see Figure~\ref{fig:case(5)}:
\[
\{(\ngraph_{(4)}(\dynA_1),\nbasis_{(4)}(\dynA_1)), (\ngraph_{(4)}(\dynA_1),\nbasis_{(4)}(\dynA_1)), 
(\ngraph_{(4)}(n-4,2,2),\nbasis_{(4)}(n-4,2,2))\}.
\]

\item If $\ell=2,3,n$ or $n+1$, then we have a pair
\[
(\ngraph'_{(4)}(\dynD_n),\nbasis'_{(4)}(\dynD_n)).
\]
See, for example, Figure~\ref{fig:case(6)}.

\item Otherwise, we have two pairs
\[
\{ (\ngraph_{(4)}(n_1,2,2),\nbasis_{(4)}(n_1,2,2)), (\ngraph_{(4)}(n_2,2,2),\nbasis_{(4)}(n_2,2,2)) \}
\]
satisfying $n_1+n_2=n-3$, see Figure~\ref{fig:case(7)}.
\end{enumerate}

\begin{figure}[ht]
\subfigure[$n=4$ and $\ell=1$]{\makebox[0.48\textwidth]{\label{fig:case(4)}
$
\begin{tikzpicture}[baseline=-.5ex,scale=0.6]
\useasboundingbox (-2.5,-3.5) rectangle (2.5,3.5);
\begin{scope}[xshift=-1cm]
\draw[violet, thick, rounded corners, fill=violet, opacity=0.2]
(-2.5,1.5) -- (-2.5,0.5) -- (-0.5,0.5) -- (-0.5,1.5) -- cycle
(-1.5,-0.5) -- (-1.5,-2.5) -- (-0.5,-2.5) -- (-0.5,-0.5) -- cycle;
\draw[line width=7, dashed] (-1,1)--(0,0) (-1,-1)--(0,0) (2,0)--(0,0);
\draw[gray,   line width=5] (-1,1)--(0,0) (-1,-1)--(0,0) (1,0)--(0,0);
\draw[green, line cap=round, line cap=round, line width=5, opacity=0.5] (-2,1)-- (-1,1) (-1,-1)--(-1,-2);
\draw[blue,thick, rounded corners] (-1.75, 3) -- (-1.5, 2.5) -- (-1.25, 3);
\draw[thick] (1, 3) -- (-2,3) to[out=180,in=90] (-3,2) --(-3,-2) to[out=-90,in=180] (-2, -3) -- (4,-3) to[out=0,in=-90] (5,-2) -- (5,2) to[out=90,in=0] (4,3)-- (1,3);
\draw[green, thick] (0,3) -- (0,-3) (0,0) -- (-3,0);
\draw[red, thick, fill] (0,0) -- (-1,1) circle (2pt) -- +(0,2) (-1,1) -- ++(-1,0) circle (2pt) -- +(-1,0) (-2,1) -- +(0,2);
\draw[red,thick, fill] (0,0) -- (-1,-1) circle (2pt) -- +(-2,0) (-1,-1) -- ++(0,-1) circle (2pt) -- +(0,-1) (-1,-2) -- +(-2,0);
\draw[red,thick] (0,0) -- (1,0);
\draw[thick, fill=white] (0,0) circle (2pt);
\end{scope}
\begin{scope}[xshift=1cm,rotate=180]
\draw[violet, thick, rounded corners, fill=violet, opacity=0.2]
(-2.5,1.5) -- (-2.5,0.5) -- (-0.5,0.5) -- (-0.5,1.5) -- cycle
(-1.5,-0.5) -- (-1.5,-2.5) -- (-0.5,-2.5) -- (-0.5,-0.5) -- cycle;
\draw[line width=7, dashed] (-1,1)--(0,0) (-1,-1)--(0,0);
\draw[gray,  line width=5] (-1,1)--(0,0) (-1,-1)--(0,0) (1,0)--(0,0);
\draw[green, line cap=round, line cap=round, line width=5, opacity=0.5] (-2,1)-- (-1,1) (-1,-1)--(-1,-2);
\draw[blue, thick] (0,3) -- (0,-3) (0,0) -- (-3,0);
\draw[red ,thick, fill] (0,0) -- (-1,1) circle (2pt) -- +(0,2) (-1,1) -- ++(-1,0) circle (2pt) -- +(-1,0) (-2,1) -- +(0,2);
\draw[red,thick, fill] (0,0) -- (-1,-1) circle (2pt) -- +(-2,0) (-1,-1) -- ++(0,-1) circle (2pt) -- +(0,-1) (-1,-2) -- +(-2,0);
\draw[red,thick] (0,0) -- (1,0);
\draw[thick, fill=white] (0,0) circle (2pt);
\end{scope}
\end{tikzpicture}
$
}}
\subfigure[$n\ge 5$ and $\ell=1$]{\makebox[0.48\textwidth]{\label{fig:case(5)}
$
\begin{tikzpicture}[baseline=-.5ex,scale=0.6]
\useasboundingbox (-2.5,-3.5) rectangle (2.5,3.5);
\begin{scope}[xshift=1cm, yscale=-1,rotate=180]
\draw[violet, thick, rounded corners, fill=violet, opacity=0.2]
(-2.5,1.5) -- (-2.5,-2.5) -- (1.5,-2.5) -- (1.5,1.5) -- cycle;
\draw[yellow, line cap=round, line width=5, opacity=0.5] (-1,-1) -- (-1,-2) (-1,1) -- (-2,1);
\draw[green, line cap=round, line width=5,opacity=0.5] (1,0)--(0,0)--(-1,-1) (0,0)--(-1,1);
\draw[blue, thick] (0,3) -- (0,-3) (0,0) -- (-3,0);
\draw[red,thick, fill] (0,0) -- (-1,1) circle (2pt) -- +(0,2) (-1,1) -- ++(-1,0) circle (2pt) -- +(-1,0) (-2,1) -- +(0,2);
\draw[red,thick, fill] (0,0) -- (-1,-1) circle (2pt) -- +(-2,0) (-1,-1) -- ++(0,-1) circle (2pt) -- +(0,-1) (-1,-2) -- +(-2,0);
\draw[red,thick] (0,0) -- (1,0);
\draw[thick, fill=white] (0,0) circle (2pt);
\end{scope}
\begin{scope}[xshift=-1cm]
\draw[violet, thick, rounded corners, fill=violet, opacity=0.2]
(-2.5,1.5) -- (-2.5,0.5) -- (-0.5,0.5) -- (-0.5,1.5) -- cycle
(-1.5,-0.5) -- (-1.5,-2.5) -- (-0.5,-2.5) -- (-0.5,-0.5) -- cycle;
\draw[line width=7, dashed] (-1,1)--(0,0) (-1,-1)--(0,0) (1,0)--(0,0);
\draw[gray, line width=5] (-1,1)--(0,0) (-1,-1)--(0,0) (1,0)--(0,0);
\draw[green, line cap=round, line width=5,opacity=0.5] (-2,1)--(-1,1) (-1,-1)--(-1,-2);
\draw[blue,thick, rounded corners] (-1.75, 3) -- (-1.5, 2.5) -- (-1.25, 3);
\draw[thick] (1, 3) -- (-2,3) to[out=180,in=90] (-3,2) --(-3,-2) to[out=-90,in=180] (-2, -3) -- (4,-3) to[out=0,in=-90] (5,-2) -- (5,2) to[out=90,in=0] (4,3)-- (1,3);
\draw[green, thick] (0,3) -- (0,-3) (0,0) -- (-3,0);
\draw[red,thick, fill] (0,0) -- (-1,1) circle (2pt) -- +(0,2) (-1,1) -- ++(-1,0) circle (2pt) -- +(-1,0) (-2,1) -- +(0,2);
\draw[red,thick, fill] (0,0) -- (-1,-1) circle (2pt) -- +(-2,0) (-1,-1) -- ++(0,-1) circle (2pt) -- +(0,-1) (-1,-2) -- +(-2,0);
\draw[red,thick] (0,0) -- (1,0);
\draw[thick, fill=white] (0,0) circle (2pt);
\end{scope}
\draw[red, thick, fill] (0,0) circle (2pt) -- (0,-3);
\end{tikzpicture}
$}}
\\
\subfigure[$\ell=n$]{\makebox[\textwidth]{\label{fig:case(6)}
$
\begin{tikzpicture}[baseline=-.5ex,scale=0.6]
\useasboundingbox (-7.5,-3.5) rectangle (7.5,3.5);
\draw[violet, thick, rounded corners, fill=violet, opacity=0.2]
(-6,1.5) -- (-6,-2.5) -- (6,-2.5) -- (6,1.5) -- cycle;
\begin{scope}[xshift=-3.5cm]
\draw[yellow, line cap=round, line width=5, opacity=0.5] (-1,1)--(0,0) (-1,-1)--(0,0) (1,0)--(0,0);
\draw[green, line cap=round, line width=5, opacity=0.5] (-1,1)--(-2,1) (-1,-1)--(-1,-2) (1,0)--(2,0)
(3,0)--(3.5,0) (4.5,0)--(5,0) (6,0)--(7,0) (7,0)--(8,1) (7,0)--(8,-1);
\draw[blue,thick, rounded corners] (-1.75, 3) -- (-1.5, 2.5) -- (-1.25, 3);
\draw[green, thick] (0,3) -- (0,-3) (0,0) -- (-3,0);
\draw[red,thick, fill] (0,0) -- (-1,1) circle (2pt) -- +(0,2) (-1,1) -- ++(-1,0) circle (2pt) -- +(-1,0) (-2,1) -- +(0,2);
\draw[red,thick, fill] (0,0) -- (-1,-1) circle (2pt) -- +(-2,0) (-1,-1) -- ++(0,-1) circle (2pt) -- +(0,-1) (-1,-2) -- +(-2,0);
\draw[red,thick] (0,0) -- (1,0);
\draw[thick, fill=white] (0,0) circle (2pt);
\draw[thick](1, 3) -- (-2,3) to[out=180,in=90] (-3,2) --(-3,-2) to[out=-90,in=180] (-2, -3) -- (9,-3) to[out=0,in=-90] (10,-2) -- (10,2) to[out=90,in=0] (9,3)-- (1,3);
\end{scope}
\begin{scope}[xshift=3.5cm,rotate=180]
\draw[line width=7, dashed] (-1,-1) -- (-1,-2);
\draw[gray, line width=5] (-1,-1) -- (-1,-2) ;
\draw[yellow, line cap=round, line width=5, opacity=0.5] (-1,1)--(-2,1);
\draw[blue, thick] (0,3) -- (0,-3) (0,0) -- (-3,0);
\draw[red,thick, fill] (0,0) -- (-1,1) circle (2pt) -- +(0,2) (-1,1) -- ++(-1,0) circle (2pt) -- +(-1,0) (-2,1) -- +(0,2);
\draw[red,thick, fill] (0,0) -- (-1,-1) circle (2pt) -- +(-2,0) (-1,-1) -- ++(0,-1) circle (2pt) -- +(0,-1) (-1,-2) -- +(-2,0);
\draw[red,thick] (0,0) -- (1,0);
\draw[thick, fill=white] (0,0) circle (2pt);
\end{scope}
\begin{scope}[xshift=0.5cm]
\draw[yellow, line cap=round, line width=5, opacity=0.5] (-2,0) -- (-1,0) (1,0) -- (2,0);
\draw[red, thick, fill] (-3,0) circle (2pt) -- (-3, -3) (-3,0) -- (-2,0) circle (2pt) -- (-2,3) (-2,0) -- (-1,0) circle (2pt) -- (-1,-3) (-1,0) -- (-0.5,0) (0.5,0) -- (1,0) circle (2pt) -- (1,3) (1,0) -- (2, 0) circle (2pt) -- (2,-3);
\draw[red, thick, dashed] (-0.5,0) -- (0.5,0);
\end{scope}
\end{tikzpicture}
$}}
\\
\subfigure[$n\ge 6$ and $\ell=4,\dots,n-2$]{\makebox[\textwidth]{\label{fig:case(7)}
$
\begin{tikzpicture}[baseline=-.5ex,scale=0.6]
\useasboundingbox (-7.5,-3.5) rectangle (7.5,3.5);
\draw[violet, thick, rounded corners, fill=violet, opacity=0.2]
(-6.5,1.5) -- (-6.5,-2.5) -- (-0.5,-2.5) -- (-0.5,1.5) -- cycle
(6.5,2.5) -- (0.5,2.5) -- (0.5,-1.5) -- (6.5,-1.5) -- cycle;
\begin{scope}[xshift=-4cm]
\draw[line width=7, dashed] (3,0)--(5,0);
\draw[gray, line width=5] (3,0)--(5,0);
\draw[yellow, line cap=round, line width=5, opacity=0.5] (-1,1)--(0,0) (-1,-1)--(0,0) (1,0)--(0,0);
\draw[green, line cap=round, line width=5, opacity=0.5] (-1,1)--(-2,1) (-1,-1)--(-1,-2) (1,0)--(2,0)
(6,0)--(7,0) (9,1)--(9,2) (9,-1)--(10,-1);
\draw[blue,thick, rounded corners] (-1.75, 3) -- (-1.5, 2.5) -- (-1.25, 3);
\draw[green, thick] (0,3) -- (0,-3) (0,0) -- (-3,0);
\draw[red,thick, fill] (0,0) -- (-1,1) circle (2pt) -- +(0,2) (-1,1) -- ++(-1,0) circle (2pt) -- +(-1,0) (-2,1) -- +(0,2);
\draw[red,thick, fill] (0,0) -- (-1,-1) circle (2pt) -- +(-2,0) (-1,-1) -- ++(0,-1) circle (2pt) -- +(0,-1) (-1,-2) -- +(-2,0);
\draw[red,thick] (0,0) -- (1,0);
\draw[thick, fill=white] (0,0) circle (2pt);
\draw[thick](1, 3) -- (-2,3) to[out=180,in=90] (-3,2) --(-3,-2) to[out=-90,in=180] (-2, -3) -- (10,-3) to[out=0,in=-90] (11,-2) -- (11,2) to[out=90,in=0] (10,3)-- (1,3);
\end{scope}
\begin{scope}[xshift=4cm,rotate=180]
\draw[yellow, line cap=round, line width=5, opacity=0.5] (-1,1)--(0,0) (-1,-1)--(0,0) (1,0)--(0,0);
\draw[blue, thick] (0,3) -- (0,-3) (0,0) -- (-3,0);
\draw[red,thick, fill] (0,0) -- (-1,1) circle (2pt) -- +(0,2) (-1,1) -- ++(-1,0) circle (2pt) -- +(-1,0) (-2,1) -- +(0,2);
\draw[red,thick, fill] (0,0) -- (-1,-1) circle (2pt) -- +(-2,0) (-1,-1) -- ++(0,-1) circle (2pt) -- +(0,-1) (-1,-2) -- +(-2,0);
\draw[red,thick] (0,0) -- (1,0);
\draw[thick, fill=white] (0,0) circle (2pt);
\end{scope}
\begin{scope}
\draw[yellow, line cap=round, line width=5, opacity=0.5] (-2,0) -- (-1,0) (1,0) -- (2,0);
\draw[red, thick, fill] (-3,0) circle (2pt) -- (-3, -3) (-2,0) circle (2pt) -- (-2,3) (-2,0) -- (-1,0) circle (2pt) -- (-1,-3) (-1,0) -- (1,0) circle (2pt) -- (1,3) (1,0) -- (2, 0) circle (2pt) -- (2,-3) (3,0) circle (2pt) -- (3,3);
\draw[red, thick, dashed] (-3,0) -- (-2,0) (2,0) -- (3,0);
\end{scope}
\end{tikzpicture}
$}}
\caption{Induced pairs from $(\ngraph(\exdynD),\nbasis(\exdynD))$}
\label{figure:induced pairs}
\end{figure}

In each picture of Figure~\ref{figure:induced pairs}, the gray shaded cycles represent the avoiding $\ell$-th cycle, and the violet-shaded regions represent the induced pairs of a $N$-graph and a tuple of cycles.

Recall from \cite[Proposition 5.15]{ABL2021} that the initial seed in the cluster pattern of finite type, i.e. $(\ngraph_{(\bullet)}(a,b,c),\nbasis_{(\bullet)}(a,b,c))$ with $\frac{1}{a}+\frac{1}{b}+\frac{1}{c}>1$ including $(\ngraph_{(\bullet)}(\dynA_*),\nbasis_{(\bullet)}(\dynA_*))$, admit no obstruction to mutate the cycles $\nbasis_{(\bullet)}(a,b,c)$ in $\ngraph_{(\bullet)}(a,b,c)$. 
It is direct to check that all the above cases except~(6) satisfy this assumption.

Now it is enough to show the $N$-graph realizability for case (6). Note that the induced pair consists of a $4$-graph and a tuple of cycles which is of type $\dynD_n$. It is direct to check that one can obtain $(\ngraph_{(4)}(n-2,2,2),\nbasis_{(4)}(n-2,2,2))$ as a subregion from $(\ngraph'_{(4)}(\dynD_n),\nbasis'_{(4)}(\dynD_n))$ by applying a sequence of Move~(II).
\end{proof}

\begin{proof}[Proof of Theorem~\ref{thm:seed many fillings}]
A direct combination of Corollary~\ref{corollary:distinct seeds imples distinct fillings} and Proposition~\ref{prop:every seeds come from N-graphs} implies that there are at least as many exact embedded Lagrangian fillings as seeds of the corresponding cluster structure.
\end{proof}

\section{Foldings}\label{section:folding}

In this section, we will consider the cluster patterns of non-simply-laced affine type $\dynY$ which is obtained by folding a cluster pattern of type $\dynX=\exdynD$ or $\exdynE$ under the $G$-action.
More precisely, except for the first column in Table~\ref{figure:all possible foldings}, each and every column correspond to all possible triple $(\dynX, G, \dynY)$ we will consider.

\subsection{\texorpdfstring{$N$-graphs}{N-graphs} of the folded cluster pattern}

Let $(\dynX, G, \dynY)$ be a triple with $\dynX$ of type $\exdynD\exdynE$ and let $\psi:\field\to \field^G$ be a field homomorphism.
By Theorem~\ref{thm:seed many fillings}, for each seed $\seed_t$ in the cluster pattern of type $\dynX$, there exists an $N$-graph $(\ngraph_t, \nbasis_t)$ whose image under $\Psi$ becomes $\seed_t$.

Now by collecting an $N$-graph corresponding to each $(G,\psi)$-invariant seed, we have a subset which is bijectively mapped via $\Psi$ to the set of $(G,\psi)$-invariant seeds in the cluster pattern of type $\dynX$.
However, since $\dynX$ is globally foldable with respect to $G$ and every $(G,\psi)$-invariant seed is $(G,\psi)$-admissible by Theorem~\ref{thm_invariant_seeds_form_folded_pattern}, the latter is isomorphic to the cluster pattern of type $\dynY = \dynX^G$.

As a direct consequence, we have the following theorem:
\begin{theorem}\label{theorem:folded seed}
For each triple $(\dynX, G, \dynY)$ with $\dynX$ of type $\exdynD\exdynE$ and $\dynY=\dynX^G$ in Table~\ref{figure:all possible foldings}, there is a subset of $N$-graphs of type $\dynX$ which is isomorphic to the cluster pattern of type $\dynY$.
\end{theorem}

One of the natural question is then as follows:
can we find geometric properties of $N$-graphs which are equivalent to the $(G,\psi)$-admissibility (or $(G,\psi)$-invariance) of the corresponding seed?
For example, an invariance (or a symmetry) under a certain $G$-action is a 
possible candidate.

If so, then we can find a subset of $N$-graphs corresponding to the folded cluster pattern without passing through $\Psi$.

Unfortunately, we have no successful candidates when $\dynX$ is of type $\exdynD$.
One of difficulties comes from the obvious asymmetricity of the Coxeter padding of type $\exdynD$.
Even though it corresponds to a Legendrian loop as shown in Figure~\ref{fig:legendrian loop of D_intro}, it does not seem to be helpful to see any symmetry under $G$-action for almost all cases.

\subsection{\texorpdfstring{$G$-admissibilities on $\exdynD_{2n},\exdynE_6$ and $\exdynE_7$}{G-admissibilities on affine D2n, E6 and E7}}

For $\dynX$ is of type $\exdynD$, the only one successful attempt is when $\dynY=\exdynB_n=\exdynD_{2n}^{\Z/2\Z}$. That is,
\[
(\dynX, G, \dynY) = (\exdynD_{2n}, \Z/2\Z, \exdynB_{n}).
\]

Let $(\ngraph(\exdynD_{2n}),\nbasis(\exdynD_{2n}))$ be the initial $N$-graph. 
Since the top-left blue arc is isolated and not contained in $\nbasis(\exdynD_{2n})$, it remains the same after any (realizable) Legendrian mutations and so we may assume that the top-left blue arc is in a small enough collar neighborhood $U(\boundary\disk^2)$ of the boundary $\boundary \disk^2$.

Now the $G$-action on $(\ngraph, \nbasis)$ of type $\exdynD_{2n}$ is defined as follows:
let $\ngraph_0=\ngraph\cap \disk^2_0$ be the subgraph of $\ngraph$ contained in $\disk^2_0=\disk^2\setminus U(\boundary\disk^2)$ and $\nbasis_0=\nbasis$.
\begin{enumerate}
\item Switch colors of $\ngraph_1$ and $\ngraph_3$. In other words, if $\ngraph=(\ngraph_1,\ngraph_{2},\ngraph_{3})$, then the new $N$-graph is 
\[
\bar{\ngraph}_0=(\ngraph_{3}, \ngraph_{2}, \ngraph_1).
\]
\item Rotate $(\bar{\ngraph}_0, \bar{\nbasis}_0)$ by $\pi$ to obtain $\tau(\ngraph_0)$ and $\tau(\nbasis_0)$, and
\item Replace $(\ngraph_0,\nbasis_0)$ from $(\ngraph,\nbasis)$ with $(\tau(\ngraph_0),\tau(\nbasis_0))$.
\end{enumerate}
The result will be denoted by $\tau\cdot(\ngraph,\nbasis)$.
It is obvious that $\tau$ is involutive and so the action of $G=\Z/2\Z$ generated by $\tau$ is well-defined.

\begin{figure}[ht]
\begin{tikzcd}
\begin{tikzpicture}[baseline=-.5ex,scale=0.5]
\draw[violet, thick, rounded corners, fill=violet, opacity=0.2]
(-6,2.5) -- (-6,-2.5) -- (6,-2.5) -- (6,2.5) -- cycle;
\begin{scope}[xshift=-3.5cm]
\draw[yellow, line cap=round, line width=5, opacity=0.5] (-1,1)--(0,0) (-1,-1)--(0,0) (1,0)--(0,0);
\draw[green, line cap=round, line width=5, opacity=0.5] (-1,1)--(-2,1) (-1,-1)--(-1,-2) (1,0)--(2,0)
(3,0)--(3.5,0) (4.5,0)--(5,0) (6,0)--(7,0) (7,0)--(8,1) (7,0)--(8,-1);
\draw[blue,thick, rounded corners] (-1.75, 3) -- (-1.5, 2.5) -- (-1.25, 3);
\draw[green, thick] (0,3) -- (0,-3) (0,0) -- (-3,0);
\draw[red,thick, fill] (0,0) -- (-1,1) circle (2pt) -- +(0,2) (-1,1) -- ++(-1,0) circle (2pt) -- +(-1,0) (-2,1) -- +(0,2);
\draw[red,thick, fill] (0,0) -- (-1,-1) circle (2pt) -- +(-2,0) (-1,-1) -- ++(0,-1) circle (2pt) -- +(0,-1) (-1,-2) -- +(-2,0);
\draw[red,thick] (0,0) -- (1,0);
\draw[thick, fill=white] (0,0) circle (2pt);
\draw[thick](1, 3) -- (-2,3) to[out=180,in=90] (-3,2) --(-3,-2) to[out=-90,in=180] (-2, -3) -- (9,-3) to[out=0,in=-90] (10,-2) -- (10,2) to[out=90,in=0] (9,3)-- (1,3);
\end{scope}
\begin{scope}[xshift=3.5cm,rotate=180]
\draw[yellow, line cap=round, line width=5, opacity=0.5] (-1,-1) -- (-1,-2) (-1,1)--(-2,1);
\draw[blue, thick] (0,3) -- (0,-3) (0,0) -- (-3,0);
\draw[red,thick, fill] (0,0) -- (-1,1) circle (2pt) -- +(0,2) (-1,1) -- ++(-1,0) circle (2pt) -- +(-1,0) (-2,1) -- +(0,2);
\draw[red,thick, fill] (0,0) -- (-1,-1) circle (2pt) -- +(-2,0) (-1,-1) -- ++(0,-1) circle (2pt) -- +(0,-1) (-1,-2) -- +(-2,0);
\draw[red,thick] (0,0) -- (1,0);
\draw[thick, fill=white] (0,0) circle (2pt);
\end{scope}
\begin{scope}[xshift=0.5cm]
\draw[yellow, line cap=round, line width=5, opacity=0.5] (-2,0) -- (-1,0) (1,0) -- (2,0);
\draw[red, thick, fill] (-3,0) circle (2pt) -- (-3, -3) (-3,0) -- (-2,0) circle (2pt) -- (-2,3) (-2,0) -- (-1,0) circle (2pt) -- (-1,-3) (-1,0) -- (-0.5,0) (0.5,0) -- (1,0) circle (2pt) -- (1,3) (1,0) -- (2, 0) circle (2pt) -- (2,-3);
\draw[red, thick, dashed] (-0.5,0) -- (0.5,0);
\end{scope}
\draw[white, fill, thick, rounded corners]
(-5.75,2.25) -- (-5.75,-2.25) -- (5.75,-2.25) -- (5.75,2.25) -- cycle;
\draw(0,0) node {$(\ngraph_0, \nbasis_0)$};
\draw(0,-3) node[below] {$(\ngraph, \nbasis)$};
\end{tikzpicture}
\arrow[r, leftrightarrow, "\tau"] &
\begin{tikzpicture}[baseline=-.5ex,scale=0.5]
\draw[violet, thick, rounded corners, fill=violet, opacity=0.2]
(-6,2.5) -- (-6,-2.5) -- (6,-2.5) -- (6,2.5) -- cycle;
\begin{scope}[xshift=-3.5cm]
\draw[yellow, line cap=round, line width=5, opacity=0.5] (-1,1)--(0,0) (-1,-1)--(0,0) (1,0)--(0,0);
\draw[green, line cap=round, line width=5, opacity=0.5] (-1,1)--(-2,1) (-1,-1)--(-1,-2) (1,0)--(2,0)
(3,0)--(3.5,0) (4.5,0)--(5,0) (6,0)--(7,0) (7,0)--(8,1) (7,0)--(8,-1);
\draw[blue,thick, rounded corners] (-1.75, 3) -- (-1.5, 2.5) -- (-1.25, 3);
\draw[green, thick] (0,3) -- (0,-3) (0,0) -- (-3,0);
\draw[red,thick, fill] (0,0) -- (-1,1) circle (2pt) -- +(0,2) (-1,1) -- ++(-1,0) circle (2pt) -- +(-1,0) (-2,1) -- +(0,2);
\draw[red,thick, fill] (0,0) -- (-1,-1) circle (2pt) -- +(-2,0) (-1,-1) -- ++(0,-1) circle (2pt) -- +(0,-1) (-1,-2) -- +(-2,0);
\draw[red,thick] (0,0) -- (1,0);
\draw[thick, fill=white] (0,0) circle (2pt);
\draw[thick](1, 3) -- (-2,3) to[out=180,in=90] (-3,2) --(-3,-2) to[out=-90,in=180] (-2, -3) -- (9,-3) to[out=0,in=-90] (10,-2) -- (10,2) to[out=90,in=0] (9,3)-- (1,3);
\end{scope}
\begin{scope}[xshift=3.5cm,rotate=180]
\draw[yellow, line cap=round, line width=5, opacity=0.5] (-1,-1) -- (-1,-2) (-1,1)--(-2,1);
\draw[blue, thick] (0,3) -- (0,-3) (0,0) -- (-3,0);
\draw[red,thick, fill] (0,0) -- (-1,1) circle (2pt) -- +(0,2) (-1,1) -- ++(-1,0) circle (2pt) -- +(-1,0) (-2,1) -- +(0,2);
\draw[red,thick, fill] (0,0) -- (-1,-1) circle (2pt) -- +(-2,0) (-1,-1) -- ++(0,-1) circle (2pt) -- +(0,-1) (-1,-2) -- +(-2,0);
\draw[red,thick] (0,0) -- (1,0);
\draw[thick, fill=white] (0,0) circle (2pt);
\end{scope}
\begin{scope}[xshift=0.5cm]
\draw[yellow, line cap=round, line width=5, opacity=0.5] (-2,0) -- (-1,0) (1,0) -- (2,0);
\draw[red, thick, fill] (-3,0) circle (2pt) -- (-3, -3) (-3,0) -- (-2,0) circle (2pt) -- (-2,3) (-2,0) -- (-1,0) circle (2pt) -- (-1,-3) (-1,0) -- (-0.5,0) (0.5,0) -- (1,0) circle (2pt) -- (1,3) (1,0) -- (2, 0) circle (2pt) -- (2,-3);
\draw[red, thick, dashed] (-0.5,0) -- (0.5,0);
\end{scope}
\draw[white, fill, thick, rounded corners]
(-5.75,2.25) -- (-5.75,-2.25) -- (5.75,-2.25) -- (5.75,2.25) -- cycle;
\draw(0,0) node {\rotatebox{180}{$(\bar{\ngraph}_0, \bar{\nbasis}_0)$}};
\draw(0,-3) node[below] {$\tau\cdot(\ngraph, \nbasis)$};
\end{tikzpicture}
\end{tikzcd}
\caption{$\Z/2\Z$-action on $N$-graphs of type $\exdynD_{2n}$}
\label{figure:Z/2Z-action on D_2n}
\end{figure}

On the other hand, when $\dynX$ is of type $\exdynE$, the action of $G=\Z/2\Z$ or $\Z/3\Z$ can be defined by the same way as described in \cite[Section~6]{ABL2021}.
Let us consider the cluster patterns of type $ \exdynG_2, \dynE_6^{(2)}$ and 
$\exdynF_4$ which can be obtained by folding cluster patterns of type 
$\exdynE_6$ and $\exdynE_7$. 

More precisely, we consider the following three cases: let $(\dynX, G, \dynY)$ be one of
\begin{align*}
&(\exdynE_6, \Z/3\Z, \exdynG_2),&
&(\exdynE_6, \Z/2\Z, \dynE_6^{(2)}),&
&(\exdynE_7, \Z/2\Z, \exdynF_4).
\end{align*}
Then as seen earlier, each $\dynX$ has a $G$-action as depicted in Figure~\ref{figure:G-actions} which makes $\dynX$ globally foldable with respect to $G$. From now on, we denote the generator of $G$ by $\tau$.

Now let $(\ngraph, \nbasis)$ be a pair of a $3$-graph and a good tuple of cycles of type $\dynX$. We say that $(\ngraph, \nbasis)$ has the \emph{ray symmetry} if it has the $2\pi/3$-rotation symmetry on the subset
\begin{align*}
R_{2\pi/3}\cup R_{4\pi/3}\cup R_{2\pi}&\subset \disk^2,&
R_\theta&=\{(r,\theta)\in\disk\subset\C\mid 0\le r\le 1\}
\end{align*}
as depicted in Figure~\ref{figure:ray symmetricity}.

\begin{figure}[ht]
\[
\begin{tikzpicture}[baseline=-.5ex,scale=0.8]
\draw[thick] (0,0) circle (3);
\foreach \i in {0,120, 240} {
\begin{scope}[rotate=\i]
\foreach \j in {1,2,4,5} {
\draw[blue, thick] (\j*20:2) -- (\j*20:3);
\draw[red, thick] (0:2) -- (0:3);
\draw[red, thick] (-20:1.5) arc (-20:20:1.5);
\draw[blue, thick] (-20:1) arc (-20:20:1);
}
\draw[blue, thick, dotted] (50:2.5) arc (50:70:2.5);
\end{scope}
}
\draw[double] (0,0) circle (2);
\foreach \i in {0,120, 240} {
\begin{scope}[shift=({\i+60}:0.2)]
\fill[white,opacity=0.5] (0,0) -- (\i:3) arc (\i:{\i+120}:3) -- cycle;
\end{scope}
}
\draw[orange, opacity=0.2, line width=5] (0,0) -- (0:3);
\draw[blue, opacity=0.1, line width=5] (0,0) -- (120:3);
\draw[violet, opacity=0.1, line width=5] (0,0) -- (240:3);
\draw[orange] (0:3.7) node {$R_0$};
\draw[blue] (120:3.7) node {$R_{2\pi/3}$};
\draw[violet] (240:3.7) node {$R_{4\pi/3}$};
\end{tikzpicture}\qquad
\begin{tikzpicture}[baseline=-.5ex,scale=0.8]
\begin{scope}[yshift=2cm]
\draw[orange, opacity=0.2, line width=5] (0,0) -- (3,0);
\draw[orange] (0,0) node[left] {$(R_0, (\ngraph,\nbasis)\cap R_0)=$};
\draw[thick] (-10:3) arc (-10:10:3);
\draw[red, thick] (0:2) -- (0:3);
\draw[double] (-10:2) arc (-10:10:2);
\draw[red, thick] (-10:1.5) arc (-10:10:1.5);
\draw[blue, thick] (-10:1) arc (-10:10:1);
\fill[white, opacity=0.5] (0,0.15) rectangle (3,0.6);
\fill[white, opacity=0.5] (0,-0.15) rectangle (3,-0.6);
\end{scope}
\begin{scope}
\draw[blue, opacity=0.1, line width=5] (0,0) -- (3,0);
\draw[blue] (0,0) node[left] {$(R_{2\pi/3}, (\ngraph,\nbasis)\cap R_{2\pi/3})=$};
\draw[thick] (-10:3) arc (-10:10:3);
\draw[red, thick] (0:2) -- (0:3);
\draw[double] (-10:2) arc (-10:10:2);
\draw[red, thick] (-10:1.5) arc (-10:10:1.5);
\draw[blue, thick] (-10:1) arc (-10:10:1);
\fill[white, opacity=0.5] (0,0.15) rectangle (3,0.6);
\fill[white, opacity=0.5] (0,-0.15) rectangle (3,-0.6);
\end{scope}
\begin{scope}[yshift=-2cm]
\draw[violet, opacity=0.1, line width=5] (0,0) -- (3,0);
\draw[violet] (0,0)  node[left] {$(R_{4\pi/3}, (\ngraph,\nbasis)\cap R_{4\pi/3})=$};
\draw[thick] (-10:3) arc (-10:10:3);
\draw[red, thick] (0:2) -- (0:3);
\draw[double] (-10:2) arc (-10:10:2);
\draw[red, thick] (-10:1.5) arc (-10:10:1.5);
\draw[blue, thick] (-10:1) arc (-10:10:1);
\fill[white, opacity=0.5] (0,0.15) rectangle (3,0.6);
\fill[white, opacity=0.5] (0,-0.15) rectangle (3,-0.6);
\end{scope}
\draw (1.5,1) node[rotate=90] {$\cong$};
\draw (1.5,-1) node[rotate=90] {$\cong$};
\end{tikzpicture}
\]
\caption{Ray-symmetricity}
\label{figure:ray symmetricity}
\end{figure}

For each ray symmetric $(\ngraph, \nbasis)$, we define the $G$-action according to $G$.
\begin{enumerate}
\item If $G=\Z/3\Z$, then $(\tau(\ngraph),\tau(\nbasis))$ is defined by the $2\pi/3$-rotation.
\item If $G=\Z/2\Z$, then $(\tau(\ngraph),\tau(\nbasis))$ is defined by the \emph{partial rotation} as follows:
\begin{enumerate}
\item Cut $(\ngraph, \nbasis)$ into $3$-pieces $(\ngraph_i,\nbasis_i)$ for $1\le i\le 3$ along the rays $R_{2\pi i/3}$ for $1\le i\le 3$, where $(\nbasis_i, \nbasis_i)$ is in between $R_{2\pi (i-1)/3}$ and $R_{2\pi i/3}$.
\item Interchange the last two pieces $(\ngraph_2,\nbasis_2)$ and $(\ngraph_3,\nbasis_3)$ by the rotation.
\end{enumerate}
\end{enumerate}
We define the action of $\tau\in G$ as
\[
\tau\cdot(\ngraph,\nbasis) = (\tau(\ngraph), \tau(\nbasis)).
\]
The pictorial definition of the $G$-action is shown in Figure~\ref{figure:G actions}.

\begin{figure}[ht]
\subfigure[$\Z/3\Z$-action on $(\ngraph,\nbasis)$ of type $\exdynE_6$]{
\begin{tikzcd}[ampersand replacement=\&]
\begin{tikzpicture}[baseline=-.5ex, scale=0.5]
\draw[thick] (0,0) circle (3);
\fill[orange, opacity=0.2] (0,0) -- (0:3) arc(0:120:3) -- cycle;
\fill[violet, opacity=0.1] (0,0) -- (-120:3) arc(-120:0:3) -- cycle;
\fill[blue, opacity=0.1] (0,0) -- (120:3) arc (120:240:3) -- cycle;
\foreach \i in {0, 120, 240} {
\begin{scope}[rotate=\i]
\draw[blue, thick] (30:3) -- (30:2) (60:3) -- (60:2) (90:3) -- (90:2);
\draw[red, thick] (0:3) -- (0:2);
\end{scope}
}
\draw[double] (0,0) node {$(\ngraph,\nbasis)$} circle (2);
\end{tikzpicture}
\arrow[r,"\tau"]\&
\begin{tikzpicture}[baseline=-.5ex, scale=0.5]
\draw[thick] (0,0) circle (3);
\fill[orange, opacity=0.2] (0,0) -- (120:3) arc(120:240:3) -- cycle;
\fill[violet, opacity=0.1] (0,0) -- (0:3) arc(0:120:3) -- cycle;
\fill[blue, opacity=0.1] (0,0) -- (-120:3) arc (-120:0:3) -- cycle;
\foreach \i in {0, 120, 240} {
\begin{scope}[rotate=\i]
\draw[blue, thick] (30:3) -- (30:2) (60:3) -- (60:2) (90:3) -- (90:2);
\draw[red, thick] (0:3) -- (0:2);
\end{scope}
}
\draw[double] (0,0) node[rotate=120] {$(\ngraph,\nbasis)$} circle (2);
\end{tikzpicture}
\arrow[r,"\tau"]\&
\begin{tikzpicture}[baseline=-.5ex, scale=0.5]
\draw[thick] (0,0) circle (3);
\fill[orange, opacity=0.2] (0,0) -- (-120:3) arc(-120:0:3) -- cycle;
\fill[violet, opacity=0.1] (0,0) -- (120:3) arc(120:240:3) -- cycle;
\fill[blue, opacity=0.1] (0,0) -- (0:3) arc (0:120:3) -- cycle;
\foreach \i in {0, 120, 240} {
\begin{scope}[rotate=\i]
\draw[blue, thick] (30:3) -- (30:2) (60:3) -- (60:2) (90:3) -- (90:2);
\draw[red, thick] (0:3) -- (0:2);
\end{scope}
}
\draw[double] (0,0) node[rotate=-120] {$(\ngraph,\nbasis)$} circle (2);
\end{tikzpicture}
\arrow[ll,"\tau", bend left=35]
\end{tikzcd}
}

\subfigure[$\Z/2Z$-action on $(\ngraph, \nbasis)$ of type $\exdynE_6$ or $\exdynE_7$]{
$
\begin{tikzcd}[ampersand replacement=\&, column sep=5pc, row sep=3pc]
\begin{tikzpicture}[baseline=-.5ex, scale=0.5]
\draw[thick] (0,0) circle (3);
\fill[orange, opacity=0.2] (0,0) -- (0:3) arc(0:120:3) -- cycle;
\fill[violet, opacity=0.1] (0,0) -- (-120:3) arc(-120:0:3) -- cycle;
\fill[blue, opacity=0.1] (0,0) -- (120:3) arc (120:240:3) -- cycle;
\foreach \j in {1,2,4,5} {
\draw[blue, thick] (\j*20:2) -- (\j*20:3);
\draw[red, thick] (0:2) -- (0:3);
}
\draw[blue, thick, dotted] (50:2.5) arc (50:70:2.5);
\foreach \i in {120, 240} {
\begin{scope}[rotate=\i]
\foreach \j in {1,3} {
\draw[blue, thick] (\j*30:2) -- (\j*30:3);
\draw[red, thick] (0:2) -- (0:3);
}
\draw[blue, thick, dotted] (50:2.5) arc (50:70:2.5);
\end{scope}
}
\draw[double] (0,0) circle (2);
\draw (60:1) node[rotate=-30] {$(\ngraph_1,\nbasis_1)$};
\begin{scope}[rotate=120]
\draw (60:1) node[rotate=90] {$(\ngraph_2,\nbasis_2)$};
\end{scope}
\begin{scope}[rotate=240]
\draw (60:1) node[rotate=-150] {$(\ngraph_3,\nbasis_3)$};
\end{scope}
\end{tikzpicture}
\arrow[r,"\tau", yshift=.5ex]\arrow[d, "\text{cut}"', xshift=-.5ex]\&
\begin{tikzpicture}[baseline=-.5ex, scale=0.5]
\draw[thick] (0,0) circle (3);
\fill[orange, opacity=0.2] (0,0) -- (0:3) arc(0:120:3) -- cycle;
\fill[blue, opacity=0.1] (0,0) -- (-120:3) arc(-120:0:3) -- cycle;
\fill[violet, opacity=0.1] (0,0) -- (120:3) arc (120:240:3) -- cycle;
\foreach \j in {1,2,4,5} {
\draw[blue, thick] (\j*20:2) -- (\j*20:3);
\draw[red, thick] (0:2) -- (0:3);
}
\draw[blue, thick, dotted] (50:2.5) arc (50:70:2.5);
\foreach \i in {120, 240} {
\begin{scope}[rotate=\i]
\foreach \j in {1,3} {
\draw[blue, thick] (\j*30:2) -- (\j*30:3);
\draw[red, thick] (0:2) -- (0:3);
}
\draw[blue, thick, dotted] (50:2.5) arc (50:70:2.5);
\end{scope}
}
\draw[double] (0,0) circle (2);
\draw (60:1) node[rotate=-30] {$(\ngraph_1,\nbasis_1)$};
\begin{scope}[rotate=120]
\draw (60:1) node[rotate=90] {$(\ngraph_3,\nbasis_3)$};
\end{scope}
\begin{scope}[rotate=240]
\draw (60:1) node[rotate=-150] {$(\ngraph_2,\nbasis_2)$};
\end{scope}
\end{tikzpicture}\arrow[l, "\tau", yshift=-.5ex]\arrow[d, "\text{cut}", xshift=.5ex]\\
\begin{tikzpicture}[baseline=-.5ex, scale=0.5]
\begin{scope}[shift=(60:1)]
\fill[orange, opacity=0.2] (0,0) -- (0:3) arc(0:120:3) -- cycle;
\draw[thick] (0:3) arc (0:120:3);
\foreach \j in {1,2,4,5} {
\draw[blue, thick] (\j*20:2) -- (\j*20:3);
\draw[red, thick] (0:2) -- (0:3);
\draw[red, thick] (120:2) -- (120:3);
}
\draw[double] (0:2) arc (0:120:2);
\draw[blue, thick, dotted] (50:2.5) arc (50:70:2.5);
\draw (60:1) node[rotate=-30] {$(\ngraph_1,\nbasis_1)$};
\end{scope}
\begin{scope}[rotate=120]
\begin{scope}[shift=(60:1)]
\fill[blue, opacity=0.1] (0,0) -- (0:3) arc (0:120:3) -- cycle;
\draw[thick] (0:3) arc (0:120:3);
\foreach \j in {1,3} {
\draw[blue, thick] (\j*30:2) -- (\j*30:3);
\draw[red, thick] (0:2) -- (0:3);
\draw[red, thick] (120:2) -- (120:3);
}
\draw[blue, thick, dotted] (50:2.5) arc (50:70:2.5);
\draw[double] (0:2) arc (0:120:2);
\draw (60:1) node[rotate=90] {$(\ngraph_2,\nbasis_2)$};
\end{scope}
\end{scope}
\begin{scope}[rotate=240]
\begin{scope}[shift=(60:1)]
\fill[violet, opacity=0.1] (0,0) -- (0:3) arc (0:120:3) -- cycle;
\draw[thick] (0:3) arc (0:120:3);
\foreach \j in {1,3} {
\draw[blue, thick] (\j*30:2) -- (\j*30:3);
\draw[red, thick] (0:2) -- (0:3);
\draw[red, thick] (120:2) -- (120:3);
}
\draw[blue, thick, dotted] (50:2.5) arc (50:70:2.5);
\draw[double] (0:2) arc (0:120:2);
\draw (60:1) node[rotate=-150] {$(\ngraph_3,\nbasis_3)$};
\end{scope}
\end{scope}
\end{tikzpicture}\arrow[r, "\text{partial rot.}", yshift=.5ex]\arrow[u, "\text{glue}"', xshift=.5ex]
\&
\begin{tikzpicture}[baseline=-.5ex, scale=0.5]
\begin{scope}[shift=(60:1)]
\fill[orange, opacity=0.2] (0,0) -- (0:3) arc(0:120:3) -- cycle;
\draw[thick] (0:3) arc (0:120:3);
\foreach \j in {1,2,4,5} {
\draw[blue, thick] (\j*20:2) -- (\j*20:3);
\draw[red, thick] (0:2) -- (0:3);
\draw[red, thick] (120:2) -- (120:3);
}
\draw[double] (0:2) arc (0:120:2);
\draw[blue, thick, dotted] (50:2.5) arc (50:70:2.5);
\draw (60:1) node[rotate=-30] {$(\ngraph_1,\nbasis_1)$};
\end{scope}
\begin{scope}[rotate=240]
\begin{scope}[shift=(60:1)]
\fill[blue, opacity=0.1] (0,0) -- (0:3) arc (0:120:3) -- cycle;
\draw[thick] (0:3) arc (0:120:3);
\foreach \j in {1,3} {
\draw[blue, thick] (\j*30:2) -- (\j*30:3);
\draw[red, thick] (0:2) -- (0:3);
\draw[red, thick] (120:2) -- (120:3);
}
\draw[blue, thick, dotted] (50:2.5) arc (50:70:2.5);
\draw[double] (0:2) arc (0:120:2);
\draw (60:1) node[rotate=-150] {$(\ngraph_2,\nbasis_2)$};
\end{scope}
\end{scope}
\begin{scope}[rotate=120]
\begin{scope}[shift=(60:1)]
\fill[violet, opacity=0.1] (0,0) -- (0:3) arc (0:120:3) -- cycle;
\draw[thick] (0:3) arc (0:120:3);
\foreach \j in {1,3} {
\draw[blue, thick] (\j*30:2) -- (\j*30:3);
\draw[red, thick] (0:2) -- (0:3);
\draw[red, thick] (120:2) -- (120:3);
}
\draw[blue, thick, dotted] (50:2.5) arc (50:70:2.5);
\draw[double] (0:2) arc (0:120:2);
\draw (60:1) node[rotate=90] {$(\ngraph_3,\nbasis_3)$};
\end{scope}
\end{scope}
\end{tikzpicture}
\arrow[l, "\text{partial rot.}", yshift=-.5ex]\arrow[u, "\text{glue}", xshift=-.5ex]
\end{tikzcd}
$
}

\caption{$G$-action on a ray-symmetric $(\ngraph, \nbasis)$}
\label{figure:G actions}
\end{figure}

From now on, we assume that the triple $(\dynX, G, \dynY)$ is one of the following:
\begin{align*}
&(\exdynD_{2n}, \Z/2\Z, \exdynB_{n}),&
&(\exdynE_6, \Z/3\Z, \exdynG_2),&
&(\exdynE_6, \Z/2\Z, \dynE_6^{(2)}),&
&(\exdynE_7, \Z/2\Z, \exdynF_4).
\end{align*}

\begin{definition}[$G$-admissibility]\label{definition:G-admissibility for ngraphs}
We say that $(\ngraph,\nbasis)$ of type $\dynX$ is $G$-admissible if
\begin{enumerate}
\item the $N$-graph $\ngraph$ is invariant under $G$-action, 
\item the tuples of cycles $\nbasis$ and $\tau(\nbasis)$ are identical up to relabelling as follows:
\begin{enumerate}
\item if $\dynX=\exdynD_{2n}$ and $G=\Z/2\Z$, then
\begin{align*}
\gamma_1&\stackrel{\tau}{\longleftrightarrow}\gamma_{2n-1},&
\gamma_2&\stackrel{\tau}{\longleftrightarrow}\gamma_{2n+1},&
\gamma_3&\stackrel{\tau}{\longleftrightarrow}\gamma_{2n},&
\gamma_j&\stackrel{\tau}{\longleftrightarrow}\gamma_{2n-j},
\end{align*}
where $3<j<2n-1$.
\item if $\dynX=\exdynE_6$ and $G=\Z/3\Z$, then 
\begin{align*}
\gamma_1&\stackrel{\tau}{\longleftrightarrow}\gamma_1,&
\gamma_2&\stackrel{\tau}{\longmapsto}\gamma_4\stackrel{\tau}{\longmapsto}\gamma_6\stackrel{\tau}{\longmapsto}\gamma_2,&
\gamma_3&\stackrel{\tau}{\longmapsto}\gamma_5\stackrel{\tau}{\longmapsto}\gamma_7\stackrel{\tau}{\longmapsto}\gamma_3.
\end{align*}
\item if $\dynX=\exdynE_6$ and $G=\Z/2\Z$, then 
\begin{align*}
\gamma_i&\stackrel{\tau}{\longleftrightarrow}\gamma_i\quad i\le 3,&
\gamma_4&\stackrel{\tau}{\longleftrightarrow}\gamma_6,&
\gamma_5&\stackrel{\tau}{\longleftrightarrow}\gamma_7.
\end{align*}
\item if $\dynX=\exdynE_7$ and $G=\Z/2\Z$, then 
\begin{align*}
\gamma_i&\stackrel{\tau}{\longleftrightarrow}\gamma_i\quad i\le 2,&
\gamma_3&\stackrel{\tau}{\longleftrightarrow}\gamma_6,&
\gamma_4&\stackrel{\tau}{\longleftrightarrow}\gamma_7,&
\gamma_5&\stackrel{\tau}{\longleftrightarrow}\gamma_8.
\end{align*}
\end{enumerate}
\end{enumerate}
\end{definition}

\begin{proposition}
Let $(\ngraph, \nbasis)$ be of type $\dynX$. If $(\ngraph, \nbasis)$ is $G$-admissible, then so is the quiver~$\quiver(\Legendrian(\ngraph),\nbasis)$.
\end{proposition}
\begin{proof}
If $(\ngraph, \nbasis)$ is $G$-admissible, then the quiver 
$\quiver(\Legendrian(\ngraph), \nbasis)$ is $G$-invariant by definition. 
Moreover, it is $G$-admissible by 
Theorem~\ref{thm_invariant_seeds_form_folded_pattern} and we are done.
\end{proof}

Let us recall the globally foldability for $N$-graphs defined in \cite[Section~6.3]{ABL2021}. We say that $(\ngraph, \nbasis)$ of type $\dynX$ is \emph{globally foldable} with respect to $G$ if $(\ngraph,\nbasis)$ is $G$-admissible and for any sequence of mutable $G$-orbits $I_1,\dots, I_\ell$, there eists a $G$-admissible $(\ngraph',\nbasis')$ such that
\[
\quiver(\Legendrian(\ngraph'),\nbasis')=
(\mutation_{I_\ell}\cdots\mutation_{I_1})(\quiver(\Legendrian(\ngraph),\nbasis)).
\]

\begin{remark}
Since we already know that $\dynX$ is globally foldable with respect to $G$, this definition requires only the realizability of $N$-graphs.
\end{remark}

\begin{theorem}\label{theorem:globally foldable Ngraphs}
The $N$-graph with a good tuple of cycles $(\ngraph(\dynX), \nbasis(\dynX))$ is globally foldable with respect to $G$.
\end{theorem}
\begin{proof}
Let $(\ngraph(\dynX), \nbasis(\dynX))$ be given as depicted in Table~\ref{table:affine E type} and denote the initial seed $\seed_{t_0}$ via $\Psi$ as follows:
\begin{align*}
\seed_{t_0} &= \Psi(\ngraph(\dynX), \nbasis(\dynX), \flags)=(\bfx_{t_0}, \quiver_{t_0}).
\end{align*}
By Theorem~\ref{thm_invariant_seeds_form_folded_pattern}, any 
$(G,\psi)$-admissible seed $\seed=(\bfx, \quiver)$ can be reached from the 
initial seed $\seed_{t_0}=(\bfx_{t_0}, \quiver_{t_0})$ via a sequence of orbit 
mutations. Indeed, for each $\seed$, by Lemma~\ref{lemma:normal form}, there 
exist an integer $r$ and a sequence of mutations 
$\mutation^\dynY_{j_1},\dots,\mutation^{\dynY}_{j_L}$ between folded seeds 
$\seed_{t_0}^G$ and $\seed^G$ of type $\dynY$
\[
\seed^G = \left(\mutation^{\dynY}_{j_L}\cdots\mutation^{\dynY}_{j_1}\right)((\mutation_\quiver^{\dynY})^r(\seed_{t_0}^G)),
\]
where the sequence $j_1,\dots,j_L$ misses at least one index for $\dynY$. Equivalently, there is a unique lift of the sequence of orbit mutations $\mu^\dynX_{I_1},\dots,\mu^{\dynX}_{I_L}$ from $\seed_{t_0}$ to $\seed$
\[
\seed = \left(\mutation^{\dynX}_{I_L}\cdots\mutation^{\dynX}_{I_1}\right)((\mutation_\quiver^{\dynX})^r(\seed_{t_0})),
\]
where $I_\ell$ is the $G$-orbit corresponding to $j_\ell$, and the sequence 
$I_1,\dots, I_L$ misses at least one $G$-orbit, say $J$.

Furthermore, Theorem~\ref{thm:seed many fillings} tells us that there exists a pair $(\ngraph, \nbasis)$ that realizes the seed $\seed$ via $\Psi$
\[
\seed = \Psi(\ngraph, \nbasis, \flags).
\]
Since $\seed$ is already $(G,\psi)$-admissible, the quiver $\quiver(\Legendrian(\ngraph),\nbasis)$ is $G$-admissible. However, the pair $(\ngraph,\nbasis)$ itself is not yet known to be $G$-admissible in the sense of Definition~\ref{definition:G-admissibility for ngraphs}, and therefore it suffices to show the $G$-admissibility for $(\ngraph, \nbasis)$.

The rest of the proof is essentially the same as the proof of Theorem~6.10 in 
\cite{ABL2021}. We will show the existence of the Legendrian mutation
\[
\left(\mutation_{I_L}\cdots\mutation_{I_1}\right)
(\mutation_\ngraph^r(\ngraph(\dynX),\nbasis(\dynX))).
\]

The realizability under $\mutation_\ngraph^r$ is guaranteed by 
Corollaries~\ref{cor:coxeter realization E-type} and \ref{cor:coxeter 
realization D-type} and the resulting $N$-graph 
$\mutation_\ngraph^r(\ngraph(\dynX),\nbasis(\dynX))$ is the same as the initial 
$N$-graph up to Coxeter padding attachment, and so it is $G$-admissible. On the 
other hand, since the sequence $I_1,\dots, I_L$ misses the orbit $J$, we 
separate the resulting $N$-graph into $\{(\ngraph^{(i)},\nbasis^{(i)})\}$ by 
using the cycles corresponding to the set $J$ as before so that each piece 
$(\ngraph^{(i)},\nbasis^{(i)})$ becomes an $N$-graph of finite type $\dynA_n$, 
$\dynD_n$, or $\dynE_n$.
Now the orbit mutations $\mutation_{I_\ell}$ will be separated into several 
sequences $\mutation^{(i)}$ of single mutations on separated $N$-graphs. Hence 
the realizability under orbit mutations follows from the realizability of each 
piece $(\ngraph^{(i)},\nbasis^{(i)})$ under $\mutation^{(i)}$, which are done 
already by \cite[Proposition~5.15]{ABL2021}. Finally, the $G$-admissibility of 
the final $N$-graph obviously follows from the construction.
\end{proof}

\begin{theorem}\label{theorem:folded Lagrangian fillings}
The following holds:
\begin{enumerate}
\item There exists a set of $\Z/2\Z$-admissible $4$-graphs of the Legendrian link $\legendrian(\exdynD_{2n})$ admits the cluster pattern of type $\exdynB_n$.
\item There exists a set of $\Z/3\Z$-admissible $3$-graphs of the Legendrian link $\legendrian(\exdynE_6)$ admits the cluster pattern of type $\exdynG_2$.
\item There exists a set of $\Z/2\Z$-admissible $3$-graphs of the Legendrian link $\legendrian(\exdynE_6)$ admits the cluster pattern of type $\dynE_6^{(2)}$.
\item There exists a set of $\Z/2\Z$-admissible $3$-graphs of the Legendrian link $\legendrian(\exdynE_7)$ admits the cluster pattern of type $\exdynF_4$.
\end{enumerate}
\end{theorem}
\begin{proof}
This is a combination of Theorem~\ref{theorem:globally foldable Ngraphs} and 
Theorem~\ref{thm_invariant_seeds_form_folded_pattern}.
\end{proof}

\appendix

\section{Coxeter paddings \texorpdfstring{$\coxeterpadding(\exdynD_n)$}{for Affine Dn}}\label{sec:Coxeter padding affine D_n}

Let us recall the pair $(\ngraph(\exdynD_n), \nbasis(\exdynD_n))$ given in Table~\ref{table:affine D type}.
We will perform the Legendrian Coxeter mutation $\mutation_\ngraph$ on $(\ngraph(\exdynD_n), \nbasis(\exdynD_n))$ in order to provide the pictorial proof of Proposition~\ref{proposition:coxeter realization D-type}.

Before we take mutations, we first introduce a useful operation on $N$-graphs 
described below, called the \emph{move} $\mathrm{(Z)}$.
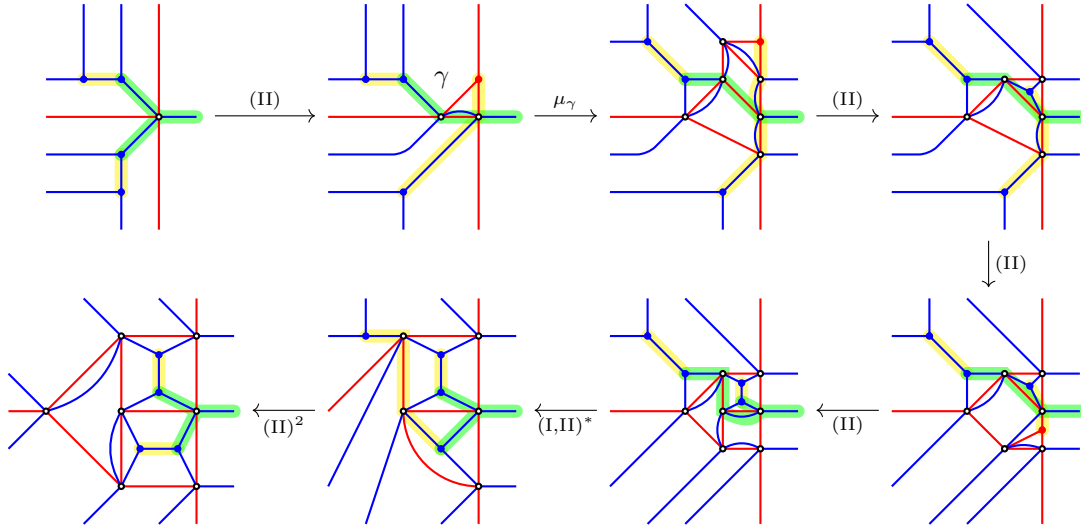
\[
\begin{tikzcd}
\begin{tikzpicture}[baseline=-.5ex,scale=0.5]
\draw[yellow, line cap=round, line width=5, opacity=0.5] (-2,1)--(-1,1) (-1,-1)--(-1,-2);
\draw[green, line cap=round, line width=5, opacity=0.5,] (-1,1)--(0,0) (-1,-1)--(0,0) (0,0)--(1,0);
\draw[red, thick] (0,3) -- (0,-3) (0,0) -- (-3,0);
\draw[blue,thick, fill] (0,0) -- (-1,1) circle (2pt) -- +(0,2) (-1,1) -- ++(-1,0) circle (2pt) -- +(-1,0) (-2,1) -- +(0,2);
\draw[blue,thick, fill] (0,0) -- (-1,-1) circle (2pt) -- +(-2,0) (-1,-1) -- ++(0,-1) circle (2pt) -- +(0,-1) (-1,-2) -- +(-2,0);
\draw[blue,thick] (0,0) -- (1,0);
\draw[thick, fill=white] (0,0) circle (2pt);
\end{tikzpicture}\arrow[r,"\mathrm{(II)}"]&
\begin{tikzpicture}[baseline=-.5ex,scale=0.5]
\draw[yellow, line cap=round, line width=5, opacity=0.5] (-2,1)--(-1,1) (-1,-2)--(1,0)--(1,1);
\draw[green, line cap=round, line width=5, opacity=0.5] (-1,1)--(0,0) (0,0)--(2,0);
\draw[red, thick, fill] (1,3) -- (1,1) circle (2pt) -- (0,0) (1,1) -- (1,0) (1,0) -- (1,-3) (1,0) -- (-3,0);
\node at (-0.5,0.5)[above right] {$\gamma$};
\draw[blue, thick] (0,0) to[out=30,in=150] (1,0);
\draw[blue,thick, fill] (0,0) -- (-1,1) circle (2pt) -- +(0,2) (-1,1) -- ++(-1,0) circle (2pt) -- +(-1,0) (-2,1) -- +(0,2) (-1,-2) circle (2pt);
\draw[blue,thick]  (-1,-3) -- (-1,-2) -- (1,0) (-1,-2) -- (-3,-2);
\draw[blue,thick,rounded corners](0,0) -- (-1,-1) -- +(-2,0);
\draw[blue,thick] (1,0) -- (2,0);
\draw[thick, fill=white] (0,0) circle (2pt) (1,0) circle (2pt);
\end{tikzpicture}\arrow[r,"\mu_\gamma"]&
\begin{tikzpicture}[baseline=-.5ex,scale=0.5]
\draw[yellow, line cap=round, line width=5, opacity=0.5] (-2,2)--(-1,1) (0,-2)--(1,-1)--(1,0) to[out=120,in=-120] (1,1) -- (1,2);
\draw[green, line cap=round, line width=5, opacity=0.5,] (-1,1)--(0,1)--(1,0)--(2,0);
\draw[red, thick, fill] (1,3) -- (1,-3) (-3,0) -- (-1,0) (1,2) circle (2pt) -- (0,2);
\draw[red, thick] (1,1) -- (0,2) -- (0,1) -- (1,0) (0,1) -- (-1,0) -- (1,-1);
\draw[blue, thick, fill] (0,2) -- (-1,3) (0,1) -- (-1,1) circle (2pt) -- (-2,2) circle (2pt) -- (-2,3) (-1,1) -- (-1,0) (-2,2) -- (-3,2) (1,-1) -- (0,-2) circle (2pt) -- (-3,-2) (0,-2) -- (0,-3);
\draw[blue,thick, rounded corners] (-1,0) -- (-2,-1) -- (-3,-1) ;
\draw[blue, thick] (1,1) -- (2,1) (1,0) -- (2,0) (1,-1) -- (2,-1);
\draw[blue, thick] (-1,0) to[out=15,in=-105] (0,1) (0,2) to[out=-15,in=105] (1,1) (1,1) to[out=-120,in=120] (1,0) (1,0) to[out=-120,in=120] (1,-1) (0,2) to[out=-60,in=60] (0,1);
\draw[thick, fill=white] (-1,0) circle (2pt) (1,0) circle (2pt) (0,1) circle (2pt) (0,2) circle (2pt) (1,1) circle (2pt) (1,-1) circle (2pt);
\end{tikzpicture}\arrow[r,"\mathrm{(II)}"]&
\begin{tikzpicture}[baseline=-.5ex,scale=0.5]
\draw[yellow, line cap=round, line width=5, opacity=0.5] (-2,2)--(-1,1) (0,-2)--(1,-1)--(1,0)--(0.67, 0.67);
\draw[green, line cap=round, line width=5, opacity=0.5,] (-1,1)--(0,1)--(1,0)--(2,0);
\draw[red, thick, fill] (1,3)--(1,-3) (-3,0) -- (-1,0) (-1,0) -- (1,-1) (1,1) -- (0,1);
\draw[red, thick] (-1,0) -- (0,1) -- (1,0);
\draw[blue, thick, fill] (1,1) -- (0.67, 0.67) circle (2pt) -- (0,1) (0.67, 0.67) -- (1,0);
\draw[blue, thick] (-1,0) to[out=15, in=-105] (0,1) (1,0) to[out=-120,in=120] (1,-1);
\draw[blue, thick, fill] (0,1) -- (-1,1) circle (2pt) -- (-2,2) circle (2pt) -- (-2,3) (-1,0) -- (-1,1) (-2,2) -- (-3,2);
\draw[blue, thick] (1,1) -- (0,2) -- (-1,3);
\draw[blue, thick] (1,1) -- ++(1,0) (1,0) -- ++(1,0) (1,-1) -- ++(1,0);
\draw[blue, thick, fill] (1,-1) -- (0,-2) circle (2pt) -- (-3,-2) (0,-2) -- (0,-3) ;
\draw[blue,thick, rounded corners] (-1,0) -- (-2,-1) -- (-3,-1) ;
\draw[thick, fill=white] (-1,0) circle (2pt) (0,1) circle (2pt) (1,-1) circle (2pt) (1,0) circle (2pt) (1,1) circle (2pt);
\end{tikzpicture}\arrow[d, "\mathrm{(II)}"]\\
\begin{tikzpicture}[baseline=-.5ex,scale=0.5]
\draw[yellow, line cap=round, line width=5, opacity=0.5, line cap=round] 
(1,0.5) -- (1,1.5)
(0.5,-1) --(1.5,-1)
;
\draw[green, line cap=round, line width=5, opacity=0.5] 
(1,0.5) -- (2,0) --(1.5,-1)
(2,0) -- (3,0)
;
\draw[red, thick] 
(2,3) -- (2,-3)
(2,2) -- (0,2) -- (-2,0) -- (-3,0)
(2,-2) -- (0,-2) -- (-2,0)
(0,2) -- (0,-2)
(0,0) -- (2,0)
;
\draw[blue, thick] 
(-3,1) -- (-2,0) to[out=15,in=-105] (0,2) -- (-1,3)
(0,2) -- (1,1.5) -- (2,2) -- (3,2)
(1,3) -- (2,2)
(1,1.5) -- (1,0.5) -- (2,0) -- (3,0)
(1,0.5) -- (0,0) -- (0.5,-1) -- (1.5,-1) -- (2,0)
(0,0) to[out=-120,in=120] (0,-2) -- (-1,-3)
(0,-2) -- (0.5,-1)
(1.5,-1) -- (2,-2) -- (3,-2)
(2,-2) -- (1,-3)
(-3,-1) -- (-2,0)
;
\draw[thick, fill=white] 
(-2,0) circle (2pt) 
(0,2) circle (2pt) 
(0,-2) circle (2pt) 
(0,0) circle (2pt) 
(2,2) circle (2pt) 
(2,0) circle (2pt) 
(2,-2) circle (2pt);
\draw[blue, thick, fill]
(1,0.5) circle (2pt)
(1,1.5) circle (2pt)
(0.5,-1) circle (2pt)
(1.5,-1) circle (2pt)
; 
\end{tikzpicture}
&
\begin{tikzpicture}[baseline=-.5ex,scale=0.5]
\draw[yellow, line cap=round, line width=5, opacity=0.5] 
(-2,2)--(-1,2)--(-1,0)--(0,-1) 
(0,0.5)--(0,1.5);
\draw[green, line cap=round, line width=5, opacity=0.5,] (0,-1)--(1,0)--(2,0) (0,0.5)--(1,0);
\draw[red, thick] (1,3)--(1,-3) (-3,0) -- (-1,2)--(1,2) (-1,2) --(-1,0)--(1,0) (-1,0) to[out=-90,in=180] (1,-2);
\draw[blue, thick] 
(-3,2) -- (-2,2) -- (-2,3) 
(-2,2) -- (-1,2) -- (0,1.5) -- (1,2) -- (0,3)
(2,2)--(1,2) 
(0,1.5)--(0,0.5)--(1,0)--(2,0) 
(0,0.5)--(-1,0) -- (-2,-3)
(0,-3) --(1,-2) -- (2,-2)
(1,-2) -- (-1,0)
(0,-1) -- (1,0)
(-1,2) -- (-3,-2)
;
\draw[blue, thick, fill] 
(-2,2) circle (2pt)
(0,1.5) circle (2pt)
(0,0.5) circle (2pt)
(0,-1) circle (2pt)
;
\draw[thick, fill=white] 
(-1,2) circle (2pt) 
(-1,0) circle (2pt) 
(1,2) circle (2pt) 
(1,0) circle (2pt) 
(1,-2) circle (2pt) 
;
\end{tikzpicture}\arrow[l,"\mathrm{(II)^2}"]
&
\begin{tikzpicture}[baseline=-.5ex,scale=0.5]
\draw[yellow, line cap=round, line width=5, opacity=0.5] (-2,2)--(-1,1) (0.5, 0.75) -- (0.5,0.25);
\draw[green, line cap=round, line width=5, opacity=0.5,] (-1,1)--(0,1)--(0,0) to[out=-30,in=-120] (1,0) --(2,0) (1,0) -- (0.5,0.25);
\draw[red, thick, fill] (1,3)--(1,-3) (-3,0) -- (-1,0) (-1,0) -- (0,-1) (1,1) -- (0,1) (0,0) -- (0,-1);
\draw[red, thick] (-1,0) -- (0,1) -- (0,0) -- (1,0) (0,-1) -- (1,-1);
\draw[blue, thick, fill] (1,1) -- (0.5, 0.75) circle (2pt) -- (0.5,0.25) circle (2pt) -- (1,0) (0.5, 0.75) -- (0,1) (0.5, 0.25) -- (0,0);
\draw[blue, thick] (-1,0) to[out=15, in=-105] (0,1) (0,0) to[out=-120,in=120] (0,-1) (0,-1) to[out=30, in=150] (1,-1) (0,0) to[out=-30, in=-150] (1,0);
\draw[blue, thick, fill] (0,1) -- (-1,1) circle (2pt) -- (-2,2) circle (2pt) -- (-2,3) (-1,0) -- (-1,1) (-2,2) -- (-3,2);
\draw[blue, thick] (1,1) -- (0,2) -- (-1,3);
\draw[blue, thick] (1,1) -- ++(1,0) (1,0) -- ++(1,0) (1,-1) -- ++(1,0);
\draw[blue, thick] (0,-1) -- (-2,-3) (1,-1) -- (-1,-3) (-1,0) -- (-3,-2);
\draw[thick, fill=white] (-1,0) circle (2pt) (0,-1) circle (2pt) (0,0) circle (2pt) (0,1) circle (2pt) (1,-1) circle (2pt) (1,0) circle (2pt) (1,1) circle (2pt);
\end{tikzpicture}\arrow[l,"\mathrm{(I,II)^*}"]&
\begin{tikzpicture}[baseline=-.5ex,scale=0.5]
\draw[yellow, line cap=round, line width=5, opacity=0.5] (-2,2)--(-1,1) (1,-0.5)--(1,0)--(0.67, 0.67);
\draw[green, line cap=round, line width=5, opacity=0.5,] (-1,1)--(0,1)--(1,0)--(2,0);
\draw[red, thick, fill] (1,3)--(1,-3) (-3,0) -- (-1,0) (-1,0) -- (0,-1) (1,1) -- (0,1) (1,-0.5) circle (2pt) -- (0,-1);
\draw[red, thick] (-1,0) -- (0,1) -- (1,0) (0,-1) -- (1,-1);
\draw[blue, thick, fill] (1,1) -- (0.67, 0.67) circle (2pt) -- (0,1) (0.67, 0.67) -- (1,0) (1,0) -- (0,-1);
\draw[blue, thick] (-1,0) to[out=15, in=-105] (0,1) (0,-1) to[out=20, in=160] (1,-1);
\draw[blue, thick, fill] (0,1) -- (-1,1) circle (2pt) -- (-2,2) circle (2pt) -- (-2,3) (-1,0) -- (-1,1) (-2,2) -- (-3,2);
\draw[blue, thick] (1,1) -- (0,2) -- (-1,3);
\draw[blue, thick] (1,1) -- ++(1,0) (1,0) -- ++(1,0) (1,-1) -- ++(1,0);
\draw[blue, thick] (0,-1) -- (-2,-3) (1,-1) -- (-1,-3) (-1,0) -- (-3,-2);
\draw[thick, fill=white] (-1,0) circle (2pt) (0,-1) circle (2pt) (0,1) circle (2pt) (1,-1) circle (2pt) (1,0) circle (2pt) (1,1) circle (2pt);
\end{tikzpicture}\arrow[l,"\mathrm{(II)}"]
\end{tikzcd}
\]

\begin{remark}
The reader should not confuse that even though we call this operation the 
\emph{move}, it does not induce any equivalence on $N$-graphs since it involves 
a mutation $\mutation_\cycle$.
\end{remark}

One important observation is that one can take the move $\mathrm{(Z)}$ instead of the Legendrian mutation $\mutation_\cycle$ on the $\sfY$-like cycle $\cycle$, and after the move, the $\sfY$-like cycle becomes the $\sfY$-like cycle and $\sfI$-cycles become $\sfI$-cycles again.
\begin{remark}
We use an ambiguous terminology `$\sfY$-like cycle' since the global shape of $\cycle$ is unknown.
However, the meaning is obvious and we omit the detail.
\end{remark}

Equipped with the move $\mathrm{(Z)}$ as a (local) mutation, the Legendrian 
Coxeter mutation $\mutation_\ngraph$ can be explicitly performed as follows: we 
will explain only the positive Legendrian Coxeter mutation~$\mutation_\ngraph$.

\pagebreak
\smallskip
\noindent(1) $n=4$. For the mutation $\mutation_1$ on the central cycle $\cycle_1$, we will perform the move $\mathrm{(Z)}$ twice at both six valent vertices.
Then all other cycles become (short) $\sfI$-cycles which can mutate easily.
See Figure~\ref{figure:Legendrian Coxeter mutations for affine D4}.

\begin{figure}[H]
\begin{tikzcd}[column sep=2pc, row sep=2pc]
&\begin{tikzpicture}[baseline=-.5ex,scale=0.5,rotate=90]
\begin{scope}[xshift=-2cm]
\begin{scope}[xshift=-1cm]
\draw[yellow,   line width=5, opacity=0.5] (-1,1)--(0,0) (-1,-1)--(0,0) (1,0)--(0,0);
\draw[green, line cap=round, line width=5, opacity=0.5] (-2,1)-- (-1,1) (-1,-1)--(-1,-2)
(3,1)--(3,2) (3,-1)--(4,-1);
\draw[blue,thick, rounded corners] (-1.75, 3) -- (-1.5, 2.5) -- (-1.25, 3);
\draw[thick] (1, 3) -- (-2,3) to[out=180,in=90] (-3,2) --(-3,-2) to[out=-90,in=180] (-2, -3) -- (4,-3) to[out=0,in=-90] (5,-2) -- (5,2) to[out=90,in=0] (4,3)-- (1,3);
\draw[green, thick] (0,3) -- (0,-3) (0,0) -- (-3,0);
\draw[red, thick, fill] (0,0) -- (-1,1) circle (2pt) -- +(0,2) (-1,1) -- ++(-1,0) circle (2pt) -- +(-1,0) (-2,1) -- +(0,2);
\draw[red,thick, fill] (0,0) -- (-1,-1) circle (2pt) -- +(-2,0) (-1,-1) -- ++(0,-1) circle (2pt) -- +(0,-1) (-1,-2) -- +(-2,0);
\draw[red,thick] (0,0) -- (1,0);
\draw[thick, fill=white] (0,0) circle (2pt);
\end{scope}
\begin{scope}[xshift=1cm,rotate=180]
\draw[yellow,  line width=5, opacity=0.5] (-1,1)--(0,0) (-1,-1)--(0,0) (1,0)--(0,0);
\draw[blue, thick] (0,3) -- (0,-3) (0,0) -- (-3,0);
\draw[red ,thick, fill] (0,0) -- (-1,1) circle (2pt) -- +(0,2) (-1,1) -- ++(-1,0) circle (2pt) -- +(-1,0) (-2,1) -- +(0,2);
\draw[red,thick, fill] (0,0) -- (-1,-1) circle (2pt) -- +(-2,0) (-1,-1) -- ++(0,-1) circle (2pt) -- +(0,-1) (-1,-2) -- +(-2,0);
\draw[red,thick] (0,0) -- (1,0);
\draw[thick, fill=white] (0,0) circle (2pt);
\end{scope}
\end{scope}
\end{tikzpicture}\arrow[dl, bend right, "\mutation_{\nbasis_-}"']
\arrow[dr, bend left, "\mutation_{\nbasis_+}"]\\[-5pc]
\begin{tikzpicture}[baseline=-.5ex,scale=0.5,rotate=90]
\begin{scope}[xshift=-1cm]
\draw[yellow,   line width=5, opacity=0.5] (-1,1)--(0,0) (-1,-1)--(0,0) (1,0)--(0,0);
\draw[green, line cap=round, line width=5, opacity=0.5] (-1,2)-- (-1,1) (-1,-1)--(-2,-1) (3,1)--(4,1) (3,-1)--(3,-2);
\draw[blue,thick, rounded corners] (-1.75, 3) -- (-1.5, 2.5) -- (-1.25, 3);
\draw[thick] (1, 3) -- (-2,3) to[out=180,in=90] (-3,2) --(-3,-2) to[out=-90,in=180] (-2, -3) -- (4,-3) to[out=0,in=-90] (5,-2) -- (5,2) to[out=90,in=0] (4,3)-- (1,3);
\draw[green, thick] (0,3) -- (0,-3) (0,0) -- (-3,0);
\draw[red, thick, fill] (0,0) -- (-1,1) circle (2pt) -- +(0,2) (-1,1) -- +(-2,0) (-1,2) circle(2pt) -- +(-2,0);
\draw[red,thick, fill] (0,0) -- (-1,-1) circle (2pt) -- +(-2,0) (-1,-1) -- +(0,-2) (-2,-1) circle(2pt) -- +(0,-2);
\draw[red,thick] (0,0) -- (1,0);
\draw[thick, fill=white] (0,0) circle (2pt);
\end{scope}
\begin{scope}[xshift=1cm,rotate=180]
\draw[yellow,  line width=5, opacity=0.5] (-1,1)--(0,0) (-1,-1)--(0,0) (1,0)--(0,0);
\draw[blue, thick] (0,3) -- (0,-3) (0,0) -- (-3,0);
\draw[red, thick, fill] (0,0) -- (-1,1) circle (2pt) -- +(0,2) (-1,1) -- +(-2,0) (-1,2) circle(2pt) -- +(-2,0);
\draw[red,thick, fill] (0,0) -- (-1,-1) circle (2pt) -- +(-2,0) (-1,-1) -- +(0,-2) (-2,-1) circle(2pt) -- +(0,-2);
\draw[red,thick] (0,0) -- (1,0);
\draw[thick, fill=white] (0,0) circle (2pt);
\end{scope}
\end{tikzpicture}\arrow[d,"\mutation_{\nbasis_+}"']
& &
\begin{tikzpicture}[baseline=-.5ex,scale=0.5,rotate=90]
\begin{scope}[xshift=-3cm]
\begin{scope}
\draw[green, line width=5, opacity=0.5, line cap=round] 
(1,0.5) -- (1,1.5)
(0.5,-1) --(1.5,-1)
;
\draw[yellow, line width=5, opacity=0.5] 
(1,0.5) -- (2,0) --(1.5,-1)
(2,0) -- (3,0)
;
\draw[blue,thick, rounded corners]
(1/3,3)-- (2/3,8/3) --(1/3,7/3)-- (0,8/3) -- (-1/3,3)
;
\draw[green, thick] 
(2,3) -- (2,-3)
(2,2) -- (0,2) -- (-2,0) -- (-3,0)
(2,-2) -- (0,-2) -- (-2,0)
(0,2) -- (0,-2)
(0,0) -- (2,0)
;
\draw[red, thick] 
(-3,1) -- (-2,0) to[out=15,in=-105] (0,2) -- (-1,3)
(0,2) -- (1,1.5) -- (2,2) -- (3,2)
(1,3) -- (2,2)
(1,1.5) -- (1,0.5) -- (2,0) -- (3,0)
(1,0.5) -- (0,0) -- (0.5,-1) -- (1.5,-1) -- (2,0)
(0,0) to[out=-120,in=120] (0,-2) -- (-1,-3)
(0,-2) -- (0.5,-1)
(1.5,-1) -- (2,-2) -- (3,-2)
(2,-2) -- (1,-3)
(-3,-1) -- (-2,0)
;
\draw[thick, fill=white] 
(-2,0) circle (2pt) 
(0,2) circle (2pt) 
(0,-2) circle (2pt) 
(0,0) circle (2pt) 
(2,2) circle (2pt) 
(2,0) circle (2pt) 
(2,-2) circle (2pt);
\draw[red, thick, fill]
(1,0.5) circle (2pt)
(1,1.5) circle (2pt)
(0.5,-1) circle (2pt)
(1.5,-1) circle (2pt)
;
\draw[thick,]
(3,3) -- (-1,3) to[out=180, in=90] (-3,1) to (-3,-1) to[out=-90,in=180] (-1,-3) to (7,-3) to[out=0, in=-90] (9,-1) to (9,1) to[out=90,in=0] (7,3) to (3,3) 
;
\end{scope}
\begin{scope}[xshift=6cm,rotate=180]
\draw[green, line width=5, opacity=0.5, line cap=round] 
(1,0.5) -- (1,1.5)
(0.5,-1) --(1.5,-1)
;
\draw[yellow, line width=5, opacity=0.5] 
(1,0.5) -- (2,0) --(1.5,-1)
(2,0) -- (3,0)
;
\draw[blue, thick] 
(2,3) -- (2,-3)
(2,2) -- (0,2) -- (-2,0) -- (-3,0)
(2,-2) -- (0,-2) -- (-2,0)
(0,2) -- (0,-2)
(0,0) -- (2,0)
;
\draw[red, thick] 
(-3,1) -- (-2,0) to[out=15,in=-105] (0,2) -- (-1,3)
(0,2) -- (1,1.5) -- (2,2) -- (3,2)
(1,3) -- (2,2)
(1,1.5) -- (1,0.5) -- (2,0) -- (3,0)
(1,0.5) -- (0,0) -- (0.5,-1) -- (1.5,-1) -- (2,0)
(0,0) to[out=-120,in=120] (0,-2) -- (-1,-3)
(0,-2) -- (0.5,-1)
(1.5,-1) -- (2,-2) -- (3,-2)
(2,-2) -- (1,-3)
(-3,-1) -- (-2,0)
;
\draw[thick, fill=white] 
(-2,0) circle (2pt) 
(0,2) circle (2pt) 
(0,-2) circle (2pt) 
(0,0) circle (2pt) 
(2,2) circle (2pt) 
(2,0) circle (2pt) 
(2,-2) circle (2pt);
\draw[red, thick, fill]
(1,0.5) circle (2pt)
(1,1.5) circle (2pt)
(0.5,-1) circle (2pt)
(1.5,-1) circle (2pt)
;
\end{scope}
\end{scope}
\end{tikzpicture}\arrow[d,"\mutation_{\nbasis_-}"]\\
\begin{tikzpicture}[baseline=-.5ex,scale=0.5,rotate=90]
\draw[blue,thick, rounded corners]
(1/3,3)-- (4/3,2) --(1,4/3)-- (2/3,2) -- (-1/3,3)
;
\draw[dashed, rounded corners]
(3,1.75) -- (0.25,1.75) -- (0.25, -1.75) -- (5.75,-1.75) -- (5.75,1.75) -- (3,1.75)
;
\begin{scope}[yscale=-1]
\draw[green, line width=5, opacity=0.5, line cap=round] 
(1,0.5) -- (1,1.5)
(0.5,-1) --(1.5,-1)
;
\draw[yellow, line width=5, opacity=0.5] 
(1,0.5) -- (2,0) --(1.5,-1)
(2,0) -- (3,0)
;
\draw[green, thick] 
(2,3) -- (2,-3)
(2,2) -- (0,2) -- (-2,0) -- (-3,0)
(2,-2) -- (0,-2) -- (-2,0)
(0,2) -- (0,-2)
(0,0) -- (2,0)
;
\draw[red, thick] 
(-3,1) -- (-2,0) to[out=15,in=-105] (0,2) (0,-2) -- (-1,-3)
(0,2) -- (1,1.5) -- (2,2) -- (3,2)
(1,3) -- (2,2)
(1,1.5) -- (1,0.5) -- (2,0) -- (3,0)
(1,0.5) -- (0,0) -- (0.5,-1) -- (1.5,-1) -- (2,0)
(0,0) to[out=-120,in=120] (0,-2) (0,2) -- (-1,3)
(0,-2) -- (0.5,-1)
(1.5,-1) -- (2,-2) -- (3,-2)
(2,-2) -- (1,-3)
(-3,-1) -- (-2,0)
;
\draw[thick, fill=white] 
(-2,0) circle (2pt) 
(0,2) circle (2pt) 
(0,-2) circle (2pt) 
(0,0) circle (2pt) 
(2,2) circle (2pt) 
(2,0) circle (2pt) 
(2,-2) circle (2pt);
\draw[red, thick, fill]
(1,0.5) circle (2pt)
(1,1.5) circle (2pt)
(0.5,-1) circle (2pt)
(1.5,-1) circle (2pt)
;
\draw[thick,]
(3,3) -- (-1,3) to[out=180, in=90] (-3,1) to (-3,-1) to[out=-90,in=180] (-1,-3) to (7,-3) to[out=0, in=-90] (9,-1) to (9,1) to[out=90,in=0] (7,3) to (3,3) 
;
\end{scope}
\begin{scope}[xshift=6cm,rotate=180,yscale=-1]
\draw[green, line width=5, opacity=0.5, line cap=round] 
(1,0.5) -- (1,1.5)
(0.5,-1) --(1.5,-1)
;
\draw[yellow, line width=5, opacity=0.5] 
(1,0.5) -- (2,0) --(1.5,-1)
(2,0) -- (3,0)
;
\draw[blue, thick] 
(2,3) -- (2,-3)
(2,2) -- (0,2) -- (-2,0) -- (-3,0)
(2,-2) -- (0,-2) -- (-2,0)
(0,2) -- (0,-2)
(0,0) -- (2,0)
;
\draw[red, thick] 
(-3,1) -- (-2,0) to[out=15,in=-105] (0,2) (0,-2) -- (-1,-3)
(0,2) -- (1,1.5) -- (2,2) -- (3,2)
(1,3) -- (2,2)
(1,1.5) -- (1,0.5) -- (2,0) -- (3,0)
(1,0.5) -- (0,0) -- (0.5,-1) -- (1.5,-1) -- (2,0)
(0,0) to[out=-120,in=120] (0,-2) (0,2) -- (-1,3)
(0,-2) -- (0.5,-1)
(1.5,-1) -- (2,-2) -- (3,-2)
(2,-2) -- (1,-3)
(-3,-1) -- (-2,0)
;
\draw[thick, fill=white] 
(-2,0) circle (2pt) 
(0,2) circle (2pt) 
(0,-2) circle (2pt) 
(0,0) circle (2pt) 
(2,2) circle (2pt) 
(2,0) circle (2pt) 
(2,-2) circle (2pt);
\draw[red, thick, fill]
(1,0.5) circle (2pt)
(1,1.5) circle (2pt)
(0.5,-1) circle (2pt)
(1.5,-1) circle (2pt)
;
\end{scope}
\end{tikzpicture}
& &
\begin{tikzpicture}[baseline=-.5ex,scale=0.5,rotate=90]
\draw[blue,thick, rounded corners]
(1/3,3)-- (4/3,2) --(1,4/3)-- (2/3,2) -- (-1/3,3)
;
\draw[dashed, rounded corners]
(3,1.75) -- (0.25,1.75) -- (0.25, -1.75) -- (5.75,-1.75) -- (5.75,1.75) -- (3,1.75)
;
\begin{scope}[yscale=-1]
\draw[green, line width=5, opacity=0.5, line cap=round] 
(1,0.5) -- (1,1.5)
(0.5,-1) --(1.5,-1)
;
\draw[yellow, line width=5, opacity=0.5] 
(1,0.5) -- (2,0) --(1.5,-1)
(2,0) -- (3,0)
;
\draw[green, thick] 
(2,3) -- (2,-3)
(2,2) -- (0,2) -- (-2,0) -- (-3,0)
(2,-2) -- (0,-2) -- (-2,0)
(0,2) -- (0,-2)
(0,0) -- (2,0)
;
\draw[red, thick] 
(-3,1) -- (-2,0) to[out=-15,in=105] (0,-2) -- (-1,-3)
(0,2) -- (1,1.5) -- (2,2) -- (3,2)
(1,3) -- (2,2)
(1,1.5) -- (1,0.5) -- (2,0) -- (3,0)
(1,0.5) -- (0,0) -- (0.5,-1) -- (1.5,-1) -- (2,0)
(0,0) to[out=120,in=-120] (0,2) -- (-1,3)
(0,-2) -- (0.5,-1)
(1.5,-1) -- (2,-2) -- (3,-2)
(2,-2) -- (1,-3)
(-3,-1) -- (-2,0)
;
\draw[thick, fill=white] 
(-2,0) circle (2pt) 
(0,2) circle (2pt) 
(0,-2) circle (2pt) 
(0,0) circle (2pt) 
(2,2) circle (2pt) 
(2,0) circle (2pt) 
(2,-2) circle (2pt);
\draw[red, thick, fill]
(1,0.5) circle (2pt)
(1,1.5) circle (2pt)
(0.5,-1) circle (2pt)
(1.5,-1) circle (2pt)
;
\draw[thick,]
(3,3) -- (-1,3) to[out=180, in=90] (-3,1) to (-3,-1) to[out=-90,in=180] (-1,-3) to (7,-3) to[out=0, in=-90] (9,-1) to (9,1) to[out=90,in=0] (7,3) to (3,3) 
;
\end{scope}
\begin{scope}[xshift=6cm,rotate=180,yscale=-1]
\draw[green, line width=5, opacity=0.5, line cap=round] 
(1,0.5) -- (1,1.5)
(0.5,-1) --(1.5,-1)
;
\draw[yellow, line width=5, opacity=0.5] 
(1,0.5) -- (2,0) --(1.5,-1)
(2,0) -- (3,0)
;
\draw[blue, thick] 
(2,3) -- (2,-3)
(2,2) -- (0,2) -- (-2,0) -- (-3,0)
(2,-2) -- (0,-2) -- (-2,0)
(0,2) -- (0,-2)
(0,0) -- (2,0)
;
\draw[red, thick] 
(-3,1) -- (-2,0) to[out=-15,in=105] (0,-2) -- (-1,-3)
(0,2) -- (1,1.5) -- (2,2) -- (3,2)
(1,3) -- (2,2)
(1,1.5) -- (1,0.5) -- (2,0) -- (3,0)
(1,0.5) -- (0,0) -- (0.5,-1) -- (1.5,-1) -- (2,0)
(0,0) to[out=120,in=-120] (0,2) -- (-1,3)
(0,-2) -- (0.5,-1)
(1.5,-1) -- (2,-2) -- (3,-2)
(2,-2) -- (1,-3)
(-3,-1) -- (-2,0)
;
\draw[thick, fill=white] 
(-2,0) circle (2pt) 
(0,2) circle (2pt) 
(0,-2) circle (2pt) 
(0,0) circle (2pt) 
(2,2) circle (2pt) 
(2,0) circle (2pt) 
(2,-2) circle (2pt);
\draw[red, thick, fill]
(1,0.5) circle (2pt)
(1,1.5) circle (2pt)
(0.5,-1) circle (2pt)
(1.5,-1) circle (2pt)
;
\end{scope}
\end{tikzpicture}
\end{tikzcd}
\caption{Legendrian Coxeter mutations for $(\ngraph({\exdynD}_4),\nbasis({\exdynD}_4))$.}
\label{figure:Legendrian Coxeter mutations for affine D4}
\end{figure}

\pagebreak
\smallskip
\noindent(2) $n=5$. As before, we perform the move $\mathrm{(Z)}$ near the cycle $\gamma_1$ instead of the mutation $\mutation_1$.
Then two adjacent cycles $\gamma_2$ and $\gamma_3$ become short $\sfI$-cycles, and so do two cycles $\gamma_5$ and $\gamma_6$ in other side.
After the mutations $\mutation_5$ and $\mutation_6$, the move $\mathrm{(Z)}$ near the cycle $\gamma_4$ is still applicable. Since the last move preserves short $\sfI$-cycles $\gamma_2$ and $\gamma_3$, one can easily take mutations there.
See Figure~\ref{figure:Legendrian Coxeter mutations for affine D5}.

\begin{figure}[H]
\begin{tikzcd}[column sep=2pc, row sep=2pc]
&\begin{tikzpicture}[baseline=-.5ex,scale=0.5,rotate=90]
\begin{scope}[xshift=-3cm]
\begin{scope}[xshift=-1cm]
\draw[yellow, line width=5,opacity=0.5] (-1,1)--(0,0) (-1,-1)--(0,0) (1,0)--(0,0);
\draw[green, line width=5,opacity=0.5] (-2,1)--(-1,1) (-1,-1)--(-1,-2) (1,0)--(2,0)--(3,1) (2,0)--(3,-1);
\draw[blue,thick, rounded corners] (-1.75, 3) -- (-1.5, 2.5) -- (-1.25, 3);
\draw[thick] (1, 3) -- (-2,3) to[out=180,in=90] (-3,2) --(-3,-2) to[out=-90,in=180] (-2, -3) -- (4,-3) to[out=0,in=-90] (5,-2) -- (5,2) to[out=90,in=0] (4,3)-- (1,3);
\draw[green, thick] (0,3) -- (0,-3) (0,0) -- (-3,0);
\draw[red,thick, fill] (0,0) -- (-1,1) circle (2pt) -- +(0,2) (-1,1) -- ++(-1,0) circle (2pt) -- +(-1,0) (-2,1) -- +(0,2);
\draw[red,thick, fill] (0,0) -- (-1,-1) circle (2pt) -- +(-2,0) (-1,-1) -- ++(0,-1) circle (2pt) -- +(0,-1) (-1,-2) -- +(-2,0);
\draw[red,thick] (0,0) -- (1,0);
\draw[thick, fill=white] (0,0) circle (2pt);
\end{scope}
\draw[red, thick, fill] (0,0) circle (2pt) -- (0,-3);
\begin{scope}[xshift=1cm, yscale=-1,rotate=180]
\draw[yellow, line width=5, opacity=0.5] (-1,-1) -- (-1,-2) (-1,1) -- (-2,1);
\draw[blue, thick] (0,3) -- (0,-3) (0,0) -- (-3,0);
\draw[red,thick, fill] (0,0) -- (-1,1) circle (2pt) -- +(0,2) (-1,1) -- ++(-1,0) circle (2pt) -- +(-1,0) (-2,1) -- +(0,2);
\draw[red,thick, fill] (0,0) -- (-1,-1) circle (2pt) -- +(-2,0) (-1,-1) -- ++(0,-1) circle (2pt) -- +(0,-1) (-1,-2) -- +(-2,0);
\draw[red,thick] (0,0) -- (1,0);
\draw[thick, fill=white] (0,0) circle (2pt);
\end{scope}
\end{scope}
\end{tikzpicture}
\arrow[dr,bend left,"\mu_{\nbasis_+}"]
\arrow[dl,bend right,"\mu_{\nbasis_-}"'] 
\\[-5pc]
\begin{tikzpicture}[baseline=-.5ex,xscale=-0.5,yscale=0.5,rotate=-90]
\begin{scope}[xshift=-4cm]
\begin{scope}
\draw[yellow, line width=5, opacity=0.5, line cap=round] 
(1,0.5) -- (1,1.5)
(0.5,-1) --(1.5,-1)
;
\draw[green, line width=5, opacity=0.5] 
(1,0.5) -- (2,0) --(1.5,-1)
(2,0) -- (3,0)
;
\draw[blue,thick, rounded corners] (5.75, 3) -- (5.5, 2.5) -- (5.25, 3);
\draw[blue, thick] 
(2,3) -- (2,-3)
(2,2) -- (0,2) -- (-2,0) -- (-3,0)
(2,-2) -- (0,-2) -- (-2,0)
(0,2) -- (0,-2)
(0,0) -- (2,0)
;
\draw[red, thick] 
(-3,1) -- (-2,0) to[out=15,in=-105] (0,2) -- (-1,3)
(0,2) -- (1,1.5) -- (2,2)
(1,3) -- (2,2)
(1,1.5) -- (1,0.5) -- (2,0) -- (3,0)
(1,0.5) -- (0,0) -- (0.5,-1) -- (1.5,-1) -- (2,0)
(0,0) to[out=-120,in=120] (0,-2) -- (-1,-3)
(0,-2) -- (0.5,-1)
(1.5,-1) -- (2,-2)
(2,-2) -- (1,-3)
(-3,-1) -- (-2,0)
;
\draw[red, thick, rounded corners]
(2,2) -- (3,2) -- (3,0)
(2,-2) -- (3,-2) -- (3,-3)
;
\draw[thick, fill=white] 
(-2,0) circle (2pt) 
(0,2) circle (2pt) 
(0,-2) circle (2pt) 
(0,0) circle (2pt) 
(2,2) circle (2pt) 
(2,0) circle (2pt) 
(2,-2) circle (2pt);
\draw[red, thick, fill]
(1,0.5) circle (2pt)
(1,1.5) circle (2pt)
(0.5,-1) circle (2pt)
(1.5,-1) circle (2pt)
(3,0) circle (2pt)
;
\draw[thick]
(3,3) -- (-2,3) to[out=180, in=90] (-3,2) to (-3,-2) to[out=-90,in=180] (-2,-3) to (6,-3) to[out=0, in=-90] (7,-2) to (7,2) to[out=90,in=0] (6,3) to (3,3) 
;
\end{scope}
\draw[yellow, opacity=0.5, line width=5] (3,0)--(4,0)--(5,1) (4,0)--(5,-1);
\begin{scope}[xshift=4cm,rotate=180]
\draw[green, line width=5, opacity=0.5] (-1,-1) -- (-1,-2) (-1,1) -- (-2,1);
\draw[green, thick] (0,3) -- (0,-3) (0,0) -- (-3,0);
\draw[red,thick, fill] (0,0) -- (-1,1) circle (2pt) -- +(0,2) (-1,1) -- ++(-1,0) circle (2pt) -- +(-1,0) (-2,1) -- +(0,2);
\draw[red,thick, fill] (0,0) -- (-1,-1) circle (2pt) -- +(-2,0) (-1,-1) -- ++(0,-1) circle (2pt) -- +(0,-1) (-1,-2) -- +(-2,0);
\draw[red,thick] (0,0) -- (1,0);
\draw[thick, fill=white] (0,0) circle (2pt);
\end{scope}
\end{scope}
\end{tikzpicture}
\arrow[d,"\mu_{\nbasis_+}"'] & &
\begin{tikzpicture}[baseline=-.5ex,scale=0.5,rotate=90]
\begin{scope}
\draw[green, line width=5, opacity=0.5, line cap=round] 
(1,0.5) -- (1,1.5)
(0.5,-1) --(1.5,-1)
;
\draw[yellow, line width=5, opacity=0.5] 
(1,0.5) -- (2,0) --(1.5,-1)
(2,0) -- (3,0)
;
\draw[blue,thick, rounded corners]
(1/3,3)-- (2/3,8/3) --(1/3,7/3)-- (0,8/3) -- (-1/3,3)
;
\draw[green, thick] 
(2,3) -- (2,-3)
(2,2) -- (0,2) -- (-2,0) -- (-3,0)
(2,-2) -- (0,-2) -- (-2,0)
(0,2) -- (0,-2)
(0,0) -- (2,0)
;
\draw[red, thick] 
(-3,1) -- (-2,0) to[out=15,in=-105] (0,2) -- (-1,3)
(0,2) -- (1,1.5) -- (2,2)
(1,3) -- (2,2)
(1,1.5) -- (1,0.5) -- (2,0) -- (3,0)
(1,0.5) -- (0,0) -- (0.5,-1) -- (1.5,-1) -- (2,0)
(0,0) to[out=-120,in=120] (0,-2) -- (-1,-3)
(0,-2) -- (0.5,-1)
(1.5,-1) -- (2,-2)
(2,-2) -- (1,-3)
(-3,-1) -- (-2,0)
;
\draw[red, thick, rounded corners]
(2,2) -- (3,2) -- (3,0)
(2,-2) -- (3,-2) -- (3,-3)
;
\draw[thick, fill=white] 
(-2,0) circle (2pt) 
(0,2) circle (2pt) 
(0,-2) circle (2pt) 
(0,0) circle (2pt) 
(2,2) circle (2pt) 
(2,0) circle (2pt) 
(2,-2) circle (2pt);
\draw[red, thick, fill]
(1,0.5) circle (2pt)
(1,1.5) circle (2pt)
(0.5,-1) circle (2pt)
(1.5,-1) circle (2pt)
(3,0) circle (2pt)
;
\draw[thick]
(3,3) -- (-2,3) to[out=180, in=90] (-3,2) to (-3,-2) to[out=-90,in=180] (-2,-3) to (6,-3) to[out=0, in=-90] (7,-2) to (7,2) to[out=90,in=0] (6,3) to (3,3) 
;
\end{scope}
\draw[green, opacity=0.5, line width=5] (3,0)--(4,0)--(5,1) (4,0)--(5,-1);
\begin{scope}[xshift=4cm,rotate=180]
\draw[yellow, line width=5, opacity=0.5] (-1,-1) -- (-1,-2) (-1,1) -- (-2,1);
\draw[blue, thick] (0,3) -- (0,-3) (0,0) -- (-3,0);
\draw[red,thick, fill] (0,0) -- (-1,1) circle (2pt) -- +(0,2) (-1,1) -- ++(-1,0) circle (2pt) -- +(-1,0) (-2,1) -- +(0,2);
\draw[red,thick, fill] (0,0) -- (-1,-1) circle (2pt) -- +(-2,0) (-1,-1) -- ++(0,-1) circle (2pt) -- +(0,-1) (-1,-2) -- +(-2,0);
\draw[red,thick] (0,0) -- (1,0);
\draw[thick, fill=white] (0,0) circle (2pt);
\end{scope}
\end{tikzpicture}
\arrow[d,"\mu_{\nbasis_-}"]\\
\begin{tikzpicture}[baseline=-.5ex,xscale=-0.5,yscale=0.5,rotate=-90]
\begin{scope}[xshift=-7cm]
\draw[blue,thick, rounded corners]
(7-1/3,3)-- (7-4/3,2) --(7-1,4/3)-- (7-2/3,2) -- (7+1/3,3)
;
\draw[dashed, rounded corners]
(3,1.75) -- (0.25,1.75) -- (0.25, -1.75) -- (6.75,-1.75) -- (6.75,1.75) -- (3,1.75)
;
\begin{scope}[yscale=-1]
\draw[yellow, line width=5, opacity=0.5, line cap=round] 
(1,0.5) -- (1,1.5)
(0.5,-1) --(1.5,-1)
;
\draw[green, line width=5, opacity=0.5] 
(1,0.5) -- (2,0) --(1.5,-1)
(2,0) -- (4,0)
;
\draw[blue, thick] 
(2,3) -- (2,-3)
(2,2) -- (0,2) -- (-2,0) -- (-3,0)
(2,-2) -- (0,-2) -- (-2,0)
(0,2) -- (0,-2)
(0,0) -- (2,0)
;
\draw[red, thick] 
(-3,1) -- (-2,0) to[out=-15,in=105] (0,-2) -- (-1,-3)
(0,2) -- (1,1.5) -- (2,2)
(1,3) -- (2,2)
(1,1.5) -- (1,0.5) -- (2,0) -- (4,0)
(1,0.5) -- (0,0) -- (0.5,-1) -- (1.5,-1) -- (2,0)
(0,0) to[out=120,in=-120] (0,2) -- (-1,3)
(0,-2) -- (0.5,-1)
(1.5,-1) -- (2,-2)--(4,-2)
(2,-2) -- (1,-3)
(-3,-1) -- (-2,0)
;
\draw[red, thick, rounded corners]
(2,2) -- (3,2) -- (3,3)
;
\draw[thick, fill=white] 
(-2,0) circle (2pt) 
(0,2) circle (2pt) 
(0,-2) circle (2pt) 
(0,0) circle (2pt) 
(2,2) circle (2pt) 
(2,0) circle (2pt) 
(2,-2) circle (2pt);
\draw[red, thick, fill]
(1,0.5) circle (2pt)
(1,1.5) circle (2pt)
(0.5,-1) circle (2pt)
(1.5,-1) circle (2pt)
;
\draw[thick,]
(3,3) -- (-1,3) to[out=180, in=90] (-3,1) to (-3,-1) to[out=-90,in=180] (-1,-3) to (8,-3) to[out=0, in=-90] (10,-1) to (10,1) to[out=90,in=0] (8,3) to (3,3) 
;
\end{scope}

\begin{scope}[xshift=7cm,rotate=180]
\draw[green, line width=5, opacity=0.5, line cap=round] 
(1,0.5) -- (1,1.5)
(0.5,-1) --(1.5,-1)
;
\draw[yellow, line width=5, opacity=0.5] 
(1,0.5) -- (2,0) --(1.5,-1)
(2,0) -- (3,0)
;
\draw[green, thick] 
(2,3) -- (2,-3)
(2,2) -- (0,2) -- (-2,0) -- (-3,0)
(2,-2) -- (0,-2) -- (-2,0)
(0,2) -- (0,-2)
(0,0) -- (2,0)
;
\draw[red, thick] 
(-3,1) -- (-2,0) to[out=15,in=-105] (0,2) -- (-1,3)
(0,2) -- (1,1.5) -- (2,2)
(1,3) -- (2,2)
(1,1.5) -- (1,0.5) -- (2,0) -- (3,0)
(1,0.5) -- (0,0) -- (0.5,-1) -- (1.5,-1) -- (2,0)
(0,0) to[out=-120,in=120] (0,-2) -- (-1,-3)
(0,-2) -- (0.5,-1)
(1.5,-1) -- (2,-2) -- (3,-2)
(2,-2) -- (1,-3)
(-3,-1) -- (-2,0);
\draw[red, thick, rounded corners]
(2,2) -- (3,2) -- (3,0)
;
\draw[thick, fill=white] 
(-2,0) circle (2pt) 
(0,2) circle (2pt) 
(0,-2) circle (2pt) 
(0,0) circle (2pt) 
(2,2) circle (2pt) 
(2,0) circle (2pt) 
(2,-2) circle (2pt);
\draw[red, thick, fill]
(1,0.5) circle (2pt)
(1,1.5) circle (2pt)
(0.5,-1) circle (2pt)
(1.5,-1) circle (2pt)
(3,0) circle (2pt)
;
\end{scope}
\end{scope}
\end{tikzpicture} 
& & \begin{tikzpicture}[baseline=-.5ex,scale=0.5,rotate=90]
\draw[blue,thick, rounded corners]
(1/3,3)-- (4/3,2) --(1,4/3)-- (2/3,2) -- (-1/3,3)
;
\draw[dashed, rounded corners]
(3,1.75) -- (0.25,1.75) -- (0.25, -1.75) -- (6.75,-1.75) -- (6.75,1.75) -- (3,1.75)
;
\begin{scope}[yscale=-1]
\draw[green, line width=5, opacity=0.5, line cap=round] 
(1,0.5) -- (1,1.5)
(0.5,-1) --(1.5,-1)
;
\draw[yellow, line width=5, opacity=0.5] 
(1,0.5) -- (2,0) --(1.5,-1)
(2,0) -- (4,0)
;
\draw[green, thick] 
(2,3) -- (2,-3)
(2,2) -- (0,2) -- (-2,0) -- (-3,0)
(2,-2) -- (0,-2) -- (-2,0)
(0,2) -- (0,-2)
(0,0) -- (2,0)
;
\draw[red, thick] 
(-3,1) -- (-2,0) to[out=-15,in=105] (0,-2) -- (-1,-3)
(0,2) -- (1,1.5) -- (2,2)
(1,3) -- (2,2)
(1,1.5) -- (1,0.5) -- (2,0) -- (4,0)
(1,0.5) -- (0,0) -- (0.5,-1) -- (1.5,-1) -- (2,0)
(0,0) to[out=120,in=-120] (0,2) -- (-1,3)
(0,-2) -- (0.5,-1)
(1.5,-1) -- (2,-2)--(4,-2)
(2,-2) -- (1,-3)
(-3,-1) -- (-2,0)
;
\draw[red, thick, rounded corners]
(2,2) -- (3,2) -- (3,3)
;
\draw[thick, fill=white] 
(-2,0) circle (2pt) 
(0,2) circle (2pt) 
(0,-2) circle (2pt) 
(0,0) circle (2pt) 
(2,2) circle (2pt) 
(2,0) circle (2pt) 
(2,-2) circle (2pt);
\draw[red, thick, fill]
(1,0.5) circle (2pt)
(1,1.5) circle (2pt)
(0.5,-1) circle (2pt)
(1.5,-1) circle (2pt)
;
\draw[thick,]
(3,3) -- (-1,3) to[out=180, in=90] (-3,1) to (-3,-1) to[out=-90,in=180] (-1,-3) to (8,-3) to[out=0, in=-90] (10,-1) to (10,1) to[out=90,in=0] (8,3) to (3,3) 
;
\end{scope}

\begin{scope}[xshift=7cm,rotate=180]
\draw[yellow, line width=5, opacity=0.5, line cap=round] 
(1,0.5) -- (1,1.5)
(0.5,-1) --(1.5,-1)
;
\draw[green, line width=5, opacity=0.5] 
(1,0.5) -- (2,0) --(1.5,-1)
(2,0) -- (3,0)
;
\draw[blue, thick] 
(2,3) -- (2,-3)
(2,2) -- (0,2) -- (-2,0) -- (-3,0)
(2,-2) -- (0,-2) -- (-2,0)
(0,2) -- (0,-2)
(0,0) -- (2,0)
;
\draw[red, thick] 
(-3,1) -- (-2,0) to[out=15,in=-105] (0,2) -- (-1,3)
(0,2) -- (1,1.5) -- (2,2)
(1,3) -- (2,2)
(1,1.5) -- (1,0.5) -- (2,0) -- (3,0)
(1,0.5) -- (0,0) -- (0.5,-1) -- (1.5,-1) -- (2,0)
(0,0) to[out=-120,in=120] (0,-2) -- (-1,-3)
(0,-2) -- (0.5,-1)
(1.5,-1) -- (2,-2) -- (3,-2)
(2,-2) -- (1,-3)
(-3,-1) -- (-2,0);
\draw[red, thick, rounded corners]
(2,2) -- (3,2) -- (3,0)
;
\draw[thick, fill=white] 
(-2,0) circle (2pt) 
(0,2) circle (2pt) 
(0,-2) circle (2pt) 
(0,0) circle (2pt) 
(2,2) circle (2pt) 
(2,0) circle (2pt) 
(2,-2) circle (2pt);
\draw[red, thick, fill]
(1,0.5) circle (2pt)
(1,1.5) circle (2pt)
(0.5,-1) circle (2pt)
(1.5,-1) circle (2pt)
(3,0) circle (2pt)
;
\end{scope}
\end{tikzpicture}
\end{tikzcd}
\caption{Legendrian Coxeter mutations for $(\ngraph({\exdynD}_5,\nbasis({\exdynD}_5))$.}
\label{figure:Legendrian Coxeter mutations for affine D5}
\end{figure}

\smallskip
\noindent(3) $n\ge 6$. All other cases are essentially the same as above. More precisely, two $\mathrm{(Z)}$-moves happen simultaneously or sequentially according to the parity of $n$.
Since the move $\mathrm{(Z)}$ preserves the types of cycles such as $\sfI$ and $\sfY$, there are no obstructions to take mutations. See Figures~\ref{figure:Legendrian Coxeter mutations for affine D even} and \ref{figure:Legendrian Coxeter mutations for affine D odd}.

\begin{figure}[H]
\begin{tikzcd}[row sep=0.5pc]

\end{tikzcd}
\caption{Legendrian Coxeter mutation $\mutation_\ngraph^{\pm1}$ for $(\ngraph({\exdynD}_{2\ell+4},\nbasis({\exdynD}_{2\ell+4}))$.}
\label{figure:Legendrian Coxeter mutations for affine D even}
\end{figure}

\begin{figure}[H]
\begin{tikzcd}[row sep=0.5pc]
%
\end{tikzcd}
\caption{Legendrian Coxeter mutation $\mutation_\ngraph^{\pm1}$ for $(\ngraph({\exdynD}_{2\ell+5},\nbasis({\exdynD}_{2\ell+5}))$.}
\label{figure:Legendrian Coxeter mutations for affine D odd}
\end{figure}

\end{document}